\documentclass[11pt, reqno]{amsart}
\usepackage{fullpage}

\newif\ifArxiv
\Arxivtrue

 \newif\ifHideFoot
\HideFoottrue

\ifHideFoot
\else
\usepackage[notref,notcite]{showkeys}
\fi 

\usepackage{amssymb,amsmath,amsthm,amscd,mathrsfs,graphicx, color}
\usepackage[cmtip,all,matrix,arrow,tips,curve]{xy}
 \usepackage{hyperref}
 \usepackage{comment}
\usepackage[usenames,dvipsnames]{xcolor}
\usepackage{xypic}
\hypersetup{colorlinks=true,citecolor=ForestGreen,linkcolor=Maroon,urlcolor=NavyBlue}
 \usepackage{mathtools}
\usepackage{mathpazo}
\usepackage[normalem]{ulem}
\usepackage{thmtools}
\usepackage{cleveref}
\usepackage{enumitem}
\usepackage{tikz-cd}
  
\numberwithin{equation}{section}
\newtheorem{teo}{Theorem}[section]
\newtheorem{pro}[teo]{Proposition}
\newtheorem{lem}[teo]{Lemma}
\newtheorem{cor}[teo]{Corollary}

\newtheorem{teoalpha}{Theorem}

\newtheorem{coralpha}[teoalpha]{Corollary}

\theoremstyle{definition}
\newtheorem{exa}[teo]{Example}
\newtheorem{dfn}[teo]{Definition}
\theoremstyle{remark}
\newtheorem{rem}[teo]{Remark}

\ifHideFoot

\newcommand{\Yano}[1]{}
\newcommand{\Shend}[1]{}

\else 

\newcommand{\marg}[1]{\normalsize{{
\color{red}\footnote{{\color{blue}#1}}}{\marginpar[\vskip
-.25cm{\color{red}\hfill\tiny\thefootnote$\implies$}]{\vskip
-.2cm{\color{red}$\impliedby$\tiny\thefootnote}}}}}
\newcommand{\Yano}[1]{\marg{(Yano) #1}}
\newcommand{\Shend}[1]{\marg{(Shend) #1}}

\fi

\newcommand{\m}[1]{\mathcal{#1}}

\newcommand{\X}{\mathcal{X}}
\newcommand{\XX}{\mathcal{X}}
\newcommand{\Z}{\mathcal{Z}}
\newcommand{\Y}{\mathcal{Y}}

\newcommand{\D}{\mathcal{D}}

\newcommand{\C}{\mathbb{C}}
\newcommand{\Q}{\mathbb{Q}}

\newcommand{\ra}{\rightarrow}

\title[Moduli of KSBA stable pairs]{Moduli spaces of snc klt KSBA stable pairs are naturally of log general type}

\author{Sebastian Casalaina-Martin}
\address{University of Colorado, Department of Mathematics, 
Boulder, CO 80309, USA }
\email{casa@math.colorado.edu}

\author{Shend Zhjeqi}
\address{University of Michigan, Department of Mathematics, 
Ann Arbor, MI 48109, USA }
\email{shendzh@umich.edu}

\thanks{Research of the first named author is supported in part by a grant from the Simons Foundation (SFI-MPS-TSM-00013682). The second named author was partially supported by the Simons Collaboration grant Moduli of Varieties.}

\date{\today}

\begin{document}

\begin{abstract}
We generalize work of Wei--Wu on Viehweg hyperbolicity for families of stable pairs over smooth projective varieties to the case of smooth Deligne--Mumford stacks.  The results of Wei--Wu build on results of Popa--Schnell and Viehweg--Zuo, utilize a result of Campana–P\u{a}un, and extend  results of Kebekus--Kov\'acs and Patakfalvi.  We apply our results to smooth proper moduli stacks parameterizing generically klt relatively simple normal crossings KSBA stable pairs, showing that such moduli stacks naturally have big log canonical bundles. We also consider implications for their coarse moduli spaces.
\end{abstract}

\maketitle

\ifHideFoot
\setcounter{page}{1}
\fi

\section*{Introduction}

Recent work of a number of authors \cite{VZ03, KK08, Pat12, CPFol19, PS17, WW23} has established Viehweg hyperbolicity for families of smooth projective varieties of general type, as well as for families of KSBA stable pairs.  
These results build on strategies developed in \cite{Zuo00negativity, VZ01isotriv, VZ02, VZ03},   as well as positivity  results for variations of Hodge structures  due to Griffiths and  \cite{fujita_fiber_spaces, Kawamata, FF14variationsHMS, FFS14Remarks, Zuo00negativity,PW16}, and  for push forwards of relative canonical bundles due to \cite{kollar_sub, EV90effective, KP17proj}.
Specifically, in the case of KSBA stable pairs,   \cite[Thm.~A]{WW23} implies that if $X$ is a smooth projective variety and  
 $f:( Y, D)\to X$ is a family of  KSBA stable pairs of maximal variation,  
  then for any reduced divisor  $\Delta\subseteq X$ such that
  over the complement $X- \Delta$ the morphism $f$ is relative simple normal crossings (see \Cref{D:RelsncDef}) and the fibers are klt pairs,   
one has that  $K_{X}+\Delta$ is big.

While one can apply these results to coarse moduli spaces of smooth moduli stacks of KSBA stable pairs by ignoring loci with extra automorphisms,  our goal in this paper is to generalize the results to smooth proper DM stacks with projective coarse moduli spaces, so that we can apply the results directly to the moduli stacks.  This seems more natural in the context of moduli spaces, and can give sharper statements when applied to the coarse moduli space of the stack.  We discuss this comparison further below.
This paper is the fourth in a series with the aim of generalizing results of \cite{PS17, WW23} on Viehweg hyperbolicity  to
the case of Deligne–Mumford stacks; the previous articles in this series are 
\cite{CMZpositivity, CMZslope_stability, CMZfoliations}. In forthcoming work we will address the results of \cite{PS17}.

Our main result in this paper is the following generalization of  the result of Wei--Wu on KSBA stable pairs mentioned above to the case of smooth proper DM stacks:

\begin{teoalpha}[Viehweg hyperbolicity for DM stacks of stable pairs]
 \label{T:main-pairs}  
 Let $\mathcal X$ be a smooth proper integral DM stack over $\mathbb C$ with projective coarse moduli space, and let $f:(\mathcal Y,\mathcal D)\to \mathcal X$ be a family of KSBA stable pairs having maximal variation, in the sense that $\operatorname{Var}(f)=\dim \mathcal X$ (see \S \ref{S:variation}).
  Then for any reduced divisor   $\mathbf \Delta\subseteq \mathcal X$ such that
  over the complement $\mathcal X-\mathbf \Delta$ the morphism $f$ is relative simple normal crossings (see \Cref{D:RelsncDef}) and the fibers are klt pairs,   
 one has that  $K_{\mathcal X}+\mathbf \Delta$ is big. 
\end{teoalpha}

The proof of \Cref{T:main-pairs} is given in \S \ref{S:ProofMain}. 
We discuss the notion of big line bundles on stacks in \cite[Def.~2.4]{CMZpositivity};  in short, a line bundle on the stack is said to be big if a positive  tensor power  descends to a big line bundle on the coarse moduli space.  
Maximal variation of $f$ is  the statement that the morphism from $\mathcal X$ to the moduli stack of KSBA stable pairs is generically finite (see \S \ref{S:variation}).
The condition that $f$ be relative simple normal crossings 
over the complement $\mathcal U=\mathcal X-\mathbf \Delta$  means that the family $f|_{\mathcal U}:f^{-1}(\mathcal U)\to \mathcal U$ is smooth, the divisor $\mathcal D|_{f^{-1}(\mathcal U)}$ is an snc $\mathbb Q$-divisor on $f^{-1}(\mathcal U)$, and every stratum of $\mathcal D|_{f^{-1}(\mathcal U)}$ is smooth over $\mathcal U$ (see \Cref{D:RelsncDef}).  This condition is called log smooth in \cite{WW23} (although it is a special type of log smooth morphism; see also \Cref{R:LogSm}).  

The case of \Cref{T:main-pairs} 
where $\mathcal X=X$ is a smooth projective variety follows from \cite[Thm.~A]{WW23}; in fact the case 
where $\mathcal X=[V/G]$ is the quotient of a smooth projective variety $V$ by a finite group $G$ follows directly from \cite[Thm.~A]{WW23};  the content of \Cref{T:main-pairs}  is to extend this result to more general smooth Deligne--Mumford stacks.  In \cite{WW23}, they also consider the case where $\mathcal X=X$ is assumed to be a smooth \emph{quasi}-projective variety, and weaken the condition on the generic fibers from being klt to being lc. 
  They also consider the case where the fibers are only assumed to be of log general type (rather than KSBA stable), but they then require the family to have a relative good minimal model.  Elsewhere we will consider the natural extension of these Wei--Wu results to the case where $\mathcal X$ is assumed to be a smooth separated DM stack of finite type over $\mathbb C$ with quasi-projective coarse moduli space.

One might be tempted to weaken the snc conditions on $f$ outside of $\mathbf \Delta$, for instance,  by weakening the relative snc condition described above with the condition that the \emph{fibers} be smooth with snc boundary divisor; however, there is one key step of the proof of \Cref{T:main-pairs}, namely  \Cref{L:coherence}, which relies on a result of Wei--Wu, \cite[Lem.~6.2]{WW23}, and that result uses  the stronger relative snc hypotheses in its proof (see \Cref{R:relsnc}).  Finally, one cannot easily weaken the hypothesis that $f$ be smooth outside of $\mathbf \Delta$, as is shown by the example of the moduli space of curves.  
For instance, for the moduli stack of marked stable  curves $\overline {\mathcal M}_{1,1}$, one cannot weaken the hypothesis to the condition that the fibers of $f$ be normal crossings outside of  $\mathbf \Delta$, as one could then take $\mathbf \Delta$ to be empty, and one can easily see that $K_{\overline {\mathcal M}_{1,1}}$ is not big.

We apply \Cref{T:main-pairs}  to integral components of moduli spaces of KSBA stable pairs in the following way, where we interpret an integral component of a moduli space quite loosely to mean essentially any family whose base is integral and proper,  and admits a generically finite morphism to one of Koll\'ar's  moduli spaces \cite[Thm.~8.1]{kollar_families}:

\begin{coralpha}[Viehweg hyperbolicity for KSBA moduli spaces]\label{main-pairs-cor}
Let $\overline{\mathcal M}$ be a proper integral DM stack over $\mathbb C$ admitting a projective coarse moduli space, and a generically finite  morphism $\overline {\mathcal M}\to \mathcal {SP}(\mathbf a,n,\nu)$ to a moduli stack of marked KSBA stable pairs  \cite[Thm.~8.1]{kollar_families} (see also \S \ref{S:SGSF}).
Assume that the  relative snc discriminant (see \Cref{dfn:non-relative snc locus})  $\bar{\mathbf \Delta} \subseteq \overline {\mathcal M}$ of the family of KSBA stable pairs corresponding to the morphism $\overline {\mathcal M}\to \mathcal {SP}(\mathbf a,n,\nu)$ is not equal to the whole of $ \overline {\mathcal M}$, and let $(\mathcal X,\mathbf \Delta)\to (\overline {\mathcal M},\bar{\mathbf \Delta})$ be a log resolution of singularities, with $\mathbf \Delta$ the reduced pre-image of $\bar{\mathbf \Delta}$.  If the family  $f:(\mathcal Y,\mathcal D)\to \mathcal X$ of KSBA stable pairs corresponding to the morphism $\mathcal X \to \overline {\mathcal M}\to \mathcal {SP}(\mathbf a,n,\nu)$ is such that the fibers of $f$ over $\mathcal X-\mathbf \Delta$ are klt pairs, then   $K_{\mathcal X}+ {\mathbf \Delta}$ is big.  
 \qed
\end{coralpha}

Our proof of   \Cref{T:main-pairs}  closely follows the strategies developed in \cite{PS17} and \cite{WW23},
which in turn build on the strategies developed by  Viehweg and Zuo in  \cite{Zuo00negativity, VZ01isotriv, VZ02, VZ03}. 
   The first observation is that one can replace $\mathcal X$  by a log resolution (see \Cref{L:sncRdx}), so we may as well assume that the divisor $\mathbf \Delta$ is a normal crossings divisor. 
The main goal is then to establish the existence of a so-called Viehweg--Zuo sheaf, i.e., a sheaf $\mathcal H$ with big determinant that sits in a short exact sequence 
$$
\xymatrix{
0 \ar[r]& \mathcal H \ar[r]& (\Omega^1_{\mathcal X}(\log \mathbf \Delta))^{\otimes s} \ar[r] & \mathcal Q \ar[r]& 0,
}
$$
where  $s$ is a positive integer, 
and for simplicity, one may always replace $\mathcal H$ with its saturation.  It follows in the  case of varieties from 
a theorem of Campana--P\u{a}un  \cite[Thm.~7.6, Thm.~1.2]{CPFol19},   and from   
\cite[Thm.~A]{CMZfoliations} in the case of stacks, that   $\det \mathcal Q$ is pseudo-effective.  
Taking determinants in the short exact sequence above, one obtains
that a positive multiple of $K_{\mathcal X}+\mathbf \Delta$ is ``big plus pseudo-effective'', and therefore is big  (e.g., \cite[Cor.~B]{CMZfoliations}).

\medskip 

   The proof of \Cref{T:main-pairs} is therefore reduced to constructing Viehweg--Zuo sheaves.  The strategy for constructing these sheaves is easiest to explain in the somewhat artificial situation where  the boundary $\mathcal D$ is trivial and  the family $f:\mathcal Y\to \mathcal X$ is smooth (i.e., $\mathbf \Delta =0$).   In this case, one has that $f_* \omega_{\mathcal Y/\mathcal X}$ is a vector bundle, and moreover, from Koll\'ar \cite{kollar_sub}, and  more generally from Kov\'acs--Patakfalvi \cite{KP17proj},  one has that $\det (f_*\omega_{\mathcal Y/\mathcal X}^{\otimes m})$ is big for $m$ sufficiently large.
   Let us also assume that $m=1$; otherwise, one must use a cyclic covering construction, which complicates the discussion.   If $r$ is the rank of $f_* \omega_{\mathcal Y/\mathcal X}$, then for the morphism
\begin{equation}\label{E:IntroFibProdTrick}
\mathcal Y^r:=\mathcal Y\times_{\mathcal X}\cdots\times_{\mathcal X}\mathcal Y \stackrel{f^r}{\longrightarrow } \mathcal X
\end{equation}
one has an inclusion
\begin{equation}\label{E:IntrobigDetTrick}
\xymatrix{
\mathcal A:=\det f_*\omega_{\mathcal Y/\mathcal X} \ar@{^(->}[r]& (f_*\omega_{\mathcal Y/\mathcal X})^{\otimes r}=f^r_*(\omega_{\mathcal Y^r/\mathcal X}),
}
\end{equation}
where the right-most term in \eqref{E:IntrobigDetTrick} is the lowest piece of the Hodge filtration for the middle variation of Hodge structure associated to $f^r$, or equivalently,   the lowest piece of the Hodge filtration for the  Hodge module $$\mathsf M:=\mathcal H^0f^r_+\mathbb Q_{\mathcal Y^r}[\dim \mathcal Y^r].$$
  Letting $(\mathcal M,F_\bullet)$ be the filtered regular holonomic $D$-module underlying $\mathsf M$ on $\mathcal X$, one has the corresponding graded Higgs bundle $\mathcal E_\bullet :=(\operatorname{gr}^F_\bullet \mathcal M,\theta_\bullet)$, where $\theta_k:\operatorname{gr}^F_k \mathcal M\to \operatorname{gr}^F_{k+1} \mathcal M \otimes_{\mathcal O_{\mathcal X}}\Omega^1_{\mathcal X}$ is the induced  Kodaira--Spencer type  map.   

The next key observation is that setting 
$\mathcal K_k(\mathcal M) := \ker \theta_k$, 
then from Griffiths' study of semi-positive metrics on vector bundles obtained as the lowest term in the Hodge filtration of a variation of Hodge structures, as interpreted by \cite{Zuo00negativity} in terms of kernels of Kodaira--Spencer maps (see also \cite{PW16}, and \cite{CMZpositivity} in the case of stacks), one has that $\mathcal K_k(\mathcal M)^\vee$ is weakly positive. 
Considering the diagram
$$
\xymatrix@C=3em{
0\ar[r]& \operatorname{gr}_{-d}^F\mathcal M \ar[r]^<>(0.5){\theta_{-d}}& \operatorname{gr}_{-d+1}^F\mathcal M\otimes_{\mathcal O_{\mathcal X}}\Omega^1_{\mathcal X} \ar[r]^<>(0.5){\theta_{-d+1}\otimes 1}&  \operatorname{gr}_{-d+2}^F\mathcal M\otimes_{\mathcal O_{\mathcal X}}(\Omega^1_{\mathcal X})^{\otimes 2} \ar[r]^<>(0.5) {\theta_{-d+2}\otimes 1}& \cdots 
}
$$
where $d=\dim \mathcal Y-\dim \mathcal X$, and recalling that by definition $\mathcal A\subseteq \operatorname{gr}_{-d}^F\mathcal M$, one sees that there exists a non-negative  integer $s$ such that 
$
\mathcal A \hookrightarrow \mathcal K_{-d+s}(\mathcal M) \otimes (\Omega^1_{\mathcal X})^{\otimes s}$.  
From this we see there is a non-zero morphism $\mathcal A \otimes  \mathcal K_{-d+s}(\mathcal M)^\vee  \to (\Omega^1_{\mathcal X})^{\otimes s}$, so that factoring this morphism into a surjection followed by an inclusion,   we obtain the desired Viehweg--Zuo sheaf $\mathcal H$:
\[
\xymatrix{
\mathcal A \otimes  \mathcal K_{-d+s}(\mathcal M)^\vee \ar@{->>}[r] &  \mathcal{H} \ar@{^(->}[r] &  (\Omega^1_{\mathcal X})^{\otimes s},
}
\]
the key point being that since $\mathcal A$ is big and $\mathcal K_{-d+s}(\mathcal M)^\vee$ is weakly positive, we have that $\mathcal H$ is big (e.g., \cite[Lem.~2.19(1),(2)]{CMZpositivity}), and therefore that the determinant of  $ \mathcal H$ is big (e.g., \cite[Lem.~2.19(3)]{CMZpositivity}).   
Note that this implies $s>0$; if $s=0$, then the above would imply that  $\mathcal O_{\mathcal X}$ was big, giving a contradiction.

\medskip 
For the general case of  \Cref{T:main-pairs}, the construction of the Viehweg--Zuo sheaf  is quite intricate, and we sketch the outline here following \cite{PS17, WW23}, which, again, builds on techniques developed in \cite{VZ01isotriv, VZ02, VZ03}.
The first step  involves taking fibered products as in \eqref{E:IntroFibProdTrick}, and then taking a certain resolution of singularities of the family.  
 This is done in \Cref{P:TWW5.2} and \Cref{L:PWW5.3}, and relies on the positivity results of \cite{KP17proj}. 
 The outcome is that one replaces the original family with a  new morphism  $f:(\mathcal Y,\mathcal D)\to \mathcal X$ with $(\mathcal Y,\mathcal D)$ an snc pair (although now the pair need not be a family of stable pairs), and shows that one has a certain non-zero section of the line bundle $\omega_{\mathcal Y/\mathcal X}^{\otimes m}(m\mathcal D)\otimes f^*\mathcal A^{\otimes -m}$ on $\mathcal Y$, where $m$ is some sufficiently divisible positive integer, and $\mathcal A$ is a line bundle on $\mathcal X$ such that $\mathcal A(-\mathbf \Delta)$ is big.

With the construction in the previous paragraph, one uses the given section  to construct a cyclic cover of $\mathcal Y$, and one then replaces the given  family with a resolution of singularities of this cyclic cover, obtaining a  new morphism  $f:(\mathcal Y,\mathcal D)\to \mathcal X$ with $(\mathcal Y,\mathcal D)$ an snc pair.  
Using the theory of Hodge modules, this cyclic cover allows one to  construct an effective snc divisor $\mathbf E\subseteq \mathcal X$, containing $\mathbf \Delta$, together with a graded logarithmic Higgs bundle on $\mathcal X$
$$
\theta_\bullet: \mathcal E_\bullet \longrightarrow \mathcal E_\bullet \otimes \Omega^1_{\mathcal X}(\log \mathbf E)
$$ 
containing a graded sub-sheaf $\mathcal F_\bullet\subseteq \mathcal E_\bullet$ such that $\theta_p(\mathcal F_p)\subseteq \mathcal F_{p+1}\otimes\Omega^1_{\mathcal X}(\log \mathbf \Delta)$, and such  that the first nonzero sheaf in the filtration $\mathcal F_\bullet$ is a line bundle containing  the big sub-line bundle $\mathcal A(-\mathbf \Delta)$.
This is established  in  \Cref{T:refineHiggs}, which relies on a construction with Hodge modules given in \Cref{T:Existence-of-Hodge-module-and-G-subsheaf-theorem}. 
This step is one  of several  key additions that Popa--Schnell \cite{PS17} made to the strategy for constructing Viehweg--Zuo sheaves, by using  Saito's theory of Hodge modules, and the compatibility of Hodge and
$V$-filtrations, to show  the existence of the  Higgs sheaf $\mathcal F_\bullet$  that is 
logarithmic with respect to the discriminant locus $\mathbf \Delta$, as opposed to the
singular locus of the torsion-free quotient of the Hodge module $\mathcal H^0f^r_+\mathbb Q_{\mathcal Y^r}[\dim \mathcal Y^r]$, where the latter could be bigger.

\smallskip 
This construction with Hodge modules requires the use of the push forward of Hodge modules on DM stacks, which is discussed  in \S \ref{S:PFHMStack}. 
In general, if one wants to
obtain a $6$-functor formalism, there are more sophisticated treatments of push forwards, e.g., \cite{tubach_2024}, or \cite{paulin13} for $D$-modules; however, for our purposes here, where we only need to push forward the trivial 
Hodge module along a schematic morphism, the more elementary approach we take here suffices.  The approach we take is also well-suited to obtaining a result we want, namely, identifying the associated graded of the push forward in terms of Spencer complexes
(\Cref{L:PSP10.2}). 
The connection to the push forward constructed in \cite{tubach_2024} is explained in \S \ref{S:tubachf*}.

 At this point, the proof relies on some of our previous results in \cite{CMZpositivity, CMZfoliations}, generalizing some results of \cite{PW16,PS17} to the case of DM stacks.  
 More precisely, the existence of the  
 graded logarithmic Higgs bundle $\mathcal E_\bullet$ 
  and the graded subsheaf $\mathcal F_\bullet$  
   implies from \cite[Thm.~B]{CMZpositivity} 
   the existence of a Viehweg--Zuo sheaf, completing the construction.

\medskip 
The above results also have implications for the coarse moduli spaces of the stacks.  We explain this briefly in the situation where we have a normal integral proper $\mathbb Q$-factorial DM stack $\mathcal X$ over $\mathbb C$ with projective coarse moduli space $X$, and we denote by $$\pi:\mathcal X\to X$$ the canonical morphism.  
This morphism comes equipped with a so-called ramification divisor $\mathbf R$ on $\mathcal X$, and a branch divisor $R$ on $X$ (a $\mathbb Q$-divisor), with the property that $\mathbf R=\pi^*R$. 
The divisor $\mathbf R$ is the divisorial locus corresponding to points with extra automorphisms; 
moreover,  
$$
K_{\mathcal X} = \pi^*K_X+\mathbf R=\pi^*(K_X+R).
$$
If
$\mathbf R_i$ is an irreducible component of $\mathbf R$, with $r_i$ being  the order of the stabilizer at the generic point of $\mathbf R_i$ (modulo the order of the stabilizer at the generic point of $\mathcal X$), and $R_i$ is  the corresponding irreducible component in the support of $R$, then 
$$
R=\sum_i (1-\frac{1}{r_i})R_i.
$$
In addition, given an effective divisor $\mathbf \Delta\subseteq \mathcal X$, we let $\Delta$ be the $\mathbb Q$-divisor on $X$ such that $\mathbf \Delta=\pi^*\Delta$.  
If 
$\mathbf \Delta_i$ is an irreducible component of $\mathbf \Delta$, with $r_i$ being  the order the stabilizer at the generic point of $\mathbf \Delta_i$ (modulo the order of the stabilizer at the generic point of $\mathcal X$), and $\Delta_i$ is  the corresponding irreducible component in the support of $\Delta$, then 
$$
\Delta =\sum_i \frac{1}{r_i}\Delta_i.
$$
Note in particular that the coefficients of the irreducible components in the sum  of the divisors $R+\Delta$ are at most $1$ (i.e., if $\Delta_i=R_i$ is a component of both $\Delta$ and $R$, then all together we have $(1-\frac{1}{r_i})R_i+\frac{1}{r_i}\Delta_i=\Delta_i$).  
All together, we have
$$
K_{\mathcal X}+\mathbf \Delta = \pi^*(K_X+R+\Delta).
$$
Consequently,  $K_{\mathcal X}+\mathbf \Delta $ is big if and only if $K_X+R+\Delta$ is big  (e.g.,  \cite[Def.~2.4]{CMZpositivity}). 
Moreover, the pair $(X,R+\Delta)$ is log canonical by 
\cite[Rem.~3.1]{CMZfoliations}. 

In the situation we are in, i.e., working with a  normal integral complex projective  $\mathbb Q$-factorial variety $X$ with an effective $\mathbb Q$-divisor $\Delta$ on $X$ with coefficients in $(0,1]$, we will say that the pair $(X,\Delta)$ is of log general type if for any log resolution $\mu:X'\to X$ of the pair $(X,\Delta)$, taking $\Delta'$ to be the strict transform of $\Delta$,  \emph{union the reduced exceptional locus of $\mu$}, one has that $K_{X'}+\Delta'$ is big.  If $(X,\Delta)$ is log canonical and $K_X+\Delta$ is big, then it follows that $(X,\Delta)$ is of log general type (see also \cite[Prop.~1.27]{kollar_singularities_MMP}).
In summary, we have the following:

\begin{coralpha}\label{main-coarseMS-cor}
If $(\mathcal X,\mathbf \Delta)$ is the log resolution of singularities of a
 moduli space of KSBA stable pairs as in \Cref{main-pairs-cor}, then the pair $(X,R+\Delta)$ has log canonical singularities and  $K_X+R+\Delta$ is big; i.e., $(X,R+\Delta)$ is of log general type.
\qed 
\end{coralpha}

With this result, one can clearly explain the difference between our result and the earlier result
 \cite[Thm.~A]{WW23}. 
Assume we are in the situation where $(\mathcal X,\mathbf \Delta)$ is as in 
 \Cref{T:main-pairs} or \Cref{main-pairs-cor}.   Assuming that the stack $\mathcal X$ has trivial automorphisms at the generic point, then $\mathcal X$ and $X$ agree on $X-(R\cup \Delta)$ outside of codimension $2$, and so, using the family over this open sub-variety, the result  
  \cite[Thm.~A]{WW23}  implies that $K_X +\lceil R+\Delta \rceil$ is big.  There seem to be two benefits to the stronger statement in \Cref{main-coarseMS-cor}, i.e.,  regarding the pair $(X,R+\Delta)$ rather than $(X,\lceil R+\Delta \rceil)$.  First, the coefficients of 
$R+\Delta$ may be smaller than those of $\lceil R+\Delta \rceil$, so that one has a more precise statement about how much of the boundary divisor should be  added to the canonical divisor to  obtain a big  divisor.  Second, 
 it seems plausible that there could be situations where $(X,\lceil R+\Delta \rceil)$ is not log canonical, whereas we know from \Cref{main-coarseMS-cor} that  $(X,R+\Delta)$ is. 

A familiar, elementary example makes the difference between the results apparent.  Consider the moduli stack $\overline{\mathcal M}_{1,1}$ of elliptic curves, with coarse moduli space $\overline M_{1,1}$, which is isomorphic to the $j$-line $\mathbb P^1$. 
Viewing the marked point of the genus $1$ curve as having weight $1-\epsilon$ for $0<\epsilon\ll 1$,  one  can view $\overline{\mathcal M}_{1,1}$ as
a moduli space of klt KSBA stable pairs, so that we can apply   \Cref{main-coarseMS-cor}, as well as  \cite[Thm.~A]{WW23}, since the universal family descends over $\mathbb P^1-\{p_0,p_{1728}, p_\infty\}$, where here $p_j$ denotes the point parameterizing elliptic curves with $j$-invariant $j$.  
 In this example, we have 
   $R=(1-\frac{1}{3})p_{0}+ (1-\frac{1}{2})p_{1728}$ and $\Delta = p_\infty$, 
 so that taking sums we have
 $ R+\Delta  \equiv \frac{13}{6}p_0$ and  $\lceil R+\Delta \rceil  \equiv 3p_0$.
From \Cref{main-coarseMS-cor},   the assertion is that
$$
K_{\mathbb P^1}+(R+\Delta) \equiv K_{\mathbb P^1}+\frac{13}{6}p_0\equiv \frac{1}{6}p_0 \quad \text{ is big,}
$$
while the result \cite[Thm.~A]{WW23} implies   that
$$
K_{\mathbb P^1}+\lceil R+\Delta\rceil \equiv K_{\mathbb P^1}+3p_0\equiv p_0 \quad \text{ is big.}
$$
In this case, the two statements, that $\frac{1}{6}p_0$ and $p_0$ are big, are equivalent, but the former shows that it suffices to add only $\frac{13}{6}p_0$ to the canonical bundle to obtain a big line bundle (as opposed to $3p_0$ in the latter).

\subsection*{Acknowledgements}
The first named author thanks Mihnea Popa for  conversations on the topic, which led to this project.  He also thanks Jonathan Wise and David Rydh for conversations about the geometry of stacks, Swann Tubach for conversations on push forwards of Hodge modules, and J\'anos Koll\'ar for comments on an earlier draft of this manuscript.  The second named author thanks his advisor,  Mircea Musta\c{t}\u{a}, for useful discussions and all the support provided. The authors are also grateful to the organizers of the Simons Collaborations on Moduli of Varieties Workshop at the University of Utah in November 2024, where their work on this project began.

\section{Preliminaries}

\subsection{Terminology}
We use the same conventions as in \cite{CMZpositivity, CMZslope_stability, CMZfoliations}.
We work over $\mathbb C$.  
A \emph{variety} is an integral separated scheme of finite type over $\mathbb C$. 
We use the definition of a \emph{Deligne--Mumford (DM) stack} in \cite[Def.~4.1]{LMB}. Note that this differs from the definition in \cite{stacks-project} in that there is the additional hypothesis in \cite[Def.~4.1]{LMB} that the diagonal be representable, separated, and quasi-compact.  We direct the reader to \cite[App.~B]{CMW18} for a discussion of the relationship among various definitions of DM stacks in the literature (see in particular \cite[Fig.~1]{CMW18}).  We emphasize that, with the definition of DM stack that we are using, a morphism from a scheme to a DM stack is schematic (representable by schemes); see e.g., \cite[Lem.~B.20 and Lem.~B.12]{CMW18}.

\subsection{Structure of DM stacks}\label{S:DM-intro}
Again, we use the same conventions as in \cite{CMZpositivity, CMZslope_stability, CMZfoliations}.
  The general set-up will be a smooth proper (resp.~separated) integral DM stack $\mathcal X$ of finite type over $\mathbb C$ with coarse moduli space $\pi: \mathcal X\to X$, with the added assumption that the algebraic space $X$ be a projective (resp.~quasi-projective) variety. 
Recall that such a stack admits a finite flat 
morphism $q:V\to \mathcal X$ from a smooth projective  (resp.~quasi-projective) 
variety $V$  (\cite[Thm.~1]{KV04} and \cite[Thm.~4.4]{kresch09}, see also \cite[\href{https://stacks.math.columbia.edu/tag/03B6}{\S 03B6}]{stacks-project}); note that $q$ is schematic and projective.  
Note also that by the proof in \cite[Thm.~1]{KV04}, one can take $q$ to be \'etale over any given finite collection of points of $\mathcal X$. 
In this situation we have that $X$ is normal, $\mathbb Q$-factorial, with at worst klt singularities.  The morphism $\pi:\mathcal X\to X$ is flat over the smooth locus of $X$; flatness is an \'etale local property, and so it suffices to consider the case $\mathcal X=[U/G]$ for some smooth variety $U$ and a finite group $G$.  Then, from say \cite[Cor.~14.12]{GW20}, it suffices to show that $U\to U/G$ is flat over the smooth locus of the quotient, which follows from the miracle of flatness \cite[Thm.~23.1, p.179]{matsumura}.

For brevity, we will say that such a stack $\mathcal X$ is a global finite quotient stack if there is a smooth projective (resp.~quasi-projective) variety $V$ over $\mathbb C$ and a finite algebraic group $G$ over $\mathbb C$ acting on $V$ such that $\mathcal X\cong [V/G]$; note that this implies that $X=V/G$.  For context, recall that  $\mathcal X$ is a global finite quotient stack  if and only if there exists a smooth projective (resp.~quasi-projective) variety  $V'$ and a finite \emph{\'etale} morphism $q':V'\to \mathcal X$ (see the \emph{proof} of \cite[Thm.~(6.1)]{LMB}).

\subsection{Families of KSBA stable pairs, discriminants, and relative normal crossings}\label{S:SGSF}
We will consider  families of KSBA stable pairs following \cite[Thm.~8.1]{kollar_families}, and will work over the \'etale site of $\mathbb C$-schemes.    Fixing notation, for a coefficient vector $\mathbf a=(a_1,\dots, a_r)\in ([0,1]\cap \mathbb Q)^r$ and a rational number $\nu$, we let $\mathcal S\mathcal P(\mathbf a,n,\nu)$ be the proper DM stack of finite type over $\mathbb C$   consisting of families of $\mathbf a$-marked stable pairs of dimension $n$ and volume $\nu$.  This stack admits a projective coarse moduli space 
 $\mathcal S\mathcal P(\mathbf a,n,\nu) \to SP(\mathbf a,n,\nu)$ \cite[Thm.~1.1]{KP17proj}.   
We denote the universal family by
$$
(\mathcal Y(\mathbf a,n,\nu), a_1\mathcal D_1+\cdots+a_r\mathcal D_r)\to  \mathcal S\mathcal P(\mathbf a,n,\nu).
$$

  For a $\mathbb C$-scheme $X$, we will denote a stable family in  $\mathcal S\mathcal P(\mathbf a,n,\nu)$ over $X$, i.e., a $\mathbb C$-morphism $X\to  \mathcal S\mathcal P(\mathbf a,n,\nu)$, as
$$
f:(Y,a_1D_1+\cdots +a_rD_r)\to X.
$$
For brevity, abusing notation, we will often simply write $D=a_1D_1+\cdots+a_rD_r$, and will write $f:(Y,D)\to X$ for the family.  

For a reduced DM stack $\mathcal X$ of finite type over $\mathbb C$, by a stable family over $\mathcal X$ we mean a $\mathbb C$-morphism $\mathcal X\to  \mathcal S\mathcal P(\mathbf a,n,\nu)$.  Denote the pull back of  the universal family by
$$
f:(\mathcal Y,a_1\mathcal D_1+\cdots+a_r\mathcal D_r)\to \mathcal X.
$$
There is a technical point that one should be careful about what is meant by the pull back of the universal divisor to $\mathcal Y$.  The discussion in \cite[\S 4.1]{kollar_families} can be used to define this precisely.  
The stable family of $\mathcal X$ is equivalent to having for any  \'etale presentation $p:U\to \mathcal X$ from a reduced scheme $U$, that the  pull back of the family gives a stable family over $U$  with descent data.  
 Again for brevity, we will often refer to such families as $f:(\mathcal Y,\mathcal D)\to \mathcal X$.

We note that for any family of KSBA stable pairs  $f:(\mathcal Y,\mathcal D)\to \mathcal X$, the morphism is schematic and projective.  
The reason for this is that the index of $K_{\mathcal Y/\mathcal X}+\mathcal D$ is finite, i.e., there is some positive integer  $m$ such that $m(K_{\mathcal Y/\mathcal X}+\mathcal D)$ is Cartier  
(see, e.g., \cite[Def.-Thm. 4.7]{kollar_families}); 
thus $m(K_{\mathcal Y/\mathcal X}+\mathcal D)$  provides an $f$-ample line bundle on $\mathcal Y$ over $\mathcal X$; i.e., for every $\mathbb C$-morphism $S\to \mathcal X$, one has that the fibered product $\mathcal Y_S=S\times_{\mathcal X}\mathcal Y$ is  a scheme and the pull back of $m(K_{\mathcal Y/\mathcal X}+\mathcal D)$  is relatively ample (see, e.g.,  \cite[Thm.~4.38]{FGAE}).     Equivalently,  for any \'etale cover $p:U\to \mathcal X$ from a smooth variety $U$, if we consider the fibered product diagram,
\begin{equation}\label{E:stab-fam-diag}
\xymatrix{
\mathcal Y_U\ar[r]^{p'} \ar[d]^{f'}& \mathcal Y\ar[d]^f\\
U\ar[r]^p&\mathcal X
}
\end{equation}
we have that $\mathcal Y_U$ is a scheme and $p'^*m(K_{\mathcal Y/\mathcal X}+\mathcal D)$ is an $f'$-ample line bundle.

\begin{dfn}[Relative snc morphism]\label{D:RelsncDef}
Let   $f:(\mathcal Y,\mathcal D)\to \mathcal X$ be a surjective schematic projective  morphism of   smooth separated integral DM stacks of finite type over $\mathbb C$ with quasi-projective coarse moduli spaces, with $\mathcal D$ an effective $\mathbb Q$-divisor on $\mathcal Y$.
We say that \emph{$f$ is relative simple normal crossings (snc)} if $f$ is smooth, the support of $\mathcal D$ is an snc  divisor on $\mathcal Y$, and every stratum of the support of  $\mathcal D$ is smooth over $\mathcal X$.  
\end{dfn}

\begin{rem}\label{R:LogSm}
We believe the correct formulation in \Cref{D:RelsncDef} should be made in terms of log structures, in which case one would replace relative snc morphisms with more general log smooth morphisms, but we do not pursue this further here.
Note that in \cite{WW23} (in the case of varieties), they call a relative snc morphism a ``log smooth morphism'', even though relative snc morphisms are a special class of log smooth morphism (see their definition of log smooth morphism \cite[p.711]{WW23}).    
\end{rem}

\begin{dfn}[Relative snc discriminant]\label{dfn:non-relative snc locus}
Let   $f:(\mathcal Y,\mathcal D)\to \mathcal X$ be a surjective schematic projective  morphism of  separated integral DM stacks of finite type over $\mathbb C$ with quasi-projective coarse moduli space, with $\mathcal D$ an effective $\mathbb Q$-divisor on $\mathcal Y$.  
The locus where $f$ fails to be  relative snc is closed in $\mathcal X$, and  will be called the \emph{relative snc discriminant locus}.
  We denote by $\mathbf \Delta_f\subseteq \mathcal X$ the closed substack with the reduced induced structure on the relative snc discriminant locus, and call it the \emph{relative snc discriminant}.  
\end{dfn}

\begin{exa}
For clarity, let us consider the concrete example of the  Hassett moduli space \cite{hassett_weighted} $f:(\mathcal Y,\mathcal D)\to \overline {\mathcal M}_{0,\mathcal A}$ for $\mathcal A=(1-\epsilon,1-\epsilon,\frac{1}{3},\frac{1}{3})$; i.e.,  the moduli space of genus $0$ curves marked with four points with weights $1-\epsilon$, $1-\epsilon$, $1/3$, and $1/3$, respectively, with $0<\epsilon \ll 1$.   Then $\overline {\mathcal M}_{0,\mathcal A}\cong \overline {M}_{0,\mathcal A}\cong \mathbb P^1$ and  $\mathbf \Delta_f$ consists of three points.  There is one point of $\mathbf \Delta_f$ corresponding to the partition $\{1,3\}\cup \{2,4\}$, which parameterizes the reducible curve consisting of two copies of $\mathbb P^1$ attached at a node, with the first and third point on one component, and the second and fourth on the other.  Hassett denotes this in \cite[\S 7]{hassett_weighted} as $D_{\{1,3\}\cup \{2,4\}}(\mathcal A)$.   Similarly there is a point $D_{\{1,4\}\cup \{2,3\}}(\mathcal A)$ of $\mathbf  \Delta_f$.  Hassett denotes the union as $\nu = D_{\{1,3\}\cup \{2,4\}}(\mathcal A)\cup D_{\{1,4\}\cup \{2,3\}}(\mathcal A)$, for the ``nodal'' divisor; i.e., the part of the discriminant arising from nodes in the fibers of the family.  Then there is also the point $D_{\{1,2\}\cup \{3,4\}}(\mathcal A)$ of $\mathbf \Delta_f$ parameterizing a copy of $\mathbb P^1$ where the third and fourth points have collided, denoted by $\gamma$ in \cite{hassett_weighted}. 
The point $\gamma$  is included in $\mathbf  \Delta_f$ since the stratum of $\mathcal D$ to which this corresponds does not surject onto $\overline{\mathcal M}_{0,\mathcal A}$, and is therefore, not smooth over $\overline{\mathcal M}_{0,\mathcal A}$.  In Hassett's notation, we have $ \mathbf \Delta_f=\delta=\nu\cup \gamma$.   
\end{exa}

\begin{rem}[Total space is integral]\label{R:int-total}
Let $f:(\mathcal Y,\mathcal D)\to \mathcal X$ be  a family of KSBA stable pairs as in \Cref{T:main-pairs}.    Then $\mathcal Y$ is integral with connected fibers. Indeed, the morphism is flat, and the base is irreducible, so all the components appearing in $\mathcal{Y}$ can be detected generically over $\mathcal{X}$. But generically, the map is smooth and so there is only one irreducible component.
\end{rem}

\begin{rem}[Total space is klt]\label{R:klt-total}
Let $f:(\mathcal Y,\mathcal D)\to \mathcal X$ be  a family of KSBA stable pairs as in \Cref{T:main-pairs}.    Then the pair $(\mathcal Y,\mathcal D)$ is klt.  By definition, it suffices to show this after an \'etale base change as in \eqref{E:stab-fam-diag}, in which case this reduces to the case of schemes.  
For this, recall that for a stable family over a smooth base, if the generic fiber of the pair  is klt, then the pair consisting of the total family with the boundary divisor is klt \cite[Cor.~4.56 and Prop.~2.15]{kollar_families} (note that stable implies locally stable \cite[p.149]{kollar_families}).  
\end{rem}

\subsection{Variation}\label{S:variation}

Let $f:(\mathcal Y,\mathcal D)\to \mathcal X$ be a family of KSBA stable pairs  over
a smooth separated integral DM stack $\mathcal X$ of finite type over $\mathbb C$ with quasi-projective coarse moduli space $X$.  This induces a natural morphism of stacks $\mathcal X\to \mathcal {SP}(\mathsf a,n,\nu)$, as well as a morphism of coarse moduli spaces $X\to SP(\mathbf a,n,\nu)$.  The \emph{variation of $f$}, denoted $\operatorname{Var}(f)$, is defined to be the dimension of the image of the morphism $\mathcal X\to \mathcal {SP}(\mathbf a,n,\nu)$, or equivalently, the dimension of the image of the morphism $X\to SP(\mathbf a,n,\nu)$.  This agrees with the definition in \cite[Def.~6.16]{KP17proj}.  
Note that if  $q':V'\to \mathcal X$ is a morphism from a smooth projective variety $V'$, and $f':(Y',D')\to V'$ is the stable family obtained by pull back, then $\operatorname{Var}(f)=\operatorname{Var}(f')$.

\subsection{Reduction to normal crossings pairs}

We recall the following generalization of a standard result for varieties:
\begin{lem}\label{L:sncRdx}
Let $\mu:\mathcal X'\to  \mathcal X$ be a log resolution (see \cite[Thm.~1.3]{CMZslope_stability}) of the pair $(\mathcal X,\mathbf \Delta)$ with $\mathcal X$ a smooth proper DM stack over $\mathbb C$ with projective coarse moduli space and $\mathbf \Delta$ an effective $\mathbb Q$-divisor.
Let 
 $\widetilde {\mathbf \Delta}\subseteq \mathcal X'$ be the strict transform of $\mathbf \Delta$,
let $\mathcal E'\subseteq \mathcal X'$ denote the reduced divisor with support the exceptional locus of $\mu$, and 
let $\mathbf \Delta'\subseteq \mathcal X'$ be any effective $\mathbb Q$-divisor such that $0\le \mathbf \Delta'\le \widetilde{\mathbf \Delta} +\mathcal E'$.   Then if $K_{\mathcal X'}+\mathbf \Delta'$ is big, so is $K_{\mathcal X}+\mathbf \Delta$.
\end{lem}

\ifArxiv

\begin{proof}
Let $\mathbf \Delta_{\mathcal X'}$ be the $\mathbb Q$-divisor on $\mathcal X'$ such that
$
K_{\mathcal X'}+\mathbf \Delta_{\mathcal X'}=\mu^*(K_{\mathcal X}+\mathbf \Delta)
$
and write
$
\mathbf \Delta_{\mathcal X'}=\mathbf \Delta_{\mathcal X'}^+-\mathbf \Delta_{\mathcal X'}^-
$
with $\mathbf \Delta_{\mathcal X'}^{\pm}$ being effective divisors with disjoint support.  Note that clearly  we have $\widetilde {\mathbf \Delta}\le \mathbf \Delta_{\mathcal X'}^+$, and the support of $\mathbf \Delta_{\mathcal X'}^-$ is contained in the exceptional locus of $\mu$.  
Now, since we have $\widetilde {\mathbf \Delta}\le \mathbf \Delta_{\mathcal X'}^+$, 
we obtain
$
\mathbf \Delta'\le  \widetilde{\mathbf \Delta} +\mathcal E'\le \mathbf \Delta_{\mathcal X'}^+ + \mathcal E'. 
$
Thus,  if $K_{\mathcal X'}+\mathbf \Delta'$ is big, so is $K_{\mathcal X'}+ \mathbf \Delta_{\mathcal X'}^+ + \mathcal E'$. 

Also, for sufficiently divisible integers $m$ we have
\begin{align*}
\mu_*\mathcal O_{\mathcal X'}(m(K_{\mathcal X'}+\mathbf \Delta_{\mathcal X'}^+ + \mathcal E'))& \cong \mu_*\mathcal O_{\mathcal X'}(m(K_{\mathcal X'}+\mathbf \Delta_{\mathcal X'} +\mathbf \Delta_{\mathcal X'}^- + \mathcal E'))\\
  &\cong \mathcal O_{\mathcal X}(m(K_{\mathcal X}+\mathbf \Delta)) \otimes \mu_*\mathcal O_{\mathcal X'}(m(\mathbf \Delta_{\mathcal X'}^- + \mathcal E'))\\
    &\cong  \mathcal O_{\mathcal X}(m(K_{\mathcal X}+\mathbf \Delta)).
\end{align*}
Consequently, pushing forward to a point, global sections of the sheaves  $\mathcal O_{\mathcal X'}(m(K_{\mathcal X'}+\mathbf \Delta_{\mathcal X'}^+ + \mathcal E'))$ and  $ \mathcal O_{\mathcal X}(m(K_{\mathcal X}+\mathbf \Delta))$ agree, and we may conclude that $K_{\mathcal X}+\mathbf \Delta$ is big.
\end{proof}

\else
\begin{proof}
The proof is identical to the proof in the case of varieties, and is left to the reader. 
\end{proof}
\fi

\section{Some geometric constructions}\label{S:GeomConstRoots}

In order to make some geometric constructions with Hodge modules later, we will need some geometric constructions that modify our families to ensure sections of certain line bundles over our families.  This generalizes some results from \cite{PS17, WW23}.

\subsection{A geometric construction for stable pairs}\label{S:geom-con-sp}

Here we use ideas from  \cite[\S 4,5]{WW23} and the results in \cite{KP17proj} to obtain a version of  \cite[Thm.~5.2]{WW23} for DM stacks:

\begin{pro}
\label{P:TWW5.2} 
Let $f:(\mathcal Y,\mathcal D)\to \mathcal X$ and $\mathbf \Delta\subseteq \mathcal X$ be as in \Cref{T:main-pairs}, 
and let $\mathcal A$ be a line bundle on $\mathcal X$. 
There exists  a surjective schematic projective morphism $f^{(r)}:(\mathcal Y^{(r)},\mathcal D^{(r)})\to \mathcal X$ from a smooth proper integral DM stack $\mathcal Y^{(r)}$ over $\mathbb C$ with projective coarse moduli space, with $\mathcal D^{(r)}$ an effective $\mathbb Q$-divisor on $\mathcal Y^{(r)}$ with coefficients in $(0,1)$,  with pair discriminant $\mathbf \Delta_{f^{(r)}}$ contained in $\mathbf \Delta$, with $f^{(r)}$ having connected  fibers, so that setting $\mathcal U=\mathcal X-\mathbf \Delta$, the restriction $f^{(r)}|_{\mathcal U}:(\mathcal Y^{(r)}|_{\mathcal U},\mathcal D^{(r)}|_{\mathcal U})\to \mathcal U$ is relative snc, and  
 such that for the $\mathbb Q$-line bundle 
\begin{equation}\label{E:P:TWW5.2eq1}
\mathcal B:=\omega_{\mathcal Y^{(r)}/\mathcal X}(\mathcal D^{(r)})\otimes f^{(r)*}\mathcal A^{-1}, 
\end{equation}
we have 
\begin{equation*}\label{E:P:TWW5.2eq2}
H^0\left(\mathcal Y^{(r)},\mathcal B^{\otimes m}\right)\ne 0 \ \ \ \text{ for some sufficiently divisible } m\ge 1.
\end{equation*}
\end{pro}

\begin{proof}
The first observation is that it follows from \cite[Thm.~7.1]{KP17proj}
 that $\det f_*((\omega_{\mathcal Y/\mathcal X}^{\otimes m}(m \mathcal D))^{\vee \vee})$ is big.  For instance, let $q:V\to \mathcal X$ be a finite flat morphism from a smooth projective variety $V$, let $f_V:Y_V=V\times_{\mathcal X}\mathcal Y\to V$ be the base change of $f$, and let $q':Y_V\to \mathcal Y$ be the base change of $q$.  The pair $(Y_V,D_V:=q'^*\mathcal D)$ is a stable pair by definition of the moduli functor; here we are using  the so-called Weil divisor pull back $q'^*\mathcal D$, or one of the equivalent notions, in \cite[Ch.~4]{kollar_families}. 
 Then as $q$ is flat, we clearly have $ f_{V*}((\omega_{Y_V/V}^{\otimes m}(mD_V))^{\vee \vee})\cong q^*f_*((\omega_{\mathcal Y/\mathcal X}^{\otimes m}(m \mathcal D))^{\vee\vee})$.  The former has big determinant by\cite[Thm.~7.1]{KP17proj}, and so the latter has big determinant on $\mathcal X$, by virtue of  
 \cite[Lem.~2.5]{CMZpositivity} 
 and the functoriality of the determinant for torsion-free sheaves under finite flat pull back (see e.g., \cite[\S 1.4]{CMZslope_stability}).

Now, for any given $r$, we consider the  $r$-fold fiber product $(\mathcal Y ^{r}, \mathcal D^{r})$, where $\mathcal D^r:=\sum_{i}pr_i^*\mathcal  D$ with its induced morphism $f ^{r} : \mathcal Y ^{r}\to \mathcal X$.  
  This family is again a stable family \cite[Prop.~2.12]{BHPS13}, and $\mathcal Y^{r}$ is irreducible; we are assuming that $f$ is generically smooth so that there is only one irreducible component over the locus that $f$ is smooth,  and then use \cite[Lem.~7.11]{KP17proj}.

Let $\mu^{(r)}:\mathcal Y^{(r)} \to \Y^{r}$ be a strong  log resolution of singularities of the pair $(\mathcal Y^{r},f^{r*}\mathbf \Delta+\mathcal D^{r})$, and let $f^{(r)}= f^{r}\circ  \mu^{(r)}: \mathcal Y^{(r)}\to \mathcal X$.  Then $f^{(r)}$ is schematic and projective, $\mathcal Y^{(r)}$ is smooth and $f^{(r)*}\mathbf \Delta +\mu^{(r)*}\mathcal D^{r}$ is an snc divisor, which is relative SNC over $\X- \mathbf\Delta$. For that last part, we are using that we do not make any birational modifications away from fibers over $\mathbf\Delta$ and so it is sufficient to show that the $r$-fold fiber product of a relative SNC family is still relative SNC, which is easy to check locally.  Note that the fibers of $f^{(r)}$ are connected; one can see this by taking, for instance, the Stein factorization (e.g., \cite[Thm.~4.6.14]{AlperStacksBook}).

We have a commutative diagram
\begin{equation}\label{E:Pair-r-(r)}
\xymatrix{
\mathcal Y^{(r)} \ar[r]^{\mu^{(r)}} \ar[d]^{f^{(r)}}& \mathcal Y^{r}\ar[d]^{f^{r}}\\
\mathcal X \ar@{=}[r]& \mathcal X.
}
\end{equation}
Define an snc divisor $\mathcal D_{\mathcal Y^{(r)}}^{(r)}$ on $\mathcal Y^{(r)}$ by requiring
$$
K_{\mathcal Y^{(r)}}+\mathcal D^{(r)}_{\mathcal Y^{(r)}}=\mu^{(r)*}(K_{\mathcal Y^{r}}+\mathcal D^{r}).$$  
Write
$
\mathcal D^{(r)}_{\mathcal Y^{(r)}}=\mathcal D^{(r)+}_{\mathcal Y^{(r)}}-\mathcal D^{(r)-}_{\mathcal Y^{(r)}}
$
with $\mathcal D_{\mathcal Y^{(r)}}^{(r)\pm}$ being effective divisors with disjoint support.  
  We then define\footnote{Wei--Wu \cite{WW23} use the convention of taking the divisor $\mathcal D^{(r)}_{\mathcal Y^{(r)}}-\lfloor -\mathcal D^{(r)-}_{\mathcal Y^{(r)}}\rfloor=\mathcal D^{(r)+}_{\mathcal Y^{(r)}}+ \{\mathcal D^{(r)-}_{\mathcal Y^{(r)}}\}=\mathcal D^{(r)}+  \{\mathcal D^{(r)-}_{\mathcal Y^{(r)}}\}$ in place of the divisor $\mathcal D^{(r)}$, which also suffices for the argument here.} 
$\mathcal D^{(r)}:=\mathcal D^{(r)+}_{\mathcal Y^{(r)}}=\mathcal D^{(r)}_{\mathcal Y^{(r)}}+ \mathcal D^{(r)-}_{\mathcal Y^{(r)}}$.

First note that the pair $(\mathcal Y^r,\mathcal D^r)$ is klt (see \Cref{R:klt-total}).
 Consequently, since $\mathcal D^{(r)+}_{\mathcal Y^{(r)}}$ consists of the strict transform of $\mathcal D^r$, together with some exceptional divisors arising with negative discrepancies, we have from the definition of klt that the coefficients of $\mathcal D^{(r)}:=\mathcal D^{(r)+}_{\mathcal Y^{(r)}}$ are in $(0,1)$.

 In addition, since clearly $\mathcal D^{(r)-}_{\mathcal Y^{(r)}}$ has support in the exceptional locus of $\mu^{(r)}$, for any sufficiently divisible integer  $m>0$ we have
 \begin{align*}
 \mu^{(r)}_*\mathcal O_{\mathcal Y^{(r)}}(m(K_{\mathcal Y^{(r)}}+\mathcal D^{(r)}))= \mu^{(r)}_*\mathcal O_{\mathcal Y^{(r)}}(m(K_{\mathcal Y^{(r)}}+\mathcal D^{(r)}_{\mathcal Y^{(r)}}+ \mathcal D^{(r)-}_{\mathcal Y^{(r)}}))\\
 \cong \mu^{(r)}_*\mathcal O_{\mathcal Y^{(r)}}(m(K_{\mathcal Y^{(r)}}+\mathcal D^{(r)}_{\mathcal Y^{(r)}}))\cong  \mathcal O_{\mathcal Y^r}(m(K_{\mathcal Y^r}+\mathcal D^r)),
 \end{align*}
 where the last isomorphism is by the projection formula. From this we have
 $$
\mu^{(r)}_*\mathcal O_{\mathcal Y^{(r)}}(m(K_{\mathcal Y^{(r)}/\mathcal X}+\mathcal D^{(r)}))\cong \mu^{(r)}_*\mathcal O_{\mathcal Y^{(r)}}(m(K_{\mathcal Y^{(r)}}-\mu^{(r)*}f^{r*}K_{\mathcal X}+\mathcal D^{(r)}))
$$
$$
\cong \mathcal O_{\mathcal Y^r}(m(K_{\mathcal Y^r}-f^{r*}K_{\mathcal X}+\mathcal D^r))
\cong \mathcal O_{\mathcal Y^r}(m(K_{\mathcal Y^r/\mathcal X}+\mathcal D^r)).
$$
 Switching to sheaf notation  and taking tensor powers, for any sufficiently divisible $m>0$, 
 \begin{equation}\label{E:WWKPstep}
 \mu^{(r)}_*\omega_{\mathcal Y^{(r)}/\mathcal X}^{\otimes m}(m\mathcal D^{(r)})\cong \omega_{\mathcal Y^r/\mathcal X}^{\otimes m}(m \mathcal D^r)^{\vee \vee}.
 \end{equation}
 Finally, pushing forward by $f^r_*$ gives (for sufficiently divisible $m>0$) an isomorphism 
  \begin{equation}\label{E:WWKPstepf}
 f^{(r)}_*\omega_{\mathcal Y^{(r)}/\mathcal X}^{\otimes m}(m\mathcal D^{(r)})\cong f^r_*(\omega_{\mathcal Y^r/\mathcal X}^{\otimes m}(m \mathcal D^r)^{\vee \vee}),
 \end{equation}
 where we know that the sheaf on the right is reflexive, as $f^r$ is equidimensional (e.g., \cite[Cor.~1.7]{Hart80}; take an \'etale cover, then use \cite[\href{https://stacks.math.columbia.edu/tag/0339}{Lem.~0339}]{stacks-project} to deduce that the total space of the family is $S_2$, and observe that $S_2$ rather than $R_1+S_2$ is a sufficient hypothesis on the source in  the proof of  \cite[Cor.~1.7]{Hart80}), 
 and therefore the sheaf on the left is reflexive, too, by the isomorphism.

Finally from \eqref{E:WWKPstepf}, 
we find, using the short hand $[\ell]$ for the reflexive hull of the $\ell$-th tensor power,  that 
\begin{equation}\label{E:f'tof(r)'WW}
  (f_*\omega_{\mathcal Y/\mathcal X}^{m}(m \mathcal D))^{[r]}\cong  f^{r}_*(\omega_{\mathcal Y^{r}/\mathcal X}^{[m]}(m \mathcal D^{r})^{\vee\vee}) \cong  f^{(r)}_*\omega_{\mathcal Y^{(r)}/\mathcal X}^{m}(m\mathcal D^{(r)}),
\end{equation}
where the first isomorphism is from  \cite[(7.16.2) and Lem.~3.6]{KP17proj}. 

Denote by $r_0$ the rank of  the reflexive sheaf $f_*\omega_{\mathcal Y/\mathcal X}^{\otimes m}(m\mathcal D)$; as $\det f_*\omega_{\mathcal Y/\mathcal X}^{\otimes m}(m \mathcal D)$ is big,  we have that $f_*\omega_{\mathcal Y/\mathcal X}^{\otimes m}(m\mathcal D)\ne 0$, so that $r_0>0$.  Set 
$r:=N\cdot r_0$.
Then, via the standard arguments regarding determinants and tensor products  of vector bundles, together with push forwards of vector bundles over codimension $2$ loci, 
 there is a natural  inclusion of sheaves
\begin{equation}\label{E:PS-20.1incWW}
(\det f_*\omega_{\mathcal Y/\mathcal X}^{\otimes m}(m \mathcal D))^{\otimes N} \hookrightarrow (f_*\omega_{\mathcal Y/\mathcal X}^{\otimes m}(m\mathcal D))^{[r]}.
\end{equation}
Using that $\det f_*\omega_{\mathcal Y/\mathcal X}^{\otimes m}(m \mathcal D)$ is big, we have that for $N$ sufficiently large we can write
\begin{equation}\label{E:PS-20.1LMB-WW}
(\det f^{(r)}_*\omega_{\mathcal Y^{(r)}/\mathcal X}^{\otimes m}(m \mathcal D^{(r)}))^{\otimes N} \cong \mathcal A^{\otimes m} \otimes \mathcal O_{\mathcal X}(\mathcal E)
\end{equation}
where $\mathcal E$ is an effective divisor on $\mathcal X$. 
By \eqref{E:PS-20.1LMB-WW}, \eqref{E:PS-20.1incWW},  and  \eqref{E:f'tof(r)'WW},  we have an inclusion 
$$
\mathcal A^{\otimes m} \otimes \mathcal O_{\mathcal X}(\mathcal E) \hookrightarrow f^{(r)}_*\omega_{\mathcal Y^{(r)}/\mathcal X}^{\otimes m}(m\mathcal D^{(r)})
$$
showing that
$$
H^0(\mathcal X,f^{(r)}_*\omega_{\mathcal Y^{(r)}/\mathcal X}^{\otimes m}(m\mathcal D^{(r)}) \otimes \mathcal A^{\otimes -m})\ne 0.
$$
\end{proof}

\subsubsection{A further log resolution of the base for the pairs case}

For the arguments later with Hodge modules and Higgs bundles, we will actually need to make a further modification of the family.  The following lemma generalizes \cite[Prop.~5.3]{WW23} to the setting here:

\begin{lem}

\label{L:PWW5.3}
In the same situation as \Cref{P:TWW5.2},  fix the line bundle $\mathcal A$, fix the morphism $$f^{(r)}:\mathcal Y^{(r)}\to \mathcal X,$$ fix the divisor  $\mathcal D^{(r)}$, and fix a nonzero section $s\in H^0(\mathcal Y^{(r)},\mathcal B^{\otimes m})$ for some sufficiently divisible $m\ge 1$, where $\mathcal B:=\omega_{\mathcal Y^{(r)}/\mathcal X}( \mathcal D^{(r)} )\otimes f^{(r)*}\mathcal A^{-1}$ is as in \eqref{E:P:TWW5.2eq1}.

Let $\mathcal S\subseteq \mathcal X$ be an effective divisor,  and let $\sigma:\widetilde {\mathcal X}\to \mathcal X$ be a strong log resolution of the pair $(\mathcal X,\mathbf \Delta+\mathcal S)$.  Denote by $\mathbf \Lambda\subseteq \mathcal X$ the (codimension at least $2$) locus where $\mathbf \Delta+\mathcal S$ is not snc, and let $\mathcal X^\circ:=\mathcal X-\mathbf \Lambda$. Also, let  $\widetilde {\mathcal X}^\circ:=\sigma^{-1}(\mathcal X^\circ)\cong \mathcal X^\circ$.  Then there is a commutative diagram
$$
\xymatrix{
(\widetilde {\mathcal Y}^{(r)},\widetilde {\mathcal D}^{(r)}) \ar[r]^{\tilde \sigma } \ar[d]^{\tilde f^{(r)}}& (\mathcal Y^{(r)},\mathcal D^{(r)}) \ar[d]^{f^{(r)}}\\
\widetilde {\mathcal X}\ar[r]^\sigma&\mathcal X
}
$$
that satisfies the following conditions:

\begin{enumerate}[label=(\alph*)]
\item  The morphism  $\tilde f^{(r)}:\tilde {\mathcal Y}^{(r)}\to \widetilde {\mathcal X}$ is  schematic and projective with connected fibers, from a smooth proper integral DM stack $\widetilde {\mathcal Y}^{(r)}$ over $\mathbb C$ with projective coarse moduli space.

\item  The divisor 
 $\widetilde {\mathcal D}^{(r)}$ is an effective snc $\mathbb Q$-divisor with coefficients in $[0,1)$.

\item The morphism 
$\tilde \sigma $ is schematic projective and birational, 
 and   setting $\widetilde {\mathcal Y}^{(r)\circ}:=(\tilde f^{(r)})^{-1}(\widetilde {\mathcal X}^{(r)\circ})$ and $\mathcal Y^{(r)\circ}:=(f^{(r)})^{-1}(\mathcal X^\circ)$, we have that $\tilde \sigma  $ restricts to an isomorphism 
 $
 \tilde \sigma |_{\widetilde {\mathcal Y}^{(r)\circ}}: \widetilde {\mathcal Y}^{(r)\circ} \stackrel{\sim}{\longrightarrow} \mathcal Y^{(r)\circ}
 $
 over $\mathcal X^\circ$, and  
 $  \widetilde {\mathcal D}^{(r)}
 |_{\widetilde {\mathcal Y}^{(r)\circ}} =\mathcal D^{(r)}|_{\mathcal Y^{(r)\circ}}$.

\item Let $\tilde{\mathcal A}=\sigma^*\mathcal A$, and define 
\begin{align*}
\widetilde {\mathcal B}:=\omega_{\widetilde {\mathcal Y}^{(r)}/\widetilde {\mathcal X}}(  \widetilde{\mathcal D}^{(r)})\otimes_{\mathcal O_{\widetilde {\mathcal Y}^{(r)}}} \tilde f^{(r)*}\tilde {\mathcal A}^{-1}.
\end{align*}
There is an identification $\widetilde {\mathcal B}|_{\widetilde {\mathcal Y}^{(r)\circ}}= \tilde \sigma|_{\widetilde {\mathcal Y}^{(r)\circ}}^*(\mathcal B|_{\mathcal Y^{(r)\circ}})$, and 
there exists a section $\tilde s\in H^0(\widetilde {\mathcal Y}^{(r)}, \widetilde{\mathcal B}^{\otimes m})$ such that 
$
\tilde s|_{\widetilde {\mathcal Y}^{(r)\circ}}= \tilde \sigma|_{\widetilde {\mathcal Y}^{(r)\circ}}^*(s|_{\mathcal Y^{(r)\circ}})$.

\item  Denote by $\tilde {\mathbf \Delta}\subseteq \widetilde {\mathcal X}$ the reduced pre-image of $\mathbf \Delta$.  
Setting $\widetilde {\mathcal U}:=\widetilde {\mathcal X}-\tilde {\mathbf \Delta}$, the restriction $\widetilde {\mathcal Y}|_{\widetilde {\mathcal U}}\to \widetilde {\mathcal U}$ is smooth and all of the strata of the snc divisor $\widetilde {\mathcal D}|_{\widetilde {\mathcal U}}$ are smooth over $\widetilde {\mathcal U}$.  

\end{enumerate}
Moreover, if $\mathcal A$ was chosen so that $\mathcal A(-\mathbf \Delta)$ is big, then $\tilde {\mathcal A}(-\tilde {\mathbf \Delta})$ is big. 
\end{lem}

\begin{proof}
We start in the proof of  \Cref{P:TWW5.2} with diagram \eqref{E:Pair-r-(r)}, and construct a diagram 

\begin{equation}\label{E:Pair-r-(r)2}
\xymatrix@C=1em@R=1em{
&\mathcal Y^{(r)} \ar[rr]^{\mu^{(r)}} \ar@{-}[d]^{f^{(r)}}&& \mathcal Y^{r}\ar[dd]^{f^{r}}\\
\widetilde {\mathcal Y}^{(r)}\ar[rr]_<>(0.25){\tilde \mu^{(r)}} \ar[dd]_{\tilde f^{(r)}} \ar[ru]^{\tilde \sigma }&\ar[d]&\widetilde {\mathcal Y}^r \ar[dd]^<>(0.25){\tilde f^r} \ar[ru]^{\sigma^r}&\\
&\mathcal X \ar@{=}[r]& \ar@{=}[r]& \mathcal X\\
\widetilde {\mathcal X}\ar@{=}[rr] \ar[ru]^\sigma&&\widetilde {\mathcal X} \ar[ru]_\sigma& \\
}
\end{equation} 
In the above $\widetilde {\mathcal Y}^r:=\widetilde {\mathcal X}\times_{\mathcal X}\mathcal Y^r$ is defined to be the stable family obtained by pull back, with boundary divisor $\widetilde {\mathcal D}^r$ obtained by pull back, and we take $\widetilde {\mathcal Y}^{(r)}$ to be a strong resolution of singularities of the fibered product $\widetilde {\mathcal X}\times_{\mathcal X}\mathcal Y^{(r)}$ with respect to the pull backs of all of the divisors in question so that 
$\tilde f^{(r)*} ({\mathbf \Delta_{\tilde f^{(r)}}} +\sigma^*\mathbf \Delta_{f^{(r)}})+ \tilde \sigma ^*\mathcal D^{(r)}$, together with the exceptional divisors for $\widetilde \mu^{(r)}$, is an snc divisor.   

As before, define an snc divisor $\widetilde {\mathcal D}_{\widetilde {\mathcal Y}^{(r)}}^{(r)}$ on $\widetilde {\mathcal Y}^{(r)}$ by requiring
\begin{equation}\label{E:TDY(r)def}
K_{\widetilde {\mathcal Y}^{(r)}}+\widetilde {\mathcal D}^{(r)}_{\widetilde {\mathcal Y}^{(r)}}=\widetilde \mu^{(r)*}(K_{\widetilde {\mathcal Y}^{r}}+\widetilde {\mathcal D}^{r}).
\end{equation}
Write
$
\widetilde {\mathcal D}^{(r)}_{\widetilde {\mathcal Y}^{(r)}}=\widetilde {\mathcal D}^{(r)+}_{\widetilde {\mathcal Y}^{(r)}}-\widetilde {\mathcal D}^{(r)-}_{\widetilde {\mathcal Y}^{(r)}}
$
with $\widetilde {\mathcal D}_{\widetilde {\mathcal Y}^{(r)}}^{(r)\pm}$ being effective divisors with disjoint support.  We then define
$\widetilde {\mathcal D}^{(r)}:=\widetilde {\mathcal D}^{(r)+}_{\widetilde {\mathcal Y}^{(r)}}=\widetilde {\mathcal D}^{(r)}_{\widetilde {\mathcal Y}^{(r)}}+ \widetilde {\mathcal D}^{(r)-}_{\widetilde {\mathcal Y}^{(r)}}$.
Arguing as before, we have that the coefficients of $\widetilde {\mathcal D}^{(r)}$ are all in $(0,1)$, and  
 \begin{equation}\label{E:WWKPstep-2-2}
 \tilde f^{(r)}_*\omega_{\widetilde {\mathcal Y}^{(r)}/\widetilde {\mathcal X}}^{\otimes m}(m\widetilde {\mathcal D}^{(r)})\cong \tilde f^{r}_*(\omega_{\widetilde {\mathcal Y}^{r}/\widetilde {\mathcal X}}^{\otimes m}(m \widetilde {\mathcal D}^{r}))^{\vee \vee},
\end{equation}
 where, as before, the sheaf on the right is reflexive, hence, so is the sheaf on the left. 
To see that \eqref{E:WWKPstep-2-2} is well defined, note  that we assumed that $m(K_{{\mathcal Y}^{r}}+\mathcal D^r)$ was integral and Cartier; by base change for stable families, we   therefore have   $m(K_{\widetilde {\mathcal Y}^{r}}+\widetilde {\mathcal D}^{r})$ is integral and Cartier, making sense of the right hand side of \eqref{E:WWKPstep-2-2}.  
At the same time,  this together with  \eqref{E:TDY(r)def} then implies that $m(K_{\widetilde {\mathcal Y}^{(r)}}+\widetilde {\mathcal D}^{(r)}_{\widetilde {\mathcal Y}^{(r)}})$ is integral and Cartier, as well.  In particular, the left hand side of \eqref{E:WWKPstep-2-2} is also well defined.  Note further that because  $\widetilde {\mathcal Y}^{(r)}$ is smooth, we have that $K_{\widetilde {\mathcal Y}^{(r)}}$ is integral, so that $m \widetilde {\mathcal D}^{(r)}_{\widetilde {\mathcal Y}^{(r)}}$ is integral, and therefore that $\widetilde {\mathcal D}_{\widetilde {\mathcal Y}^{(r)}}^{(r)\pm}$ are integral.  In particular, $m\widetilde {\mathcal D}^{(r)}$ is integral.

 Note also that if $\mathcal A(-\mathbf \Delta)$ is big, then $\sigma^*\mathcal A(-\mathbf \Delta)= \tilde {\mathcal A}(-\sigma^*\mathbf \Delta)$ is big.  Consequently, since $0\le \tilde {\mathbf \Delta}\le \sigma^{-1}\mathbf \Delta$, we have that $\tilde{\mathcal A}(-\tilde {\mathbf \Delta})$ is big.
  Thus the morphism $\tilde f^{(r)}$ and the divisor $\widetilde {\mathcal D}^{(r)}$ satisfy all of the conditions of the lemma, except perhaps the condition on global sections. 
 
The key observation is that, being a base change of stable families, there is an isomorphism (e.g., \cite[Rem.~3.8]{KP17proj}) 

$$
\sigma^{r*}\omega^{\otimes m}_{\mathcal Y^r/\mathcal X}(m\mathcal D^r)\cong \omega^{\otimes m}_{\widetilde {\mathcal Y}^r/\widetilde {\mathcal X}}(m\widetilde {\mathcal D}^r).
$$
We then consider the natural morphism 
$$
\sigma^*f^r_*\omega^{\otimes m}_{\mathcal Y^r/\mathcal X}(m\mathcal D^r) \to \tilde f^r_* \omega^{\otimes m}_{\widetilde {\mathcal Y}^r/\widetilde {\mathcal X}}(m\widetilde {\mathcal D}^r),
$$
which is generically an isomorphism (i.e., an isomorphism over $\mathcal X^\circ$), and take reflexive hulls.  Using the isomorphisms \eqref{E:WWKPstepf} and \eqref{E:WWKPstep-2-2}, and tensoring with $\tilde {\mathcal A}=\sigma^*\mathcal A$,  we obtain a morphism
$$
\sigma^*(f^{(r)}_*\omega^{\otimes m}_{\mathcal Y^{(r)}/\mathcal X}(m\mathcal D^{(r)})\otimes \mathcal A^{-m})
\to 
 \tilde f^{(r)}_*\omega_{\widetilde {\mathcal Y}^{(r)}/\widetilde {\mathcal X}}^{\otimes m}(m\widetilde {\mathcal D}^{(r)})\otimes \tilde {\mathcal A}^{-m},
$$
which is an isomorphism over $\mathcal X^\circ$.    

Taking global sections gives a morphism
$$
\sigma^*: H^0(\mathcal X, f^{(r)}_*\omega^{\otimes m}_{\mathcal Y^{(r)}/\mathcal X}(m\mathcal D^{(r)})\otimes \mathcal A^{-m})
\to 
 H^0(\widetilde {\mathcal X}, \tilde f^{(r)}_*\omega_{\widetilde {\mathcal Y}^{(r)}/\widetilde {\mathcal X}}^{\otimes m}(m\widetilde {\mathcal D}^{(r)})\otimes \tilde {\mathcal A}^{-m}).
$$
Therefore, given our section $s\in H^0(\mathcal X, f^{(r)}_*\omega^{\otimes m_0}_{\mathcal Y^{(r)}/\mathcal X}(m\mathcal D^{(r)})\otimes \mathcal A^{-m})$, 
which is by assumption non-zero over a dense open substack of $\mathcal X^\circ$, we have that $\sigma^*s=\tilde s$ satisfies the conditions of the lemma.
\end{proof}

\begin{rem}[Replacing $\mathcal D^{(r)}$ with $\lceil \mathcal D^{(r)}\rceil$]\label{R:L([D])section}
We make the following observation for later.  
In the notation of \Cref{P:TWW5.2}, so that we may assume that 
 $H^0(\mathcal{Y}^{(r)}, \omega^{\otimes m}_{\mathcal{Y}^{(r)}/\mathcal{X}}(m \mathcal D^{(r)})\otimes f^{(r)*}\m{A}^{\otimes -m}) \ne 0$,
define 
\begin{equation}\label{E:HdgMdLL'def}
\m{L}':=\omega_{\mathcal{Y}^{(r)}/\mathcal{X}}( \lceil{\mathcal D^{(r)}} \rceil )\otimes f^{(r)*}\m{A}^{-1},
\end{equation}
and fix an  $m$ such that $\mathcal L'^{\otimes m}$ has a non-trivial section $s$ 
with zero divisor $(s)$ containing the support of $\mathcal D^{(r)}$.  Clearly such an $m$ exists since $$\mathcal L'^{\otimes m}=\omega^{\otimes m}_{\mathcal{Y}^{(r)}/\mathcal{X}}(m\mathcal D^{(r)})\otimes f^{(r)*}\m{A}^{\otimes -m} \otimes \mathcal O_{\mathcal Y^{(r)}}(m \lceil{\D}^{(r)}  \rceil -m\mathcal D^{(r)}),$$ so that if $t_1$ is a non-trivial section of $\omega^{\otimes m}_{\mathcal{Y}^{(r)}/\mathcal{X}}(m\mathcal D^{(r)})\otimes f^{(r)*}\m{A}^{\otimes -m}$  and $t_2$ is a non-trivial section of the line bundle $\mathcal O_{\mathcal Y^{(r)}}(m \lceil {\mathcal D}^{(r)} \rceil -m\mathcal D^{(r)})$ determined by the effective integral divisor $m\left \lceil{\mathcal D}^{(r)}\right \rceil -m\mathcal D^{(r)}=m(\left \lceil{\mathcal D}^{(r)}\right \rceil -\mathcal D^{(r)})$, then $s= t_1t_2$ is  a non-trivial section of $\mathcal L'^{\otimes m}$ such that $(s)$ contains the support of $\mathcal D^{(r)}$.
 Now, in the notation of \Cref{L:PWW5.3}, 
define
$$
\tilde{\mathcal L} :=\omega_{\widetilde {\mathcal{Y}}^{(r)}/\widetilde {\mathcal{X}}}(\lceil{\widetilde {\mathcal D}}^{(r)}\rceil )\otimes_{\mathcal O_{\widetilde {\mathcal Y}^{(r)}}} \tilde f^{(r)*}\widetilde{\m{A}}^{\otimes -1}.
$$
 By \Cref{L:PWW5.3}, $t_1$ extends to a global section $\tilde t_1\in H^0(\omega_{\widetilde {\mathcal{Y}}^{(r)}/\widetilde {\mathcal{X}}}(\widetilde {\mathcal D}^{(r)})\otimes_{\mathcal O_{\widetilde {\mathcal Y}^{(r)}}} \tilde f^{(r)*}\widetilde{\m{A}}^{\otimes m})$; i.e., 
$
\tilde t_1|_{\widetilde {\mathcal Y}^\circ}=\tilde \sigma^*(s'|_{\mathcal Y^\circ})$. 
Now let $\tilde t_2$ be the section of $\mathcal O_{\widetilde {\mathcal Y}^{(r)}}(m\lceil\widetilde {\mathcal D}^{(r)} \rceil -m\widetilde {\mathcal D}^{(r)})$ determined by the effective integral divisor $m \lceil\widetilde {\mathcal D}^{(r)} \rceil -m\widetilde {\mathcal D}^{(r)}=m( \lceil\widetilde {\mathcal D}^{(r)} \rceil -\widetilde {\mathcal D}^{(r)})$; then $\tilde s= \tilde t_1\tilde t_2$ is  a non-trivial section of $\tilde{\mathcal L}^{\otimes m}$ such that $(s)$ contains the support of $\widetilde {\mathcal D}^{(r)}$.
In addition, let $\widetilde {\mathcal E}$ be the exceptional locus of  $\tilde \sigma$, which we note is contained in $\tilde f^{(r)*}\sigma^{-1}(\mathcal S)$.  Therefore, by construction, we have that $
\tilde s |_{\widetilde {\mathcal Y}^{(r)}-\widetilde {\mathcal E}}=\tilde \sigma^*s|_{\widetilde {\mathcal Y}^{(r)}-\widetilde {\mathcal E}}$.

\end{rem}

\section{Push forward of Hodge modules on stacks} \label{S:PFHMStack}

In this section we consider the push forward of Hodge modules on stacks. 
In general, if one wants to
obtain a $6$-functor formalism, there are more sophisticated treatments, e.g., \cite{tubach_2024}, or \cite{paulin13} for $D$-modules; however, for our purposes here, where we only need to push forward the trivial 
Hodge module along schematic morphisms, the more elementary approach we take here,
closely following say the treatment of $D$-modules on stacks in \cite{BDstacks}, suffices.
  In fact, our end goal is to use explicit descriptions of associated graded modules of push forwards of filtered $D$-modules underlying push forwards of Hodge modules, in terms of push forwards of relative Spencer complexes (e.g., \Cref{L:PSP10.2}), and the approach we take here is  well suited to this.  
  The connection to the push forward constructed in \cite{tubach_2024} is explained in \S \ref{S:tubachf*}.
Throughout this section, 
let $\mathcal X$ be a smooth separated integral DM stack locally of finite type over $\mathbb C$.

\subsection{Preliminaries on $D$-modules and Hodge modules}\label{S:prelim-Dmod}

We use the definition of $D$-modules and Hodge modules on $\mathcal X$  from \cite{CMZpositivity}. 
Recall that these are taken to be \'etale sheaves of the given objects on the small \'etale site of $\mathcal X$.  
We denote by $\mathsf {Mod_{rh}}(D_{\mathcal X})$ the category of (left) regular holonomic $D_{\mathcal X}$-modules on $\mathcal X$, and by $\mathsf {HM^p}(\mathcal X,w)$ the category of polarizable Hodge modules of weight $w$ on $\mathcal X$.  When $(\Y, \D)$ is a smooth log pair, we denote the graded algebras
$$\mathscr{A}^\bullet_{\Y}(-\log \D):=\operatorname{Sym}^\bullet(\mathcal {T}_{\Y}(-\log \D)) \subseteq \mathscr{A}^\bullet_{\Y}:=\textup{Sym}^\bullet(\m{T}_{\Y}),$$
where $\mathcal T_{\mathcal Y}(-\log \mathcal D):=(\Omega^1_{\mathcal Y}(\log \mathcal D))^\vee$.  

We briefly mention that in \cite{CMZpositivity} we confirmed that Hodge modules on stacks form an abelian category, that the torsion sub-sheaf of a Hodge module is a sub-Hodge module, and that there is a notion of  minimal extension of Hodge modules, which can be characterized as in \cite[Lem.~3.17]{CMZpositivity}. Also, the singular locus
 of a pure Hodge module with full support is pure of codimension-$1$, or empty (e.g., \cite[Prop.~3.21]{CMZpositivity}).

\subsection{Spencer complexes of $D$-modules}\label{S:dR+SC}

Given a \emph{right} $D_{\mathcal X}$-module $\mathcal M$, the \emph{Spencer complex of $\mathcal M$}, denoted  $\operatorname{Sp}^\bullet(\mathcal M)$, or $(\operatorname{Sp}^\bullet (\mathcal M),\delta)$,  is the bounded complex with entries that are $\mathcal O_{\mathcal X}$-modules, and morphisms that are  $\mathbb C$-linear: 
$$
\operatorname{Sp}^\bullet(\mathcal M):= \left[
\xymatrix@R=0em{
\cdots 0 \ar[r] &\displaystyle \mathcal M\otimes_{\mathcal O_{\mathcal X}}\bigwedge^{n}\mathcal T_{\mathcal X}  \ar[r]^<>(0.5)\delta & \displaystyle \mathcal M \otimes_{\mathcal O_X}\bigwedge ^{n-1}\mathcal T_{\mathcal X} \ar[r]^<>(0.5)\delta & \cdots \ar[r]^<>(0.5)\delta& \mathcal M \ar[r] & 0\cdots\\
&&&&0
} \right]
$$
with $\mathcal M$ in degree $0$, and $n:=\dim \mathcal X$.  The $\mathbb C$-linear morphisms $\delta$ are defined \'etale locally, and so are given as in the case of varieties; see \cite[Def.~8.4.3]{MHM-project} or \cite[Lem.~1.5.27]{HTT08}.  The fact that $\delta\circ \delta = 0$ follows from the case of varieties \cite[Lem.~1.5.27]{HTT08}.  
If, in addition to being a right $D_{\mathcal X}$-module,  $\mathcal M$ is also a left $D_{\mathcal X}$-module, then $\operatorname{Sp}^\bullet(\mathcal M)$ is naturally a complex with entries that are \emph{left} $\mathcal D_{\mathcal X}$-modules.

\begin{exa}[Spencer complex for $D_{\mathcal X}$]
If we consider the Spencer complex associated  to the right and left $D_{\mathcal X}$-module $D_{\mathcal X}$, we have a quasi-isomorphism of left $D_{\mathcal X}$-modules:  
\begin{equation}\label{E:SpQI-OO}
\operatorname{Sp}^\bullet (D_{\mathcal X})\simeq \mathcal O_{\mathcal X}.
\end{equation}
Indeed, the complex of left $D_{\mathcal X}$-modules and $\mathbb C$-linear morphisms 
$$
\xymatrix{
\cdots 0 \ar[r] &\displaystyle D_{\mathcal X}\otimes_{\mathcal O_{\mathcal X}}\bigwedge^{n}\mathcal T_{\mathcal X}  \ar[r]^<>(0.5)\delta & \displaystyle D_{\mathcal X} \otimes_{\mathcal O_X}\bigwedge ^{n-1}\mathcal T_{\mathcal X} \ar[r]^<>(0.5)\delta & \cdots \ar[r]^<>(0.5)\delta& D_{\mathcal X} \ar[r] & \mathcal O_{\mathcal X}\ar[r]& 0\cdots 
}
$$
is exact.  The $\mathbb C$-linear morphism  $D_{\mathcal X} \to \mathcal O_{\mathcal X}$ is the canonical morphism $P\mapsto P(1)$; i.e., the natural projection splitting the inclusion $\mathcal O_X\to D_{\mathcal X}$.  The exactness of the complex can be checked \'etale locally, and  then follows from the case of varieties; e.g.,  \cite[Lem.~1.5.27]{HTT08}.
\end{exa}

Let $f:\mathcal X\to \mathcal Y$ be a schematic morphism of  
smooth separated integral DM stacks locally of finite type over $\mathbb C$.  For a \emph{right} $D_{\mathcal X}$-module $\mathcal M$, 
the \emph{relative Spencer complex of $\mathcal M$}, denoted  $\operatorname{Sp}^\bullet_{\mathcal X\to \mathcal Y}(\mathcal M)$, or $(\operatorname{Sp}^\bullet _{\mathcal X\to \mathcal Y}(\mathcal M),\delta_{\mathcal X\to \mathcal Y})$,  is the bounded complex with entries that are right $f^{-1}D_{\mathcal Y}$-modules, and morphisms that are  $\mathbb C$-linear: 
$$
\operatorname{Sp}_{\mathcal X\to \mathcal Y}^\bullet(\mathcal M):= 
$$

$$
\left[\xymatrix@C=1em@R=0em{
\cdots 0 \ar[r] &\displaystyle \mathcal M\otimes_{\mathcal O_{\mathcal X}}\bigwedge^{n}\mathcal T_{\mathcal X} \otimes_{f^{-1}\mathcal O_{\mathcal Y}}f^{-1}D_{\mathcal Y} \ar[r]^<>(0.5){\delta_{\mathcal X\to \mathcal Y}} & \displaystyle \mathcal M \otimes_{\mathcal O_X}\bigwedge ^{n-1}\mathcal T_{\mathcal X} \otimes_{f^{-1}\mathcal O_{\mathcal Y}}f^{-1}D_{\mathcal Y}  \ar[r]^<>(0.5){\delta_{\mathcal X\to \mathcal Y}} & \cdots & \\
&\ 
}\right.
$$
$$
\left.\xymatrix@C=1em@R=0em{
\cdots \ar[r]^<>(0.5){\delta_{\mathcal X\to \mathcal Y}}&\displaystyle \mathcal M \otimes_{\mathcal O_X}\bigwedge ^{1}\mathcal T_{\mathcal X} \otimes_{f^{-1}\mathcal O_{\mathcal Y}}f^{-1}D_{\mathcal Y}  \ar[r]^<>(0.5){\delta_{\mathcal X\to \mathcal Y}} & \mathcal M \otimes_{f^{-1}\mathcal O_{\mathcal Y}}f^{-1}D_{\mathcal Y}  \ar[r] & 0\cdots \\ 
&&0
}\right]
$$
with $\mathcal M \otimes_{f^{-1}\mathcal O_{\mathcal Y}}f^{-1}D_{\mathcal Y}$ in degree $0$, and $n:=\dim \mathcal X$.  The $\mathbb C$-linear morphisms $\delta_{\mathcal X\to \mathcal Y}$ are defined \'etale locally, and so are given as in the case of varieties; see \cite[Exe.~8.40]{MHM-project}.  The fact that $\delta_{\mathcal X\to \mathcal Y}\circ \delta_{\mathcal X\to \mathcal Y} = 0$ can be checked \'etale locally, and so also follows from the case of varieties \cite[Exe.~8.40]{MHM-project}.  
If, in addition to being a right $D_{\mathcal X}$-module, $\mathcal M$ is also a left $D_{\mathcal X}$-module, then $\operatorname{Sp}^\bullet_{\mathcal X\to \mathcal Y}(\mathcal M)$ is naturally a complex with entries that are \emph{left} $D_{\mathcal X}$-modules.  

There is an isomorphism  \cite[Exe.~8.40]{MHM-project} 
$\operatorname{Sp}^\bullet _{\mathcal X\to \mathcal Y} (\mathcal M)\simeq \mathcal M\otimes_{D_{\mathcal X}}\operatorname{Sp}^\bullet_{\mathcal X\to \mathcal Y}(D_{\mathcal X})$.
The key observation is then that 
\begin{equation}\label{E:RelSpVersion}
\operatorname{Sp}^\bullet _{\mathcal X\to \mathcal Y} (\mathcal M)\simeq \mathcal M\otimes_{D_{\mathcal X}}\operatorname{Sp}^\bullet_{\mathcal X\to \mathcal Y}(D_{\mathcal X}) \simeq \mathcal M\otimes_{D_{\mathcal X}}\operatorname{Sp}^\bullet(D_{\mathcal X}) \otimes_{f^{-1}\mathcal O_{\mathcal Y}}f^{-1}D_{\mathcal Y}.
\end{equation}
Recalling the notation $D_{\mathcal X\to \mathcal Y}:=\mathcal O_{\mathcal X}\otimes_{f^{-1}\mathcal O_{\mathcal Y}}f^{-1}D_{\mathcal Y}$ (e.g., \cite[\S 4.1.1]{CMZpositivity}) and combining it with the quasi-isomorphism $\operatorname{Sp}^\bullet (D_{\mathcal X})\simeq \mathcal O_{\mathcal X}$ of \eqref{E:SpQI-OO}, we see also that
\begin{equation}\label{E:SpDXY}
\operatorname{Sp}^\bullet _{\mathcal X\to \mathcal Y} (\mathcal M)\simeq \mathcal M\otimes_{D_{\mathcal X}}^LD_{\mathcal X\to \mathcal Y}.
\end{equation}
The benefit of \eqref{E:SpDXY} is that $\operatorname{Sp}^\bullet _{\mathcal X\to \mathcal Y} (\mathcal M)$ provides an explicit complex to work with for push forwards, as we discuss in \S \ref{S:adhoc-push-D}, below.

\subsection{\emph{Ad hoc} push forward of $D$-modules}\label{S:adhoc-push-D}
Recall that for a proper morphism $f:X\to Y$ of smooth varieties, one has a push forward (denoted by $_Df_*$ in \cite[Def.~8.7.3]{MHM-project} and $\int_f$ in \cite[p.41]{HTT08}) 
$$
f_+:\operatorname{Mod}(D_X)^{\operatorname{right}}\longrightarrow D^b(D_Y)^{\operatorname{right}} 
$$
\begin{equation}\label{E:D:push-with-spence}
M\mapsto Rf_*\operatorname{Sp}^\bullet_{X\to Y}(M)
\end{equation}
taking right $D_X$-modules to objects in the bounded derived category of right $D_Y$-modules; note that $Rf_*$ is the derived functor $Rf_*:D^b(f^{-1}D_Y)\to D^b(D_Y)$; see e.g., \cite[p.40]{HTT08}, and note that  one wants to use the quasi-isomorphism  \eqref{E:SpDXY} to compare. 

For our purposes, rather than delving in to issues of derived categories on stacks, we make the following \emph{ad hoc} definition for the push forward of a $D$-module on a stack, which is sufficient for our purposes, and from which it will be  easy, later,  to derive the geometric consequences we wish. 

To start, we consider a fibered product diagram
\begin{equation}\label{E:PresDpf-var}
\xymatrix{
U_X \ar[r]^{p_X}\ar[d]_{f'}& X \ar[d]_f\\
U \ar[r]^p& Y
}
\end{equation}
where $p:U\to Y$ is an \'etale morphism of smooth varieties, and observe that for all right $D_X$-modules $M$, and  for all $i$, there is a natural isomorphism of right $D_U$-modules
\begin{equation}\label{E:DescDataDPush}
p^*\mathcal H^i(f_+M) \stackrel{\sim}{\longrightarrow}\mathcal H^i(f'_+p_X^*M).
\end{equation}
\ifArxiv
While this is standard, and can essentially be derived from the presentation in say \cite[\S1.5]{HTT08}, for clarity, and to compare to other related natural isomorphisms, we review this in the appendix, in  \Cref{L:CmodDescADmodHi}, as well as what we mean precisely by naturality.
\else
This is standard, and can essentially be derived from the presentation in say \cite[\S1.5]{HTT08}.
\fi

Now, let $f: \mathcal X\to \mathcal Y$ be a schematic proper morphism of smooth separated integral DM stacks locally of finite type over $\mathbb C$, and let  $\mathcal M$ be a \emph{right}  $D_{\mathcal X}$-module on $\mathcal X$.  
For every \'etale $p:U\to \mathcal Y$ from a smooth variety $U$, we consider a fiber product diagram:
\begin{equation}\label{E:TubPresDpf0}
\xymatrix{
 U_{\mathcal X} \ar[r]^{p_{\mathcal X}} \ar[d]_{f'}& \mathcal X \ar[d]_f\\
 U \ar[r]^p& \mathcal Y
}
\end{equation}
and for a fixed choice of $i$, consider the collection of $D_U$-modules $\mathcal H^i(f'_+\mathcal M|_{U_{\mathcal X}})$.  For every  fibered product diagram
\begin{equation}\label{E:TubPresDpf}
\xymatrix{
U''_{\mathcal X}\ar[r]^{\phi'_{\mathcal X}} \ar[d]_{f'''}& U'_{\mathcal X}\ar[r]^{\phi_{\mathcal X}} \ar[d]_{f''}& U_{\mathcal X} \ar[r]^{p_X} \ar[d]_{f'}& \mathcal X \ar[d]_f\\
U''\ar[r]^{\phi'} & U'\ar[r]^\phi & U \ar[r]^p& \mathcal Y
}
\end{equation}
where the bottom row consists of \'etale morphisms from smooth varieties, we have, 
using \eqref{E:DescDataDPush} applied to the $D$-module $\mathcal M|_{U_{\mathcal X}}$ on the variety $U_{\mathcal X}$, natural isomorphisms of right $D$-modules  on $U'$
$$
\tau_\phi: \phi^*\mathcal H^i(f'_+\mathcal M|_{U_{\mathcal X}})\stackrel{\sim}{\to} \mathcal H^i(f''_+\mathcal M|_{ U'_{\mathcal X}}).
$$
The fact that the isomorphisms in \eqref{E:DescDataDPush} are natural   implies 
\ifArxiv
(see, e.g., \Cref{C:PervDescADmod})
\fi
 that 
\begin{equation}\label{E:cocycle-push}
\tau_{\phi\circ \phi'}= \tau_{\phi'}\circ \phi'^*\tau_\phi.\end{equation}
In other words,  the collection of $D$-modules $\mathcal H^i(f'_+\mathcal M|_{U_{\mathcal X}})$ together with the $\tau_\phi$ give descent data, and therefore define a  $D$-module on $\mathcal Y$, which we denote
\begin{equation}\label{D:Hif*M}
\mathcal H^i(f^{ah}_+\mathcal M).
\end{equation}
In fact, one can check that the construction above defines a 
\ifArxiv
functor (see, e.g.,  \Cref{C:PervDescADmod} for details):
\else 
functor:
\fi

\begin{dfn}[\emph{Ad hoc} push forward of $D$-modules]
\label{D:ADPFDmodMain}
Let $f: \mathcal X\to \mathcal Y$ be a schematic proper morphism of smooth separated integral DM stacks locally of finite type over $\mathbb C$ (or the analytification of such a morphism). 
The construction above (i.e., culminating in \eqref{D:Hif*M}) defines a functor on $D$-modules:
\begin{equation}\label{D:Hif*MFun}
{\mathcal H^i}(f^{ah}_+(-)): \operatorname{Mod}(D_{\mathcal X})^{\operatorname{right}} \longrightarrow \operatorname{Mod}(D_{\mathcal Y})^{\operatorname{right}}
\end{equation}
$$
\mathcal M\mapsto {\mathcal{H}^i}(f^{ah}_+\mathcal M),
$$
which we call the ($i$-th) \emph{ad hoc} push forward. 
\end{dfn}

\begin{rem}
 We use the term \emph{ad hoc} to emphasize that, in order to avoid issues of descent for derived categories on stacks,  we did \emph{not} define a derived push forward ``$f_+\mathcal M$'' on the derived categories of the stacks. 
 Nevertheless, for every diagram as in \eqref{E:TubPresDpf0}, we have
$$
{\mathcal{H}^i}(f^{ah}_+\mathcal M)|_U= R^if'_*((\operatorname{Sp}^\bullet_{\mathcal X\to \mathcal Y}(\mathcal M))|_{U_X})=\mathcal H^i(f'_+\mathcal M|_{\mathcal U_{\mathcal X}}),
$$
where $R^if'_*$ is the $i$-th derived functor of  $Rf'_*:D^b(f'^{-1}D_U)\to D^b(D_{U})$. 
\ifArxiv
In the appendix, in \S \ref{S:A:BDpushforward},   we recall a  treatment of the push forward at the level of derived categories for  $D$-modules on stacks following Beilinson--Drinfeld \cite{BDstacks}, and explain 
 the connection to the \emph{ad hoc} push forward we construct here (see \Cref{C:Hif+=Hifah+}).
\fi 
\end{rem}

 Given a \emph{left}  $D_{\mathcal X}$-module $\mathcal M$ on $\mathcal X$, we define the \emph{ad hoc} push forward of $\mathcal M$ as 
\begin{equation}\label{E:D:AdHocPFDR}
\mathcal H^i(f_+^{ah}\mathcal M):= (\mathcal H^i(f_+^{ah}\mathcal M^{\operatorname{right}}))\otimes_{\mathcal O_{\mathcal Y}} \omega_{\mathcal Y}^{-1}.
\end{equation}

\subsection{Spencer complexes  for filtered $D$-modules} \label{S:dR+SCFilt}
We now discuss the Spencer complexes for filtered $D$-modules. 
Given a filtered \emph{right} $D_{\mathcal X}$-module $(\mathcal M,F_\bullet)$, we define the \emph{filtered Spencer complex of $(\mathcal M,F_\bullet)$} to be the filtered complex
$(\operatorname{Sp}^\bullet(\mathcal M),\delta,F_\bullet)$ defined by
$$
F_k\operatorname{Sp}^\bullet(\mathcal M):= 
$$
$$
\left[\xymatrix@R=0em{
\cdots 0 \ar[r] &\displaystyle F_{k-n}\mathcal M\otimes_{\mathcal O_{\mathcal X}}\bigwedge^{n}\mathcal T_{\mathcal X}  \ar[r]^<>(0.5)\delta & \displaystyle F_{k-n+1}\mathcal M \otimes_{\mathcal O_X}\bigwedge ^{n-1}\mathcal T_{\mathcal X} \ar[r]^<>(0.5)\delta & \cdots \ar[r]^<>(0.5)\delta& \mathcal F_{k}\mathcal M \ar[r] & 0\cdots \\
&&&&0
}\right]
$$

We also put a filtration on the relative Spencer complex via the tensor product as filtered complexes:
$$
(\operatorname{Sp}^\bullet_{\mathcal X\to \mathcal Y} (\mathcal M ),F_\bullet):=(\operatorname{Sp}^\bullet (\mathcal M),\delta,F_\bullet)\otimes_{f^{-1}\mathcal O_{\mathcal Y}} (f^{-1}D_{\mathcal Y},f^{-1}F_\bullet)
$$
where $(f^{-1}D_{\mathcal Y},f^{-1}F_\bullet)
$ is given the filtration defined by applying $f^{-1}$ to the standard filtration on $D_{\mathcal Y}$.  
In concrete terms
\begin{equation}\label{E:Filt-on-Rel-Sp}
F_k\operatorname{Sp}_{\mathcal X\to \mathcal Y}^\bullet(\mathcal M)=
\end{equation}

$$
\left[\xymatrix@C=1em@R=0em{
\cdots 0 \ar[r] &\displaystyle (\sum_{i+j=k-n}F_i\mathcal M\otimes_{\mathcal O_{\mathcal X}}\bigwedge^{n}\mathcal T_{\mathcal X} \otimes_{f^{-1}\mathcal O_{\mathcal Y}}f^{-1}F_jD_{\mathcal Y}) \ar[r]^<>(0.5){\delta_{\mathcal X\to \mathcal Y}} 
& \cdots & \\
&\ 
}\right.
$$
$$
\left.\xymatrix@C=1em@R=0em{
\cdots \ar[r]^<>(0.5){\delta_{\mathcal X\to \mathcal Y}}&\displaystyle \sum_{i+j =k-1}(F_i\mathcal M \otimes_{\mathcal O_X}\bigwedge ^{1}\mathcal T_{\mathcal X} \otimes_{f^{-1}\mathcal O_{\mathcal Y}}f^{-1}F_jD_{\mathcal Y})  \ar[r]^<>(0.5){\delta_{\mathcal X\to \mathcal Y}} & \displaystyle \sum_{i+j=k}(F_i\mathcal M \otimes_{f^{-1}\mathcal O_{\mathcal Y}}f^{-1}F_jD_{\mathcal Y})  \ar[r] & 0\cdots \\ 
&&0
}\right]
$$
with $ \sum_{i+j=k}(F_i\mathcal M \otimes_{f^{-1}\mathcal O_{\mathcal Y}}f^{-1}F_jD_{\mathcal Y})$ in degree $0$, and $n:=\dim \mathcal X$.

In the same way that we used the relative Spencer complex to define push forwards for $D$-modules, we will use the filtered relative Spencer complex to push forward filtered $D$-modules.  While one would like to do this by simply pushing  forward the filtered pieces of the filtered relative Spencer complex, these pieces only have an $\mathcal O_{\mathcal X}$-structure (not a $D_{\mathcal X}$-structure, as the differential operators move one between the filtered pieces).  For this reason, we will utilize the standard trick of using Rees modules.  We review this  in the next subsection.

\subsection{Rees rings and modules}\label{S:Tilde-notation-from-MHM-project}

We briefly review some of the standard treatment of Rees modules, following the presentation in \cite[\S 5]{MHM-project}.  The starting point is a $\mathbb Z$-filtered $\mathbb C$-algebra $(A,F_\bullet)$, with an increasing filtration $F_\bullet$ of sub-$\mathbb C$-vector spaces such that $\bigcup_{p\in \mathbb Z} F_pA=A$ and such that the image of the structure morphism for the algebra $\mathbb C\to A$ is contained in $F_pA$ for all $p\in \mathbb Z_{\ge 0}$. 
By convention, the filtration satisfies $F_pA\cdot F_qA \subseteq F_{p+q}A$ for all $p,q \in \mathbb Z$; in particular, $F_0A$ is a sub-$\mathbb C$-algebra of $A$.  

 We will be interested in $\mathbb Z$-filtered $(A,F_\bullet)$-modules $(M,F_\bullet)$, where $F_\bullet$ is an increasing filtration of $\mathbb C$-subspaces  such that $F_pA\cdot F_qM \subseteq F_{p+q}M$ for all $p,q \in \mathbb Z$; note that this gives all of the $F_qM$ the structure of an $F_0A$-module.   Morphisms are defined in the obvious way. 

We then consider the associated Rees ring, 
\begin{equation}\label{E:Rees-Ring}
\widetilde A:=\bigoplus_{p\in  \mathbb Z}F_pA \cdot z^p\subseteq A\otimes_{\mathbb C}\mathbb C[z,z^{-1}],
\end{equation}
 which we view as a graded ring in the obvious way.  We note that with our conventions, we naturally have an inclusion of graded rings $\mathbb C[z]\subseteq \widetilde A$.  We also obtain the Rees module associated with $(M,F_\bullet)$:
\begin{equation}\label{E:Rees-Module}
\widetilde M:=\bigoplus_{p\in \mathbb Z}F_pM \cdot z^p\subseteq M\otimes_{\mathbb C} \mathbb C[z,z^{-1}],
\end{equation}
which is naturally a graded $\widetilde A$-module. 

At the same time, given a graded $\widetilde A$-module $\widetilde N=\bigoplus_{p\in \mathbb Z} N_p$,   we are free to write it as $\widetilde N=\bigoplus_{p\in \mathbb Z} N_p\cdot z^p$. Each of the $N_p$ is an $F_0A$-module, and the inclusion $\mathbb C[z]\subseteq \widetilde {A}$ induces, by multiplication by $z$, a $\mathbb C$-linear morphism $N_p\to N_{p+1}$, for all $p\in \mathbb Z$,  which one can check is in fact $F_0A$-linear.
We say that $\widetilde N$ is \emph{strict} if all of these morphisms are injective; this is equivalent to the standard definition that $\widetilde N$ has no torsion as a $\mathbb C[z]$-module.  
   The $F_0A$-module $\varinjlim N_p$ is naturally an $A$-module, with a filtration $F_\bullet \varinjlim N_p$ given by the images of the canonical morphisms 
\begin{equation}\label{E:FiltOnInjLim}
 F_q \varinjlim N_p :=\operatorname{Im}\left(N_q\to \varinjlim N_p\right).
\end{equation}
      The main point is that $\widetilde N$ is strict if and only if it is the Rees module of some filtered $(A,F_\bullet)$-module $(M,F_\bullet)$; in this case one has $\varinjlim N_p=M$, and the morphisms in \eqref{E:FiltOnInjLim} are injective for all $q$, so that $ F_q \varinjlim N_p  = N_q= F_qM$  (e.g., \cite[Prop.~5.1.9]{MHM-project}).
   
 Finally, for a filtered $(A,F_\bullet)$-module $(M,F_\bullet)$, one has the identities \cite[Exe.~5.2(5)]{MHM-project}

\begin{equation}\label{E:zTrick}
M=\widetilde M/(1-z)\widetilde M, \quad \operatorname{gr}^F_\bullet M = \widetilde M/z\widetilde M=\operatorname{gr}\widetilde M.
\end{equation}

We will use the same notation for Rees algebras and modules in the setting of sheaves, as well as in the setting of complexes of sheaves.

\subsection{\emph{Ad hoc} push forward of filtered $D$-modules}
Recall that for a proper morphism $f:X\to Y$ of smooth varieties, one has a push forward $f_+$ (denoted by $_Df_*$ in \cite[Def.~8.7.3]{MHM-project}) of  $\widetilde {D_X}$-modules

\begin{equation}\label{E:Rf*tildeMdef}
f_+\widetilde{M}:=Rf_*({\operatorname{Sp}^\bullet_{X\to Y}(\widetilde M)})
\end{equation}
where   $Rf_*$ is the derived functor $Rf_*: D^b(f^{-1}\widetilde {D_Y})\to D^b(\widetilde {D_Y})$, and we refer the reader to \cite[8.7.a]{MHM-project} for the definition of ${\operatorname{Sp}^\bullet_{X\to Y}(\widetilde M)}$.  Note that for our purposes, we will always be in the situation where   $\widetilde M$ is the Rees module of a filtered module $(M,F_\bullet)$, in which case we have the identification $\operatorname{Sp}^\bullet_{X\to Y}(\widetilde M) =\widetilde{\operatorname{Sp}^\bullet_{X\to Y}(M)}$, where  the latter complex is given by the Rees construction applied to  the filtration $F_\bullet \operatorname{Sp}^\bullet_{X\to Y}(M)$ defined in \eqref{E:Filt-on-Rel-Sp}. 

To convert $f_+\widetilde{M}$ from the category of $\widetilde D_Y$-modules to the category of filtered $D_Y$-modules, we can take limits as described in \S \ref{S:Tilde-notation-from-MHM-project}, above; however, to avoid taking limits in derived categories, we prefer here to take cohomology, and use $i$-th derived push forward functors.  In other words,  we obtain a  push forward on filtered $D$-modules \begin{equation}\label{E:push-filt-lim}
 \mathcal H^i f_+(M,F_\bullet):= (\varinjlim (R^if_*(\widetilde{\operatorname{Sp}^\bullet_{X\to Y}(M)}))_k,F_\bullet)
\end{equation}
where the limit and the filtration are described in the previous sub-section in \eqref{E:FiltOnInjLim} in the context of Rees modules.

We want to describe \eqref{E:push-filt-lim} more concretely.  To begin,  the terms on the right hand side of  \eqref{E:push-filt-lim} can be described as:
\begin{equation}\label{E:filt-is-naive-OO}
(R^if_*(\widetilde{\operatorname{Sp}^\bullet_{X\to Y}(M)}))_k = R^if_*(F_k\operatorname{Sp}^\bullet_{X\to  Y}( M)),
\end{equation}
where on the right in \eqref{E:filt-is-naive-OO}, the functor $R^if_*$ is the $i$-th derived functor  
\ifArxiv
$Rf_*:D^b(\mathcal O_X)\to D^b(\mathcal O_Y)$  
(see \Cref{FilteredDmodPush}).
\else
$Rf_*:D^b(\mathcal O_X)\to D^b(\mathcal O_Y)$.
\fi
Then using the fact that, in the situation here, colimits commute with higher direct images 
\ifArxiv
(e.g., \cite[\href{https://stacks.math.columbia.edu/tag/07TA}{\S 07TA}]{stacks-project}; see \Cref{L:limRif=Rif} for the details in this setting), 
\else
(e.g., \cite[\href{https://stacks.math.columbia.edu/tag/07TA}{\S 07TA}]{stacks-project}), 
\fi
we have an identification $\varinjlim (R^if_*(\widetilde{\operatorname{Sp}^\bullet_{X\to Y}(M)}))_k= R^if_*\operatorname{Sp}^\bullet_{X\to Y}(M)$.
Moreover, there is a canonical identification of $R^if_*\operatorname{Sp}^\bullet_{X\to Y}(M)$ as the underlying $\mathcal O_Y$-module of the $D_Y$-module $R^if_*\operatorname{Sp}^\bullet_{X\to Y}(M)$ where here,  in this second expression, we mean the $i$-th derived functor  
\ifArxiv
$Rf_*:D^b(f^{-1}D_Y)\to D^b(D_Y)$ (see \Cref{FilteredDmodPush}).
\else
$Rf_*:D^b(f^{-1}D_Y)\to D^b(D_Y)$.
\fi 

In summary,  \eqref{E:push-filt-lim}  is defining a filtration on the $D_Y$-module $\mathcal H^if_+M$ of \Cref{D:ADPFDmodMain}, where, considering the Rees module construction of the filtration on the limit (e.g., \eqref{E:FiltOnInjLim}), the filtration on $\mathcal H^if_+M$ in \eqref{E:push-filt-lim} 
is given by the image of the natural map 
\begin{equation}\label{E:concrete-filt-f+i}
F_k\mathcal H^i(f_+(M,F_\bullet)):=\operatorname{Im}\left(R^if_*(F_k\operatorname{Sp}^\bullet_{ X\to  Y}( M))\rightarrow R^if_*(\operatorname{Sp}^\bullet_{ X\to  Y}( M))=\mathcal H^i f_+M\right),
\end{equation}
where, here, again,  $R^if_*$ is the push forward at the level of $\mathcal O_X$-modules. 

In other words, denoting by $MF(D_X)$  the category of pairs $(M,F_\bullet)$ such that $M$ is a $D_{X}$-module and $F_\bullet$ is a filtration on $\mathcal M$, 
then if $f:X\to  Y$ is  schematic proper morphism  of varieties,  then for each $i$ we have a functor
\begin{equation*}
\mathcal H^if_+: MF(D_{\mathcal X})\longrightarrow MF(D_{\mathcal Y})
\end{equation*}
\begin{equation}\label{E:PushFiltDBody}
\mathcal H^if_+(M,F_\bullet):=(\mathcal H^if_+M,F_\bullet)
\end{equation}   
where  the filtration on $\mathcal H^if_+M$ is defined in \eqref{E:concrete-filt-f+i}.  
We note that if $(M,F_\bullet)$ is coherent (resp.~regular holonomic), then so is $\mathcal H^if_+(M,F_\bullet)$.

\begin{rem}
There are constructions of  push forwards $f_+(M,F_\bullet)$ at the level of derived categories, e.g., \cite[\S 10.4]{MHM-project}, so that the image is a filtered object in the derived category of $D_Y$-modules.  One then has that the $i$-th cohomology of that object agrees with $\mathcal H^if_+(M,F_\bullet)$ as defined above. 
Concretely, one defines this as $Rf_*\operatorname{Sp}^\bullet _{  X\to   Y}(  M^{\operatorname{right}})$ where here $Rf_*$ is a functor $Rf_*:D^b(Ff^{-1}D_Y)\to D^b(FD_Y)$  between  categories as described in  \cite[\S 10.4]{MHM-project}.
 For our purposes, where we will be working on stacks, it will be easier to work directly as above, with $\mathcal H^if_+(M,F_\bullet)$, rather than at the level of derived categories.  
\end{rem}

\begin{rem}\label{R:Hodge-Strict-Saito}
From the definition of strictness (see \S \ref{S:Tilde-notation-from-MHM-project}), for a filtered $D_X$-module $(M,F_\bullet)$, we see that  $\mathcal H^i(f_+\widetilde M)$ being strict is equivalent to the morphisms $(\mathcal H^i(f_+\widetilde M))_q\to \varinjlim (\mathcal H^i(f_+\widetilde M))_p = \mathcal H^if_+M $ being injective.  Putting together
\eqref{E:filt-is-naive-OO} with \eqref{E:concrete-filt-f+i}, we see that 
this is equivalent to the morphisms in  \eqref{E:concrete-filt-f+i} being injective,  i.e.,  equivalent to 
$$
F_k\mathcal H^i(f_+(M,F_\bullet))=R^if_*(F_k\operatorname{Sp}^\bullet_{ X\to  Y}( M)).
$$
This last equality does not always hold. However, it is known to hold when $(M,F_{\bullet})$ underlies a Hodge module \cite[Thm. 1(i)]{Saito}.
\end{rem}

To construct an \emph{ad hoc} push forward of filtered $D$-modules on stacks, we
start by considering diagram  \eqref{E:PresDpf-var}, and observe that for all right  $\widetilde {D_X}$-modules $\widetilde M$, and  for all $i$, there is a natural isomorphism of right $\widetilde{D_Y}$-modules
\begin{equation}\label{E:DescDataDPushTilde}
p^*\mathcal H^i(f_+\widetilde{M}) \stackrel{\sim}{\longrightarrow}\mathcal H^i(f'_+p_X^*(\widetilde{M})).
\end{equation}

One establishes this in the same way as for \eqref{E:DescDataDPush}, but now in the Rees module setting.  
Moreover, for filtered right  $D_X$-modules $(M,F_\bullet)$, and  for all $i$, this induces natural isomorphism
 of filtered right ${D_Y}$-modules
\begin{equation}\label{E:DescDataDPushFilt}
p^*\mathcal H^i(f_+(M,F_\bullet)) \stackrel{\sim}{\longrightarrow}\mathcal H^i(f'_+p_X^*(M,F_\bullet))
\end{equation}
 by taking limits using the isomorphism of Rees modules \eqref{E:DescDataDPushTilde},  and then using \eqref{E:concrete-filt-f+i} for the statement on the filtrations.

Now, let $f: \mathcal X\to \mathcal Y$ be a schematic proper morphism of smooth separated integral DM stacks locally of finite type over $\mathbb C$ (or the analytification of such a morphism), and let  $(\mathcal M,F_\bullet)$ be a filtered \emph{right}  $D_{\mathcal X}$-module on $\mathcal X$.  
We consider a fiber product diagram as in \eqref{E:TubPresDpf0}, 
and for a fixed choice of $i$, consider the collection of $D$-modules $\mathcal H^i(f'_+(\mathcal M,F_\bullet)|_{U_{\mathcal X}})$.  For every  fibered product diagram as in \eqref{E:TubPresDpf}, 
we have, 
using \eqref{E:DescDataDPushFilt} applied to the filtered $D$-module $(\mathcal M,F_\bullet)|_{U_{\mathcal X}}$ on the variety $U_{\mathcal X}$, natural isomorphisms of filtered right $D$-modules  on $U'$
$$
\tau_\phi: \phi^*\mathcal H^i(f'_+(\mathcal M,F_\bullet)|_{U_{\mathcal X}})\stackrel{\sim}{\to} \mathcal H^i(f''_+(\mathcal M,F_\bullet) |_{ U'_{\mathcal X}}).
$$
The fact that the isomorphisms \eqref{E:DescDataDPushFilt} are canonical implies 
\ifArxiv
 (see, e.g., \Cref{R:FiltahPF} for details) 
\fi
that  $ \tau_{\phi\circ \phi'}= \tau_{\phi'}\circ \phi'^*\tau_\phi$. 
In other words,  the collection of filtered $D$-modules $\mathcal H^i(f'_+(\mathcal M,F_\bullet)|_{U_{\mathcal X}})$ together with the $\tau_\phi$ give descent data, and therefore define a filtered  $D$-module on 
 $\mathcal Y$, which we denote
\begin{equation}\label{D:Hif*MFilt}
\mathcal H^if^{ah}_+(\mathcal M,F_\bullet).
\end{equation}
In fact, we obtain a 
\ifArxiv
functor (see e.g., \eqref{E:PushFiltD} and \Cref{R:FiltahPF} for details):
\else
functor:
\fi 

\begin{dfn} [\emph{Ad hoc} push forward of filtered $D$-modules]
 \label{E:AHPFfdm}
 Let $f: \mathcal X\to \mathcal Y$ be a schematic proper morphism of smooth separated integral DM stacks locally of finite type over $\mathbb C$ (or the analytification of such a morphism).
 The construction above (i.e., culminating in \eqref{D:Hif*MFilt}) defines a functor on filtered $D$-modules:
\begin{equation*}
\mathcal H^if^{ah}_+: MF(D_{\mathcal X})^{\operatorname{right}}\longrightarrow MF(D_{\mathcal Y})^{\operatorname{right}}
\end{equation*}
$$
(\mathcal M,F_\bullet) \mapsto {\mathcal{H}^i}f^{ah}_+(\mathcal M,F_\bullet),
$$
which we call the  ($i$-th)  \emph{ad hoc} push forward. 
\end{dfn}

 Given a \emph{left}  $D_{\mathcal X}$-module $\mathcal M$ on $\mathcal X$, we define the  ($i$-th)  \emph{ad hoc}  push forward of $\mathcal M$ as 
\begin{equation}\label{E:D:AdHocPFDRFilt}
\mathcal H^i(f_+^{ah}(\mathcal M ,F_\bullet)):= \mathcal H^i(f_+^{ah}(\mathcal M^{\operatorname{right}},F_\bullet))\otimes_{\mathcal O_{\mathcal Y}} \omega_{\mathcal Y}^{-1}.
\end{equation}

If we consider the category of filtered regular holonomic $D$-modules, then the \emph{ad hoc} push forward of filtered $D$-modules defines a functor on filtered regular holonomic $D$-modules:
\begin{equation}\label{D:Hif*MFunRH}
{\mathcal H^i}(f^{ah}_+(-)): MF_{{rh}}(D_{\mathcal X})^{\operatorname{right}} \longrightarrow MF_{{rh}}(D_{\mathcal Y})^{\operatorname{right}},
\end{equation}
and similarly for left filtered regular holonomic $D$-modules. 

\subsection{\emph{Ad hoc} push forward of Hodge modules}
Recall that for a proper morphism $f:X\to Y$ of smooth varieties, one has a push forward  (denoted by $f_\star$ in  \cite[p.233]{HTT08}) 
$$
f_+:D^b(\operatorname{MHM}(X))\longrightarrow D^b(\operatorname{MHM}(Y))
$$
on the bounded derived categories of mixed Hodge modules.  If $\mathsf M$ is a pure Hodge module of weight $w$ on $X$ with underlying filtered $D_X$-module $(M,F_\bullet)$, then the underlying filtered $D_Y$-module of $f_+\mathsf M$ is given by the filtered $D_Y$-module $f_+(M,F_\bullet)$ described above, e.g., \cite[p.222]{HTT08}. 
Recall also from \cite[(4.3.3) and (4.5.2)]{Saito1990} that if $\mathsf M$ is a pure Hodge module of weight $w$ on $X$, then for each $i$ one has that $\mathcal H^if_+\mathsf M$ is a pure Hodge module of weight $w+i$.

Now, considering any diagram as in \eqref{E:PresDpf-var}, we observe there is a natural isomorphism of pure Hodge modules of weight $w+i$:
\begin{equation}\label{E:DescDataDPushHdg}
p^*\mathcal H^i(f_+\mathsf M) \stackrel{\sim}{\longrightarrow}\mathcal H^i(f'_+p_X^*\mathsf M).
\end{equation}
\ifArxiv
While this is standard,  for clarity, and to compare to other related natural isomorphisms, we review this in the appendix, in  \Cref{R:FiltahPFHodge}. 
\fi 

Now, let $f: \mathcal X\to \mathcal Y$ be a schematic proper morphism of smooth separated integral DM stacks locally of finite type over $\mathbb C$.
We consider a fiber product diagram as in \eqref{E:TubPresDpf0}, 
and for a fixed choice of $i$, consider the collection of pure Hodge modules $\mathcal H^i(f'_+\mathsf M|_{U_{\mathcal X}})$ of weight $w+i$.  For every  fibered product diagram as in \eqref{E:TubPresDpf}, 
we have, 
using \eqref{E:DescDataDPushHdg} applied to the pure Hodge module $\mathsf M|_{U_{\mathcal X}}$  on the variety $U_{\mathcal X}$, natural isomorphisms of pure Hodge modules  on $U'$
$$
\tau_\phi: \phi^*\mathcal H^i(f'_+\mathsf M|_{U_{\mathcal X}})\stackrel{\sim}{\to} \mathcal H^i(f''_+\mathsf M |_{ U'_{\mathcal X}}).
$$
The fact that the isomorphisms \eqref{E:DescDataDPushHdg} are canonical implies 
\ifArxiv
(see, e.g., \Cref{R:FiltahPFHodge})
\fi
that 
$
\tau_{\phi\circ \phi'}= \tau_{\phi'}\circ \phi'^*\tau_\phi$.  
In other words,  the collection of pure Hodge modules  $\mathcal H^i(f'_+\mathsf M|_{U_{\mathcal X}})$ together with the $\tau_\phi$ give descent data, and therefore define a pure Hodge module of weight $w+i$ on 
 $\mathcal Y$, which we denote
\begin{equation}\label{D:Hif*MHdg}
\mathcal H^i(f^{ah}_+\mathsf M),
\end{equation}
and call the ($i$-th) \emph{ad hoc} push forward of  $\mathsf M$.

In fact, we obtain a 
\ifArxiv
functor (see e.g., \eqref{E:Hif+Hodge} and \Cref{R:FiltahPFHodge} for details):
\else
functor:
\fi 

\begin{dfn} [\emph{Ad hoc} push forward of Hodge modules]
 \label{D:AHPFGHodge}
 Let $f: \mathcal X\to \mathcal Y$ be a schematic proper morphism of smooth separated integral DM stacks locally of finite type over $\mathbb C$. 
 The construction above (i.e., culminating in \eqref{D:Hif*MHdg}) defines a functor on pure Hodge modules:
\begin{equation*}
\mathcal H^if^{ah}_+: \mathsf {HM}(\mathcal X,w)\longrightarrow \mathsf {HM}(\mathcal Y,w+i)
\end{equation*}
$$
\mathsf  M\mapsto {\mathcal{H}^i}f^{ah}_+\mathsf M,
$$
which we call the  ($i$-th)  \emph{ad hoc} push forward. 
\end{dfn}

\begin{rem}
By construction, if $(\mathcal M,F_\bullet)$ is  the underlying filtered regular holonomic $D$-module  of $\mathsf M$, then  the underlying filtered regular holonomic $D$-module of the  ($i$-th)  \emph{ad hoc} push forward, $\mathcal H^if^{ah}_+\mathsf M$, is given by $\mathcal H^i(f^{ah}_+(\mathcal M, F_\bullet))$ of \Cref{E:AHPFfdm}.
\end{rem}

\subsubsection{Connection to the push forward functor in \cite{tubach_2024}}\label{S:tubachf*}

Swann Tubach explained to us the following connection between our \emph{ad hoc} push forward and the push forward he constructs in \cite{tubach_2024}. 
The starting point is that for a scheme $X$, there are two  $t$-structures on  Saito's bounded derived category of mixed Hodge modules $D^b(\mathsf {MHM}(X))$ on $X$,
namely, Saito's perverse $t$-structure, as well as what Tubach calls the standard (i.e., constructible) $t$-structure.
Denoting by $\mathsf {MHM}_{\operatorname{std}}(X)$ the heart of this standard $t$-structure, there is a canonical equivalence $D^b(\mathsf {MHM}_{\operatorname{std}}(X))\simeq D^b(\mathsf {MHM}(X))$; see \cite[Thm.~2.11]{tubach_2024}. We emphasize that this standard $t$-structure, in general, is not the same as the perverse $t$-structure. 

Tubach then defines a functor on $\mathbb C$-schemes, $\operatorname{D}^b_{H,c}(-)$, as the stable $\infty$-categorical bounded derived category associated to the heart of the standard $t$-structure; i.e., on a scheme $X$, the $\infty$-category $\operatorname{D}^b_{H,c}(X):=D^b_\infty(\mathsf {MHM}_{\operatorname{std}}(X))$ can be  modeled by the sub-$\infty$-category consisting of the bounded objects in the $\infty$-derived category 
$N_{dg}(\operatorname{Ch}_{dg}(\mathsf {MHM}_{\operatorname{std}}(X)))[W^{-1}]$. 
Recall that the $\infty$-homotopy category $K_\infty(\mathsf {MHM}_{\operatorname{std}}(X))$ can be modeled on $N_{dg}(\operatorname{Ch}_{dg}(\mathsf {MHM}_{\operatorname{std}}(X)))$, the differential graded nerve of the differential graded category of chain complexes of objects in $\mathsf {MHM}_{\operatorname{std}}(X)$, and there is a canonical equivalence (e.g., \cite[Rem.~1.3.1.11]{LurieHA}) $hK_\infty(\mathsf {MHM}_{\operatorname{std}}(X))\simeq K(\mathsf {MHM}_{\operatorname{std}}(X))$ of the homotopy category of the $\infty$-homotopy category with the classical homotopy category.    
The $\infty$-derived category $D_\infty(\mathsf {MHM}_{\operatorname{std}}(X))$ can be modeled on $N_{dg}(\operatorname{Ch}_{dg}(\mathsf {MHM}_{\operatorname{std}}(X)))[W^{-1}]$, the localization of the $\infty$-homotopy category $K_\infty(\mathsf {MHM}_{\operatorname{std}}(X))$ at the
 collection of quasi-isomorphisms $W$, and since taking homotopy categories commutes with localization (e.g., \cite[\href{https://kerodon.net/tag/01MV}{Rem.~01MV}]{kerodon}), one has a canonical equivalence $h\operatorname{D}_{H,c}(X)\simeq D(\mathsf {MHM}_{\operatorname{std}}(X))$.  The bounded objects in the $\infty$-derived category  are the complexes with cohomology in only finitely many degrees.
Consequently, together with Tubach's result above, 
one has canonical equivalences
\begin{equation}\label{E:hDbHcDbDb}
h\operatorname{D}^b_{H,c}(X)\simeq D^b(\mathsf {MHM}_{\operatorname{std}}(X))\simeq D^b(\mathsf {MHM}(X)).
\end{equation}

The functor $\operatorname{D}^b_{H,c}(-)$ defines an \'etale hypersheaf (even $h$-hypersheaf) on the category of $\mathbb C$-schemes (see \cite[\S 4.1]{tubach_2024}). Considering the category of $\mathbb C$-schemes as a full subcategory of the category of DM stacks over $\mathbb C$, one can take the right Kan extension  of the functor $\operatorname{D}^b_{H,c}(-)$ to obtain a functor, which we also denote by $\operatorname{D}^b_{H,c}(-)$, on the category of DM stacks over $\mathbb C$. Concretely, in the situation of a smooth separated DM stack $\mathcal X$ of finite type over $\mathbb C$, restricting $\operatorname{D}^b_{H,c}(-)$ to the category of smooth varieties, \'etale over $\mathcal X$, and using either $f^!$ or $f^*$ functoriality in \cite[Thm.~4.1.8]{tubach_2024}, we have (e.g., \cite[\href{https://kerodon.net/tag/02Y9}{Def.~02Y9}]{kerodon}) that $\operatorname{D}^b_{H,c}(\mathcal X)$ is the $\infty$-categorical limit of the $\operatorname{D}^b_{H,c}(U)$ for $U\to \mathcal X$ \'etale; i.e., given an \'etale presentation $p:U\to \mathcal X$, we have (e.g., \cite[Thm.~A.6]{aoki23}) 
\begin{equation}\label{E:DbHclim}
\operatorname{D}^b_{H,c}(\mathcal X)=\lim_{\Delta}\operatorname{D}^b_{H,c}(U_\bullet),
\end{equation} 
where $U_n=U\times_{\mathcal X}\cdots \times_{\mathcal X}U$ is the $n$-fold fibered product.   
One can therefore identify $\operatorname{D}^b_{H,c}(\mathcal X)$ with the category of compatible systems of objects in $\operatorname{D}^b_{H,c}(U)$; for example one may use the fact that the limit may be computed as the cocartesian sections of the associated cofibration (e.g., \cite[\href{https://kerodon.net/tag/02TK}{Cor.~02TK}]{kerodon}). 

Using \eqref{E:hDbHcDbDb}, it is clear that for a scheme $X$ there are two natural $t$-structures on $\operatorname{D}^b_{H,c}( X)$, namely the one given by the standard  $t$-structure  on $h\operatorname{D}^b_{H,c}(X)\simeq D^b(\mathsf {MHM}(X))$, and the one given by the perverse $t$-structure. 
Since a limit of stable $\infty$-categories with $t$-structures and $t$-exact
transition functors is endowed with a canonical $t$-structure (e.g., \cite[Lem.~3.2.18]{RS20Intersection}), we see from \eqref{E:DbHclim} that the $t$-structures defined for schemes naturally induce $t$-structures on $\operatorname{D}^b_{H,c}(\mathcal X)$ (all covers are \'etale and so the pull backs are $t$-exact for both $t$-structures), which we call the standard and perverse $t$-structures.   
The two halves of the induced
$t$-structure are detected component-wise (e.g.,  \cite[Lem.~3.2.18(i)]{RS20Intersection});  i.e., an object
$
K=(K_n)
$ 
in 
$\lim_{\Delta}\operatorname D^b_{H,c}(U_\bullet)$ (viewed as an object in $h\operatorname{D}^b_{H,c}(\mathcal X)$) 
lies in the perverse heart
 if and only if
$K_n$ is in $\mathsf{MHM}(U_n)$ 
for every $n$.  Consequently, 
\begin{equation}\label{E:limMHMdesc}
(h\operatorname D^b_{H,c}(\mathcal X))^{\heartsuit_p}
\simeq
\lim_{\Delta}\mathsf{MHM}(U_\bullet).
\end{equation}
The right-hand side is precisely the category $\mathsf {MHM}(\mathcal X)$ of mixed Hodge modules on
$\mathcal X$ defined by \'etale descent (see, e.g., the proof of \cite[\href{https://stacks.math.columbia.edu/tag/0D7I}{Lem.~0D7I}]{stacks-project}). 
 Consequently, taking perverse cohomology, one obtains for each $i$ a functor 
 \begin{equation}\label{E:HiTubach}
 \mathcal H^i:h\operatorname{D}^b_{H,c}(\mathcal X) \to \mathsf {MHM}(\mathcal X).
\end{equation}

In \cite[Thm.~4.1.8]{tubach_2024}, Tubach also provides a $6$-functor formalism in the $\infty$-category setting.  For this he introduces a related  functor \cite[Def.~4.1.5]{tubach_2024}, denoted $\operatorname{D}_{H,c}(-)$, of cohomologically constructible \emph{unbounded} complexes, and in particular, defines push forward functors for $\operatorname{D}_{H,c}(-)$ under representable morphisms.  For a schematic morphism, one can check on charts that this push forward takes bounded complexes to bounded complexes, and so restricts to give a push forward for the functor  $\operatorname{D}^b_{H,c}(-)$.
In other words, in the situation in \Cref{D:AHPFGHodge}, where one is given 
a schematic proper morphism $f:\mathcal X\to \mathcal Y$ of smooth separated integral DM stacks locally of finite type over $\mathbb C$, there is a push forward functor $$ f_*:\operatorname{D}^b_{H,c}(\mathcal X)\longrightarrow \operatorname{D}^b_{H,c}(\mathcal Y), $$ which induces a push forward functor on homotopy categories $$ hf_*:h\operatorname{D}^b_{H,c}(\mathcal X)\longrightarrow h\operatorname{D}^b_{H,c}(\mathcal Y). $$ In the situation where $f:X\to Y$ is a proper morphism of varieties, then $hf_*$ agrees with Saito's push forward functor; i.e., after the canonical  identifications $D^b(\mathsf {MHM}(X))\simeq h\operatorname{D}^b_{H,c}(X)$ and $D^b(\mathsf {MHM}(Y))\simeq h\operatorname{D}^b_{H,c}(Y)$, one has a canonical identification of functors
\begin{equation}\label{E:hf*=f+} 
hf_*=f_+:D^b(\mathsf{MHM}(X))\longrightarrow D^b(\mathsf{MHM}(Y)). 
\end{equation} 
Viewing $\mathsf {HM}(\mathcal X)$ as a subcategory of $h\operatorname{D}^b_{H,c}(\mathcal X)$, we then have a functor 
$$ 
\mathcal H^ihf_*:\mathsf {HM}(\mathcal X)\longrightarrow \mathsf {MHM}(\mathcal Y),
$$ 
and, in fact, one has a canonical identification of functors 
\begin{equation}\label{E:Hihf*=Hifah+}
\mathcal H^ihf_*(-)= \mathcal H^if^{ah}_+(-): \mathsf {HM}(\mathcal X,w)\longrightarrow \mathsf {HM}(\mathcal Y,w+i),
\end{equation}
 i.e., the \emph{ad hoc} push forward $\mathcal H^if^{ah}_+(-)$ on Hodge modules that we constructed in \Cref{D:AHPFGHodge} agrees with Tubach's push forward $\mathcal H^ihf_*(-)$, when restricted to Hodge modules. 

We sketch the argument here.  
Given an \'etale presentation 
$
p:U\rightarrow\mathcal Y
$
and the cartesian square \eqref{E:TubPresDpf0}, 
then for  $\mathsf M$ in $\mathsf{HM}(\mathcal X)$, one has canonical isomorphisms
\begin{equation}\label{E:CanIsofhf+}
p^*\mathcal H^i(hf_*\mathsf M)
\cong
\mathcal H^i(p^*hf_*\mathsf M)
\cong
\mathcal H^i(hf'_*p_{\mathcal X}^*\mathsf M)\cong \mathcal H^i(f'_+\mathsf M|_{U_{\mathcal X}}),
\end{equation}
where the  first isomorphism follows from the perverse $t$-exactness of \'etale
pull back, the second follows from \'etale base change, and the last is from the identification in \eqref{E:hf*=f+}.

As $U\to\mathcal Y$ varies, the objects $ \mathcal H^i(f'_+\mathsf M|_{U_{\mathcal X}})$ in $
\mathsf{MHM}(U)$
form a compatible system, determining the object
$\mathcal H^i(hf_*\mathsf M)$.
The $ \mathcal H^i(f'_+\mathsf M|_{U_{\mathcal X}})$ all in fact lie in $\mathsf {HM}(U)$, and so the compatible system can also be identified with   
an
 object of $\lim_{(U\to\mathcal Y)_{\mathrm{\acute et}}}
\mathsf{HM}(U)$.    Under the equivalence $
\mathsf{HM}(\mathcal Y)
\simeq
\lim_{(U\to\mathcal Y)_{\mathrm{\acute et}}}
\mathsf{HM}(U),
$
we see that 
$\mathcal H^i(hf_*\mathsf M)$ determines a  Hodge module on $\mathcal Y$.  
The $ \mathcal H^i(f'_+\mathsf M|_{U_{\mathcal X}})$ also give the compatible system used in the definition of the \emph{ad hoc} 
pushforward $
\mathcal H^i f^{ah}_+(\mathsf M)$ in \Cref{D:AHPFGHodge}. 
Therefore, there is a canonical isomorphism 
$\mathcal H^i(hf_*\mathsf M)\cong 
\mathcal H^i f^{ah}_+(\mathsf M)$. A similar argument shows the compatibility of $\mathcal H^i(hf_*(-))$ and $ 
\mathcal H^i f^{ah}_+(-)$ on morphisms.

\subsection{Associated graded Spencer complexes} 
The associated graded Spencer complex of a  \emph{right}  $D_{\mathcal X}$-module $\mathcal M$, namely $\operatorname{gr}_\bullet^F \operatorname{Sp}^\bullet (\mathcal M)$, has  $k$-th graded piece given by 
$$
\operatorname{gr}_k^F \operatorname{Sp}^\bullet (\mathcal M)=
\left[
\xymatrix@C=.8em@R=0em{
\cdots 0 \ar[r] &\displaystyle \operatorname{gr}^F_{k-n}\mathcal M\otimes_{\mathcal O_{\mathcal X}}\bigwedge^{n}\mathcal T_{\mathcal X}  \ar[r]^<>(0.5)\delta & \displaystyle \operatorname{gr}^F_{k-n+1} {\mathcal M} \otimes_{\mathcal O_X}\bigwedge ^{n-1}\mathcal T_{\mathcal X} \ar[r]^<>(0.5)\delta & \cdots \ar[r]^<>(0.5)\delta& \operatorname{gr}_k^F\mathcal M \ar[r] 
& 0\cdots \\
&&&&0
}\right],
$$
with $n=\dim \mathcal X$.

For the associated graded of the relative Spencer complex,  there is a quasi-isomorphism (see \cite[Exe.~8.55(1)]{MHM-project})
\begin{equation}\label{E:MHMp8.55}
\operatorname{gr}_\bullet^F\operatorname{Sp}^\bullet_{\mathcal X\to \mathcal Y}(\mathcal M)\simeq \operatorname{gr}_\bullet^F \mathcal M\otimes^L_{\mathcal A_{\mathcal X}}f^{*}\mathscr A_{\mathcal Y}.
\end{equation}
Note that in the expression above on the right, we are using the identification between quasi-coherent sheaves on the tangent stack $T_{\mathcal X}$ and quasi-coherent $\mathscr A_{\mathcal X}$-sheaves on $\mathcal X$ to obtain the derived tensor product over $\mathscr A_{\mathcal X}$ (there is a derived tensor product on $T_{\mathcal X}$ over $\mathcal O_{T_{\mathcal X}}$; see e.g., \cite[Prop.~(13.2.6)]{LMB}).

While the associated graded relative Spencer complex  may be somewhat difficult to describe concretely from  \eqref{E:MHMp8.55},  it has a simpler form in the special case where there is an integer $d$ such that 
$$
F_k\mathcal M=\left\{
\begin{array}{ll}
0 & k<d,\\
\mathcal M & d\le k,
\end{array}
\right.
$$
so that $\operatorname{gr}^F_\bullet\mathcal M$ is equal to $\mathcal M$ in degree $d$; in this case we get, in concrete terms
\begin{equation}\label{E:grSpRelM}
\operatorname{gr}^F_k\operatorname{Sp}_{\mathcal X\to \mathcal Y}^\bullet(\mathcal M)=
\end{equation}

$$
\left[\xymatrix@C=1.5em@R=0em{
\cdots 0 \ar[r] &\displaystyle \mathcal M\otimes_{\mathcal O_{\mathcal X}} f^*\mathscr A_{\mathcal Y}^{d+k-n} \otimes_{\mathcal O_{\mathcal X}}\bigwedge^{n}\mathcal T_{\mathcal X} \ar[r]^<>(0.5){\delta_{\mathcal X\to \mathcal Y}} 
& \displaystyle \mathcal M  \otimes_{\mathcal O_{\mathcal X}} f^*\mathscr A_{\mathcal Y}^{d+k-n+1}\otimes_{\mathcal O_{\mathcal X}}  \bigwedge ^{n-1}\mathcal T_{\mathcal X}   \ar[r]^<>(0.5){\delta_{\mathcal X\to \mathcal Y}} 
& \cdots & \\
&\ 
}\right.
$$
$$
\left.\xymatrix@C=1.5em@R=0em{
\cdots \ar[r]^<>(0.5){\delta_{\mathcal X\to \mathcal Y}}&\displaystyle \mathcal M \otimes_{\mathcal O_{\mathcal X}} f^*\mathscr A_{\mathcal Y}^{d+k-1}  \otimes_{\mathcal O_X}\bigwedge ^{1}\mathcal T_{\mathcal X}   \ar[r]^<>(0.5){\delta_{\mathcal X\to \mathcal Y}} & \displaystyle \mathcal M \otimes_{\mathcal O_{\mathcal X}} f^*\mathscr A_{\mathcal Y}^{d+k}   \ar[r] & 0\cdots \\ 
&&0
}\right]
$$
with $\mathcal M \otimes_{\mathcal O_{\mathcal X}} f^*\mathscr A_{\mathcal Y}^{d+k}$ in degree $0$, and $n:=\dim \mathcal X$.
Note that the differential $\delta_{\mathcal X\to \mathcal Y}$ in the complex above agrees with the one induced by the natural map $\mathcal T_{\mathcal X}\to f^*\mathcal T_{\mathcal Y}$.

For later use, following \cite[\S 2.3]{PS17} we define a complex of graded $f^*\mathscr A_{\mathcal Y}$-modules: 
\begin{equation}\label{D:CXXYY}
C_{\mathcal X\to \mathcal Y,k}:= \operatorname{gr}^F_{k-m+n}\operatorname{Sp}_{\mathcal X\to \mathcal Y}^\bullet(\mathcal O _{\mathcal X})=
\end{equation}
\begin{equation*}
\left[
\xymatrix@C=1em@R=0em{
\cdots 0 \ar[r] &\displaystyle f^*\mathscr A^{k-m}_{\mathcal Y}\otimes_{\mathcal O_{\mathcal X}}\bigwedge^{n}\mathcal T_{\mathcal X}  \ar[r]^<>(0.5){\delta_{\mathcal X\to \mathcal Y}} & \displaystyle f^*\mathscr A^{k-m+1}_{\mathcal Y} \otimes_{\mathcal O_X}\bigwedge ^{n-1}\mathcal T_{\mathcal X} \ar[r]^<>(0.5){\delta_{\mathcal X\to \mathcal Y}} & \cdots \ar[r]^<>(0.5){\delta_{\mathcal X\to \mathcal Y}}& f^*\mathscr A^{k-m+n}_{\mathcal Y} \ar[r] 
& 0\cdots \\
&&&&0
}\right]
\end{equation*}
with $ f^*\mathscr A^{k-m+n}_{\mathcal Y} $ in degree $0$, $n=\dim \mathcal X$, and $m=\dim \mathcal Y$.

\begin{rem}
Considering \eqref{E:grSpRelM} and \eqref{D:CXXYY},
the indexing is set up so that if we start with a filtered \emph{left} $D_{\mathcal X}$-module $(\mathcal M,F_\bullet)$ with $\operatorname{gr}^F_\bullet \mathcal M$ \emph{concentrated in degree $0$}, and recall that for $(\mathcal M^{\operatorname{right}},F_\bullet)$ we have 
$F_\bullet  \mathcal  M^{\operatorname{right}} := \omega_{  \mathcal X}\otimes_{ \mathcal  O_{ \mathcal  X}}F_{\bullet +\dim   \mathcal X}  \mathcal M$, then we have: 
\begin{equation}\label{E:grSpCXY}
\operatorname{gr}_k^F\operatorname{Sp}^\bullet_{  \mathcal X\to  \mathcal  Y} (  \mathcal M^{\operatorname{right}})=   \mathcal M^{\operatorname{right}}\otimes_{\mathcal O_{  \mathcal X}} C_{ \mathcal  X\to  \mathcal  Y,k+m}.
\end{equation}
\end{rem}

Similarly, given a normal crossings divisor $\mathcal D\subseteq \mathcal X$, following  \cite{WW23} we define a complex of graded $f^*\mathscr A_{\mathcal Y}$-modules for the pair $(\mathcal X,\mathcal D)$ as:
\begin{equation}\label{D:CXXYYlog}
C_{(\mathcal X,\mathcal D)\to \mathcal Y,k}:=
\end{equation}
$$
\left[\xymatrix@C=1.5em@R=0em{
\cdots 0 \ar[r] &\displaystyle f^*\mathscr A_{\mathcal Y}^{k-m} \otimes_{\mathcal O_{\mathcal X}}\bigwedge^{n}\mathcal T_{\mathcal X}(-\log \mathcal D) \ar[r]^<>(0.5){\delta_{\mathcal X\to \mathcal Y}} 
& \displaystyle  f^*\mathscr A_{\mathcal Y}^{k-m+1}\otimes_{\mathcal O_{\mathcal X}}  \bigwedge ^{n-1}\mathcal T_{\mathcal X}(-\log \mathcal D)   \ar[r]^<>(0.5){\delta_{\mathcal X\to \mathcal Y}} 
& \cdots & \\
&\ 
}\right.
$$
$$
\left.\xymatrix@C=1.5em@R=0em{
\cdots \ar[r]^<>(0.5){\delta_{\mathcal X\to \mathcal Y}}&\displaystyle f^*\mathscr A_{\mathcal Y}^{k-m+n-1}  \otimes_{\mathcal O_X}\bigwedge ^{1}\mathcal T_{\mathcal X}(-\log \mathcal D)   \ar[r]^<>(0.5){\delta_{\mathcal X\to \mathcal Y}} & \displaystyle  f^*\mathscr A_{\mathcal Y}^{k-m+n}   \ar[r] & 0\cdots \\ 
&&0
}\right]
$$
with $ f^*\mathscr A^{k-m+n}_{\mathcal Y} $ in degree $0$, $n=\dim \mathcal X$, and $m=\dim \mathcal Y$.  Recall  $\mathcal T_{\mathcal X}(-\log \mathcal D):=(\Omega_{\mathcal X}^1(\log \mathcal D))^\vee$.

\subsection{Associated graded of the push forward of Hodge modules}
Let $f:X\to Y$ be a schematic projective  morphism of smooth varieties and consider a  Hodge module $\mathsf M$ on $X$ with associated regular holonomic filtered \emph{left} $D_{ X}$-module $(M,F_\bullet)$.  Let $(M^{\operatorname{right}},F_\bullet)$ be the associated \emph{right} $D_{X}$-module.

A theorem of Saito \cite[Thm.~2.14]{Saito1990} implies that the filtration on the derived push forward is strict, and one can then check that  the associated graded and derived  push forward commute. 
Consequently, assuming that $X$ and $Y$ are varieties, then  we have   (see also \cite[Exe.~8.55(2)]{MHM-project}) 
\begin{align}\label{E:MHM-E8.55(2)}
\begin{split}    
\operatorname{gr}_\bullet^Ff_+M^{\operatorname{right}} &=\operatorname{gr}_\bullet^FRf_*\widetilde{\operatorname{Sp}^\bullet _{  X\to   Y}(  M^{\operatorname{right}})} \\
& = Rf_*\operatorname{gr}_\bullet^F\operatorname{Sp}^\bullet_{  X\to   Y}(  M^{\operatorname{right}})\\
& \simeq Rf_*(\operatorname{gr}_\bullet^F   M^{\operatorname{right}}\otimes^L_{  \mathscr A_{  X}}f^{*} \mathscr A_{  Y}),
\end{split}
\end{align}
where the last isomorphism comes from \eqref{E:MHMp8.55}.  Note that 
the first $Rf_*$ is  the derived functor $Rf_*:D^b(\widetilde{f^{-1}D_Y})\to D^b(\widetilde{D_{Y}})$, whereas the last two are the derived  functor  $Rf_*:D^b(\mathcal O_X)\to D^b(\mathcal O_Y)$.  

Converting this into  a statement for the filtered left $D_{  X}$-module $(  M,F_\bullet)$, and recalling that we have 
$
F_\bullet   M^{\operatorname{right}} := \omega_{  X}\otimes_{ \mathcal  O_{  X}}F_{\bullet +\dim   X}  M
$, 
one obtains from \eqref{E:MHM-E8.55(2)} \emph{Laumon's formula} (see  \cite[Exe.~8.55(2)]{MHM-project}, as well as \cite[Thm.~2.4]{PopaSchnell2013}) 
\begin{equation}\label{E:Laumon}
\operatorname{gr}_\bullet^F f_+  M\simeq Rf_*(\omega_{  X/  Y} \otimes_{ \mathcal  O_{  X}} \operatorname{gr}^F_{\bullet +\dim   X-\dim   Y}  M \otimes^L_{  \mathscr A_{  X}}f^{*}  \mathscr A_{  Y}).
\end{equation}

If $\operatorname{gr}_\bullet^F  M$ is \emph{concentrated in degree $0$}, then we can describe things explicitly in terms of the complexes $C_{X\to Y,\bullet}$ \eqref{D:CXXYY}.   
Using the identification $\operatorname{gr}_k^F\operatorname{Sp}^\bullet_{  X\to   Y} (  M^{\operatorname{right}})=   M^{\operatorname{right}}\otimes_{\mathcal O_{  \mathcal X}} C_{  X\to   Y,k+m}$ from \eqref{E:grSpCXY} together with  the identification $\operatorname{gr}_\bullet^FRf_*\operatorname{Sp}^\bullet_{  X\to   Y}(  M^{\operatorname{right}})=  Rf_*\operatorname{gr}_\bullet ^F\operatorname{Sp}^\bullet_{  X\to   Y}(  M^{\operatorname{right}})$ from \cite[Thm.~2.14]{Saito1990} 
we have canonical identifications 
$$
\operatorname{gr}^F_\bullet f_+  M^{\operatorname{right}} = \operatorname{gr}_\bullet^FRf_*\operatorname{Sp}^\bullet_{  X\to   Y}(  M^{\operatorname{right}}) = Rf_*(  M^{\operatorname{right}}\otimes_{\mathcal O_{  X}}C_{  X\to   Y,\bullet +m}),
$$
  and converting to the statement for the left $D_{  X}$-module, we have a canonical identification 
\begin{equation}\label{E:PS-Prop10.2}
\operatorname{gr}^F_\bullet f_+  M= Rf_*(\omega_{  X/  Y}\otimes_{\mathcal O_{  X}}  M\otimes_{\mathcal O_{  X}} C_{  X\to   Y,\bullet}).
\end{equation}
 Taking cohomology of the complexes, and using that strictness implies that taking cohomology commutes with taking the associated graded,  we have a canonical identification
\begin{equation}\label{E:PS-Prop10.2Coh}
\operatorname{gr}^F_\bullet\mathcal H^i f_+  M= \mathcal H^i\operatorname{gr}^F_\bullet f_+  M=  R^if_*(\omega_{  X/  Y}\otimes_{\mathcal O_{  X}}  M\otimes_{\mathcal O_{  X}} C_{  X\to   Y,\bullet}),
\end{equation}
where again, the $\mathcal H^if_+$ on the left is the push forward for filtered $D$-modules and the $R^if_*$ on the right is the derived push forward of coherent sheaves.

We would like to have a version of this in the situation where $f:\mathcal X\to \mathcal Y$ is a schematic projective  morphism of  
smooth separated integral DM stacks locally of finite type over $\mathbb C$, and $\mathsf M$ is a pure Hodge module  on $\mathcal X$ with associated regular holonomic filtered \emph{left} $D_{\mathcal X}$-module $(\mathcal M,F_\bullet)$, and $\operatorname{gr}^F_\bullet \mathcal M$ concentrated in degree $0$:

\begin{lem}\label{L:PSP10.2}
Let $f:\mathcal X\to \mathcal Y$ be a schematic projective morphism of  smooth separated integral DM stacks locally of finite type over $\mathbb C$, and with $n=\dim \mathcal X$, let $\mathsf M$ be a pure polarized Hodge module on $\mathcal X$ with underlying filtered left regular holonomic $D$-module $(\mathcal M_{\mathcal X},F_\bullet)$, and $\operatorname{gr}^F_\bullet\mathcal M$ concentrated in degree $0$.  With respect to the \emph{ad hoc} push forward 
$\mathcal H^i(f^{ah}_+\mathsf M)$  (\Cref{D:AHPFGHodge})
with  underlying filtered regular holonomic $D$-module $\mathcal H^i(f^{ah}_+(\mathcal M, F_\bullet))$  (\Cref{E:AHPFfdm}), 
there  is a canonical isomorphism of graded $\mathscr A_{\mathcal Y}$-modules 
$$
\operatorname{gr}^F_\bullet {\mathcal{H}^i}(f^{ah}_+ (\mathcal M,F_\bullet)) \cong 
R^if_*(\omega_{\mathcal X/\mathcal Y}\otimes_{\mathcal O_{\mathcal X}} \mathcal M\otimes_{\mathcal O_{\mathcal X}} C_{\mathcal X\to \mathcal Y,\bullet}),
$$
where $Rf_*$ on the right is the derived push forward 
$Rf_*:D^b(\mathcal O_{\mathcal X})\to D^b(\mathcal O_{\mathcal Y})$.
\end{lem}

\begin{proof}

  Consider the diagram \eqref{E:TubPresDpf}.  Considering the pull back $p_{\mathcal X}^*\mathcal M$, and \eqref{E:PS-Prop10.2Coh}, we have a canonical identification 
\begin{align*}
\operatorname{gr}^F_\bullet\mathcal H^i f'_+  p_{\mathcal X}^*\mathcal M &=  R^if'_*(\omega_{  U_{\mathcal X}/ U}\otimes_{\mathcal O_{  U_{\mathcal X}}}  p_{\mathcal X}^*\mathcal M\otimes_{\mathcal O_{  U_{\mathcal X}}} C_{  U_{\mathcal X}\to   U,\bullet})\\
&=  R^if'_*p_{\mathcal X}^*(\omega_{  \mathcal X/  \mathcal Y}\otimes_{\mathcal O_{  \mathcal X}}  \mathcal M\otimes_{\mathcal O_{  \mathcal X}} C_{  \mathcal X\to   \mathcal Y,\bullet}),
\end{align*}
where we are using in the second equality that since, in the notation of diagram \eqref{E:TubPresDpf},  $p$ is  \'etale, we have $p_{\mathcal X}^*\omega_{\mathcal X/\mathcal Y}= \omega_{  U_{\mathcal X}/ U}$ and $p_{\mathcal X}^*C_{\mathcal X\to \mathcal Y,\bullet}= C_{  U_{\mathcal X}\to   U,\bullet}$.  As the identification above is canonical, it respects the descent data for the two sheaves, and so gives an isomorphism of sheaves on the stacks:
$$
\operatorname{gr}^F_\bullet\mathcal H^i (f^{ah}_+\mathcal M)  \cong  R^if_*(\omega_{  \mathcal X/  \mathcal Y}\otimes_{\mathcal O_{  \mathcal X}}  \mathcal M\otimes_{\mathcal O_{  \mathcal X}} C_{  \mathcal X\to   \mathcal Y,\bullet}).
$$
\end{proof}

\ifArxiv
\subsection{De Rham complexes}\label{S:dR+SC'} Although our presentation focuses on Spencer complexes, we briefly include some discussion of de Rham complexes for completeness.  
Given a \emph{left} $D_{\mathcal X}$-module $\mathcal M$, following the conventions of  \cite{MHM-project}, the \emph{de Rham complex of $\mathcal M$}, denoted  $\operatorname{DR}^\bullet(\mathcal M)$, or $(\operatorname{DR}^\bullet (\mathcal M),\nabla)$,  is the bounded complex with entries that are $\mathcal O_{\mathcal X}$-modules, and morphisms that are  $\mathbb C$-linear: 
$$
\operatorname{DR}^\bullet(\mathcal M):= \left[
\xymatrix@R=0em{
\cdots 0 \ar[r] & \mathcal M \ar[r]^<>(0.5)\nabla & \Omega^1_{\mathcal X}\otimes_{\mathcal O_X}\mathcal M \ar[r]^<>(0.5)\nabla & \cdots \ar[r]^<>(0.5)\nabla& \Omega^{n}_{\mathcal X}\otimes_{\mathcal O_{\mathcal  X}}\mathcal M \ar[r] & 0\cdots \\
&0
}\right]
$$
with $\mathcal M$ in degree $0$, and $n:=\dim \mathcal X$.  The $\mathbb C$-linear morphisms $\nabla$ are defined \'etale locally, and so are given as in the case of varieties; see \cite[p.104]{HTT08}.  The fact that $\nabla\circ \nabla = 0$ can be checked locally, and so also follows from the case of varieties.   Putting things another way, the de Rham complex realizes  the standard equivalence between left $D_{\mathcal X}$-modules and $\mathcal O_{\mathcal X}$-modules with the data of a flat connection.

We are following the minus sign conventions of \cite{MHM-project} for our complexes, and  so, following \cite[Def.~8.4.1]{MHM-project}, we define the \emph{shifted de Rham complex of $\mathcal M$} to be $$^p\operatorname{DR}^\bullet (\mathcal M):=(\operatorname{DR}^\bullet (\mathcal M)[\dim \mathcal X],(-1)^{\dim \mathcal X}\nabla),$$ so that now $\Omega^{n}_{\mathcal X}\otimes_{\mathcal O_{\mathcal X}}\mathcal M$ is in degree $0$.  
If, in addition to being a left $D_{\mathcal X}$-module,  $\mathcal M$ is also a right $D_{\mathcal X}$-module, then $\operatorname{DR}^\bullet(\mathcal M)$ is naturally a complex with entries that are \emph{right} $D_{\mathcal X}$-modules.  

\begin{exa}[De Rham complex for $D_{\mathcal X}$]
If we consider the de Rham complex associated  to the left and right $D_{\mathcal X}$-module $D_{\mathcal X}$, we have a quasi-isomorphism of complexes of right $D_{\mathcal X}$-modules:
$$
{}^p\operatorname{DR}^\bullet (D_{\mathcal X})\simeq \omega_{\mathcal X}.
$$
Indeed, the complex of right $D_{\mathcal X}$-modules and $\mathbb C$-linear morphisms 
$$
\xymatrix{
\cdots 0 \ar[r] & D_{\mathcal X} \ar[r]^<>(0.5)\nabla & \Omega^1_{\mathcal X}\otimes_{\mathcal O_X} D_{\mathcal X} \ar[r]^<>(0.5)\nabla & \cdots \ar[r]^<>(0.5)\nabla& \Omega^{\dim \mathcal X}_{\mathcal X}\otimes_{\mathcal O_{\mathcal  X}} D_{\mathcal X} \ar[r] &\omega_{\mathcal X}\ar[r]& 0\cdots 
}
$$
is exact.  The $\mathbb C$-linear morphism  $\Omega^{\dim \mathcal X}_{\mathcal X}\otimes_{\mathcal O_{\mathcal  X}} D_{\mathcal X} \to \omega_{\mathcal X}$ 
is  defined \'etale locally, and so is given as in the case of varieties; see \cite[Lem.~1.5.27]{HTT08}. The exactness of the complex then follows from the case of varieties; e.g.,  \cite[Lem.~1.5.27]{HTT08}.
\end{exa}

\begin{rem}
The de Rham and Spencer complexes are closely related.  
Given a left $D_{\mathcal X}$-module $\mathcal M$, with associated right $D_{\mathcal X}$-module $\mathcal M^{\operatorname{right}}:=\mathcal M \otimes _{\mathcal O_{\mathcal X}}\omega_{\mathcal X}$, there is an isomorphism (see \cite[(8.4.8) and Exe.~8.26]{MHM-project}) 
\begin{equation}\label{E:dR-SpComp}
\operatorname{Sp}^\bullet (\mathcal M^{\operatorname{right}}) \stackrel{\sim}{\longrightarrow} \  ^p\operatorname{DR}^\bullet (\mathcal M),
\end{equation}
which is term-wise $\mathcal O_{\mathcal X}$-linear. 
\end{rem}

\subsubsection{De Rham complexes  for filtered $D$-modules} 

Given a filtered \emph{left} $D_{\mathcal X}$-module $(\mathcal M,F_\bullet)$, we define the 
 \emph{shifted filtered de Rham complex} by setting 
 $$
F_k{}^p\operatorname{DR}^\bullet(\mathcal M):= 
$$
$$
\left[
\xymatrix@C=2.5em@R=0em{
\cdots 0 \ar[r] & F_k\mathcal M \ar[r]^<>(0.5){(-1)^n\nabla}& \Omega^1_{\mathcal X}\otimes_{\mathcal O_X}\mathcal F_{k+1}M \ar[r]^<>(0.5){(-1)^n\nabla} & \cdots \ar[r]^<>(0.5){(-1)^n\nabla}& \Omega^{n}_{\mathcal X}\otimes_{\mathcal O_{\mathcal  X}}F_{k+n}\mathcal M \ar[r] & 0\cdots \\
&&&&0
}\right]
$$
 where we recall that $n=\dim \mathcal X$ and  $ \Omega^{n}_{\mathcal X}\otimes_{\mathcal O_{\mathcal  X}}F_{k+n}\mathcal M$ is in degree $0$.

\begin{rem}[Comparing shifted de Rham and Spencer] 
For a \emph{left} $D_{\mathcal X}$-module $\mathcal M$, from the comparison \eqref{E:dR-SpComp}, and using the filtration 
$
F_\bullet \mathcal M^{\operatorname{right}} := \omega_{\mathcal X}\otimes_{\mathcal O_{\mathcal X}}F_{\bullet +n}\mathcal M
$, 
we see that we have 
\begin{equation}\label{E:dR-SpCompF}
(\operatorname{Sp}^\bullet (\mathcal M^{\operatorname{right}}),\delta,F_\bullet) \stackrel{\sim}{\longrightarrow} \  (^p\operatorname{DR}^\bullet (\mathcal M),(-1)^{n}\nabla, F_\bullet),
\end{equation}
which  is term-wise $\mathcal O_{\mathcal X}$-linear. 
\end{rem}

\else
\fi

\section{Geometric construction of Hodge modules}

In this section we generalize some geometric constructions  from \cite{PS17,WW23} to the case of stacks, which allow us to construct certain Hodge modules. 

\subsection{The set-up}\label{S:set-up}
In order to perform the geometric constructions  that will allow us to construct certain Hodge modules, we need to make certain assumptions on the families we consider.  As the assumptions are fairly lengthy, we lay them out in detail in this subsection.   The results in \S \ref{S:GeomConstRoots} will be used in the proofs of the main theorems in \S \ref{S:ProofMain} to ensure that these assumptions hold for the families that we will need to consider, at least after some birational modifications.

\subsubsection{Initial family for the constructions} \label{S:InnitialFamConst}

Throughout this section we assume that 
\begin{equation}\label{E:infamf}
f\colon (\mathcal{Y},\D)\to \mathcal{X}
\end{equation}
 is a surjective schematic projective morphism between smooth separated integral DM stacks of finite type over $\mathbb C$, admitting quasi-projective coarse moduli spaces, that   $\mathcal D$ is an effective normal crossing $\mathbb Q$-divisor on $\mathcal Y$ with
 coefficients in $(0,1)$ on $\mathcal{Y}$, that 
\begin{equation}\label{E:infamdelta}
\mathbf \Delta\subseteq \mathcal X
\end{equation}
is a normal crossings divisor containing the relative snc discriminant  $\mathbf \Delta_f$ (\Cref{dfn:non-relative snc locus}), and that the fibers of $f$ are connected.

We also assume that we have an inclusion $\mathcal X\subseteq \overline{\mathcal X}$ of $\mathcal X$ into a smooth proper integral DM stack $\overline{\mathcal X}$ over $\mathbb C$ with projective coarse moduli space, and that the codimension of the complement is at least $2$. We denote by $\bar{\mathbf \Delta}$ the closure of $\mathbf \Delta$ in $\overline{\mathcal X}$.

\subsubsection{Existence of non-zero sections for a cyclic covering construction}\label{S:non-zero-section}

We fix a line bundle  $\bar {\m{A}}$ on  $\overline{\mathcal{X}}$ such that $\bar{\mathcal A}(-\bar {\mathbf \Delta})$ is big, and setting 
\begin{equation}\label{E:infamA}
\mathcal A:=\bar{\mathcal A}|_{\mathcal X},
\end{equation}
 we assume that for some  sufficiently divisible $m$,
\begin{equation}\label{as:sections'}
 H^0(\mathcal{Y}, \omega^{\otimes m}_{\mathcal{Y}/\mathcal{X}}(m\D)\otimes f^*\m{A}^{\otimes -m}) \ne 0.
\end{equation}
We then set 
\begin{equation}\label{E:HdgMdLLdef}
\m{L}:=\omega_{\mathcal{Y}/\mathcal{X}}( \lceil{\D} \rceil )\otimes f^*\m{A}^{\otimes -1},
\end{equation}
and fix an  $m$ such that $\mathcal L^{\otimes m}$ has a non-trivial section 
\begin{equation}\label{E:HdgMdsdef}
s\in H^0(\mathcal Y,\mathcal L^{\otimes m})
\end{equation}
with zero set containing the support of $\mathcal D$  (see \Cref{R:L([D])section}), 
and we require that  $s$ does not admit a root (i.e., it is not the power of a section of a smaller tensor power of $\mathcal L$).  This implies that when we take the $m$-th cyclic covering $\pi_m: \mathcal{Y}_m\to \mathcal{Y}$ determined by the section $s$, where we use the conventions in \cite[\S 3.5]{EVvan92} and note that the construction there carries over directly to the case of stacks, we have that $\mathcal Y_m^\circ$ is irreducible \cite[Lem.~3.15(a)]{EVvan92}.   One can, for instance, insure  that $s$ does not admit a root by taking $m$ minimal such that $\mathcal L^{\otimes m}$ has a non-trivial global section  with zero set containing the support of $\mathcal D$. 

\subsubsection{Non-recursive hypothesis: further log resolution of the family}
\label{S:FLRA}

The previous subsections, \S \ref{S:InnitialFamConst} and \S \ref{S:non-zero-section}, outline hypotheses we will use recursively.  In this subsection we introduce a hypothesis we will only use once, and only in the situation where $\mathcal X=\overline{\mathcal X}$: here we want to assume that  given any  divisor  $\mathcal S\subseteq {\mathcal X}=\overline{\mathcal X}$, there  is an effective divisor $\mathcal S\subseteq \mathcal S'\subseteq \mathcal X$ containing $\mathcal S$, as well as a log resolution  $\sigma: \widetilde {\mathcal X}\to {\mathcal{X}}$  of the pair $({\mathcal{X}},{\mathbf \Delta}+\mathcal S')$ with centers contained in $\mathcal S'$, 
and a closed substack $\widetilde {\mathbf T}\subseteq \widetilde {\mathcal X}$ of codimension at least $2$, 
such that there is a commutative diagram
\begin{equation}\label{E:infamTY}
\xymatrix{
(\widetilde{\mathcal Y},\widetilde {\mathcal D}) \ar[r]^<>(0.5){\tilde \sigma} \ar[d]^{\tilde f}& ({\mathcal{Y}},\mathcal D) \ar[d]^{f}\\
\widetilde{\mathcal X}\ar[r]^{\sigma}&{\mathcal{X}}
}
\end{equation}
 that satisfies the following conditions, where we fix the notation  $\widetilde {\mathcal X}^\circ:=\widetilde {\mathcal X}-\widetilde {\mathbf T}$, $\widetilde {\mathcal Y}^\circ:=\widetilde {\mathcal Y}-\tilde f^{-1}(\widetilde {\mathbf T})$, $\widetilde {\mathcal D}^\circ=\widetilde {\mathcal D}|_{\mathcal Y^\circ}$, and $\tilde f^\circ =\tilde f|_{\widetilde {\mathcal Y}^\circ}$:
 
\begin{enumerate}[label=(\arabic*)]

\item \label{S:FLRX-1}
Letting $\tilde {\mathbf \Delta}$ be the reduced snc divisor with support  $\sigma^{-1}\mathbf \Delta$, the morphism $\tilde f:(\widetilde {\mathcal Y},\widetilde {\mathcal D})\to \widetilde {\mathcal X}$ and the divisor $\tilde {\mathbf \Delta}$ satisfy the conditions in \S \ref{S:InnitialFamConst}, and if $\mathcal D=0$, then $\widetilde {\mathcal D}=0$. 

\item \label{S:FLRX-2} 
Letting $\tilde {\mathcal A}:=\sigma^*\mathcal A$,  $\tilde {\m{L}}:=\omega_{\widetilde {\mathcal{Y}}/\widetilde {\mathcal{X}}}(\lceil\widetilde {\D} \rceil )\otimes 
\tilde f^*\tilde {\m{A}}^{\otimes -1}$, and $\tilde {\mathcal L}^\circ:=\tilde{\mathcal L}|_{\mathcal Y^\circ}$, we observe that $\tilde {\mathcal A}(-\tilde{\mathbf \Delta})$ is big, and we  assume there is a section 
$$\tilde s\in H^0(\widetilde {\mathcal Y}^\circ,\tilde {\mathcal L}^{\circ \otimes m})$$ 
that satisfies the conditions in \S \ref{S:non-zero-section} for the morphism $\tilde f^\circ:(\widetilde {\mathcal Y}^\circ,\widetilde {\mathcal D}^\circ)\to \widetilde {\mathcal X}^\circ$.

 \item \label{S:FLRX-3}

The morphism 
$\tilde \sigma $ is schematic projective and birational with divisorial exceptional locus, and setting $\widetilde {\mathcal E}\subseteq \widetilde {\mathcal X}$ to be the exceptional locus of $\sigma$ (which by construction is contained in $\sigma^{-1}(\mathcal S'$)), $\mathbf Z:= \sigma(\widetilde {\mathcal E})\subseteq \mathcal X$,  $\widetilde {\mathcal U}:=\widetilde {\mathcal X}-\widetilde {\mathcal E}$, $\mathcal U:=\mathcal X-\mathbf Z$,  $\widetilde {\mathcal Y}_{\widetilde {\mathcal U}}:=\tilde f^{-1}(\widetilde {\mathcal U})$, and $\mathcal Y_{\mathcal U}:=f^{-1}(\mathcal U)$, we have that $\tilde \sigma  $ restricts to an isomorphism 
 $
 \tilde \sigma |_{\widetilde {\mathcal Y}_{\widetilde {\mathcal U}}}: \widetilde {\mathcal Y}_{\widetilde {\mathcal U}} \stackrel{\sim}{\rightarrow} \mathcal \mathcal Y_{\mathcal U}
 $
 over  $\sigma|_{\widetilde {\mathcal U}}: \widetilde {\mathcal U}\stackrel{\sim}{\to} \mathcal U$.
 
 \item \label{S:FLRX-4}
 Under the identifications  of \ref{S:FLRX-3}, we have   that  
 $  \widetilde {\mathcal D}|_{\widetilde {\mathcal Y}_{\widetilde {\mathcal U}}} =\mathcal D|_{\mathcal Y_{\mathcal U}}$, from which it follows that $  \tilde {\mathcal L}|_{\widetilde {\mathcal Y}_{\widetilde {\mathcal U}}}=\mathcal L|_{\mathcal Y_{\mathcal U}}$, and we assume that

\begin{equation}\label{E:infamTs}
 \tilde s|_{\widetilde {\mathcal Y}^\circ-\tilde f^{-1}(\sigma^{-1}\mathcal S')}=\tilde \sigma^*s|_{\widetilde {\mathcal Y}^\circ-\tilde f^{-1}(\sigma^{-1}\mathcal S')}.
 \end{equation}

\end{enumerate}

 \begin{rem} 
\label{R:S:FLRX-2} 
 If in  \ref{S:FLRX-2} we only assumed, \emph{a priori}, that the section $\tilde s$ satisfies the conditions in \S \ref{S:non-zero-section} except possibly for the condition that it not admit a root, then, \emph{a posteriori},  from \ref{S:FLRX-4} it would follow that $\tilde s$ did not admit a root.  Indeed, if $\tilde s$ admitted a root, then $s|_{\widetilde {\mathcal Y}^\circ-\tilde f^{-1}(\sigma^{-1}\mathcal S')}$ would admit a root, which would imply that the irreducible cover $\pi_m:\mathcal Y_m\to \mathcal Y$, when restricted to $\mathcal Y^\circ-\mathcal S'$ was reducible, a contradiction. 
\end{rem}
 
 \begin{rem}
 In this paper we will always be able to take $\widetilde {\mathbf T}=\emptyset$ and $\mathcal S'=\mathcal S$; we include the slightly more complicated set-up here for use in forthcoming work.
\end{rem}

\subsection{Singular cotangent vectors and conormal substacks}
We take a brief digression to discuss singular cotangent vectors and conormal substacks.  
Let $f:\mathcal Y\to \mathcal X$ be a 
surjective schematic projective morphism between smooth proper DM stacks over $\mathbb C$.
We have a diagram
$$
\xymatrix{
\mathcal Y \ar[d]_f& \mathcal  Y   \times _{\mathcal X} T^\vee \mathcal X \ar[l]_<>(0.5){p_1} \ar[d]^{p_2} \ar[r]^<>(0.5){df}& T^\vee\mathcal Y\\
\mathcal X & \ar[l]_p T^\vee\mathcal X
}
$$
where here $p$ denotes the structure morphism for the cotangent bundle, and the $p_i$ denote the projections.  
Inside the cotangent bundle $T^\vee{\mathcal X}$ of $\X$, we define the \emph{locus of singular cotangent vectors}  to be
\begin{equation}\label{E:DefSf}
S_{f}:=p_{2}\left(d f^{-1}(0)\right) \subseteq {T}^{\vee} {\X}.
\end{equation}

Given an integral substack $\mathcal Z\subseteq \mathcal X$, we define the \emph{conormal substack $T_{\mathcal Z}^\vee\mathcal X \subseteq T^\vee\mathcal X$ to $\mathcal Z$}  be the sub-stack of $T^\vee\mathcal X$ defined by the closure of the conormal bundle to the smooth locus of $\mathcal Z$.

\subsection{Hodge module construction}

The goal of this subsection is to prove the following theorem generalizing  \cite[Thm.~2.2]{PS17} and \cite[\S 6]{WW23}.
Note that for a coherent graded $\mathscr{A}_{\X}$-module $
\m{G}_{\bullet}=\bigoplus_{k \in \mathbb{Z}} \m{G}_{k}$, we use the symbol $\m{G}$, without the dot, to denote the associated coherent sheaf on $T^{\vee} \X$; it has the property that $p_{*} \m{G} \simeq \m{G}_{\bullet}$ as modules over $\mathscr{A}_{\X}$, where $p:T^\vee\mathcal X\to \mathcal X$ is the structure map.

\begin{teo}

\label{T:Existence-of-Hodge-module-and-G-subsheaf-theorem}
Assume one is in the  situation described in  \S \ref{S:InnitialFamConst} and \S\ref{S:non-zero-section}, 
and in particular, one is given  the data in \eqref{E:infamf}, \eqref{E:infamdelta}, \eqref{E:infamA}, and \eqref{E:HdgMdsdef}.  
There exists a Hodge module $\mathsf M$ on $\mathcal X$ with strict support $\mathcal X$, and a graded $\mathscr{A}_{\XX}$-module $\m{G}_{\bullet}$ that is coherent over $\mathscr{A}_{\XX}$, which  have the following properties:
 
 \begin{enumerate}[label=(\alph*)]
\item
\label{T:I:EHMGa}
  If $\mathcal D=0$, then, as a coherent sheaf on the cotangent bundle, one has  $\operatorname{Supp} \mathcal {G} \subseteq S_{f}\subseteq T^\vee\mathcal X$.

\item \label{T:I:EHMGb}
 $\m{G}_{0} \cong \m{A}$.

\item 
\label{T:I:EHMGc}
For all $k\gg 0$ one has $\mathcal G_k|_{\mathcal X-\mathbf \Delta}=0$, and if $\mathcal D=0$, then  for all $k$ one has  $\m{G}_{k}|_{\mathcal X-\mathbf \Delta}$ is torsion-free.

\item 
\label{T:I:EHMGd}
For  $\left(\mathcal{M}, F_{\bullet} \mathcal{M}\right)$ the regular holonomic filtered $D_{\mathcal X}$-module  underlying  $\mathsf M$, 
there is an inclusion of graded $\mathcal{A}_{\X}$-modules $\m{G}_{\bullet} \subseteq \operatorname{gr}_{\bullet}^{F} \mathcal{M}$.

\item \label{T:I:EHMGe}

$F_{k} \mathcal{M}=0$ for $k<0$.
\end{enumerate}
\end{teo}

The rest of this subsection is devoted to proving the theorem; the proof is summarized in \S \ref{S:PfThmHdgGG}. 

\subsubsection{Cyclic covering construction}\label{S:CycCovConst}

Following \cite[\S 2.3]{PS17}, and \cite[\S 6]{WW23} in the case of pairs,  we take the $m$-th cyclic covering $\pi_m\colon \mathcal{Y}_m\to \mathcal{Y}$ determined by the section $s$. 
As mentioned before, we use the conventions in \cite[\S 3.5]{EVvan92} and note that the construction there carries over directly to the case of stacks.   
We note that the conditions on the diagonal of the stack $\mathcal Y_m$ hold by say 
\cite[Lem.~B.25 (2) and (4)]{CMW18}, and the fact that the coarse moduli space $Y_m$ is projective follows since the induced morphism on coarse moduli spaces $Y_m\to Y$ is a dominant finite morphism of proper algebraic spaces, with $Y$ projective, and so the pull back of an ample on $Y$ is ample on the algebraic space $Y_m$ (see, e.g., \cite[Lem.~A.1]{CMZslope_stability}).

Let $\nu\colon \Z \to \mathcal{Y}_m$ be a strong log resolution of the pullback of $\D$ to the normalization of $\mathcal{Y}_m$. Set $\phi=\pi\circ\nu$ and $h=f\circ\phi$. By the cyclic covering construction, the pull back $\phi^*\frac{1}{m}(s)$ of the $\mathbb Q$-divisor $\frac{1}{m}(s)$ is an effective divisor on $\Z$ that contains $\D_{\mathcal Z}$, the support of $\phi^*(\D)$. By construction, $\D_{\Z}$ is normal crossing.

The following diagram summarizes the situation:
\begin{equation}\label{Geometry-Hodge-theory-construction}
\begin{tikzcd}
\m{Z} \arrow[r, "\nu"'] \arrow[rrd, "h", bend right] \arrow[rr, "\phi", bend left] & \m{Y}_m \arrow[r, "\pi_m"'] & \m{Y} \arrow[d, "f"] \\
                                                                                      &                           & \X                  
\end{tikzcd}
\end{equation}
Note that all of the morphisms are schematic, surjective,  and projective, that $\mathcal Z$ is a smooth proper DM stack over $\mathbb C$ with projective coarse moduli space, and that the fibers of $h$ are connected.

\subsubsection{Constructing the Hodge module $\mathcal M$ and the graded  sheaf $\mathcal G_\bullet$}

Let ${\mathcal{H}^0}(h_+^{ah}\mathbb Q^H_{\mathcal Z}[\dim \mathcal Z])$ 
be the polarizable Hodge module of weight $\dim \mathcal Z$ on $\mathcal X$, as defined in 
\eqref{D:Hif*MHdg} via the \emph{ad hoc} push forward of the constant Hodge module on $\mathcal Z$.  
When restricted to the smooth locus of $h$, and say, an \'etale cover of $\mathcal X$, this is just the polarizable variation of Hodge structure on the middle cohomology of the fibers. Let $\mathsf M$ be the polarizable Hodge module of weight $\dim \mathcal Z$ on $\mathcal X$ defined as the torsion-free quotient of  ${\mathcal{H}^0}(h_+^{ah}\mathbb Q^H_{\mathcal Z}[\dim \mathcal Z])$ (see, e.g., \cite[Lem.~3.16]{CMZpositivity}).
\begin{equation}\label{E:PS-MHdgMd-def}
{\mathcal{H}^0}(h_+^{ah}\mathbb Q^H_{\mathcal Z}[\dim \mathcal Z])\twoheadrightarrow \mathsf M;
\end{equation}
$\mathsf M$ has strict support $\X$ (see \cite[Def.~3.15]{CMZpositivity}). Let $\mathcal{M}$ denote the underlying regular holonomic left $D_{\X}$-module, and $F_{\bullet} \mathcal{M}$ its Hodge filtration. Since $F_{\bullet} \mathcal{M}$ is a good filtration, the associated graded $\mathcal{A}_{\X}$-module
$$
\operatorname{gr}_{\bullet}^{F} \mathcal{M}=\bigoplus_{k \in \mathbb{Z}} \operatorname{gr}_{k}^{F} \mathcal{M}
$$
is coherent over $\mathscr{A}_{\X}=\operatorname{Sym}^\bullet \m{T}_{\X}$. One has the following more concrete description of $\operatorname{gr}_{\bullet}^{F} \mathcal{M}$, generalizing \cite[Prop.~2.4]{PS17}:
\begin{pro}

\label{P:Generalization-prop10.2-PS} In the category of graded $\mathscr{A}_{\X}$-modules, using the complex $C_{\mathcal Z\to \mathcal X,\bullet}$ defined in \eqref{D:CXXYY}, there is a surjection
 \begin{equation}\label{E:Generalization-prop10.2-PS}
 R^0h_* (\omega_{\mathcal Z/\mathcal X}\otimes_{\mathcal O_{\mathcal Z}}C_{\mathcal Z\to \mathcal X,\bullet})\twoheadrightarrow \operatorname{gr}_{\bullet}^{F} \mathcal{M}.
 \end{equation}

\end{pro}

\begin{proof}
This follows from 
\Cref{L:PSP10.2}; 
 there is an isomorphism of graded $\mathscr A_{\mathcal X}$-modules 
$$
\operatorname{gr}^F_\bullet {\mathcal{H}^0}(h^{ah}_* (\mathcal O_{\mathcal Z},F_\bullet)) \cong
R^0h_* (\omega_{\mathcal Z/\mathcal X}\otimes_{\mathcal O_{\mathcal Z}}C_{\mathcal Z\to \mathcal X,\bullet}).
$$
By construction, $\mathcal M$ is the torsion-free  quotient of ${\mathcal{H}^0}(h^{ah}_* (\mathcal O_{\mathcal Z},F_\bullet))$, so that 
$\operatorname{gr}^F_\bullet \mathcal M$ is a quotient of  $\operatorname{gr}^F_\bullet {\mathcal{H}^0}(h^{ah}_* (\mathcal O_{\mathcal Z},F_\bullet))$, completing the proof.
\end{proof}

From the cyclic covering construction, we have an induced inclusion 
\[
\phi^*\m{L}^{-1}\hookrightarrow \m{O}_{\Z}(-\D_{\Z}),
\]
where $\mathcal L=\omega_{\mathcal{Y}/\mathcal{X}}( \lceil{\D} \rceil )\otimes f^*\m{A}^{\otimes -1}$ is defined in \eqref{E:HdgMdLLdef}.
After combining this with the inclusion $\phi^*\Omega^{k}_{\mathcal{Y}}(\log \lceil \mathcal D\rceil )\hookrightarrow \Omega^k_{\Z}(\log \D_{\Z})$, we hence obtain injective morphisms
\begin{equation}
\label{eq:tautoeq'}
\phi^*\big(\m{L}^{-1}\otimes\Omega^{k}_{\mathcal{Y}}(\log \lceil \mathcal D\rceil )\big) \hookrightarrow \phi^*\mathcal L^{-1}\otimes \Omega^k_{\Z}(\log \D_{\Z})\hookrightarrow\Omega^k_{\Z}(\log \D_{\Z})(-\D_{\Z})\hookrightarrow \Omega^k_{\Z}
\end{equation}
for $k=0,1,\dots,\dim \mathcal Z$.

Recalling the definition of the complexes $ C_{\mathcal Z\to \mathcal X, \bullet}$ \eqref{D:CXXYY} and $C_{(\mathcal{Y},\lceil \mathcal D\rceil)\to \mathcal X,\bullet}$  \eqref{D:CXXYYlog}, we use \eqref{eq:tautoeq'} to give the following generalization of  \cite[Prop.~2.8]{PS17} and \cite[Lem.~6.1]{WW23}: 

\begin{lem}

\label{L:Generalization-prop2.8-PS}
The inclusion \eqref{eq:tautoeq'} induces a non-trivial morphism of complexes of graded $\mathscr{A}_{\mathcal{X}}^\bullet$-modules
 \begin{equation}\label{E:Generalization-prop2.8-PS}
{R}^0f_*\big( \m{L}^{-1}\otimes_{\mathcal O_{\mathcal Y}} \omega_{\mathcal{Y}/\mathcal{X}}( \left \lceil \D \rceil \right)\otimes_{\mathcal O_{\mathcal Y}} C_{(\mathcal{Y},\lceil \mathcal D\rceil)\to \mathcal X,\bullet}\big)\to{ R}^0h_*\big(\omega_{\Z/\mathcal{X}}\otimes_{\mathcal O_{\mathcal Z}} C_{\mathcal Z\to \mathcal X, \bullet}\big). 
\end{equation}
\end{lem}

\begin{proof}
The construction of the morphism is identical to that in  \cite[Prop.~2.8]{PS17} and \cite[Lem.~6.1]{WW23}, which we sketch here for convenience. 
 It suffices by adjunction to construct a morphism of complexes  
$
\phi^*\big( \m{L}^{-1}\otimes_{\mathcal O_{\mathcal Y}} \omega_{\mathcal{Y}/\mathcal{X}}( \left \lceil \D \rceil \right)\otimes_{\mathcal O_{\mathcal Y}} C_{(\mathcal{Y},\lceil \mathcal D\rceil)\to \mathcal X,\bullet}\big)
\to \omega_{\Z/\mathcal{X}}\otimes_{\mathcal O_{\mathcal Z}} C_{\mathcal Z\to \mathcal X, \bullet}
$.  To this end, using that $\mathcal T_{\mathcal Y}\left(-\log  \lceil \D \rceil \right)= (\Omega_{\mathcal Y}(\log  \lceil \D \rceil ))^\vee$, one constructs for each $k$ a natural isomorphism of sheaves (where the superscript on the complex $C_{(\mathcal{Y},\lceil \mathcal D\rceil)\to \mathcal X,\bullet}$ indicates the term in the complex) 
$$
\phi^*\big( \m{L}^{-1}\otimes_{\mathcal O_{\mathcal Y}} \omega_{\mathcal{Y}/\mathcal{X}} \left( \lceil \D \rceil \right)\otimes_{\mathcal O_{\mathcal Y}} C^{k-d}_{(\mathcal{Y},\lceil \mathcal D\rceil)\to \mathcal X,\bullet}\big)
$$
$$
\to \phi^*\left(\mathcal L^{-1}\otimes_{\mathcal O_{\mathcal Y}}f^*\omega_{\mathcal X}^{-1}\otimes_{\mathcal O_{\mathcal Y}}\Omega^{k}_{\mathcal Y} \left( \log \lceil \D \rceil \right)\otimes_{\mathcal O_{\mathcal Y}} f^*\mathscr A^{\bullet -n-k}_{\mathcal X}\right).
$$
Using the morphism in \eqref{eq:tautoeq'}, one obtains a morphism from the sheaf on the right above to $h^*\omega_{\mathcal X}^{-1}\otimes_{\mathcal O_{\mathcal Z}}h^*\mathscr A_{\mathcal X}^{\bullet -n-k}\otimes \Omega^k_{\mathcal Z}$, which is isomorphic to $\omega_{\mathcal Z/\mathcal X}\otimes_{\mathcal O_{\mathcal Z}}C^{k-d}_{\mathcal Z\to \mathcal X,\bullet}$, which completes the construction of the morphism.  One then checks that these morphisms of the terms in the complex in fact give a morphism on the complexes.   
  The proof that the morphism is non-trivial can be checked \'etale locally, and so follows directly from  \cite[Prop.~2.8]{PS17} and \cite[Lem.~6.1]{WW23}.
\end{proof}

We now define a graded $\mathcal A_{\mathcal X}$-module $\mathcal G_{\bullet}$ to be the image of the composition of \eqref{E:Generalization-prop2.8-PS} and \eqref{E:Generalization-prop10.2-PS}; in other words, we have a commutative diagram
\begin{equation}\label{E:PS-GG-def}
\xymatrix@C=1.5em{
{R}^0f_*\big( \m{L}^{-1}\otimes_{\mathcal O_{\mathcal Y}} \omega_{\mathcal{Y}/\mathcal{X}}( \left \lceil \D \rceil \right)\otimes_{\mathcal O_{\mathcal Y}} C_{(\mathcal{Y},\lceil \mathcal D\rceil) \to \mathcal X,\bullet}\big) \ar[r] \ar@{->>}[rrd]& { R}^0h_*\big(\omega_{\Z/\mathcal{X}}\otimes_{\mathcal O_{\mathcal Z}} C_{\mathcal Z\to \mathcal X, \bullet}\big) \ar@{->>}[r] & \operatorname{gr}^F_\bullet \mathcal M\\
&&\mathcal G_\bullet  \ar@{^(->}[u]
}
\end{equation}

\begin{rem}
Note that in addition to the morphism  $f\colon (\mathcal{Y},\D)\to \mathcal{X}$, the construction of  $\m{G}_{\bullet}$ depends on the line bundle $\mathcal A$, as well as on the chosen section $s$ of $\mathcal L^{\otimes m}$.
\end{rem}

The remainder of the subsection will be devoted to showing that $\mathcal M$ and $\mathcal G_\bullet$ satisfy the conditions of \Cref{T:Existence-of-Hodge-module-and-G-subsheaf-theorem}; the proof is then summarized in \S \ref{S:PfThmHdgGG}.

\subsubsection{Properties of $\operatorname{gr}^F_\bullet \mathcal M$ and $\mathcal G_\bullet$} 

We now establish some properties of $\operatorname{gr}^F_\bullet \mathcal M$ and $\mathcal G_\bullet$, following \cite{PS17,WW23}.  Most of these can be reduced to the case of a variety after taking an \'etale cover; we include proofs when this is not the case. 

We start with properties of $\mathcal M$.  First we generalize \cite[Cor.~2.5]{PS17}: 

\begin{cor}

\label{C:10.6-PS} One has $\operatorname{gr}_{k}^{F} \mathcal{M}=0$ for $k<\dim \mathcal X-\dim \mathcal Y$, whereas $\operatorname{gr}_{k}^{F} \mathcal{M} \cong h_{*} \omega_{\m{Z} / \X}$ for $k= \dim \mathcal X-\dim \mathcal Y$.
\end{cor}

\begin{proof}
For this we use \Cref{P:Generalization-prop10.2-PS}, and in particular the surjection \eqref{E:Generalization-prop10.2-PS}. 
 The first assertion is then clear because $C_{\Z\to \mathcal X, k}=0$ for $k<\dim \mathcal X-\dim \mathcal Z$ (see \eqref{D:CXXYY}). To prove the second assertion, observe that the complex $\omega_{\mathcal Z/\mathcal X}\otimes_{\mathcal O_{\mathcal Z}} C_{\Z\to \mathcal X, k}$, for $k=\dim \mathcal X-\dim \mathcal Y$, is just $\omega_{\mathcal Z/\mathcal X}$ in degree $0$, so that again by  \eqref{E:Generalization-prop10.2-PS}, we have that $\operatorname{gr}^F_{\dim \mathcal X-\dim \mathcal Y}\mathcal M=F_{\dim \mathcal X-\dim \mathcal Y}\mathcal M$ is a quotient of $h_*\omega_{\mathcal Z/\mathcal X}$.  
 
 Now recall that $(\mathcal M,F_\bullet)$ is defined as the torsion-free quotient of ${\mathcal{H}^0}(h^{ah}_* (\mathcal O_{\mathcal Z},F_\bullet))$; denote by $(\mathcal M',F_\bullet)$ the torsion submodule of ${\mathcal{H}^0}(h^{ah}_* (\mathcal O_{\mathcal Z},F_\bullet))$.   Then we have a short exact sequence, recalling $n=\dim \mathcal X$ and $d=\dim \mathcal Y$, 
$$
0 \to F_{n-d}\mathcal M'\to F_{n-d}{\mathcal{H}^0}(h^{ah}_* (\mathcal O_{\mathcal Z},F_\bullet)) = h_*\omega_{\mathcal Z/\mathcal X}\to F_{n-d}\mathcal M\to 0.
$$
Since $h_* \omega_{\mathcal Z/\mathcal X}$ is torsion-free, we are done.
\end{proof}

We also have the following, generalizing  \cite[Prop.~2.6]{PS17} and \cite[Cor.~2.7]{PS17}:

\begin{pro}

\label{P:support}
In the notation above: 
\begin{enumerate}[label=(\alph*)]
\item  The support of $C_{\Z\ra\X,\bullet }$ is equal to $d h^{-1}(0) \subseteq \mathcal Z\times _{\mathcal X} T^{\vee}{\X}$.

\item The support of $\operatorname{gr}_{\bullet}^F{\mathcal M}$ is a union of irreducible components of $S_{h}$, the locus of singular cotangent vectors.
\end{enumerate}

\end{pro}

\begin{proof} 
As the support of sheaves can be checked after a surjective \'etale base change, this reduces to the case of varieties, and so follows from \cite[Prop.~2.6]{PS17} and \cite[Cor.~2.7]{PS17}.
\end{proof}

We now move on to properties of $\mathcal G_\bullet$.
We start with a generalization of \cite[Prop.~2.9.]{PS17} and \cite[(6.3)]{WW23}: 

\begin{pro}\label{P:PS2.9}

One has $\m{G}_{k}=0$ for $k<\dim \mathcal X-\dim 
\mathcal Y$, whereas $\m{G}_{k} \cong  \m{A} \otimes 
f_{*} \m{O}_{\mathcal Y}\cong \mathcal A$ for $k=\dim \mathcal X-\dim 
\mathcal Y$.
\end{pro}

\begin{proof} We make use of \Cref{C:10.6-PS}. The first assertion is clear because $\operatorname{gr}_{k}^{F} \mathcal{M}=$ 0 for $k<\dim \mathcal X-\dim \mathcal Y$. For the second assertion, we note that for $k=n-d = \dim \mathcal X-\dim \mathcal Y$,  then by construction, $\m{G}_{n-d}$ is a quotient of the $\m{O}_{\X}$-module
$$
{R}^0f_*\big( \m{L}^{-1}\otimes_{\mathcal O_{\mathcal Y}} \omega_{\mathcal{Y}/\mathcal{X}}( \left \lceil \D \rceil \right)\otimes_{\mathcal O_{\mathcal Y}} C_{(\mathcal{Y},\lceil \D \rceil) \to \mathcal X,n-d}\big) \cong  
\m{A} \otimes_{\mathcal O_{\mathcal X}} R^{0} f_{*} C_{(\Y,\lceil \D \rceil)\ra \X, n-d} \cong  \m{A} \otimes_{\mathcal O_{\mathcal X}} f_{*} \m{O}_{\Y},
$$
where the first isomorphism comes from the definition \eqref{E:HdgMdLLdef}  of $\mathcal L$ and adjunction, and the second comes from inspection of the definition \eqref{D:CXXYYlog} of $ C_{(\Y,\lceil \D \rceil)\ra \X, n-d} $.
As to checking whether the surjection $\m{A} \otimes_{\mathcal O_{\mathcal X}} f_{*} \m{O}_{\Y}\twoheadrightarrow \m{G}_{n-d}$ is an isomorphism, this can be done \'etale locally and
so reduces to the case of varieties, and therefore to  \cite[Prop.~2.9.]{PS17}. 

Now we use the hypothesis that  $f:\mathcal Y\to \mathcal X$ has connected fibers to conclude that the canonical morphism $\mathcal O_{\mathcal X}\to f_*\mathcal O_{\mathcal Y}$ is an isomorphism;  this can be checked \'etale locally, and so reduces to the case of varieties as our morphism is assumed to be schematic.  
\end{proof}

Next we give a generalization of 
\cite[Prop.~2.10., Lem.~2.11., Prop.~2.12.]{PS17}:

\begin{pro}

\label{P:Strong-coherence-of-Higgs}
In the notation above, if $\mathcal D=0$: 
\begin{enumerate}[label=(\alph*)]
\item      We have $\operatorname{Supp} \m{G} \subseteq S_{f}\subseteq T^\vee\mathcal X$.
\label{P:Strong-coherence-of-Higgs(1)}

\item 
 Every irreducible component of $\operatorname{Supp} \m{G}$ in $T^\vee\mathcal X$ is the conormal substack of some integral substack  of $\X$.
 \label{P:Strong-coherence-of-Higgs(2)}
 
 \item
  For every $k \in \mathbb{Z}$, the sheaf $\m{G}_{k}$ is torsion-free on $\X - \mathbf{\Delta}_f$. 
  \label{P:Torsion-freeness-of-G-sm-locus} 
 
\end{enumerate}

\end{pro}

\begin{proof}
As the support of sheaves can be checked after a surjective \'etale base change, \ref{P:Strong-coherence-of-Higgs(1)} and \ref{P:Strong-coherence-of-Higgs(2)}  reduce to the case of varieties, and so follow from \cite[Prop.~2.10 and Prop.~2.12]{PS17}. As torsion sub-sheaves can be detected after a surjective \'etale base change, \ref{P:Torsion-freeness-of-G-sm-locus} reduces to the case of varieties, and so follows from \cite[Prop.~2.12.]{PS17}.
\end{proof}

In the case where $\mathcal D\ne 0$, we will need a replacement later for \Cref{P:Strong-coherence-of-Higgs}\ref{P:Strong-coherence-of-Higgs(1)} (i.e., \cite[Prop.~2.10]{PS17}).  
Following \cite{WW23}, we will do this by investigating the term on the left in \eqref{E:PS-GG-def}, in the definition of $\mathcal G_\bullet$.  To that end, we define
\begin{equation}\label{E:tildeGGdef}
\check {\mathcal G}^0_\bullet = {R}^0f_*\big( \m{L}^{-1}\otimes_{\mathcal O_{\mathcal Y}} \omega_{\mathcal{Y}/\mathcal{X}}( \left \lceil \D \rceil \right)\otimes_{\mathcal O_{\mathcal Y}}C_{(\mathcal{Y}, \lceil \D \rceil) \to \mathcal X,\bullet}\big).
\end{equation}

While for each fixed $k$, one has that $\check {\mathcal G}^0 _k$ is a coherent $\mathcal O_{\mathcal X}$-module,  \emph{a priori}, from the  definition, the full graded module $\check {\m{G}}^0_\bullet$ is not necessarily a coherent $\m{O}_\mathcal{X}$-module. However, it is coherent on the complement of the pair discriminant locus,  as follows from \Cref{L:coherence}, below, which generalizes  \cite[Lem.~6.2]{WW23}.

\begin{lem}

\label{L:coherence}
In the notation above, we have 
\[\check{\m{G}}^0_m|_{\m{U}_s}=0\]
for every $m\gg0$, where $\m{U}_s :=(\mathcal X-\mathbf \Delta_f) \subseteq \mathcal{X}$ is  the complement of  the pair discriminant locus $\mathbf \Delta_f$ of $f$.
\end{lem}

\begin{proof}
As the triviality of $\check{\m{G}}^0_m|_{\m{U}_s}$ can  be checked after a surjective \'etale base change, this reduces to the case of varieties, and so follows from \cite[Lem.~6.2]{WW23}.
\end{proof}

\begin{rem}\label{R:relsnc}
    \Cref{L:coherence} is the only place in the paper where we use 
    the assumption that $f$ be a relatively snc morphism over an open substack, rather than just assuming the weaker condition that the general fiber of $f$ be an snc pair.   
    The point is that the proof of  \Cref{L:coherence}, or rather, the proof of \cite[Lem.~6.2]{WW23} in the case of varieties,   boils down to being able to re-write the logarithmic de Rham complex of a Hodge module in an easier way in the case where one has a relative snc morphism; see \cite[Prop. 2.3]{PTW18}.

\end{rem}

\subsubsection{Proof of \Cref{T:Existence-of-Hodge-module-and-G-subsheaf-theorem}}\label{S:PfThmHdgGG}
The proof of  \Cref{T:Existence-of-Hodge-module-and-G-subsheaf-theorem} is obtained by putting together the results about the graded $\mathcal{A}_{\X}$-module $\m{G}_{\bullet}$ 
that we have established so far.  We take $\mathsf M$ and $\mathcal G_\bullet$ to be the Hodge module and graded $\mathcal A_{\mathcal X}$-module defined in 
\S \ref{S:CycCovConst} in  \eqref{E:PS-MHdgMd-def} and \eqref{E:PS-GG-def}; let $\mathcal M$ be the underlying regular holonomic $D_{\mathcal X}$-module of $\mathsf M$.  
To align the indexing with that in \Cref{T:Existence-of-Hodge-module-and-G-subsheaf-theorem}, we replace  $\mathsf M$ (of weight $d=\dim \mathcal Y$)  with its Tate twist $\mathsf M(d-n)$  (of weight $2n-d$); 
this leaves the underlying regular holonomic ${D}_{\mathcal X}$-module $\mathcal{M}$ unchanged,
but replaces the filtration $F_{\bullet} \mathcal{M}$ by the shift $F_{\bullet+n-d} \mathcal{M}$.
Similarly, we replace $\m{G}_{\bullet}$ by the shift $\m{G}_{\bullet+n-d}$. 

With this set-up, the assertions \ref{T:I:EHMGd} and \ref{T:I:EHMGe} hold by construction.
The assertion \ref{T:I:EHMGa} is established in 
\Cref{P:Strong-coherence-of-Higgs}\ref{P:Strong-coherence-of-Higgs(1)}.  The assertion \ref{T:I:EHMGb} follows from  \Cref{P:PS2.9}.
The assertion \ref{T:I:EHMGc} is established in \Cref{P:Strong-coherence-of-Higgs}\ref{P:Torsion-freeness-of-G-sm-locus} in the case where $\mathcal D=0$, and in \Cref{L:coherence} in general. \qed

\subsection{Resolution of the singular locus of the Hodge module} In \Cref{T:Existence-of-Hodge-module-and-G-subsheaf-theorem} we have no control over the singular locus of the Hodge module $\mathsf M$.  For the Higgs bundle construction in \S\ref{S:HiggConstSection}, we will want the singular locus of the Hodge module, together with the boundary divisor, to form an snc divisor.  
At the expense of taking a further birational base change, we can achieve this:  

\begin{teo}\label{T:EHM+GFB}

Assume one is in the  situation described in
  \S \ref{S:InnitialFamConst}, \S\ref{S:non-zero-section}, and \S \ref{S:FLRA}. 
In particular, one is given  the data in \eqref{E:infamf}, \eqref{E:infamdelta}, \eqref{E:infamA}, and \eqref{E:HdgMdsdef}, 
and one assumes that $\mathcal X$ is proper over $\mathbb C$. There exists a blow-up $\sigma:\widetilde {\mathcal X}\to \mathcal X$ and a closed substack $\widetilde {\mathbf T}\subseteq \widetilde {\mathcal X}$ so that setting $\widetilde {\mathcal X}^\circ:=\widetilde {\mathcal X}-\widetilde {\mathbf T}$, one has the following.     There is a  Hodge module $\widetilde {\mathsf M}$ on $\widetilde {\mathcal X}$ with strict support $\widetilde {\mathcal X}$ and a graded ${\mathscr{A}}_{\widetilde {\XX}^\circ}$-module $\widetilde {\m{G}}^\circ_{\bullet}$ that is coherent over $\mathscr{A}_{\widetilde {\XX}^\circ}$, so that $\widetilde {\mathsf M}^\circ :=\widetilde {\mathsf M}|_{\widetilde {\mathcal X}^\circ}$ and  $\widetilde {\m{G}}^\circ_{\bullet}$ satisfy  conditions \ref{T:I:EHMGb}--\ref{T:I:EHMGe} of \Cref{T:Existence-of-Hodge-module-and-G-subsheaf-theorem} on $\widetilde {\mathcal X}^\circ$ with $\mathcal A$ replaced with $\tilde {\mathcal A}^\circ:=\tilde{\mathcal A}|_{\widetilde {\mathcal X}^\circ}$, where $\tilde {\mathcal A}:= \sigma^*\mathcal A$.   Moreover, 
letting $\tilde {\mathbf \Delta}$ be the reduced snc divisor with support  $\sigma^{-1}\mathbf \Delta$ and 
 letting $\widetilde {\mathcal S}$ be the singular locus of $\widetilde {\mathsf M}$,  the divisor $\tilde{\mathsf \Delta}\cup \widetilde {\mathcal S}$ is snc and $\tilde {\mathcal A}(-\tilde {\mathbf \Delta})$ is big.

\end{teo}

\begin{proof}
Assume one is in the  situation described in  \S \ref{S:InnitialFamConst}, \S\ref{S:non-zero-section}, and \S \ref{S:FLRA}, in particular where one assumes that $\mathcal X$ is proper over $\mathbb C$.   Applying  \Cref{T:Existence-of-Hodge-module-and-G-subsheaf-theorem} to the morphism $f:(\mathcal Y,\mathcal D)\to \mathcal X$, one obtains 
a  Hodge module $ {\mathsf M}$ on $ {\mathcal X}$ with strict support $\mathcal X$ and a graded ${\mathscr{A}}_{ {\mathcal X}}$-module $ {\mathcal {G}}_{\bullet}$ that is coherent over $\mathscr{A}_{{\mathcal X}}$ (satisfying conditions \ref{T:I:EHMGa}--\ref{T:I:EHMGe} of \Cref{T:Existence-of-Hodge-module-and-G-subsheaf-theorem}).  Let $\mathcal S$ be the singular locus of $\mathsf M$.    

Now let  $\mathcal S'$ be a choice of effective snc divisor on $\mathcal X$ containing $\mathcal S$ with the properties described in \S \ref{S:FLRA}.  In particular, let $\sigma: \widetilde {\mathcal X}\to \mathcal X$ be a choice of  corresponding log resolution of the pair $(\mathcal X,\mathbf \Delta+\mathcal S')$ yielding a diagram  \eqref{E:infamTY} with the properties in \S \ref{S:FLRA}.  
The conditions \S \ref{S:FLRA} \ref{S:FLRX-1} and \ref{S:FLRX-2} allow us to apply  \Cref{T:Existence-of-Hodge-module-and-G-subsheaf-theorem} to the morphism $\tilde f^\circ:(\widetilde {\mathcal Y}^\circ,\widetilde {\mathcal D}^\circ)\to \widetilde {\mathcal X}^\circ$.  From this we obtain 
a  Hodge module $\widetilde {\mathsf M}^\circ$ on $\widetilde {\mathcal X}^\circ$ with strict support $\widetilde {\mathcal X}^\circ$ and a graded ${\mathscr{A}}_{\widetilde {\XX}^\circ}$-module $\widetilde {\m{G}}^\circ_{\bullet}$ that is coherent over $\mathscr{A}_{\widetilde {\XX}^\circ}$, satisfying conditions \ref{T:I:EHMGb}--\ref{T:I:EHMGe} of \Cref{T:Existence-of-Hodge-module-and-G-subsheaf-theorem} on $\widetilde {\mathcal X}^\circ$ with $\mathcal A$ replaced with $\tilde {\mathcal A}^\circ$.
By construction we have that 
the divisor $\tilde{\mathsf \Delta}\cup \sigma^{-1}{\mathcal S}$ is snc and $\tilde {\mathcal A}(-\tilde {\mathbf \Delta})$ is big.  

Now let $j:\widetilde {\mathcal X}^\circ \to \widetilde {\mathcal X}$  be the natural inclusion, and let $\widetilde {\mathsf M}:=j_{!*}\widetilde {\mathsf M}^\circ$ be the minimal extension  of $\widetilde {\mathsf M}^\circ$ to $\widetilde {\mathcal X}$ (see \S \ref{S:prelim-Dmod}).  The only thing to check now is that, defining $\widetilde {\mathcal S}$ to be the singular locus of $\widetilde {\mathsf M}$, we have that  $\tilde{\mathsf \Delta}\cup \widetilde {\mathcal S}$ is snc.  For this we use the conditions \S \ref{S:FLRA}\ref{S:FLRX-3} and \ref{S:FLRX-4}.  These conditions, in particular \eqref{E:infamTs}, tell us that all the ingredients of the construction of $\mathsf M$ and $\widetilde {\mathsf M}^\circ$ in \Cref{T:Existence-of-Hodge-module-and-G-subsheaf-theorem} agree on the common  dense open substack $\mathcal V:={\widetilde {\mathcal X}^\circ-\sigma^{-1}\mathcal S'} = \mathcal X-(\sigma (\widetilde {\mathbf T})\cup \mathcal S')$ of both $\widetilde {\mathcal X}$ and $\mathcal X$.
  So in fact we have
that
$\mathsf M$ and $\widetilde {\mathsf M}^\circ$ agree on  $\mathcal V$.  By definition $\mathsf M$ cannot have singular locus contained in $\mathcal V$, so neither can $\widetilde {\mathsf M}$.  Therefore,  $\widetilde {\mathcal S}$, which is of pure codimension $1$, must be contained in the complement of $\mathcal V$, whose divisorial locus is $\sigma^{-1}\mathcal S'$.  As  $\tilde{\mathsf \Delta}\cup \sigma^{-1}{\mathcal S}'$ is snc, we have that  $\tilde{\mathsf \Delta}\cup \widetilde {\mathcal S}$ is snc. 
\end{proof}

\section{Higgs bundle construction}\label{S:HiggConstSection}

The goal of this section is to prove the following theorem about graded logarithmic Higgs bundles, which builds off of the Hodge module construction in the previous subsection, and generalizes  \cite[Thm.~2.3]{PS17} and \cite[Prop.~6.3]{WW23}. We refer the reader to \cite[\S 7]{CMZpositivity}, where we discuss graded logarithmic Higgs bundles on stacks.

\begin{teo}

\label{T:refineHiggs}
Let $\mathcal X$ be a smooth proper integral DM stack over $\mathbb C$ with projective coarse moduli space, let $\mathsf \Delta$ be an snc divisor, let $\mathcal A$ be a line bundle on $\mathcal X$,  let $ {\mathbf T}\subseteq  {\mathcal X}$ be a closed substack of codimension at least $2$, and  set ${\mathcal X}^\circ:= {\mathcal X}- {\mathbf T}$.      Assume there  is a  Hodge module ${\mathsf M}$ on $ {\mathcal X}$ with strict support $ {\mathcal X}$ and a graded ${\mathscr{A}}_{{\XX}^\circ}$-module $ {\m{G}}^\circ_{\bullet}$ that is coherent over $\mathscr{A}_{ {\XX}^\circ}$, so that ${\mathsf M}^\circ := {\mathsf M}|_{ {\mathcal X}^\circ}$ and  ${\m{G}}^\circ_{\bullet}$ satisfy  conditions \ref{T:I:EHMGb}--\ref{T:I:EHMGe} of \Cref{T:Existence-of-Hodge-module-and-G-subsheaf-theorem} on $ {\mathcal X}^\circ$ with $\mathcal A$ replaced with ${\mathcal A}^\circ:={\mathcal A}|_{ {\mathcal X}}$.   Moreover,  
 letting ${\mathcal S}$ be the singular locus of $ {\mathsf M}$,  assume the divisor $\mathbf E:={\mathbf \Delta}\cup {\mathcal S}$ is snc.

There exists a graded $\mathscr{A}^\bullet_{ {\mathcal X}}(-\log  \mathbf \Delta)$-module $\m{F}_{\bullet}$  with the  following properties:
 \begin{enumerate}[label=(\alph*)]
\item
\label{P:I:HiggsFa}
$\mathcal F_0$ is a line bundle and  $\mathcal A(-{\mathbf \Delta}) \subseteq \mathcal F_0 $.

\item
\label{P:I:HiggsFb} $\mathcal F_k$ is a reflexive coherent $\mathcal O_{\mathcal X}$-module for each $k\ge 0$.

\item
\label{P:I:HiggsFc} 
There exists a (graded logarithmic) Higgs bundle $\mathcal E_\bullet$ on ${\mathcal X}$ with Higgs field
$$
\theta_\bullet :\mathcal E_\bullet \longrightarrow \Omega^1_{{\mathcal X}}(\log { {\mathbf E}}) \otimes\mathcal E_{\bullet +1}
$$ 
such that $\mathcal F_\bullet \subseteq \mathcal E_\bullet$ and $\theta_\bullet (\mathcal F_\bullet)\subseteq \Omega^1_{{\mathcal X}}(\log  \mathbf \Delta) \otimes\mathcal F_{\bullet +1}$.

\item
\label{P:I:HiggsFd} 
The pair $(\mathcal E_\bullet ,\theta_\bullet)$ 
is the (graded logarithmic) Higgs bundle associated to the Deligne extension with eigenvalues in $[0,1)$ of a  polarizable variation of Hodge structure on $\mathcal U:= {\mathcal X}- {\mathbf E}\subseteq \mathcal X$,  and $\mathcal E_k=0$ for $k<0$. 

\end{enumerate}
\end{teo}

The rest of this section is devoted to proving the theorem; the proof is summarized in \S \ref{S:ProofRefine}.  The meaning in \Cref{T:refineHiggs}\ref{P:I:HiggsFd} of the Deligne extension on the stack is explained in \S \ref{S:ConstHiggs}.

\subsection{Construction of the Higgs bundle}\label{S:ConstHiggs}

    Denote by $\mathcal{V}$ the polarizable variation of Hodge structure obtained by restricting $\mathsf M$ to the substack $$\mathcal U:=\X - \mathbf E\subseteq \mathcal X.$$
      In other words, for each \'etale cover $p:U\to{\mathcal X}$, the restriction of the \'etale sheaf $\mathsf M$ to $U$ defines a polarizable variation of Hodge structure over $U_{\mathcal U}:=p^{-1}(\mathcal U)\subseteq U$.

On each \'etale  cover $U\ra {\mathcal X}$, let $\widetilde{\mathcal{V}}_U$ be the canonical meromorphic extension (e.g., \cite[ Prop. ~II.2.18]{Deligne1970}) of the flat bundle $(\mathcal{V}_{U_{\mathcal U}}, \nabla_{U_{\mathcal U}})$ on  $U_{\mathcal U}\subseteq U$ to $U$, and let $\widetilde{\mathcal{V}}^{\geq \alpha}_U$ and $\widetilde{\mathcal{V}}^{>\alpha}_U$ be  Deligne's canonical lattices with eigenvalues contained in the intervals $[\alpha, \alpha+1)$ and $(\alpha, \alpha+1]$, 
respectively \cite[Prop. ~I.5.4]{Deligne1970}. Note that the flat connection on $\mathcal{V}_{U}$ extends uniquely to a logarithmic connection on $\widetilde{\mathcal{V}}^{\geq \alpha}_U$ and $\widetilde{\mathcal{V}}^{>\alpha}_U$. By a local computation, the residue map commutes with \'etale pullback and hence for any map of \'etale covers $e:U'\ra U$ of $\X$, the eigenvalues are still contained within the 
intervals $[\alpha, \alpha+1)$ and $(\alpha, \alpha+1]$, respectively. By \cite[Thm.~5.2.17]{HTT08}, it follows 
that there is a unique identification between $e^*\widetilde{\mathcal{V}}^{\geq\alpha}_U$ and $\widetilde{\mathcal{V}}^{\geq\alpha}_{U'}$ extending the 
identification between $e^*\widetilde{\mathcal{V}}_U$ and $\widetilde{\mathcal{V}}_{U'}$. The uniqueness of such an extension (there are no possible automorphisms of  $\widetilde{\mathcal{V}}^{\geq\alpha}_{U'}$ over  $\widetilde{\mathcal{V}}_{U'}$) implies that the collection  $\widetilde{\mathcal{V}}^{\geq\alpha}_{U'}$ gives descent data, which is what we mean by a \emph{logarithmic 
connection on the stack $\mathcal X$}.  Similarly, the collection of meromorphic connections give descent data, which is what we mean by a  meromorphic connection on the stack $\mathcal X$.  Consequently, on  ${\mathcal X}$, we get
\begin{equation}\label{logarithmic-conn-V>=0}
\nabla: \widetilde{\mathcal{V}}^{ \geq 0} \rightarrow \Omega_{\mathcal X}^{1}(\log \mathbf E) \otimes  \widetilde{\mathcal{V}}^{ \geq 0}.
\end{equation}

On each \'etale cover of ${\mathcal X}$, as a consequence of Schmid's nilpotent orbit theorem, the Hodge filtration $F_{\boldsymbol{\bullet}} \mathcal{V}$ extends to a filtration of $\widetilde{\mathcal{V}}^{ \geq 0}$ with locally free subquotients; see \cite[(3.10.7)]{Saito1990} for a discussion of this point. Let $j: \mathcal U   \hookrightarrow {\mathcal X}$ be the inclusion. Because the Hodge filtration is given by
\begin{equation}\label{F-filtr-on-log-0-VHS}
F_{k} \widetilde{\mathcal{V}} ^{\geq 0}=\widetilde{\mathcal{V}}^{ \geq 0} \cap j_{*} F_{k} \mathcal{V} 
\end{equation}
it follows that $\widetilde{\mathcal{V}}^{ \geq 0}$ is filtered at the level of the stack.

 On the associated graded with respect to the Hodge filtration, the connection then induces an $\m{O}_{{\mathcal X }}$-linear operator
\begin{equation}\label{E:Higgs-structure}
\theta_\bullet : \operatorname{gr}_{\bullet}^{F} \widetilde{\mathcal{V}}^{\ge 0} \rightarrow \Omega_{\mathcal X }^{1}(\log \mathbf E) \otimes_{\mathcal O_{\mathcal X }}\operatorname{gr}_{\bullet+1}^{F} \widetilde{\mathcal{V}}^{\geq 0}
\end{equation}
with the property that $\theta_\bullet \wedge \theta_\bullet =0$. 
The Higgs field in \eqref{E:Higgs-structure} has been described explicitly in local coordinates, i.e., \'etale locally,  in  \cite[p.694]{PS17}.
Setting
\begin{equation}\label{E:DefHiggsEE}
\m{E}_{\bullet}=\operatorname{gr}^{F}_{\bullet} \widetilde{\mathcal{V}}^{\geq 0}
\end{equation}
we therefore obtain a (graded logarithmic) Higgs bundle  $\left(\m{E}_{\bullet}, \theta _\bullet \right)$ on $\mathcal X $. Since $F_{k} \mathcal{V}=0$ for $k<0$, it is clear that we have 
\begin{equation}\label{E:EEk=0}
\mathcal {E}_{k}=0, \ \ k<0.
\end{equation}

To connect this Higgs bundle $(\mathcal E_\bullet, \theta_\bullet)$ to the Hodge module $\mathsf M$, and its associated regular holonomic filtered $D_{\mathcal X}$-module $(\mathcal M,F_\bullet)$, 
 we have the following generalization of  \cite[Lem.~2.13]{PS17}: 

\begin{lem}

    For every $k \in \mathbb{Z}$, we have an inclusion $$\m{E}_{k}|_{\mathcal X^\circ}=\operatorname{gr}_{k}^{F} \widetilde{\mathcal{V}}^{\geq 0}|_{\mathcal X^\circ} \subseteq \operatorname{gr}_{k}^{F} \mathcal{M}|_{\mathcal X^\circ}.$$
\end{lem}

\begin{proof}
This can be checked \'etale locally and therefore follows by \cite[Lem.~2.13]{PS17}. 
\end{proof}

We can now construct a collection of reflexive subsheaves $\m{F}_{k} \subseteq \m{E}_{k}$ by intersecting the sheaves $\mathcal {G}_{k}^\circ \subseteq \operatorname{gr}_{k}^{F} \mathcal{M}|_{\mathcal X^\circ}$ and $\m{E}_{k}|_{\mathcal X^\circ} \subseteq \operatorname{gr}_{k}^{F} \mathcal{M}|_{\mathcal X^\circ}$ inside the  coherent sheaf $\operatorname{gr}_{k}^{F} \mathcal{M}|_{\mathcal X^\circ}$. Since the intersection may not be reflexive, and the intersection is only defined on $\mathcal X^\circ$, denoting by $i:\mathcal X^\circ \hookrightarrow \mathcal X$ the inclusion, we actually define
\begin{equation}\label{E:DefHiggsFF}
\m{F}_{k}:=\left(i_*\left(\m{G}_{k}^\circ \cap \m{E}_{k}|_{\mathcal X^\circ}\right)^{\vee \vee} \right)^{\vee \vee} \subseteq \m{E}_{k}
\end{equation}
as the reflexive hull of the push forward of the reflexive hull of the intersection.  Note we are using that $\mathcal E_k$ is a vector bundle to identify $(i_*(\mathcal E_k|_{\mathcal X^\circ}))^{\vee \vee} = \mathcal E_k$.

\subsection{Properties of the Higgs bundle $\mathcal E_\bullet$ and the reflexive sub-sheaf $\mathcal F_\bullet$ on $\mathcal X$}\label{S:propertyHiggs}

We start with the following generalization of \cite[Prop.~2.14]{PS17}:

\begin{pro}

\label{P:F_0-only-log-pole-along-D_f}
   The Higgs field $\theta_\bullet$ of $\mathcal E_\bullet$  maps the subsheaf $\m{F}_{k}$ into $\Omega_{\mathcal X }^{1}\left(\log \mathbf \Delta  \right) \otimes \m{F}_{k+1}$.
   \end{pro}

\begin{proof}
As  $\m{F}_{k}$ and $\Omega_{\mathcal X }^{1}\left(\log \mathbf \Delta  \right) \otimes \m{F}_{k+1}$ are reflexive sheaves, it suffices to check this on $\mathcal X^{\circ}$.  
Moreover, this can be checked \'etale locally and therefore is reduced to the case of varieties.  In the case where $\mathcal D=0$, this then follows from  \cite[Prop.~2.14.]{PS17}.  In the  general case, one uses the argument in the proof of \cite[Prop.~6.3 and p.736]{WW23}; the key point is to replace the use of \cite[Prop.~2.10]{PS17} (i.e., \Cref{P:Strong-coherence-of-Higgs}\ref{P:Strong-coherence-of-Higgs(1)}) in the proof of \cite[Prop.~2.14.]{PS17}, with \cite[Lem.~6.2]{WW23} (i.e., \Cref{L:coherence}). 
\end{proof}

We also have the following generalization of  \cite[Prop.~2.15.]{PS17}:

\begin{pro}

\label{P:Big-line-bundle-in-F_0}

    We have that $\mathcal F_0$ is a line bundle and \begin{equation}
\label{E:P:Big-line-bundle-in-F_0}
    \mathcal A\left(-\mathbf \Delta \right) \subseteq \m{F}_{0}.
\end{equation}
\end{pro}

\begin{proof}
As both $    \mathcal A\left(-\mathbf \Delta \right)$ and $ \m{F}_{0}$ are reflexive, it suffices to show that $\mathcal F_0$ is rank $1$  (e.g., \cite[Prop.~1.9]{Hart80}), and to check the inclusion \eqref{E:P:Big-line-bundle-in-F_0} on $\mathcal X^\circ$.  
By assumption (\Cref{T:Existence-of-Hodge-module-and-G-subsheaf-theorem}\ref{T:I:EHMGb}), we have an isomorphism  $\mathcal A^\circ \cong \mathcal G_0^\circ$, and by construction we have that $\mathcal G^\circ_0$ and $\mathcal F_0|_{\mathcal X^\circ}\subseteq \mathcal E_0|_{\mathcal X^\circ}$ are both contained in $F_0\mathcal M|_{\mathcal X_0}$.   In other words, $\mathcal A(-\mathbf \Delta)|_{\mathcal X^\circ}$ and $\mathcal F_0|_{\mathcal X^\circ}$ can both be viewed as subsheaves of the same sheaf $F_0\mathcal M|_{\mathcal X_0}$.  As such, 
 the containment \eqref{E:P:Big-line-bundle-in-F_0}  can be checked on $\mathcal X^\circ$  \'etale locally.
     Hence, the containment  follows from the case of varieties, which is given in \cite[Prop.~2.15]{PS17} and \cite[p.736]{WW23}.

It remains to show that $\operatorname{rk}\mathcal F_0=1$.
As this can be checked \'etale locally, this also follows from  \cite[Prop.~2.15]{PS17} and \cite[p.736]{WW23}.
  More directly, from the containment \eqref{E:P:Big-line-bundle-in-F_0}, we only have to show $\operatorname{rk}\mathcal F_0\le1$.  
As mentioned above, by assumption, we have an isomorphism  $\mathcal A^\circ \cong \mathcal G_0^\circ$.      From \eqref{E:DefHiggsFF} we have that $\mathcal F_0$ is generically contained in $\mathcal G^\circ_0$, so that 
      $\mathcal F_0$ is generically contained in $ \mathcal A$, therefore $\operatorname{rk}\mathcal F_0\le 1$.     
\end{proof}

\subsection{Proof of \Cref{T:refineHiggs}}\label{S:ProofRefine}

\begin{proof}[Proof of \Cref{T:refineHiggs}]
The proof  of \Cref{T:refineHiggs} is obtained by putting together the results in \S \ref{S:ConstHiggs}--\ref{S:propertyHiggs} about the 
 graded $\mathcal{A}_{ {\mathcal X}}(-\log  \mathbf \Delta)$-module $\m{F}_{\bullet}$ 
that we have established so far.  We take $(\mathcal E_\bullet,\theta_\bullet)$ and $\mathcal F_\bullet$ to be the graded logarithmic Higgs bundle and  graded $\mathcal{A}_{ {\mathcal X}}(-\log  \mathbf \Delta)$-module defined in \eqref{E:DefHiggsEE} and \eqref{E:DefHiggsFF}.
With this set-up, the assertions  \ref{P:I:HiggsFb} and \ref{P:I:HiggsFd}  of  \Cref{T:refineHiggs} hold by construction (see \eqref{E:EEk=0} for the vanishing).  
The assertion \ref{P:I:HiggsFa} is established in \Cref{P:Big-line-bundle-in-F_0}, and the assertion
\ref{P:I:HiggsFc} is established in \Cref{P:F_0-only-log-pole-along-D_f}.
\end{proof}

\section{Proof of the Main Theorem}\label{S:ProofMain}

\begin{proof}[Proof of \Cref{T:main-pairs}]

We start with a family $f:(\mathcal Y,\mathcal D)\to \mathcal X$ of stable pairs satisfying the conditions in  \Cref{T:main-pairs}. 
Let $\widetilde {\mathcal X}\to \mathcal X$ be a log resolution of  $(\mathcal X, \mathbf \Delta)$, and let $(\widetilde {\mathcal Y},\widetilde {\mathcal D})\to \widetilde {\mathcal X}$ be the family of stable pairs obtained by base change.  
Let $\tilde {\mathbf \Delta}\subseteq \widetilde {\mathcal X}$ be the reduced pre-image of $\mathbf \Delta$.  
Let $\mathcal U=\mathcal X-\mathbf \Delta$ be the complement of the divisor $\mathbf \Delta$; since the morphism $\widetilde {\mathcal X}\to \mathcal X$ is an isomorphism on $\mathcal U$, and $f|_{\mathcal U}:\mathcal Y|_{\mathcal U}\to \mathcal U$ is smooth  (we assumed $\mathbf \Delta_f\subseteq \mathbf \Delta$), 
 it follows that $\mathbf \Delta_{\tilde f}\subseteq \widetilde {\mathcal X}$ is contained in $\tilde {\mathbf \Delta}$.  It follows from \Cref{L:sncRdx} that in order to prove that $K_{\mathcal X}+\mathbf \Delta$ is big,  it suffices to prove that $K_{\widetilde {\mathcal X}}+\tilde {\mathbf \Delta}$ is big.  In other words, it suffices to prove  \Cref{T:main-pairs} under the assumption that $\mathbf \Delta$ is an snc divisor. Moreover, going forward we will continue to use  \Cref{L:sncRdx} without mention to replace $\mathcal X$ with a blow-up, and $\mathbf \Delta$ with its reduced pre-image.

Fix a line bundle $\mathcal A$ on $\mathcal X$ such that $\mathcal A(-\mathbf \Delta)$ is big.  
 The assertions of \Cref{P:TWW5.2} imply that, after a further blow-up of $\mathcal X$, we may replace the family of stable pairs $f:(\mathcal Y,\mathcal D)\to \mathcal X$ with a new morphism $f:(\mathcal Y,\mathcal D)\to \mathcal X$  satisfying the hypotheses of  \S \ref{S:InnitialFamConst} and \S\ref{S:non-zero-section}. 
 In particular, we now assume that $\mathcal Y$ is smooth and no longer require that $(\mathcal Y,\mathcal D)$ be a family of stable pairs over $\mathcal X$.
 \Cref{L:PWW5.3}  (see also \Cref{R:L([D])section}) implies that $f:(\mathcal Y,\mathcal D)\to \mathcal X$ satisfies \S \ref{S:FLRA}, as well; see especially \Cref{R:S:FLRX-2} regarding \S \ref{S:FLRA}\ref{S:FLRX-2}. Although we will not need this, observe that, in the notation of  \S \ref{S:FLRA}, one may take the divisor $\mathcal S'=\mathcal S$, and one may take the codimension at least $2$ substack $\widetilde {\mathbf T}$ to be empty. 
    Note that in addition to the morphism $f:(\mathcal Y,\mathcal D)\to \mathcal X$ and the snc divisor $\mathbf \Delta$, we now have  the additional data of a  line bundle $\mathcal A$ on $\mathcal X$ such that $\mathcal A(-\mathbf \Delta)$ is big.  

Now that we have $f:(\mathcal Y,\mathcal D)\to \mathcal X$ satisfying the conditions in
   \S \ref{S:InnitialFamConst}, \S\ref{S:non-zero-section}, and \S \ref{S:FLRA},
 \Cref{T:EHM+GFB}
implies that, after again replacing $\mathcal X$ with a further blow-up,  there is a closed substack $ {\mathbf T}\subseteq  {\mathcal X}$ of codimension at least $2$ (which one may take to be empty), so that   setting ${\mathcal X}^\circ:= {\mathcal X}- {\mathbf T}$,        there  is a  Hodge module ${\mathsf M}$ on $ {\mathcal X}$ with strict support $ {\mathcal X}$ and a graded ${\mathscr{A}}_{{\XX}^\circ}$-module $ {\m{G}}^\circ_{\bullet}$ that is coherent over $\mathscr{A}_{ {\XX}^\circ}$, so that ${\mathsf M}^\circ := {\mathsf M}|_{ {\mathcal X}^\circ}$ and  ${\m{G}}^\circ_{\bullet}$ satisfy  conditions \ref{T:I:EHMGb}--\ref{T:I:EHMGe} of \Cref{T:Existence-of-Hodge-module-and-G-subsheaf-theorem} on $ {\mathcal X}^\circ$ with $\mathcal A$ replaced with ${\mathcal A}^\circ:={\mathcal A}|_{ {\mathcal X}^\circ}$.   Moreover,  $\mathcal A(-\mathbf \Delta)$ is big, and  
 letting ${\mathcal S}$ be the singular locus of $ {\mathsf M}$,  the divisor $\mathbf E:={\mathbf \Delta}\cup {\mathcal S}$ is snc.

We may therefore employ \Cref{T:refineHiggs}, 
and obtain graded sheaves $\mathcal F_\bullet\subseteq  \mathcal E_\bullet$ on $ {\mathcal X}$, with $\mathcal F_\bullet$ a graded sub-module of a  graded logarithmic Higgs bundle $\mathcal E_\bullet$ that extends a variation of Hodge structure on ${\mathcal X}-{\mathbf E}$, with the property that $\theta_\bullet (\mathcal F_\bullet)\subseteq \Omega^1_{{\mathcal X}}(\log {\mathbf \Delta}) \otimes \mathcal F_{\bullet +1}$, and $\mathcal F_0$ is a big line bundle.  This data allows us to use \cite[Thm.~B]{CMZpositivity} to obtain a Viehweg--Zuo sheaf, i.e., a coherent sheaf $\mathcal H$ on ${\mathcal X}$ with big determinant, and for some $s\geq 1,$ an inclusion 
$$
\xymatrix{
\mathcal H \ar@{^(->}[r]& \Omega^1_{ {\mathcal X}}(\log  {\mathbf \Delta})^{\otimes s}.
}
$$
 Finally, taking the saturation of $\mathcal H$, then from \cite[Cor.~B]{CMZfoliations}, such an inclusion implies $K_{ {\mathcal X}}+ {\mathbf \Delta}$ is big, completing the proof. 
\end{proof}


\ifArxiv

\appendix 

\section{\emph{Ad hoc} push forward and fake derived categories}
\label[appendix]{S:A:AdHocGen}

We will be considering descent arguments for proper push forwards of sheaves in various categories.  While these descent techniques we will use are not well suited to push forward at the level of derived categories because derived categories of sheaves do not satisfy descent in general,  the techniques do still provide push forwards at the level of $i$-th derived functors, in a way that we will explain here.    These techniques also provide information about push forwards of objects in derived categories where one simply keeps track of the descent data, and we describe this in terms of what we call, for convenience,  ``fake" derived categories.

\subsection{Category of sheaves, proper push forward, and \'etale pull back} 
\label{S:PPFs}
We will assume that given a space $\mathcal X$ (a variety, complex analytic space, or DM stack) with some Grothendieck topology, we have fixed some abelian category of sheaves $\mathsf A_{\mathcal X}$ on $\mathcal X$, with derived category $\mathsf D_{\mathcal X}:=D(\mathsf A_{\mathcal X})$.  Typically $\mathsf A_X$ will be something like the category of $D$-modules on a variety $X$, and $\mathsf D_X$ will be the associated derived category. 

For morphisms $f:X\to Y$ of \emph{varieties or complex analytic spaces}, if $f$ is \'etale (for that class of spaces) then we will assume that there is a pull back functor $f^*:\mathsf A_Y\to \mathsf A_X$ (resp.~$f^*:\mathsf D_Y\to \mathsf D_X$).  On the other hand, if  $f$ is proper then we will assume there is a push forward functor  $f_*:\mathsf A_X\to \mathsf A_Y$ (resp.~$Rf_*:\mathsf D_X\to \mathsf D_Y$).   We will use $\mathcal H^i:\mathsf D_X\to \mathsf A_X$ to denote the $i$-th cohomology functor with respect to a fixed $t$-structure on $\mathsf D_X$, which will be the standard $t$-structure unless stated otherwise; we will assume that $\mathcal H^i$ commutes with \'etale pull back.

\subsection{Category of objects with descent data and fake derived categories}\label{S:ConcreteDesc}

We now recall one of the standard concrete descriptions of categories of objects with descent data that will be convenient for us to work with.

An object  $\mathsf M$ of $\mathsf A$ (resp.~$\mathsf D$) with descent data on $\mathcal X$ is the following collection of data.  First, for every \'etale morphism $$p:U\to \mathcal X,$$  i.e., for every element $(U,p: U\to \mathcal X)$ of $\mathcal X_{\text{\'etale}}$, one has an object 
\begin{equation}\label{E:MUdescData}
\mathsf M_U=\mathsf M_{(U,p:U\to \mathcal X)}
\end{equation}
 of $\mathsf A_U$ (resp.~$\mathsf D_U$) on $U$.  

Second, we consider every diagram
\begin{equation}\label{E:ConcreteDesc1step}
\xymatrix{ U' \ar[r]^{\phi } & U  \ar[r]^{p } & \mathcal X,
}
\end{equation}
where precisely we mean that we have a morphism  $(\phi,\alpha):(U,p:U\to \mathcal X)\to (U',p':U'\to \mathcal X)$ in $\mathcal X_{\text{\'etale}}$, meaning concretely that $\phi:U'\to U$ is an \'etale morphism and   $\alpha: p'\stackrel{\sim}{\to}p\circ \phi$ is an isomorphism.  
For every such diagram, one has isomorphisms 
\begin{equation}\label{E:ThetaPhidescData}
\theta_{\phi }=\theta_{(\phi,\alpha)}: \phi ^* \mathsf M_{ U }\stackrel{\sim}{\to}  \mathsf M_{U' }.
\end{equation}

Finally, we consider every diagram  of the form 
\begin{equation}\label{E:ConcreteDesc}
\xymatrix{
U''\ar[r]^{\phi' }& U' \ar[r]^{\phi } & U  \ar[r]^{p } & \mathcal X,
}
\end{equation}
where, precisely, we mean that  we have a commutative diagram of morphisms 
$$
\xymatrix{
(U'',p'')\ar[rr]^{(\phi',\alpha')} \ar[rd]_{(\phi'',\alpha'')}&&(U',p')\ar[ld]^{(\phi,\alpha)}\\
&(U,p)&
}
$$
in $\mathcal X_{\text{\'etale}}$, 
meaning concretely that $\phi''=\phi\circ \phi'$, and $\alpha''=\phi'^*\alpha \circ \alpha'$, i.e., the following diagram commutes:
$$
\xymatrix{
p'' \ar[r]^{\alpha''} \ar[d]_{\alpha'}& p\circ \phi'' \ar@{=}[d]\\
p'\circ \phi' \ar[r]^{\phi'^*\alpha}& p\circ \phi\circ \phi'.
}
$$
We require that for every diagram \eqref{E:ConcreteDesc}, the isomorphisms  
\begin{equation}\label{E:MMDescExample}
\theta_{\phi }: \phi ^* \mathsf M_{ U }\stackrel{\sim}{\to}  \mathsf M_{U' }, \quad \theta_{\phi '}: \phi '^* \mathsf M_{ U' }\stackrel{\sim}{\to}  \mathsf M_{U'' }, \quad \theta_{\phi \circ \phi '}: (\phi \circ \phi ')^* \mathsf M_{ U }\stackrel{\sim}{\to}  \mathsf M_{U'' }, 
\end{equation}
satisfy the compatibility condition 
\begin{equation}\label{E:CoCycledescData}
\theta_{\phi \circ \phi' }= \theta_{\phi' }\circ \phi '^*\theta_{\phi },
\end{equation}
i.e., the following diagram of isomorphisms commutes: 
\begin{equation}\label{E:DescDataDiagMMUUA}
\xymatrix@C=4em{
 \phi '^*\phi ^* \mathsf M_{ U }\ar[r]^<>(0.5){\phi '^*\theta_{\phi }} \ar@{=}[d]& \phi '^* \mathsf M_{ U' } \ar[d]^{\theta_{\phi' } }\\
(\phi \circ \phi ')^* \mathsf M_{ U } \ar[r]^<>(0.5){\theta_{\phi \circ \phi '}}&  \mathsf M_{U'' }.
}
\end{equation}

In summary,  an object  $\mathsf M$ of $\mathsf A$ (resp.~$\mathsf D$) on $\mathcal X$ with  descent data is the collection of data 
$$
\{\mathsf M_{U},\phi_\theta\}
$$
with the $\mathsf M_U$ as in \eqref{E:MUdescData}, and the $\theta_\phi$ as in \eqref{E:ThetaPhidescData}, all satisfying the compatibility condition of \eqref{E:CoCycledescData}.

A morphism of objects with descent data
$$
\psi:\{\mathsf M,\theta_\phi\}\longrightarrow \{\mathsf N,\tau_\phi\}
$$
is the data of morphisms $\psi_U:\mathsf M_U\to \mathsf N_U$ such that for every diagram \eqref{E:ConcreteDesc1step}, one has 
$$
\tau_\phi\circ \phi^*\psi_U = \psi_{U'}\circ \theta_\phi;
$$
i.e., the following diagram commutes:
\begin{equation}\label{E:DescDataDiagMorph}
\xymatrix{
\phi^*\mathsf M_U\ar[r]^{\phi^*\psi_U} \ar[d]_{\theta_\phi}& \phi^*\mathsf N_U \ar[d]^{\tau_\phi}\\
\mathsf M_{U'}\ar[r]^{\psi_{U'}}& \mathsf N_{U'}.
}
\end{equation}

We note that this implies that the following diagram commutes:
\begin{equation*}
\xymatrix@C=.7em@R=1em{
& \phi '^*\phi ^* \mathsf N_{ U }\ar[rr]^<>(0.5){\phi '^*\tau_{\phi }} \ar@{=}[d]&& \phi '^* \mathsf N_{ U' } \ar[dd]^{\tau_{\phi' } }\\
 \phi '^*\phi ^* \mathsf M_{ U }\ar[rr]^<>(0.75){\phi '^*\theta_{\phi }} \ar@{=}[dd] \ar[ru]^{\phi'^*\phi^*\psi_U}&\ar@{=}[d]& \phi '^* \mathsf M_{ U' } \ar[dd]^<>(0.25){\theta_{\phi' } } \ar[ru]^<>(0.25){\phi'^*\psi_{U'}}&\\
&(\phi \circ \phi ')^* \mathsf N_{ U } \ar@{-}[r]^<>(0.5){\tau_{\phi \circ \phi '}}&\ar[r]&  \mathsf N_{U'' }\\
(\phi \circ \phi ')^* \mathsf M_{ U } \ar[rr]^<>(0.5){\theta_{\phi \circ \phi '}} \ar[ru]^{(\phi\circ \phi')^*\psi_U}&&  \mathsf M_{U'' }\ar[ru]_{\psi_{U''}}&
}
\end{equation*}
Indeed, the left face obviously commutes.  The front and back face commute due to \eqref{E:DescDataDiagMMUUA}.  The bottom and right face commute due to \eqref{E:DescDataDiagMorph}, and the top face commutes by applying $\phi'^*$ to \eqref{E:DescDataDiagMorph}.

We will assume that $\mathsf A_{\mathcal X}$ is equivalent to the category of 
 objects of   $\mathsf A$ with descent data on $\mathcal X$.  On the other hand, we will use the notation 
 \begin{equation}\label{E:genDefDfake}
\mathsf D_{\mathcal X,fake}
 \end{equation}
 for the category  of  objects of $\mathsf D$ with descent data on $\mathcal X$, and will call this the \emph{fake derived category} on $\mathcal X$. 
 There is a natural functor $\operatorname{Ch}(\mathsf A_{\mathcal X})\to \mathsf D_{\mathcal X,fake}$, from the category of chain complexes of objects of $\mathsf A$ on $\mathcal X$, since such chain complexes come equipped with descent data.  This functor induces a functor on the homotopy category, and the derived category, since homotopies and quasi-isomorphisms give homotopies and quasi-isomorphisms locally; in other words there is a natural functor 
\begin{equation}\label{E:genCompareDfake}
\mathsf D_{\mathcal X}:=D(\mathsf A_{\mathcal X})\longrightarrow  \mathsf D_{\mathcal X,fake}.
\end{equation}
In general this may not be an equivalence, which we explain in \Cref{R:FakeNoGood} together with other issues that make the category $\mathsf D_{\mathcal X,fake}$ unsatisfactory; however, for our purposes, it is convenient, and suffices for what we need.

\subsection{\emph{Ad hoc} derived push forward}
\label{S:PPDiagsVDer}

We will use the following set up.  
Given  a proper morphism of smooth varieties (resp.~smooth complex analytic spaces) $f:X\to Y$, we will consider the  fibered product diagram
\begin{equation}\label{E:PresDpf-varA}
\xymatrix{
U_X \ar[r]^{p_X}\ar[d]_{f'}& X \ar[d]_f\\
U \ar[r]^p& Y
}
\end{equation}
where $p:U\to Y$ is an \'etale morphism of smooth varieties (resp.~smooth complex analytic spaces). 

Using the notation in \S \ref{S:PPFs},  given an object $\mathsf M$ in $\mathsf D_X$,  we will assume there is a natural isomorphism in $\mathsf D_U$:  
\begin{equation}\label{E:DefPushPullalpha}
\alpha_p=\alpha_{p,\mathsf M}: p^*Rf_*\mathsf M \stackrel{\sim}{\longrightarrow} Rf'_*p_X^*\mathsf M. 
\end{equation}
In particular, we will assume these isomorphisms satisfy the following compatibility properties.  First, the isomorphisms $\alpha$ are functorial in $\mathsf M$ in the sense that given a morphism  $\beta:\mathsf M\to \mathsf N$ of objects in $\mathsf D_X$, we have $Rf'_*(p_X^*\beta) \circ \alpha_{p,\mathsf M}= \alpha_{p,\mathsf N} \circ p^*(Rf_*\beta)$; i.e., we have 
 a commutative diagram
\begin{equation}\label{E:NatProp1MN}
\xymatrix@C=3em{
 p^*Rf_*\mathsf M \ar[r]^{p^*(Rf_*\beta)} 
 \ar[d]^{\alpha_{p,\mathsf M}}&  p^*Rf_*\mathsf N  \ar[d]^{\alpha_{p,\mathsf N}} \\
Rf'_*p_X^*\mathsf M \ar[r]^{Rf'_*(p_X^*\beta)}
 & Rf'_*p_X^*\mathsf N.
 }
  \end{equation}
  
  In addition, we assume they are compatible with further pull back, in the sense that if $\phi:U'\to U$ is another \'etale morphism of smooth varieties (resp.~smooth complex analytic spaces), so that we  have the  fibered product  diagram 
  \begin{equation}\label{E:PresDpf-var2}
\xymatrix{
U_X'\ar[r]^{\phi_X} \ar[d]_{f''}& U_X \ar[r]^{p_X}\ar[d]_{f'}& X \ar[d]_f\\
U'\ar[r]^{\phi}&U \ar[r]^p& Y
}
\end{equation}
then we have $\alpha_{p\circ \phi,\mathsf M}= \alpha_{\phi,p_X^*\mathsf M}\circ \phi^*(\alpha_{p,\mathsf M})$, i.e., we have a 
 a commutative diagram
\begin{equation}\label{E:PresDpf-var3}
\xymatrix@C=3em@R=2em{
\phi^*p^*Rf_*\mathsf M \ar@{=}[dd] \ar[r]^{\phi^*(\alpha_{p,\mathsf M})}& \phi^*Rf'_*p_X^*\mathsf M \ar[d]^{\alpha_{\phi,p_X^*\mathsf M}}\\
& Rf''_*\phi_X^*p_X^*\mathsf M \ar@{=}[d] \\
(p\circ \phi)^*Rf_*\mathsf M\ar[r]^<>(0.5){\alpha_{p\circ \phi,\mathsf M}}& Rf''_*(p_X\circ \phi_X)^*\mathsf M
}
\end{equation}

\begin{dfn}[\'Etale base change for proper push forwards]\label{D:EBCPPF}
We say that \emph{proper push forwards of objects of $\mathsf D$ satisfy \'etale base change for smooth varieties (resp.~smooth complex analytic spaces)} if for every proper morphism $f:X\to Y$ of smooth varieties (resp.~smooth complex analytic spaces), every object $\mathsf M$ in $\mathsf D_X$, and every \'etale morphism of smooth varieties (resp.~smooth complex analytic spaces)  $p:U\to Y$, there are natural isomorphisms $\alpha_p$ as in \eqref{E:DefPushPullalpha} satisfying  \eqref{E:PresDpf-var3} and  \eqref{E:NatProp1MN}.  
\end{dfn}

Assuming proper push forwards of objects of $\mathsf D$ satisfy \'etale base change for smooth varieties (resp.~smooth complex analytic spaces), then one can define an \emph{ad hoc} push forward on stacks. 
To explain this, we start by fixing some notation.
Given   a schematic proper morphism  $f: \mathcal X\to \mathcal Y$ of smooth separated integral DM stacks  locally of finite type over $\mathbb C$ (resp.~the analytification of such a morphism), and a sequence of \'etale morphisms 
$
\xymatrix@C=1em{
U''\ar[r]^{\phi'} & U'\ar[r]^\phi & U \ar[r]^p& \mathcal Y
}
$ with $U$, $U'$, and $U''$  smooth varieties (resp.~smooth complex analytic spaces), we 
will consider the $2$-categorical  fibered product diagram
\begin{equation}\label{E:TubPresDpf'}
\xymatrix{
U''_{\mathcal X}\ar[r]^{\phi'_{\mathcal X}} \ar[d]_{f'''}& U'_{\mathcal X}\ar[r]^{\phi_{\mathcal X}} \ar[d]_{f''}& U_{\mathcal X} \ar[r]^{p_{\mathcal X}} \ar[d]_{f'}& \mathcal X \ar[d]_f\\
U''\ar[r]^{\phi'} & U'\ar[r]^\phi & U \ar[r]^p& \mathcal Y.
}
\end{equation}
We  have the following lemma:

\begin{lem}\label{L:GenDesc}
Let $f:\mathcal X\to \mathcal Y$ be a proper morphism of smooth separated integral DM stacks of finite type over $\mathbb C$ (or the analytification of such a morphism).  
Assume that proper push forwards of objects of $\mathsf D$ satisfy \'etale base change for smooth varieties (resp.~smooth complex analytic spaces); see \Cref{D:EBCPPF}. Then there is an \emph{ad hoc} push forward  functor
$$
Rf_*^{ah}:\mathsf D_{\mathcal X,fake}\longrightarrow \mathsf D_{\mathcal Y,fake}
$$
defined as follows.
Given  $\mathsf M$  in $\mathsf D_{\mathcal X,fake}$, 
i.e., an object $\{\mathsf M_U,\theta_\phi\}$ in $\mathsf D$ with descent data on $\mathcal X$ (\S \ref{S:ConcreteDesc}), then, in the notation of   \eqref{E:TubPresDpf'} and \eqref{E:DefPushPullalpha}, the data
$$
\{Rf'_*\mathsf M_{U_{\mathcal X}},\tau_\phi:=\alpha_\phi \circ f''_*(\theta_{\phi_{\mathcal X}})\}
$$
defines an object in $\mathsf D$ with descent data on $\mathcal Y$, i.e., an object of $\mathsf D_{\mathcal Y,fake}$, which we define to be  $Rf^{ah}_*\mathsf M$ and call   the \emph{ad hoc} push forward.

\end{lem}

\begin{proof}
Consider  $Rf'_*\mathsf M_{U_{\mathcal X}}$, $Rf''_*\mathsf M_{U'_{\mathcal X}}$, and $Rf'''_*\mathsf M_{U''_{\mathcal X}}$ in $\mathsf D_{U}$, $\mathsf D_{U'}$, and $\mathsf D_{U''}$, respectively. 
Pushing forward the isomorphisms in \eqref{E:MMDescExample}, and then using  the data \eqref{E:DefPushPullalpha}, we obtain isomorphisms:

\begin{equation}\label{E:tauDescData}
\xymatrix@R=.5em@C=2em{
\tau_\phi:\phi^*Rf'_* \mathsf M_{ U_{\mathcal X}} \ar[r]^{\alpha_{\phi}} &Rf''_* \phi_{\mathcal X}^*\mathsf M_{ U_{\mathcal X}} \ar[rr]^<>(0.5){f''_*(\theta_{\phi_{\mathcal X}})}&&  Rf''_*\mathsf M_{U'_{\mathcal X}}\\
\tau_{\phi'}:\phi'^*Rf''_* \mathsf M_{ U'_{\mathcal X}} \ar[r]^{\alpha_{\phi'}}  &Rf'''_*  \phi_{\mathcal X}'^*\mathsf M_{ U'_{\mathcal X}} \ar[rr]^<>(0.5){f'''_*(\theta_{\phi_{\mathcal X}'})}&& Rf'''_*\mathsf M_{U''_{\mathcal X}} \\
\tau_{\phi\circ \phi'}:(\phi\circ \phi')^*Rf'_* \mathsf M_{ U_{\mathcal X}} \ar[r]^<>(0.5){\alpha_{\phi \circ \phi'}}
   &Rf'''_*  (\phi_{\mathcal X}\circ \phi_{\mathcal X}')^*\mathsf M_{ U_{\mathcal X}}  \ar[rr]^<>(0.5){f'''_*(\theta_{\phi_{\mathcal X}\circ \phi_{\mathcal X}'})}&&Rf'''_* \mathsf M_{U''_{\mathcal X}}. 
}
\end{equation}

We want to know if $
\tau_{\phi\circ \phi'}= \tau^*_{\phi'}\circ \phi'^*\tau_\phi
$, i.e., if the following diagram commutes:
\begin{equation}\label{E:BigComDiagDesc0'}
\xymatrix@C=2em{
\phi'^*\phi^*Rf'_* \mathsf M_{ U_{\mathcal X}} \ar[r]_{\phi'^*(\alpha_{\phi})}  \ar@/^1.5pc/[rrr]^{\phi'^*(\tau_{\phi})}  \ar@{=}[dd] &\phi'^*Rf''_* \phi_{\mathcal X}^*\mathsf M_{ U_{\mathcal X}} \ar[rr]_<>(0.5){\phi'^*(f''_*(\theta_{\phi_{\mathcal X}}))}&&  \phi'^* Rf''_*\mathsf M_{U'_{\mathcal X}} \ar[d]_{\alpha_{\phi'}}  \ar@/^4pc/[dd]^{\tau_{\phi'}} \\
&&&Rf'''_*  \phi_{\mathcal X}'^*\mathsf M_{ U'_{\mathcal X}}  \ar[d]_<>(0.5){f'''_*(\theta_{\phi_{\mathcal X}'})} \\
(\phi\circ \phi')^*Rf'_* \mathsf M_{ U_{\mathcal X}} \ar[r]^<>(0.5){\alpha_{\phi\circ \phi'}}
\ar@/_1.5pc/[rrr]_{\tau_{\phi\circ \phi'}} 
 & Rf'''_*  (\phi_{\mathcal X}\circ \phi_{\mathcal X}')^*\mathsf M_{ U_{\mathcal X}}  \ar[rr]^<>(0.5){f'''_*(\theta_{\phi_{\mathcal X}\circ \phi_{\mathcal X}'})} &&Rf'''_* \mathsf M_{U''_{\mathcal X}}. 
}
\end{equation}

Note that if we apply push forwards to the diagram \eqref{E:DescDataDiagMMUUA},  we have a commutative diagram 
\begin{equation}\label{E:f*DescData}
\xymatrix@C=6em{
 Rf'''_* \phi_{\mathcal X}'^*\phi_{\mathcal X}^*\mathsf M_{ U_{\mathcal X}}\ar[r]^<>(0.5){f'''_*(\phi_{\mathcal X}'^*\theta_{\phi_{\mathcal X}})} \ar@{=}[d]& Rf'''_*\phi_{\mathcal X}'^*\mathsf M_{ U'_{\mathcal X}} \ar[d]^{f'''_*(\theta_{\phi'_{\mathcal X}}) }\\
Rf'''_*(\phi_{\mathcal X}\circ \phi_{\mathcal X}')^*\mathsf M_{ U_{\mathcal X}} \ar[r]^<>(0.5){f'''_*(\theta_{\phi_{\mathcal X}\circ \phi_{\mathcal X}'})}& Rf'''_*\mathsf M_{U''_{\mathcal X}}.
}
\end{equation}
Using, in addition, the  isomorphism $\alpha_{\phi',\phi_{\mathcal X}^*\mathsf M_{U_{\mathcal X}}}: \phi'^*Rf''_* \phi_{\mathcal X}^*\mathsf M_{ U_{\mathcal X}}) \stackrel{\sim}{\to} Rf'''_* \phi_{\mathcal X}'^*\phi_{\mathcal X}^*\mathsf M_{ U_{\mathcal X}})  $, we can fill in diagram \eqref{E:BigComDiagDesc0'} above to
\begin{equation}\label{E:BigComDiagDesc}
\xymatrix@C=2em{
\phi'^*\phi^*Rf'_* \mathsf M_{ U_{\mathcal X}} \ar[r]_{\phi'^*(\alpha_{\phi})}  \ar@/^1.5pc/[rrr]^{\phi'^*(\tau_{\phi})}  \ar@{=}[dd] &\phi'^*Rf''_* \phi_{\mathcal X}^*\mathsf M_{ U_{\mathcal X}} \ar[rr]_<>(0.5){\phi'^*(f''_*(\theta_{\phi_{\mathcal X}}))} \ar[d]^{\alpha_{\phi',\phi_{\mathcal X}^*\mathsf M_{U_{\mathcal X}}}} &&  \phi'^* Rf''_*\mathsf M_{U'_{\mathcal X}} \ar[d]_{\alpha_{\phi'}}  \ar@/^4pc/[dd]^{\tau_{\phi'}} \\
&Rf'''_* \phi_{\mathcal X}'^*\phi_{\mathcal X}^*\mathsf M_{ U_{\mathcal X}} \ar@{->}@/^0pc/[rr]^<>(0.5){f'''_*(\phi_{\mathcal X}'^*\theta_{\phi_{\mathcal X}})} \ar@{=}[d] &&Rf'''_*  \phi_{\mathcal X}'^*\mathsf M_{ U'_{\mathcal X}}  \ar[d]_<>(0.5){f'''_*(\theta_{\phi_{\mathcal X}'})} \\
(\phi\circ \phi')^*Rf'_* \mathsf M_{ U_{\mathcal X}} \ar[r]^<>(0.5){\alpha_{\phi\circ \phi'}}
\ar@/_1.5pc/[rrr]_{\tau_{\phi\circ \phi'}} 
 & Rf'''_*  (\phi_{\mathcal X}\circ \phi_{\mathcal X}')^*\mathsf M_{ U_{\mathcal X}}  \ar[rr]^<>(0.5){f'''_*(\theta_{\phi_{\mathcal X}\circ \phi_{\mathcal X}'})}&&Rf'''_* \mathsf M_{U''_{\mathcal X}}. 
}
\end{equation}
The left hand square commutes by \eqref{E:PresDpf-var3}, the top right square commutes by \eqref{E:NatProp1MN}, and the bottom right square commutes by \eqref{E:f*DescData}.
\end{proof}

\subsection{\emph{Ad hoc} $i$-th derived push forward}
\label{S:PPDiagsV}

In this section, we consider the situation where we take the $i$-th cohomology of the push forward.  
In other words, given $\mathsf M$ in $\mathsf D_{X}$, then in the notation of diagram \eqref{E:PresDpf-varA}, we will assume there is a natural isomorphism of sheaves in $\mathsf A_U$:  
\begin{equation}\label{E:DefPushPullalphaA}
\alpha_p=\alpha_{p,\mathsf M}: p^*\mathcal H^i (Rf_*\mathsf M) \stackrel{\sim}{\longrightarrow} \mathcal H^i(Rf'_*p_X^*\mathsf M). 
\end{equation}
In particular, we assume these isomorphisms satisfy the following properties, analogous to those in the previous section.    First, the isomorphisms $\alpha$ are functorial in $\mathsf M$ in the sense that given a morphism  $\beta:\mathsf M\to \mathsf N$, we have a commutative diagram
\begin{equation}\label{E:NatProp1MNA}
\xymatrix@C=3em{
 p^*\mathcal H^i (Rf_*\mathsf M) \ar[r]^{p^*(Rf_*\beta)} 
 \ar[d]^{\alpha_{p,\mathsf M}}&  p^*\mathcal H^i (Rf_*\mathsf N)  \ar[d]^{\alpha_{p,\mathsf N}} \\
\mathcal H^i(Rf'_*p_X^*\mathsf M) \ar[r]^{Rf'_*(p_X^*\beta)}
 &  \mathcal H^i(Rf'_*p_X^*\mathsf N).
 }
  \end{equation}
  
  In addition, we assume they are compatible with further pull back, in the sense that if $\phi:U'\to U$ is another \'etale morphism of smooth varieties (resp.~smooth complex analytic spaces), so that we  have a  fibered product  diagram as in \eqref{E:PresDpf-var2}, 
then we have $\alpha_{p\circ \phi,\mathsf M}= \alpha^*_{\phi,p_X^*\mathsf M}\circ \phi^*\alpha_{p,\mathsf M}$, i.e., we have a 
 a commutative diagram
\begin{equation}\label{E:PresDpf-var3AA}
\xymatrix@C=3em@R=2em{
\phi^*p^*\mathcal H^i(Rf_*\mathsf M) \ar@{=}[dd] \ar[r]^{\phi^*(\alpha_{p,\mathsf M})}& \phi^*\mathcal H^i(Rf'_*p_X^*\mathsf M) \ar[d]^{\alpha_{\phi,p_X^*\mathsf M}}\\
& \mathcal H^i(Rf''_*\phi_X^*p_X^*\mathsf M) \ar@{=}[d] \\
(p\circ \phi)^*\mathcal H^i(Rf_*\mathsf M)\ar[r]^<>(0.5){\alpha_{p\circ \phi,\mathsf M}}& \mathcal H^i(Rf''_*(p_X\circ \phi_X)^*\mathsf M).
}
\end{equation}

\begin{dfn}[\'Etale base change for proper push forwards: $\mathcal H^i$]\label{D:EBCPPFith}
We say that \emph{$i$-th derived proper push forwards of objects of $\mathsf A$ (resp.~$\mathsf D$) satisfy \'etale base change for smooth varieties (resp.~smooth complex analytic spaces)} if for every proper morphism $f:X\to Y$ of smooth varieties (resp.~smooth complex analytic spaces), every object $\mathsf M$ in $\mathsf A_X$ (resp.~$\mathsf D_X$), and every \'etale morphism of smooth varieties (resp.~smooth complex analytic spaces)  $p:U\to Y$, there are natural isomorphisms $\alpha_p$ as in \eqref{E:DefPushPullalphaA} satisfying  \eqref{E:PresDpf-var3AA} and  \eqref{E:NatProp1MNA}.  
\end{dfn}

Assuming $i$-th derived proper push forwards of objects of $\mathsf D$ satisfy \'etale base change  for smooth varieties (resp.~smooth complex analytic spaces), then one can define an \emph{ad hoc} push forward on stacks.  
More precisely, we  have the following lemma:

\begin{lem}\label{L:GenDescA}
Let $f:\mathcal X\to \mathcal Y$ be a proper
 morphism of smooth separated integral DM stacks of finite type over $\mathbb C$ (or the analytification of such a morphism).  
Assume that $i$-th derived proper push forwards of objects of $\mathsf D$ satisfy \'etale base change for smooth varieties (resp.~smooth complex analytic spaces); see \Cref{D:EBCPPFith}. Then there is an \emph{ad hoc} $i$-th derived push forward  functor
$$
\mathcal H^i(Rf_*^{ah}):\mathsf D_{\mathcal X,fake}\longrightarrow \mathsf A_{\mathcal Y}
$$
defined as follows.
Given  $\mathsf M$  in $\mathsf D_{\mathcal X,fake}$, 
i.e., an object $\{\mathsf M_U,\theta_\phi\}$ in $\mathsf D$ with descent data on $\mathcal X$ (\S \ref{S:ConcreteDesc}), then, in the notation of   \eqref{E:TubPresDpf'} and \eqref{E:DefPushPullalphaA}, the data
$$
\{\mathcal H^i(Rf'_*\mathsf M_{U_{\mathcal X}}),\tau_\phi:=\alpha_\phi \circ \mathcal H^i(Rf''_*(\theta_{\phi_{\mathcal X}}))\}
$$
defines an object in $\mathsf A$ with descent data on $\mathcal Y$, i.e., an object of $\mathsf A_{\mathcal Y}$, which we define to be  $\mathcal H^i(Rf^{ah}_*\mathsf M)$ and call   the $i$-th \emph{ad hoc} derived push forward.

\end{lem}

\begin{proof}
The proof is similar to the proof of \Cref{L:GenDesc} and is left to the reader.
\end{proof}

\begin{rem}\label{R:ah-push-compatibleHH}
Assuming the conditions of  \Cref{L:GenDesc} also hold, then 
taking $\mathcal H^i(-)$ of the \emph{ad hoc} push forward $Rf_*^{ah}\mathsf M$ of \Cref{L:GenDesc} gives the \emph{ad hoc} push forward $\mathcal H^i(Rf_*^{ah}\mathsf M)$ defined in \Cref{L:GenDescA}.
\end{rem}

\section{\emph{Ad hoc} push forward of perverse sheaves}\label[appendix]{S:AppendixAdHodPF}

We now apply the results in the previous section to the push forward of perverse sheaves. 
For a scheme, or more generally a DM stack $\mathcal X$, and an algebra $A$, we denote by $A_{\mathcal X}$ the locally constant sheaf of algebras with value $A$. 
 For a complex analytic space $U$, and with $A=\C$ or $A=\Q$, we denote by $D_c(A_{U})$  the derived category of sheaves of $A$-modules on $U$ with constructible cohomology sheaves, and  by $\operatorname{Perv}(A_U)$ the category of $A$-perverse sheaves on $U$.

\subsection{Commutativity of perverse cohomology with \'etale pullback}

We start with the following lemma, which we will use later, that perverse cohomology commutes with \'etale base change:

\begin{lem} \label{CommuteCohwEtalepullbackA} Let $U$ be a smooth complex analytic space.
For any integer $i$, the perverse cohomology functors $^p\mathcal{H}^i$ from the derived category of constructible sheaves $D^b_c(\Q_{U^{an}})$ to $\operatorname{Perv}(\Q_U)$ commute canonically with \'etale pullback. More precisely, using the notation in diagram \eqref{E:TubPresDpf'}, the following diagram commutes

\begin{equation}\label{perverse-commutes-etale}
 \begin{tikzcd}
\phi'^*\phi^* \circ \textrm{ }^p\m{H}^0 \arrow[r, "\simeq"] \arrow[d, "\phi'^*\rho_{\phi}"'] & (\phi\phi')^* \circ \textrm{ }^p\m{H}^0 \arrow[d, "\rho_{\phi\phi'}"] \\
\phi'^*\circ \textrm{ }^p\m{H}^0 \circ \phi^*  \arrow[r, "\rho_{\phi'}"]                                 & \textrm{ }^p\m{H}^0 \circ (\phi\phi')^*                                      
\end{tikzcd}
\end{equation}
\end{lem}

\begin{proof}
  Given an arbitrary object $T\in D^b_c(\Q_{U^{an}})$, we have

\begin{center}
\begin{tikzcd}
\phi^*\tau^p_{\leq 0} T \arrow[r] \arrow[d, dotted] & \phi^*T \arrow[d, Rightarrow, no head] \arrow[r] & \phi^*\tau^p_{\geq 1}T \arrow[r, "+1"] \arrow[d, dotted] & {} \\
\tau^p_{\leq 0}\phi^*T \arrow[r]                    & \phi^*T \arrow[r]                                & \tau^p_{\geq 1} \phi^*T \arrow[r, "+1"]                  & {}
\end{tikzcd}
\end{center} 
By \cite[\S 4.2.4]{BBD82}, $\phi^*$ preserves $^pD^{\geq 0}_c(\Q_{U^{an}})$ and $^pD^{\leq 0}_c(\Q_{U^{an}})$. Thus, by \cite[Prop. 1.1.9]{BBD82}, the dotted arrows that make the triangle commutative are unique even in the derived category. Hence, all perverse truncations commute canonically with \'etale pullback. Hence, the perverse cohomologies commute canonically with \'etale pullback.
\end{proof}

\subsection{\emph{Ad hoc} push forward of $A_{\mathcal X}$-modules}\label{S:PushCmodA}
Following the
discussion in \S \ref{E:ConcreteDesc}, we  define the fake bounded  derived category of sheaves of $A$-modules on $\mathcal X$, denoted 
\begin{equation}\label{E:DbfakeAX}
 {D^b_{fake}(A_{\X})},
\end{equation}
 to be the category of objects of $D^b(A_{(-)})$ with descent data on $\mathcal X$. 

In a similar way we define the fake bounded  derived category of sheaves of $A$-modules with constructible cohomology on $\mathcal X$, denoted $ {D^b_{c,fake}(A_{\X})}$,  to be the category of objects of $D^b_c(A_{(-)})$ with descent data on $\mathcal X$.

\begin{rem}[Pathologies of the fake derived category of constructible sheaves]\label{R:FakeNoGood}
The fake derived category $D^b_{fake}(A_{\mathcal X})$  is convenient for us, but in some sense, it is not really the right category to look at geometrically:
\begin{enumerate}
\item  First, the derived category of sheaves does not satisfy \'etale descent and so in particular, even when $\X$ is a scheme, the natural functor (see \eqref{E:genCompareDfake})  $D^b(\operatorname{Mod}(A_{\mathcal X}))\to D^b_{fake}(A_{\mathcal X})$, from the bounded derived category of the abelian category of  $A_{\X}$-modules, may not be an equivalence (See \cite[\href{https://stacks.math.columbia.edu/tag/0D9F}{\S 0D9F, Lem.~85.13.7}]{stacks-project}).

\item 

Second, it is not clear that the $2$-category $D^b_{fake}(A_{\mathcal X})$ is triangulated (in fact, presumably it is not) because given a commutative diagram

\begin{center}
    \begin{tikzcd}
A \arrow[r] \arrow[d, "\psi_1"] \arrow[dd, "\psi_3"', bend right] & B \arrow[d, "\phi_1"'] \arrow[dd, "\phi_3", bend left] \\
A' \arrow[r] \arrow[d, "\psi_2"]                                  & B' \arrow[d, "\phi_2"']                                \\
A'' \arrow[r]                                                     & B''                                                   
\end{tikzcd}
\end{center}
where the vertical arrows should be thought of as descent data for $A$ and for $B$, then, taking the cones does not guarantee commutativity of the isomorphisms, i.e., the diagram below may not be commutative

\begin{center}
    \begin{tikzcd}
A \arrow[r] \arrow[d, "\psi_1"] \arrow[dd, "\psi_3"', bend right] & B \arrow[d, "\phi_1"'] \arrow[dd, "\phi_3", bend left] \arrow[r] & C \arrow[d] \arrow[dd, bend left] \\
A' \arrow[r] \arrow[d, "\psi_2"]                                  & B' \arrow[d, "\phi_2"']                                          & C' \arrow[d]                      \\
A'' \arrow[r]                                                     & B'' \arrow[r]                                                    & C''                              
\end{tikzcd}
\end{center}
because the induced morphisms do not have to commute. The failure of commutativity appears because the cone construction is unique only up to isomorphism but not via unique isomorphism.

\end{enumerate}

\end{rem}

However, despite the pathologies mentioned in the remark above, $D^b_{fake}(A_{\mathcal X})$  will be sufficient for our purposes, as the main object we are interested in, namely $A_{\X}[d_{\X}]$ ($d_{\mathcal X}=\dim \mathcal X$),  also lives in this category. In fact, there is an abelian category contained in $D^b_{fake}(A_{\mathcal{X}})$, namely the category of perverse sheaves.  More precisely, we define 
\begin{equation}\label{E:PervXXDef}
\operatorname{Perv}(A_{\X})\subseteq D^b_{fake}(A_{\mathcal{X}})
\end{equation}
to the category of objects of $\operatorname{Perv}(A_{(-)})$ with descent data on $\mathcal X$. We note that perverse sheaves satisfy \'etale descent \cite[Cor. 2.1.23]{BBD82},
  and consequently, one has that that $\operatorname{Perv}(A_{\mathcal X})$ is an abelian category.  
Moreover, the object $A_{\X}[d_{\X}]$ belongs to this category.
We also observe the following:

\begin{lem}\label{D-fake-Perv = D-fake-constructible}

We have an equivalence of categories 
  $$D_{fake}^b(\operatorname{Perv}(A_{\mathcal X}))\simeq D^b_{c,fake}(A_{\X}),$$
  where on the left we mean the category of objects of $D_{fake}^b(\operatorname{Perv}(A_{(-)}))$ with descent data on $\mathcal X$. 
\end{lem}

\begin{proof}
This follows from the definitions  since for every \'etale morphism $p:U\to \mathcal X$ from a smooth variety (resp.~complex analytic space), we have  $D^b(\operatorname{Perv}(A_{U}))=D^b_c(A_{U})$.
\end{proof}

Now let $X$ be a smooth analytic space. Following the strategy in \S \ref{S:PPDiagsVDer}, our first goal in constructing the \emph{ad hoc} push forward in this setting will be to show that given an object $K$ in $D^b_{fake}(A_X)$, then in the notation of diagram \eqref{E:PresDpf-varA}, there is a canonical isomorphism of objects in $D^b(A_U)$:  
\begin{equation}\label{E:CmodDefPushPullalphaA}
\alpha_p=\alpha_{p,K}: p^* (Rf_*K) \stackrel{\sim}{\longrightarrow} Rf'_*p_X^*K. 
\end{equation}
In particular,  these isomorphisms  satisfy the  properties  \eqref{E:NatProp1MN} and \eqref{E:PresDpf-var3}:

\begin{lem}\label{L:CmodDescA}
Proper push forwards of objects of $D^b(A_{(-)})$ satisfy \'etale base change for smooth complex analytic spaces (\Cref{D:EBCPPF}).  
In other words, there are isomorphisms as in   \eqref{E:CmodDefPushPullalphaA} satisfying 
\eqref{E:NatProp1MN} and \eqref{E:PresDpf-var3}.

\end{lem}

\begin{proof} 
This essentially follows from standard arguments in the derived category, and in principle could be deduced from the presentation in \cite[\S 4.5]{HTT08}; for clarity, and in particular to show that various constructions are all compatible, we include the details here.

Let $K\to I^\bullet_K$ be an injective resolution in $D^b(A_{X})$, so that $Rf_*K=f_*I_K^\bullet$ in the derived category, where $f_*$ is the push forward of sheaves. Let $p_X^*K \to I_{p_X^*K}^\bullet$ be an injective resolution, so that we have $Rf'_*p_X^*K = f'_*I_{p_X^*K}^\bullet$ in the derived category.  This gives us a  diagram
\begin{equation}
\xymatrix{
p_X^*K\ar[r]^<>(0.5){q.i.} \ar[rd]_{\cong }& p_X^*I_K^\bullet \ar@{-->}[d]^{\gamma_K} \\
 & I_{p_X^*K}^\bullet 
}
\end{equation}
where there is a map $\gamma$ as above that,  by say  
\cite[\href{https://stacks.math.columbia.edu/tag/013P}{Lem.~013P}]{stacks-project},  makes the diagram commute up to homotopy, and, by say  \cite[\href{https://stacks.math.columbia.edu/tag/013S}{Lem.~013S}]{stacks-project},  is unique up to homotopy.  By the commutativity of the diagram up to homotopy, the diagram commutes in the derived category. Applying $Rf_*$ to the right hand side of the diagram we have a quasi isomorphism
$$
\xymatrix{
p^*Rf_*K = p^*f_*I_K^\bullet = f'_*p_X^*I_K^\bullet   \ar[r]^<>(0.5) {f'_*(\gamma_K)}&  f'_*I_{p_X^*K}^\bullet = Rf'_*p_X^*K,
}
$$
where here one can use a  local computation with sections of  sheaves to show that with $p$ (and therefore $p_X$) \'etale, one has a canonical identification $p^*f_*=f'_*p_X^*$.  

While $\gamma_K$ is not unique, it is unique up to homotopy, and so it is unique in the derived category. This gives the canonical isomorphism:
$$
\xymatrix{
\alpha_{p,K}: p^* Rf_* K \ar[r]^<>(0.5){f'_*(\gamma_K)}_<>(0.5){\sim}&R f'_*p_X^* K
}
$$
To see that these isomorphisms satisfy \eqref{E:NatProp1MN} and \eqref{E:PresDpf-var3} we argue as follows.  Let $\beta:M\to N$ be a morphism, and at the same time, consider the diagram \eqref{E:PresDpf-var2}. Arguing as above, using similar notation, we obtain a diagram
\begin{equation}\label{E:diagramA18complexes}
\xymatrix@C=1.5em@R=1.5em{
&\phi_X^*p_X^*N\ar[rr]^<>(0.2){q.i.} \ar@{-}[d]_<>(0.5){\cong }&& \phi_X^*p_X^*I_N^\bullet \ar@{-->}[dd]^<>(0.2){\phi_X^*(\gamma_N)} \\
\phi_X^*p_X^*M\ar[rr]^<>(0.8){q.i.} \ar[dd]_<>(0.2){\cong } \ar[ru]^{\phi_X^*p_X^*(\beta)}&\ar[d]& \phi_X^*p_X^*I_M^\bullet \ar@{-->}[dd]^<>(0.2){\phi_X^*(\gamma_M)} \ar@{-->}[ru]&\\
&\phi_X^*p_X^*N \ar[rr]^<>(0.2){q.i.} \ar@{-}[d]_{\cong}&& \phi_X^*I_{p_X^*N}^\bullet  \ar@{-->}[dd]^<>(0.2){\gamma_{p_X^*N}}  \\
\phi_X^*p_X^*M \ar[rr]^<>(0.8){q.i.} \ar[ru]^{\phi_X^*p_X^*(\beta)} \ar[dd]_<>(0.2){\cong }&\ar[d]& \phi_X^*I_{p_X^*M}^\bullet \ar@{-->}[ru] \ar@{-->}[dd]^<>(0.2){\gamma_{p_X^*M}}&\\
&\phi_X^*p_X^*N \ar[rr]^<>(0.2){q.i.}&& I_{\phi_X^*p_X^*N}^\bullet  \\
\phi_X^*p_X^*M \ar[rr]^<>(0.5){q.i.} \ar[ru]^{\phi_X^*p_X^*(\beta)}&& I_{\phi_X^*p_X^*M}^\bullet \ar@{-->}[ru]&\\
}
\end{equation}
that commutes up to homotopy.  Arguing as above, the top cube shows that the isomorphisms $\alpha$ satisfy condition 
\eqref{E:NatProp1MN} 
 and the front face shows they satisfy 
 \eqref{E:PresDpf-var3}.
\end{proof}

\begin{cor}\label{L:CmodGenDescA}
Let $f:\mathcal X\to \mathcal Y$ be the analytification of a schematic proper morphism of smooth separated integral DM stacks of finite type over $\mathbb C$.
Then there is an \emph{ad hoc} push forward  functor
$$
Rf_*^{ah}:D^b_{fake}(A_{\mathcal X})\longrightarrow D^b_{fake}(A_{\mathcal Y})
$$
defined as follows.
Given  $K$  in $D^b_{fake}(A_{\mathcal X})$ \eqref{E:DbfakeAX},  
i.e., an object $\{K_U,\theta_\phi\}$ in $D^b(A_{(-)})$ with descent data on $\mathcal X$ (\S \ref{S:ConcreteDesc}), then, in the notation of   diagram \eqref{E:TubPresDpf'} and \eqref{E:CmodDefPushPullalphaA}, the data
$$
\{Rf'_*K_{U_{\mathcal X}},\tau_\phi:=\alpha_\phi \circ f''_*(\theta_{\phi_{\mathcal X}})\}
$$
defines an object in $D^b(A_{(-)})$ with descent data on $\mathcal Y$, i.e., an object of $\mathsf D_{fake}(A_{\mathcal Y})$, which we define to be  $Rf^{ah}_*K$ and call   the \emph{ad hoc} push forward.

\end{cor}

\begin{proof}
This is an immediate consequence of \Cref{L:CmodDescA} and \Cref{L:GenDesc}. 
\end{proof}

From \Cref{L:CmodGenDescA}, we also obtain an \emph{ad hoc} push forward 
$$
Rf_*^{ah}:D^b_{c,fake}(A_{\mathcal X})\longrightarrow D^b_{c,fake}(A_{\mathcal Y});
$$
the condition to have constructible cohomology sheaves is checked locally, and therefore comes from the construction.

\subsection{\emph{Ad hoc} push forward of perverse sheaves}

Let $K$ be in $ \operatorname{Perv}(A_{\mathcal X})$. By \Cref{L:CmodDescA}, in the notation of diagram \eqref{E:PresDpf-varA} we have a collection of morphisms 
 $$\alpha_p=\alpha_{p,K}: p^* (Rf_*K) \stackrel{\sim}{\longrightarrow} Rf'_*p_X^* K$$
as in \eqref{E:CmodDefPushPullalphaA} satisfying 
\eqref{E:NatProp1MN} and \eqref{E:PresDpf-var3}.  
Note that $D^b_c(A_X)$, the derived category of constructible $A_X$-sheaves, is a full and faithful subcategory of $D^b(A_X)$. So, we can apply the functors $^p\mathcal{H}^i$ to the objects and to the morphisms as above. By \Cref{CommuteCohwEtalepullbackA}, this induces canonical isomorphisms
 \begin{equation}\label{E:A2iso-replacement}
     \alpha^i_p=\alpha_{p,K,i}: p^* \, ^p \mathcal{H}^i(Rf_*K) \stackrel{\sim}{\longrightarrow}  \, ^p \mathcal{H}^i Rf'_*p_X^* K
 \end{equation}
   satisfying  
       \eqref{E:NatProp1MNA} and  \eqref{E:PresDpf-var3AA}.
Consequently, using \Cref{L:GenDescA}, we obtain:

       \begin{lem}\label{L:CmodDescAPervHi}
The $i$-th derived proper push forwards of objects of $\operatorname{Perv}(A_{(-)})$ (i.e., Perverse sheaves) satisfy \'etale base change for varieties (resp.~smooth complex analytic spaces)  (\Cref{D:EBCPPFith}).  
In other words, there are  isomorphisms as in  \eqref{E:A2iso-replacement} satisfying  
 \eqref{E:NatProp1MNA} and  \eqref{E:PresDpf-var3AA}.
\end{lem}

\begin{proof}
This follows from taking cohomology in \Cref{L:CmodDescA}.
\end{proof}

\begin{cor}\label{C:PervDescA}
Let $f:\mathcal X\to \mathcal Y$ be the analytification of a schematic proper morphism of smooth separated integral DM stacks of finite type over $\mathbb C$.  
For each  integer $i$, 
there is an \emph{ad hoc} $i$-th derived push forward  functor
$$
^p\mathcal H^i(Rf_*^{ah}):\operatorname{Perv}(A_{\mathcal X}) \longrightarrow \operatorname{Perv}(A_{\mathcal Y})
$$
defined as follows.
Given  $K$  in $\operatorname{Perv}(A_{\mathcal X})$ \eqref{E:PervXXDef}, 
i.e., an object $\{ K_U,\theta_\phi\}$ in $\operatorname{Perv}(A_{(-)})$ with descent data on $\mathcal X$ (\S \ref{S:ConcreteDesc}), then, in the notation of   \eqref{E:TubPresDpf'} and \eqref{E:A2iso-replacement}, the data
$$
\{ ^p\mathcal H^i(Rf'_*K_{U_{\mathcal X}}),\tau_\phi:=\alpha^i_\phi \circ ^p\mathcal H^i(Rf''_*(\theta_{\phi_{\mathcal X}}))\}
$$
defines an object in $\operatorname{Perv}(A_{(-)})$ with descent data on $\mathcal Y$, i.e., an object of $\operatorname{Perv}(A_{\mathcal Y})$, which we define to be  $^p\mathcal H^i(Rf^{ah}_*\mathsf M)$ and call   the $i$-th \emph{ad hoc} derived push forward.

\end{cor}

\begin{proof}
Given that the isomorphisms \eqref{E:A2iso-replacement}  satisfy  \eqref{E:NatProp1MNA} and  \eqref{E:PresDpf-var3AA}, this follows from \Cref{L:GenDescA}.

\end{proof}

We now show the compatibility between the \emph{ad hoc} push forward of the  $\Q$ and $\C$ perverse sheaves:

\begin{pro}\label{QPerv_same_asCPerv}
Let $K$ be in  $D^b_{c,fake} (\mathbb Q_{\mathcal X})$ \eqref{S:PushCmodA}  (resp.~$\operatorname{Perv}(\mathbb Q_{\mathcal X})$ \eqref{E:PervXXDef}).
There is a canonical isomorphism $Rf^{ah}_* (K)\otimes_{\Q_{\mathcal X}}\C_{\mathcal X}\simeq Rf_*^{ah}(K\otimes_{\Q_{\mathcal X}}\C_{\mathcal X})$ (resp.~$^p \mathcal H^i(Rf^{ah}_*(K\otimes_{\mathbb Q_{\mathcal X}}\mathbb C_{\mathcal X}))\cong (^p \mathcal H^i(Rf^{ah}_*K))\otimes_{\mathbb Q_{\mathcal X}}\mathbb C_{\mathcal X}$).
\end{pro}

\begin{proof}
This follows immediately because for an injective object $I_{\mathbb Q}$ over $\Q_{\mathcal X}$, we have $I_{\Q}\otimes_{\Q_{\mathcal X}}\C_{\mathcal X}$ is $f$-acyclic (write $\C$ as a colimit of finite dimensional vector spaces over $\Q$ and then use \cite[\href{https://stacks.math.columbia.edu/tag/0GQW}{Lem.~0GQW}]{stacks-project}).
\end{proof}

\section{\emph{Ad hoc} push forward of $D_{\mathcal X}$-modules}\label[appendix]{S:ahPushDmod}

Following the
discussion in \S \ref{E:ConcreteDesc}, we  define the fake bounded  derived category of sheaves of $D$-modules on $\mathcal X$, denoted 
\begin{equation}\label{E:DbfakeDX}
 {D^b_{fake}(D_{\X})},
\end{equation}
 to be the category of objects of $D^b(\operatorname{Mod}(D_{(-)})^{\operatorname{right}})$,  the bounded derived category of right $D$-modules,  with descent data on $\mathcal X$.

Now let $X$ be a smooth variety (resp.~complex analytic space). Following the strategy in \S \ref{S:PPDiagsVDer}, our first goal in constructing the \emph{ad hoc} push forward in this setting will be to show that given an object $M$ in $D^b_{fake}(D_X)$, then in the notation of diagram \eqref{E:PresDpf-varA}, there is a canonical isomorphism of sheaves in $D^b(D_U)$:  
\begin{equation}\label{E:CmodDefPushPullalphaADmod}
\alpha_p=\alpha_{p,M}: p^* (f_+M) \stackrel{\sim}{\longrightarrow} f'_+p_X^*M,
\end{equation}
where $f_+:D^b(D_X)\to D^b(D_Y)$ is the push forward, given on an object $M$ by  $f_+M=Rf_*\operatorname{Sp}_{X\to Y}M$, where  $Rf_*$ is the derived functor  $Rf_*:D^b(f^{-1}(D_Y))\to D^b(D_Y)$.
In particular,  these isomorphisms  satisfy the  properties  \eqref{E:NatProp1MN} and \eqref{E:PresDpf-var3}:

\begin{lem}\label{L:CmodDescADmod}
Proper push forwards of objects of $D^b(\operatorname{Mod}(D_{(-)})^{\operatorname{right}})$ satisfy \'etale base change for smooth varieties (resp.~smooth complex analytic spaces) (\Cref{D:EBCPPF}).  
In other words, there are  isomorphisms as in  \eqref{E:CmodDefPushPullalphaADmod} satisfying  
\eqref{E:NatProp1MN} and \eqref{E:PresDpf-var3}.
\end{lem}

\begin{proof}
The proof is identical to \Cref{L:CmodDescA}, except one takes an injective resolution  $\operatorname{Sp}^\bullet_{X\to Y}M \to I^\bullet_K$   in $D^b(f^{-1}D_{Y})$, so that $f_+M=Rf_*\operatorname{Sp}^\bullet_{X\to Y}M=f_*I_K^\bullet$ in the derived category, where $f_*$ is the push forward of sheaves. 
\end{proof}

\begin{cor}\label{L:CmodGenDescADmod}
Let $f:\mathcal X\to \mathcal Y$ be a proper morphism of smooth separated integral DM stacks of finite type over $\mathbb C$ (or the analytification of such a morphism).  
Then there is an \emph{ad hoc} push forward  functor
$$
f_+^{ah}:D^b_{fake}(D_{\mathcal X})\longrightarrow D^b_{fake}(D_{\mathcal Y})
$$
defined as follows.
Given  $\mathcal M$  in $D^b_{fake}(D_{\mathcal X})$ \eqref{E:DbfakeDX},  
i.e., an object $\{\mathcal M_U,\theta_\phi\}$ in $D^b(D_{(-)})$ with descent data on $\mathcal X$ (\S \ref{S:ConcreteDesc}), then, in the notation of   \eqref{E:TubPresDpf'} and \eqref{E:CmodDefPushPullalphaADmod}, the data
$$
\{f'_+\mathcal M_{U_{\mathcal X}},\tau_\phi:=\alpha_\phi \circ f''_+(\theta_{\phi_{\mathcal X}})\}
$$
defines an object in $D^b(D_{(-)})$ with descent data on $\mathcal Y$, i.e., an object of $\mathsf D_{fake}(D_{\mathcal Y})$, which we define to be  $f^{ah}_+\mathcal M$ and call   the \emph{ad hoc} push forward.

\end{cor}

\begin{proof}
This is an immediate consequence of \Cref{L:CmodDescADmod} and \Cref{L:GenDesc}. 
\end{proof}

\subsection{\emph{Ad hoc} $i$-th push forward of $D_{\mathcal X}$-modules}
Let $\mathcal M$ be in $ \operatorname{Mod}(D_{X})^{\operatorname{right}}$. By \Cref{L:CmodDescADmod}, we have a collection
 $$\alpha_p=\alpha_{p,K}: p^* (f_+\mathcal M) \stackrel{\sim}{\longrightarrow} f'_+p_X^* \mathcal M$$
as in \eqref{E:CmodDefPushPullalphaADmod} satisfying 
\eqref{E:NatProp1MN} and \eqref{E:PresDpf-var3}.
Taking cohomology  induces canonical isomorphisms
 \begin{equation}\label{E:A2iso-replacementDmod}
     \alpha^i_p=\alpha_{p,\mathcal M,i}: p^*  \mathcal{H}^i(f_+\mathcal M) \stackrel{\sim}{\longrightarrow}   \mathcal{H}^i f'_+p_X^* \mathcal M
 \end{equation}
 satisfying  
 \eqref{E:NatProp1MNA} and  \eqref{E:PresDpf-var3AA}:

       \begin{lem}\label{L:CmodDescADmodHi}
The $i$-th derived proper push forwards of objects of $\operatorname{Mod}(D_{(-)})^{\operatorname{right}}$ (i.e., right $D$-modules) satisfy \'etale base change for varieties (resp.~smooth complex analytic spaces)  (\Cref{D:EBCPPFith}).  
In other words, there are  isomorphisms as in  \eqref{E:A2iso-replacementDmod} satisfying  
 \eqref{E:NatProp1MNA} and  \eqref{E:PresDpf-var3AA}.

\end{lem}

\begin{proof}
This follows from taking cohomology in \Cref{L:CmodDescADmod}.
\end{proof}
       
Consequently, using \Cref{L:GenDescA}, we obtain:

\begin{cor}\label{C:PervDescADmod}
Let $f:\mathcal X\to \mathcal Y$ be a schematic proper morphism of smooth separated integral DM stacks of finite type over $\mathbb C$ (or the analytification of such a morphism).  
For each  integer $i$, 
there is an \emph{ad hoc} $i$-th derived push forward  functor
\begin{equation}\label{E:C:PervDescADmod}
\mathcal H^i(f_+^{ah}):\operatorname{Mod}(D_{\mathcal X})^{\operatorname{right}}\longrightarrow \operatorname{Mod}(D_{\mathcal Y})^{\operatorname{right}}
\end{equation}
defined as follows.
Given  $\mathcal M$  in $\operatorname{Mod}(D_{\mathcal X})^{\operatorname{right}}$, 
i.e., an object $\{ \mathcal M_U,\theta_\phi\}$ in $\operatorname{Mod}(D_{(-)})^{\operatorname{right}}$ with descent data on $\mathcal X$ (\S \ref{S:ConcreteDesc}), then, in the notation of   \eqref{E:TubPresDpf'} and \eqref{E:A2iso-replacementDmod}, the data
$$
\{\mathcal H^i(f'_+\mathcal M_{U_{\mathcal X}}),\tau_\phi:=\alpha^i_\phi \circ \mathcal H^i(f''_+(\theta_{\phi_{\mathcal X}}))\}
$$
defines an object in $\operatorname{Mod}(D_{(-)})^{\operatorname{right}}$ with descent data on $\mathcal Y$, i.e., an object of $\operatorname{Mod}(D_{\mathcal Y})^{\operatorname{right}}$, which we define to be  $\mathcal H^i(f^{ah}_+\mathsf M)$ and call   the $i$-th \emph{ad hoc} derived push forward.

\end{cor}

\begin{proof}

Given \Cref{L:CmodDescADmodHi}, 
 this follows from \Cref{L:GenDescA}.
\end{proof}

\subsection{Coherent, regular, and holonomic $D$-modules}
\label{S:Regular/holo-are-etale-ah}

We start by recalling our definition of quasi-coherent, coherent, regular, and holonomic $D$-modules on DM stacks, from \cite[\S 3]{CMZpositivity}: For  $\mathcal X$ a smooth separated integral DM stack locally of finite type over $\mathbb C$, a $D$-module $\mathcal M$ on $\mathcal X$ is \emph{quasi-coherent (resp.~coherent, resp.~regular, resp.~holonomic)} if for every \'etale morphism $U\to \mathcal X$ from a smooth variety $U$ one has that the  induced $D$-module $\mathcal M_U$ is quasi-coherent (resp.~coherent, resp.~regular, resp.~holonomic). 

We denote by $\operatorname{Mod}_c(D_{\mathcal X})^{\operatorname{right}}$ (resp.~$\operatorname{Mod}_h(D_{\mathcal X})^{\operatorname{right}}$, resp.~$\operatorname{Mod}_{rh}(D_{\mathcal X})^{\operatorname{right}}$) the abelian category of (right) coherent (resp.~holonomic, resp.~regular holonomic) $D_{\mathcal X}$-modules on $\mathcal X$, and for $\sharp \in \{c,h,rh\}$, we denote by  
$D^b_\sharp (\operatorname{Mod}(D_{\mathcal X})^{\operatorname{right}})$ the full subcategory of the bounded derived category $D^b (\operatorname{Mod}(D_{\mathcal X})^{\operatorname{right}})$ with cohomology in $\operatorname{Mod}_\sharp (D_{\mathcal X})^{\operatorname{right}}$.  We define 
 \begin{equation}\label{E:Dbsharpfake}
D^b_{\sharp,fake}(D_{\mathcal X})
\end{equation}
to be
the category of objects of $D^b(\operatorname{Mod}_\sharp(D_{(-)})^{\operatorname{right}})$  with descent data on $\mathcal X$. 

\begin{pro}[Coherent, regular, and holonomic $D$-modules]\label{P:ahCRHpush}
Suppose  $f:\mathcal X\to \mathcal Y$ is a  schematic proper morphism  of smooth separated integral DM stacks of finite type over $\mathbb C$ (or the analytification of such a morphism).
Then for $\sharp \in \{c,h,rh\}$, the functors  $f_+^{ah}$ (\Cref{L:CmodGenDescADmod})  and $\mathcal H^i(f^{ah}_+)$ (\Cref{C:PervDescADmod}) induce  functors 
$$
f_+^{ah}:D^b_{\sharp, fake}(D_{\mathcal X})\longrightarrow D^b_{\sharp,fake}(D_{\mathcal Y}),
$$
$$
\mathcal H^i(f_+^{ah}):\operatorname{Mod}_\sharp(D_{\mathcal X})^{\operatorname{right}}\longrightarrow \operatorname{Mod}_\sharp (D_{\mathcal Y})^{\operatorname{right}}.
$$
The fake derived categories are defined above in \eqref{E:Dbsharpfake}. 

\end{pro}

\begin{proof}
As the definitions are local, and in light of \Cref{R:ah-push-compatibleHH}, we are free to check that $\mathcal H^i(f^{ah}_+\mathcal M)$ is coherent (resp.~regular, resp.~holonomic)  after an \'etale base change.  This reduces the question to morphisms of varieties, where the result is standard (e.g., \cite[Thm.~2.5.1, Thm.~3.2.3, Thm.~6.1.5]{HTT08}).

\end{proof}

\section{Compatibility of \emph{ad hoc} push forwards}\label[appendix]{S:AppendixCompatible}

We now show that the \emph{ad hoc} push forwards we have defined are compatible; in \S \ref{S:AlgAnAgree}, we show that algebraic and analytic \emph{ad hoc} push forwards of $D$-modules agree, and in \S \ref{S:DandPervAgree} we show that push forwards of $D$-modules and perverse sheaves agree. 

\subsection{Algebraic and analytic \emph{ad hoc} push forwards of $D_{\mathcal X}$-modules agree}\label{S:AlgAnAgree}

For a variety $X$ we will denote by $\iota_X:(X^{an},\mathcal O_{X^{an}})\to (X,\mathcal O_X)$ the canonical morphism of locally ringed spaces, and denote the analytification by the associated pull back via $\iota$; i.e., for an $\mathcal O_X$-module $M$ we have $M^{an}=\iota^*M=\mathcal O_{X^{an}}\otimes_{\iota_X^{-1}\mathcal O_X}\iota_X^{-1}M$. 
Note that $D_{X^{an}}=\iota_X^*D_X$.  
 We also have a morphism of ringed spaces  $\iota_X:(X^{an}, D_{X^{an}})\to (X,D_X)$, and for a 
$D_X$-module $M$, we have $M^{an}=\iota_X^*M=D_{X^{an}}\otimes_{\iota_X^{-1}D_X}\iota_X^{-1}M$, where here we are using $\iota_X^*$ as the pull back as ringed spaces; this also agrees with the pull back as $\mathcal O_{\mathcal X}$-modules (see e.g., \cite[Rem.~3.5]{CMZpositivity}).   

\begin{pro}\label{P:alg-and-an-ah-push-agree}
Let $f:\mathcal X\to \mathcal Y$ be a schematic proper morphism of smooth separated integral DM stacks of finite type over $\mathbb C$, and let $\mathcal M$ be in $D^b_{fake}(D_{\mathcal X})$, e.g., $\mathcal M$ is a  $D_{\mathcal X}$-module. 
We have a canonical isomorphism 
$$
(f_+^{ah}\mathcal M)^{an}\stackrel{\sim}{\longrightarrow}(f^{an})^{ah}_+(\mathcal M^{an})
$$
\end{pro}

\begin{proof}
In light of \eqref{E:DescDataDiagMorph}, given the object with descent data  $\{f'_+\mathcal M_{U_{\mathcal X}},\tau_\phi\}$ from \Cref{L:CmodGenDescADmod} representing $f^{ah}_+\mathcal M$, and the analytic object with descent data $\{f'^{an}_+\mathcal M_{U^{an}_{\mathcal X}},\tau_{\phi^{an}}\}$ representing $(f^{an})^{ah}_+\mathcal M^{an}$, one must show that the standard isomorphisms 
\begin{equation}\label{E:StIsof+algf+an}
(f'_+M_{U_{\mathcal X}})^{an}\stackrel{\sim}{\to} f'^{an}_+\mathcal M_{U^{an}_{\mathcal X}}
\end{equation}
  (see e.g., \cite[Prop.~4.7.2]{HTT08}) are canonical in the sense that under these identifications, one has  $(\tau_\phi)^{an}= \tau_{ \phi^{an}}$.
While in principle this could be deduced from the presentation in \cite[Prop.~4.7.2]{HTT08}, for clarity, we include some of the details here.

The first step is to show, in the notation diagram of \eqref{E:PresDpf-varA},  that for  the morphisms 
$$
\alpha_p=\alpha_{p,M}: p^* (f_+M) \stackrel{\sim}{\longrightarrow} f'_+p_X^*M,
$$
given in \eqref{E:CmodDefPushPullalphaADmod}, and the corresponding morphisms in the analytic category, 
$$
\alpha_{p^{an}}=\alpha_{p^{an},M^{an}}: (p^{an})^* (f^{an}_+M^{an}) \stackrel{\sim}{\longrightarrow} f'^{an}_+(p^{an}_{X^{an}})^*M^{an},
$$
that we have $(\alpha_p)^{an}=\alpha_{p^{an}}$, after the standard identifications of the analytification of the push forward with the push forward of the analytification (i.e., \eqref{E:StIsof+algf+an}).  

For this it is convenient to consider the diagram
$$
\xymatrix@C=2em@R=1em{
&U_{X}\ar[rr]^{p_{X}} \ar@{-}[d]^{f'}&&X \ar[dd]^f\\
U^{an}_{X}\ar[rr]^<>(0.75){p_{X^{an}}} \ar[dd]_{f'^{an}} \ar[ru]^{\iota_{U_{X}}}&\ar[d]&X^{an} \ar[dd]^<>(0.25){f^{an}} \ar[ru]_{\iota_X}&\\
&U\ar@{-}[r]^p&\ar[r]&Y\\
U^{an}\ar[rr]_{p^{an}} \ar[ru]^{\iota_U}&&Y^{an} \ar[ru]_{\iota_Y}&
}
$$
Considering  the construction of the $\alpha_p$, which is explained in \Cref{L:CmodDescADmod} and \Cref{L:CmodDescA}, and now viewing the analytification functor as a pull back functor,  it is clear from the diagram above and the construction in \Cref{L:CmodDescA}, especially the front face of \eqref{E:diagramA18complexes}, that we have the following  commutative diagrams, similar to  \eqref{E:PresDpf-var3}:
\begin{equation}\label{E:PresDpf-var3ip}
\xymatrix@C=4em@R=2em{
(p^{an})^*\iota_Y^*f_+ M \ar@{=}[dd] \ar[r]^{(p^{an})^*(\alpha_{\iota_Y,M})}& (p^{an})^*f^{an}_+\iota_X^*M \ar[d]^{\alpha_{p^{an}}=\alpha_{p^{an},\iota_X^*M}}\\
& f'^{an}_+p_{X^{an}}^*\iota_X^*M \ar@{=}[d] \\
(\iota_Y\circ p^{an})^*f_+M\ar[r]^<>(0.5){\alpha_{\iota_Y\circ p^{an},M}}& f'^{an}_+(\iota_X\circ p_{X^{an}})^* M
}
\end{equation}
\begin{equation}\label{E:PresDpf-var3pi}
\xymatrix@C=4em@R=2em{
\iota_U^*p^*f_+\mathsf M \ar@{=}[dd] \ar[r]^{\iota_U^*(\alpha_{p,\mathsf M})=(\alpha_{p})^{an}}& \iota_U^*f'_+p_X^*M \ar[d]^{\alpha_{\iota_U,p_X^* M}}\\
& f'^{an}_+\iota_{U_X}^*p_X^* M \ar@{=}[d] \\
(p\circ \iota_U)^*f_+M\ar[r]^<>(0.5){\alpha_{p\circ \iota_U, M}}& f'^{an}_+(p_X\circ \iota_{U_X})^* M
}
\end{equation}
As $\iota_Y\circ p^{an}=p\circ \iota_U$, the diagrams \eqref{E:PresDpf-var3ip} and \eqref{E:PresDpf-var3pi} are equal except for the upper right corners; this implies that $\alpha_{\iota_U,p_X^* M}\circ (\alpha_p)^{an}= \alpha_{p^{an}}\circ (p^{an})^*(\alpha_{\iota_Y,M})$, which is exactly saying that $(\alpha_p)^{an}=\alpha_{p^{an}}$, after the standard identifications of the analytification of the push forward with the push forward of the analytification. 

Arguments similar to those in  \Cref{L:CmodDescA} show that the  identifications  $(\alpha_p)^{an}=\alpha_{p^{an}}$ are functorial in $M$, and stable under further \'etale base change.  With these, then arguments similar to those in \Cref{L:GenDesc} show that $(\tau_\phi)^{an}= \tau_{ \phi^{an}}$, completing the proof.   For brevity, we omit the details. 
\end{proof}

\begin{cor}\label{C:alg-and-an-ah-push-agree}
In the notation of \Cref{P:alg-and-an-ah-push-agree}, 
we have a canonical isomorphism of $D_{\mathcal Y}$-modules
$$
(\mathcal H^if_+^{ah}\mathcal M)^{an}\stackrel{\sim}{\longrightarrow}\mathcal H^i((f^{an})^{ah}_+(\mathcal M^{an})).
$$
\end{cor}

\begin{proof}
This follows directly from \Cref{P:alg-and-an-ah-push-agree}. 
\end{proof}

\subsection{Compatibility of \emph{ad hoc} push forwards of $D$-modules and perverse sheaves} \label{S:DandPervAgree}

To begin, given a smooth complex analytic space $X$, we recall the (shifted)  de Rham functor
$$
^pDR_X: D^b(D_X)\longrightarrow D^b(\mathbb C_X)
$$
\begin{equation}\label{E:def-dR-functor}
M^\bullet \mapsto \mathcal R\mathcal Hom^\bullet_{D_X}(\mathcal O_X,\omega_X^{-1}\otimes_{\mathcal O_X}M^\bullet);
\end{equation}
 see e.g., \cite[Prop.~4.2.1]{HTT08} and note that they are using \emph{left} $D_X$-modules, as well as the unshifted de Rham functor $DR_X= \ ^pDR_X[-\dim X]$. 
The  de Rham functor is functorial in $X$, as well, and so, for a smooth separated complex analytic DM stack of finite type over 
$\mathbb C$ there is an induced de Rham functor
\begin{equation}\label{E:def-drR-XX}
^pDR_{\mathcal X}: D^b_{fake}(D_{\mathcal X})\to D^b_{fake}(\mathbb C_{\mathcal X})
\end{equation}
$$
\{\mathcal M_U,\theta_\phi\}\mapsto \{ ^pDR_U (\mathcal M_U),\  ^pDR_U(\theta_\phi)\}
$$
on the fake derived categories of objects with descent data (see \eqref{E:DbfakeDX} and \eqref{E:DbfakeAX}).   Restricting to regular holonomic $D$-modules, we have induced functors 
$$
\xymatrix@R=1em{
\operatorname{Mod}_{rh}(D_{\mathcal X})^{\operatorname{right}}  \ar[r]^<>(0.5){^pDR_{\mathcal X}}_<>(0.5){\sim} \ar[d] & \operatorname{Perv}(\mathbb C_{\mathcal X}) \ar[d]\\
D^b_{rh,fake}(D_{\mathcal X}) \ar[r]^{^pDR_{\mathcal X}}_{\sim} \ar[d] &  D^b_{c,fake}(\mathbb C_{\mathcal X}) \ar[d]\\
D^b_{fake}(D_{\mathcal X}) \ar[r]^{^pDR_{\mathcal X}} &  D^b_{fake}(\mathbb C_{\mathcal X})
}
$$
where the middle row in the diagram comes from the Riemann--Hilbert correspondence (e.g., 
\cite[Thm.~4.6.3]{HTT08}) and the top row comes from the standard identification (e.g., from \cite[\S 7.2]{HTT08} showing that the de Rham functor preserves the $t$-structures).  

The \emph{ad hoc} push forward commutes with the de Rham functor:

\begin{pro}\label{P:ahpush-comm-dR}
Let $f:\mathcal X\to \mathcal Y$ be the analytification of a schematic proper morphism of smooth separated integral DM stacks of finite type over $\mathbb C$, and let $\mathcal M$ be in $D^b_{rh,fake}(D_{\mathcal X})$, e.g., $\mathcal M$ is a regular holonomic $D_{\mathcal X}$-module.    Then the de Rham functor \eqref{E:def-drR-XX} commutes with the \emph{ad hoc}  push forward functors   \eqref{L:CmodGenDescA} and \eqref{L:CmodGenDescADmod}; i.e., we have a canonical isomorphism 
$$
^pDR_{\mathcal Y}(f^{ah}_+\mathcal M)\stackrel{\sim}{\longrightarrow} Rf^{ah}_* \, ^pDR_{\mathcal X}(\mathcal M).
$$
\end{pro}

\begin{proof}
At the level of objects with descent data, this is saying that the objects $$\{^pDR_U (f'_+\mathcal M_{U_{\mathcal X}}),\, ^pDR_{U'}(\tau_\phi)\} \quad \text{and }\quad \{Rf'_* \, ^pDR_{U_{\mathcal X}}(\mathcal M_{U_{\mathcal X}}), \, ^p\tau_{\phi}\}$$
with descent data, corresponding to $^pDR_{\mathcal Y}(f^{ah}_+\mathcal M)$ and $Rf^{ah}_* \, ^pDR_{\mathcal X}(\mathcal M)$, respectively, 
are canonically isomorphic. This follows from the fact that for proper morphisms of smooth varieties, the de Rham functor commutes with the push forward functors (see e.g., \cite[Prop.~4.7.5]{HTT08}).  The details are similar to the proofs of previous compatibility results and are omitted for brevity.  
\end{proof}

\begin{cor}\label{C:ahpush-comm-dR}
In the notation of \Cref{P:ahpush-comm-dR}, we have a canonical isomorphism of perverse sheaves on $\mathcal Y$: 
$$
^pDR_{\mathcal Y}(\mathcal H^if^{ah}_+\mathcal M)\stackrel{\sim}{\longrightarrow} R^if^{ah}_* \, ^pDR_{\mathcal X}(\mathcal M).
$$
\end{cor}

\begin{proof}
This follows directly from \Cref{P:ahpush-comm-dR}.
\end{proof}

\begin{rem}[Regular holonomic $D$-modules with $\mathbb Q$-structure] \label{R:RHF-push}
At this point, we have enough to define the \emph{ad hoc} push forward of a regular holonomic $D$-module with $\mathbb Q$ structure, in the following sense.  Suppose that $f:\mathcal X\to \mathcal Y$ is  a schematic proper morphism of smooth separated integral DM stacks of finite type over $\mathbb C$, and suppose that $(\mathcal M,K)$ is a pair with $\mathcal M$ a regular holonomic $D_{\mathcal X}$-module and $K$ in $\operatorname{Perv}(\mathbb Q_{\mathcal X})$ such that $^pDR(\mathcal M^{an})= K\otimes_{\mathbb Q_{\mathcal X}}\mathbb C_{\mathcal X}$.  Then putting together 
\Cref{C:PervDescA}, 
\Cref{C:PervDescADmod},
\Cref{P:ahCRHpush},
\Cref{C:alg-and-an-ah-push-agree}, and 
\Cref{C:ahpush-comm-dR}, 
we have for the pair $(\mathcal H^if^{ah}_+\mathcal M,R^if^{ah}_*K)$  that $\mathcal H^if^{ah}_+\mathcal M$ is a regular holonomic $D_{\mathcal Y}$-module and $R^if^{ah}_*K$ in $\operatorname{Perv}(\mathbb Q_{\mathcal Y})$ is such that  $^pDR( \mathcal H^if^{ah}_+\mathcal M)^{an}= (R^if^{ah}_*K)\otimes_{\mathbb Q_{\mathcal Y}}\mathbb C_{\mathcal Y}$.  In order to discuss push forwards of Hodge modules, we will need to work with filtered $D$-modules.  
\end{rem}

\section{Push forward of $D$-modules}\label[appendix]{S:A:BDpushforward}

In this section we discuss the push forward of $D$-modules at the level of derived categories.  This is so we can discuss the push forward of \emph{filtered} $D$-modules in \S \ref{S:FiltDModPush}, and then the \emph{ad hoc} push forward of Hodge modules in \S \ref{S:PushHodge}.  Our presentation in this section for $D$-modules is essentially following \cite{BDstacks}.

\subsection{Generalities on Stacks}\label{S:Stack-derived-general}

In this section we will frequently need to use some general constructions with derived categories on stacks. Since we will be working with sheaves of modules over various types of sheaves of rings, for clarity, we collect and provide references for some standard results. 

\begin{pro}[Bounded derived category of modules on a ringed stack]\label{prop:Db-on-ringed-stack}
Let $\mathcal{X}$ be a DM stack, and let $\mathcal{X}_{\acute{e}tale}$ denote its small $\acute{e}tale$ site.
Let $\mathcal{A}$ be a sheaf of associative rings (not necessarily commutative) on $\mathcal{X}_{\acute{e}tale}$.
Let $\mathrm{Mod}(\mathcal{A})$ be the abelian category of sheaves of \emph{left} $\mathcal{A}$-modules on $\mathcal{X}_{\acute{e}tale}$.

\begin{enumerate}
\item The category $\mathrm{Mod}(\mathcal{A})$ has enough injectives; in fact there exist functorial injective embeddings.
\item Every complex $M^\bullet$ of objects of $\mathrm{Mod}(\mathcal{A})$ admits a functorial quasi-isomorphism
\[
M^\bullet \longrightarrow I^\bullet
\]
where each $I^n$ is injective in $\mathrm{Mod}(\mathcal{A})$ and $I^\bullet$ is $K$-injective.
\item We have the derived category $D(\mathrm{Mod}(\mathcal{A}))$  
and the full subcategory
\[
D^b(\mathrm{Mod}(\mathcal{A})) \subseteq D(\mathrm{Mod}(\mathcal{A}))
\]
on complexes with bounded cohomology is a well-defined triangulated category. We denote it by
\[
D^b(\mathcal{A}) := D^b(\mathrm{Mod}(\mathcal{A})).
\]
\item If $F : \mathrm{Mod}(\mathcal{A}) \to \mathcal{B}$ is an additive left exact functor to an abelian category $\mathcal{B}$,
then the right derived functor $RF : D(\mathrm{Mod}(\mathcal{A})) \to D(\mathcal{B})$ exists and can be computed by
\[
RF(M^\bullet) \cong F(I^\bullet),
\]
for any choice of $K$-injective resolution $(M^\bullet \to I^\bullet)$ as in (2).
\end{enumerate}
\end{pro}

\begin{proof}
(1) and (2) follow by\cite[\href{https://stacks.math.columbia.edu/tag/01DU}{Tag 01DU}]{stacks-project} and \cite[\href{https://stacks.math.columbia.edu/tag/079P}{Tag 079P}]{stacks-project} respectively.
 (3) Note, $\mathrm{Mod}(\mathcal{A})$ is an example of a Grothendieck abelian category \cite[\href{https://stacks.math.columbia.edu/tag/079Q}{Tag 079Q}]{stacks-project}. In a Grothendieck abelian category $\mathcal{A}$, the existence of functorial $K$-injective resolutions implies that
$\mathrm{Hom}$-classes in $D(\mathcal{A})$ are sets, since one may compute
\[
\mathrm{Hom}_{D(\mathcal{A})}(K^\bullet,L^\bullet)=
\mathrm{Hom}_{K(\mathcal{A})}(K^\bullet, I^\bullet)
\]
after replacing $L^\bullet$ by a $K$-injective $I^\bullet$ \cite[\href{https://stacks.math.columbia.edu/tag/079Q}{Tag 079Q}]{stacks-project}.
By \cite[\href{https://stacks.math.columbia.edu/tag/05RR}{\S 05RR}]{stacks-project}, $D(\mathrm{Mod}(\mathcal{A}))$ is well-defined, and $D^b(\mathcal{A})$ is its bounded subcategory, where we require that each object has non-zero cohomology only at finitely many degrees (but we allow complexes to be of infinite length). 

(4) By \cite[\href{https://stacks.math.columbia.edu/tag/05T3}{\S 05T3}]{stacks-project}, we obtain a functor $F:K(\mathcal{A})\rightarrow D(\mathcal{B})$. Then, the result follows from \cite[\href{https://stacks.math.columbia.edu/tag/070K}{Lem.~070K}]{stacks-project}
\end{proof}

\begin{pro}[Functorial injective embeddings for graded modules]\label{lem:graded-functorial-inj}
Let $(\mathcal C,\mathcal O)$ be a ringed site and let $\mathcal A$ be a $\mathbf{Z}$-graded (possibly noncommutative) $\mathcal O$-algebra.
Then,

\begin{enumerate}
\item The category $\mathrm{Mod}(\mathcal{A})$ has enough injectives; in fact there exist functorial injective embeddings.
\item Every complex $M^\bullet$ of objects of $\mathrm{Mod}(\mathcal{A})$ admits a functorial quasi-isomorphism
\[
M^\bullet \longrightarrow I^\bullet
\]
where each $I^n$ is injective in $\mathrm{Mod}(\mathcal{A})$ and $I^\bullet$ is $K$-injective.
\item Hence we have the (unbounded) derived category $D(\mathrm{Mod}(\mathcal{A}))$, and the full subcategory
\[
D^b(\mathrm{Mod}(\mathcal{A})) \subseteq D(\mathrm{Mod}(\mathcal{A}))
\]
on complexes with bounded cohomology is a well-defined triangulated category. We denote it by
\[
D^b(\mathcal{A}) := D^b(\mathrm{Mod}(\mathcal{A})).
\]
\item If $F : \mathrm{Mod}(\mathcal{A}) \to \mathcal{B}$ is an additive left exact functor to an abelian category $\mathcal{B}$,
then the right derived functor $RF : D(\mathrm{Mod}(\mathcal{A})) \to D(\mathcal{B})$ exists and can be computed by
\[
RF(M^\bullet) \cong F(I^\bullet),
\]
for any choice of $K$-injective resolution $(M^\bullet \to I^\bullet)$ as in (2).
\end{enumerate}

\end{pro}

\begin{proof}
(1) By 
\cite[\href{https://stacks.math.columbia.edu/tag/0FRD}{Lem~0FRD}]{stacks-project}, the category $\operatorname{gr-Mod}(\mathcal A)$ is an abelian Grothendieck category. Hence, (1) and (2) follow by \cite[\href{https://stacks.math.columbia.edu/tag/079H}{Thm.~079H}]{stacks-project} and \cite[\href{https://stacks.math.columbia.edu/tag/079P}{Thm.~079P}]{stacks-project}.
The rest follows as in the ungraded case.
\end{proof}

\subsection{\v{C}ech complexes}\label{S:CechNotation}
In some cases, we will want to work with \v{C}ech complexes.  For this we fix the following notation. 
Let  $p:U\ra \mathcal{Y}$ be an \'etale cover by affines, and note that $p$ is an affine morphism, as the diagonal of $\mathcal Y$ is affine.   Let $p_{\mathcal X}:U_{\mathcal X}\ra \mathcal{X}$ be  the  \'etale cover  obtained via the fibered product with $f:\mathcal X \to \mathcal Y$; while $U_{\mathcal X}$ may not be affine, $p_{\mathcal X}$ is affine by base change.  To fix notation, one has a diagram:
\begin{equation}\label{E:CechCov}
\begin{tikzcd}
\cdots &U_{\mathcal X}\times_{\mathcal X} U_{\mathcal X}\times_{\mathcal X} U_{\mathcal X} \arrow[d,"f_3"] \arrow[r, shift left=2] \arrow[r, shift right=2] \arrow[r] & U_{\mathcal X}\times_{\mathcal X} U_{\mathcal X} \arrow[d,"f_2"] \arrow[r, shift left] \arrow[r, shift right] & U_{\mathcal X} \arrow[r, "p_{\mathcal X}"] \arrow[d, "f'=f_1"] & \mathcal X \arrow[d, "f"] \\
\cdots &U\times_{\mathcal Y} U\times_{\mathcal Y} U \arrow[r, shift left=2] \arrow[r, shift right=2] \arrow[r]           & U\times_{\mathcal Y} U \arrow[r, shift left] \arrow[r, shift right]           & U \arrow[r, "p"]                & \mathcal Y               
\end{tikzcd}
\end{equation}
Let $U_{\mathcal X,n}:=U_{\mathcal X}\times_{\mathcal{X}}\cdots \times_{\mathcal{X}} U_{\mathcal X}$ be the $n$-th fiber product, and let $p_{\mathcal X,n}:U_{\mathcal X,n}\ra \mathcal{X}$ be the induced map. We define $U_n$ and  $p_n:U_n\to \mathcal Y$ similarly.  
For any complex of sheaves $\mathcal F$ on the \'etale site of $\mathcal X$, we define the \v{C}ech complex $\check{C}^\bullet (\mathcal F, U)$ by 
\begin{equation}\label{E:CechDef}
\check{C}^i(\mathcal F,U)= \bigoplus _{n\ge 0}p_{\mathcal X, n*}(\mathcal F^{i-n}_{U_{\mathcal X,n}})
\end{equation}
 with the obvious boundary maps, where $p_{\mathcal X,n*}$ is the push forward of sheaves.

\subsection{Push forward of $D$-modules following Beilinson--Drinfeld}\label{S:BDpush}

Here we recall the push forward functor for $D$-modules on stacks given in \cite{BDstacks} via differential graded algebras (dgAs); we use the conventions for dgAs in   \cite[\href{https://stacks.math.columbia.edu/tag/09JD}{Ch.~09JD}]{stacks-project}.    
We note that the presentation in  \cite{BDstacks}  is given for more general smooth Artin stacks with affine diagonals; here we continue to work with smooth DM stacks and note that the definition of DM stacks we use in this paper implies they have affine diagonal.

Following \cite[7.2--7.5]{BDstacks}, 
denote by $\Omega_{\mathcal X}$ the differential graded algebra (dga)   of differential forms on $\mathcal X$. 
 Equivalently,  $\Omega_{\mathcal{X}}=\operatorname{DR}^\bullet (\m{O}_{\mathcal{X}})$, as defined in \S \ref{S:dR+SC'}.  
We then consider the dg ringed space $(\mathcal X,\Omega_{\mathcal X})$, as well as the category of $\Omega_{\mathcal X}$-complexes, i.e., dg $\Omega_{\mathcal X}$-modules.  We say an $\Omega_{\mathcal X}$-complex $\mathcal F=(\mathcal F^\bullet, d)$ is quasi-coherent if the $\mathcal F^i$ are quasi-coherent $\mathcal O_{\mathcal X}$-modules; quasi-coherent $\Omega_{\mathcal X}$-complexes will be called $\Omega$-complexes on $\mathcal X$.  
Denote the dg category of $\Omega$-complexes on $\mathcal X$ by $C_{\operatorname{BD}}(\mathcal X,\Omega_{\mathcal X})$.
Let $C_{\operatorname{BD}}(\mathcal X,D_{\mathcal X})$ be the dg category of complexes of \emph{right} 
$D$-modules on $\mathcal X$ (which we will call right $D_{\mathcal X}$-complexes, or just $D_{\mathcal X}$-complexes for short), and let $K_{\operatorname{BD}}(\mathcal X,D_{\mathcal X})$ be the corresponding homotopy category.  We have a pair of   functors 
\begin{equation}\label{E:DDOmega-def}
\mathbf D:C_{\operatorname{BD}}(\mathcal X,\Omega_{\mathcal X}) \longrightarrow C_{\operatorname{BD}}(\mathcal X, D_{\mathcal X}), \quad \mathbf \Omega : C_{\operatorname{BD}}(\mathcal X,D_{\mathcal X}) \longrightarrow C_{\operatorname{BD}}(\mathcal X,\Omega_{\mathcal X})
\end{equation}
    \[\mathbf{D}(\mathcal F):=\mathcal F\otimes_{\Omega_{\mathcal{X}}}\operatorname{DR}(D_{\mathcal{X}}), \quad \mathbf \Omega(\mathcal  M):=\operatorname{Hom}_{D_{\mathcal{X}}}(\operatorname{DR}(D_{\mathcal{X}}),\mathcal  M)\]
which 
form an adjoint pair $\mathbf D \dashv  \mathbf \Omega$ (via tensor-hom adjunction).  By construction, these functors are compatible with \'etale base change in the sense that if $p:U\to \mathcal X$ is an \'etale morphism from a smooth variety (or complex analytic space) $U$, then we have commutative diagrams 
\begin{equation}\label{E:DO-local}
\xymatrix{
C_{BD}(\mathcal X,\Omega_{\mathcal X})\ar[r]^{p^*} \ar[d]_{\mathbf D}& C_{BD}(\mathcal X,\Omega_{U}) \ar[d]_{\mathbf D}& C_{BD}(\mathcal X,D_{\mathcal X}) \ar[r]^{p^*} \ar[d]_{\mathbf \Omega}& C_{BD}(\mathcal X,D_{U})\ar[d]_{\mathbf \Omega}\\
C_{BD}(\mathcal X,\Omega_{\mathcal X})\ar[r]^{p^*}& C_{BD}(\mathcal X,\Omega_{U}) & C_{BD}(\mathcal X,D_{\mathcal X}) \ar[r]^{p^*}& C_{BD}(\mathcal X,D_{U})\\
}
\end{equation}
where the pull backs for the \'etale morphism come from the definitions as sheaves on the \'etale site.  
The co-unit of the adjunction  
$$\mathbf D\mathbf \Omega \stackrel{q.i.}{\longrightarrow} \operatorname{Id}$$ 
is a quasi-isomorphism (the commutativity of the diagrams in \eqref{E:DO-local} implies this can be reduced to the case of varieties, which is established in \cite[7.2.2, 7.2.4, Rem.~7.3.5]{BDstacks}).

As $\operatorname{DR}(D_{\mathcal X})= \Omega_{\mathcal X}\otimes_{\mathcal O_{\mathcal X}}D_{\mathcal X}$, where we are taking tensor products as dg modules,  we have 
\begin{equation}\label{E:Alt-D}
\mathbf D(\mathcal F)= \mathcal F\otimes_{\mathcal O_{\mathcal X}}D_{\mathcal X}.
\end{equation}
  We also have that for $\mathbf \Omega(\mathcal M)$, $(\mathbf \Omega(\mathcal M))^i=\bigoplus_{a-b=i}\bigwedge ^b\mathcal T_{\mathcal X}\otimes _{\mathcal O_{\mathcal X}}\mathcal M^a$, and consequently, we call this the \emph{Spencer complex} for $\mathcal M$,   
\begin{equation}\label{E:Sp=Omega-def}
\operatorname{Sp}^\bullet (\mathcal M):=\mathbf \Omega (\mathcal M),
\end{equation}
since this agrees with the definition in \S \ref{S:dR+SC} when $\mathcal M$ is a sheaf (complex in degree zero only).

For any $\Omega_{\mathcal X}$-complex $\mathcal F$, set $$\mathcal H_{\mathbf D}^{\bullet}(\mathcal F):=\mathcal H^{\bullet}(\mathbf{D}(\mathcal F)).$$
With this, following  \cite[7.3.2]{BDstacks}, define  
\begin{equation}\label{E:DBDdef}
D_{\operatorname{BD}}(\mathcal{X}, \Omega_{\mathcal X}):=K_{BD}(\mathcal X,\Omega_{\mathcal X})[(\mathbf D\operatorname{-q.iso})^{-1}], 
\end{equation}
to be the homotopy category of $\Omega_{\mathcal X}$-complexes, localized at $\mathbf D$-quasi-isomorphisms (i.e.,  morphisms taken by $\mathbf D$ to quasi-isomorphisms). We have the subcategory $D^b_{\operatorname{BD}}(\mathcal{X}, \Omega)$ of bounded objects, i.e., objects that have nonzero cohomology with respect to $\mathcal H_{\mathbf{D}}^i(-)$ only in finitely many degrees.
There is a $t$-structure on $D_{\operatorname{BD}}(\mathcal{X}, \Omega_{\mathcal X})$ whose heart is $\operatorname{Mod}(D_{\mathcal{X}})^{\operatorname{right}}$, the category of right $D_{\mathcal X}$-modules,  and with cohomology functor $\mathcal H^\bullet_{\mathbf D}$ \cite[7.3.4]{BDstacks}.

For smooth DM stacks we have an equivalence \cite[7.3.5]{BDstacks}
\begin{equation}\label{E:Omega-DBD=D}
\mathbf \Omega:D(\operatorname{Mod}(D_{\mathcal X})^{\operatorname{right}})\stackrel{\sim}{\longrightarrow}  D_{\operatorname{BD}}(\mathcal{X},\Omega_{\mathcal X}),
\end{equation}
with inverse
\begin{equation}\label{E:D-DBD=D}
\mathbf D:  D_{\operatorname{BD}}(\mathcal{X},\Omega_{\mathcal X})\stackrel{\sim}{\longrightarrow} D(\operatorname{Mod}(D_{\mathcal X})^{\operatorname{right}}),
\end{equation}
where $\mathbf \Omega$ and $\mathbf D$ above  are induced by the functors in \eqref{E:DDOmega-def}.  
Now let $f:\mathcal X\to \mathcal Y$ be a schematic separated morphism of finite type, of smooth separated integral DM stacks of finite type over $\mathbb C$. 
If $f$ is smooth, Beilinson--Drinfeld define a $t$-exact pull back functor \cite[7.3.6]{BDstacks} 
$$
f^*_{\operatorname{BD}}:D_{\operatorname{BD}}(\mathcal Y,\Omega_{\mathcal Y})\to D_{\operatorname{BD}}(\mathcal X,\Omega_{\mathcal X}),
$$
compatible with composition, by observing that there is a natural pull back functor at the level of $\Omega$-complexes,  $f^*_\Omega:C(\mathcal Y,\Omega_{\mathcal Y})\to C(\mathcal X,\Omega_{\mathcal X})$, which preserves $\mathbf D$-quasi-isomorphisms.  

Returning to the general case (where $f:\mathcal X\to \mathcal Y$ is no longer assumed to be smooth),  in \cite[7.3.6, Rem.~7.3.11(ii)]{BDstacks} Beilinson--Drinfeld also define a push forward functor
\begin{equation}\label{E:DBpushforward}
f_{\operatorname{BD}*}:D_{\operatorname{BD}}(\mathcal X,\Omega_{\mathcal X})\to D_{\operatorname{BD}}(\mathcal Y,\Omega_{\mathcal Y})
\end{equation}
compatible with composition.  For this, they consider  the natural push forward functor on complexes  $$f_{\Omega *}:C_{\operatorname{BD}}(\mathcal X,\Omega_{\mathcal X})\to C_{\operatorname{BD}}(\mathcal Y,\Omega_{\mathcal Y}),$$ which is the right adjoint to $f_\Omega^*$, and define $f_{\operatorname{BD}*}$ to be the derived functor of $f_{\Omega *}$, which they show preserves $\mathbf D$-quasi-isomorphisms, and therefore defines a morphism $f_{\operatorname{BD}*}$ as above in \eqref{E:DBpushforward}.   By construction, $f_{BD*}$ and $f_{\Omega*}$ are compatible with \'etale base change.

\begin{rem}\label{R:CechComplex} 
We note for later that in \cite[7.3.6]{BDstacks}, they  
 show that $f_{\operatorname{BD}*}$ can be computed via \v{C}ech 
complexes, as follows, 
 where we use the notation for \v{C}ech complexes in \S \ref{S:CechNotation}.
 We note that when $\mathcal F$ is an $\Omega_{\mathcal X}$-complex in $D_{\operatorname{BD}}(\mathcal X,\Omega)$,  then in the notation of \S \ref{S:CechNotation}, one has by construction that $p_{\mathcal X,n*}=(p_{\mathcal X,n})_{\Omega*}$, so that one has that $\check{C}^\bullet (\mathcal F, U)$ is naturally in $C_{\operatorname{BD}}(\mathcal X,\Omega)$ (it is not just a complex of sheaves).  The proof of \cite[7.3.8]{BDstacks} shows  that 
 \begin{align}\label{E:Cech-computes-BDPushforward}
     f_{\operatorname{BD}*}\mathcal F \simeq f_{\Omega*}\check{C}^\bullet (\mathcal F, U).
 \end{align}
\end{rem}

We denote by $\operatorname{Mod}(D_{\mathcal X})^{\operatorname{right}}$ the abelian category of (right) $D_{\mathcal X}$-modules on $\mathcal X$, and we denote the associated derived category by   
$D(\operatorname{Mod}(D_{\mathcal X})^{\operatorname{right}})$.

\begin{dfn}[Push forward of $D$-modules]\label{D:PushD-BD}
For $f:\mathcal X\to \mathcal Y$ a  schematic proper morphism  of smooth separated integral DM stacks of finite type over $\mathbb C$ (or the analytification of such a morphism), and  $\mathcal M$ in $D(\operatorname{Mod}(D_{\mathcal X})^{\operatorname{right}})$ (e.g., $\mathcal M$ is a $D$-module on $\mathcal X$), we define the \emph{push forward} of $\mathcal M$ as
$$
f_+:D(\operatorname{Mod}(D_{\mathcal X})^{\operatorname{right}})\longrightarrow D(\operatorname{Mod}(D_{\mathcal Y})^{\operatorname{right}})
$$
\begin{equation}\label{E:PushD-BD}
f_+\mathcal M:= \mathbf D f_{\operatorname{BD}*} \mathbf \Omega \mathcal M.
\end{equation} 
\end{dfn}

To better understand this definition, 
we define a relative Spencer complex.  For  $f:\mathcal X\to \mathcal Y$  a  schematic separated morphism of finite type, of smooth separated integral  DM stacks of finite type over $\mathbb C$ (or the analytifcation of such a morphism), and $\mathcal M$ in $D(\operatorname{Mod}(D_{\mathcal X})^{\operatorname{right}})$, we define a \emph{relative Spencer complex}

\begin{equation}\label{E:DDOmegaSp}
  \operatorname{Sp}^\bullet_{\mathcal{X}\to \mathcal{Y}}(\mathcal M) :=(\mathbf{D}\mathbf \Omega\mathcal M)\otimes_{D_\mathcal{X}}f^* D_{\mathcal{Y}},
\end{equation}
where the $f^*$ on the right is the pull back as $\mathcal O_{\mathcal Y}$-modules;
from the discussion above, this agrees with the definition in  \S \ref{S:dR+SC} when $\mathcal M$ is a sheaf. Indeed, 
\begin{align}
\nonumber 
\operatorname{Sp}^\bullet_{\mathcal{X}\to \mathcal{Y}}(\mathcal M) & = (\mathbf{D}\mathbf \Omega\mathcal M) \otimes_{D_{\mathcal{X}}} f^* D_{\mathcal{Y}}\\ 
\nonumber 
& = (\mathbf \Omega\mathcal M\otimes_{\mathcal{O}_{\X}}D_{\mathcal{X}})\otimes_{D_{\X}} (\mathcal{O}_{\X}\otimes_{f^{-1}\mathcal O_{\mathcal Y}}f^{-1} D_{\mathcal{Y}}) \\
\label{E:AltSpencf-1}
& =\mathbf \Omega \mathcal M\otimes_{f^{-1}\mathcal O_{\mathcal Y}}f^{-1} D_{\mathcal{Y}} \\
\label{E:AltSpenc}
&= \operatorname{Sp}^\bullet (\mathcal M)\otimes_{f^{-1}\mathcal O_{\mathcal Y}}f^{-1} D_{\mathcal{Y}},
\end{align}
and the last line agrees with the definition in \S \ref{S:dR+SC}.
Note that for $\mathcal M$ in $D(\operatorname{Mod}(D_{\mathcal X})^{\operatorname{right}})$, the last three equalities above define an $f^{-1}D_{\mathcal Y}$-module structure on $\operatorname{Sp}^\bullet_{\mathcal X\to \mathcal Y}(\mathcal M)$.

Now if $f:X\to Y$ is a proper morphism of \emph{varieties}, and $M$ is a $D$-module on $X$, then the definition of  $f_+M$ in \Cref{D:PushD-BD}  agrees with the standard definition (see e.g., \S\ref{S:adhoc-push-D}).  In fact,   
applying \Cref{prop:Db-on-ringed-stack}, or equivalently \cite[\href{https://stacks.math.columbia.edu/tag/071J}{\S 071J}]{stacks-project}, where we let $f:\mathcal{X}\rightarrow \mathcal Y $ be a schematic morphism of smooth DM stacks,  take $\mathcal{A}=f^{-1}D_{\mathcal Y}$ and $\mathcal{B}=\operatorname{Mod}(D_{\mathcal Y})^{\operatorname{right}}$, and set $F=f_*$, yields a derived push forward 
$$Rf_*:D^b(f^{-1}D_{\Y})\rightarrow D^b(D_{\Y}).$$
Viewing $\operatorname{Sp}_{\mathcal X\to \mathcal Y}^\bullet \mathcal M $ as an object of $D^b(f^{-1}D_{\mathcal Y})$, we obtain a push forward $Rf_*\operatorname{Sp}_{\mathcal X\to \mathcal Y}^\bullet \mathcal M $. 

We claim that
\begin{equation}\label{E:BD=Rf*}
(f_+\mathcal M :=) \ \ \mathbf D f_{BD *}\mathbf{\Omega}\mathcal M= Rf_*\operatorname{Sp}_{\mathcal X\to \mathcal Y}^\bullet \mathcal M, 
\end{equation}
 where $Rf_*$ is the derived functor $Rf_*: D^b(f^{-1}D_{\mathcal Y})\to D^b(D_{\mathcal Y})$ above; the point is that the right hand side of \eqref{E:BD=Rf*} agrees with the definition of $f_+\mathcal M$ in the case of $D$-modules on varieties, reviewed in \S\ref{S:adhoc-push-D}. 
To establish \eqref{E:BD=Rf*} we have  
\begin{align*}
\mathbf D f_{BD *}\mathbf{\Omega}\mathcal M& = (f_{BD *}\mathbf{\Omega}\mathcal M)\otimes_{\Omega_{\mathcal Y}}(\Omega_{\mathcal Y}\otimes_{\mathcal{O}_{\mathcal Y}} D_{\mathcal Y}) & \text{(by \eqref{E:Alt-D})}\\ 
  & = Rf_*\left(\mathbf{\Omega}\mathcal M\otimes_{f^{-1}\Omega_{\mathcal Y}}(f^{-1}\Omega_{\mathcal Y}\otimes_{f^{-1}\mathcal{O}_{\mathcal Y}} f^{-1}D_{\mathcal Y})\right) & \text{(Projection formula)}\\
    & = Rf_*(\mathbf{\Omega}\mathcal M\otimes_{f^{-1}\mathcal{O}_{\mathcal Y}} f^{-1}D_{\mathcal Y})\\
    & = Rf_*(\mathbf{\Omega}\mathcal M\otimes_{\mathcal{O}_{\mathcal X}}\mathcal{O}_{\mathcal X}\otimes_{f^{-1}\mathcal{O}_{\mathcal Y}} f^{-1}D_{\mathcal Y})\\  
    & = Rf_* (\mathbf{\Omega}\mathcal M \otimes_{\mathcal{O}_{\mathcal X}} f^* D_{\mathcal Y}) \\
& = Rf_* (\mathbf{\Omega} \mathcal M \otimes_{\Omega_{\mathcal X}} (\Omega_{\mathcal X} \otimes_{\mathcal{O}_{\mathcal X}} D_{\mathcal X})\otimes_{D_{\mathcal X}} f^* D_{\mathcal Y}) \\ 
&=    Rf_*(\mathbf D\mathbf \Omega \mathcal M\otimes_{D_{\mathcal X}} f^*D_{\mathcal Y}) & \text{(by \eqref{E:Alt-D})}\\ 
&= Rf_*\operatorname{Sp}_{\mathcal X\to \mathcal Y}^\bullet \mathcal M & \text{(by \eqref{E:DDOmegaSp})}
\end{align*}
Note that we are using the projection formula for dgas (e.g., \cite[Thm. 4.6]{Rybakov2015DGModulesDeRham}), and $f^*D_{\mathcal Y}$ indicates the pull back as an $\mathcal O_{\mathcal Y}$-module.

The definition of the push forward is local in the following sense:

\begin{lem}\label{L:Canon-Iso+Sp}
Let $f:\mathcal X\to \mathcal Y$ be  schematic proper morphism  of smooth separated integral DM stacks of finite type over $\mathbb C$ (or the analytification of such a morphism), and  $\mathcal M$ in $D(\operatorname{Mod}(D_{\mathcal X})^{\operatorname{right}})$ (e.g., $\mathcal M$ is a $D$-module on $\mathcal X$). Let $p:U\to \mathcal Y$ be an \'etale morphism from a smooth variety (or complex analytic space).
In the notation of diagram \eqref{E:TubPresDpf'}, 
there are canonical isomorphisms 
\begin{equation}\label{E:Canon-Iso+Sp-1}
p^*f_+\mathcal M\stackrel{\sim}{\longrightarrow}f'_+(p_{\mathcal X}^*\mathcal M)
\end{equation}
\begin{equation}\label{E:Canon-Iso+Sp-2}
\operatorname{Sp}_{\mathcal X\to \mathcal Y}^\bullet (\mathcal M) |_{U_{\mathcal X_{zar}}} \stackrel{\sim}{\longrightarrow} \operatorname{Sp}_{U_{\mathcal X}\to U}p_{\mathcal X}^*\mathcal M
\end{equation}
\begin{equation}\label{E:Canon-Iso+Sp-3}
p^*f_+\mathcal M\stackrel{\sim}{\longrightarrow} Rf'_*\operatorname{Sp}_{U_{\mathcal X}\to U}p_{\mathcal X}^*\mathcal M.
\end{equation}

\end{lem}

\begin{proof}
The first isomorphism comes from the definition of $f_+$ \eqref{E:PushD-BD} and the fact that $\mathbf D$, $\mathbf \Omega$ (see \eqref{E:DO-local}), and $f_{BD*}$ all commute with \'etale base change.  The second isomorphism is obtained in the same way from the definition of the relative Spencer complex \eqref{E:DDOmegaSp}. 
\end{proof}

\begin{cor}\label{C:Im f+=fah+}
Let $f:\mathcal X\to \mathcal Y$ be  schematic proper morphism  of smooth separated integral DM stacks of finite type over $\mathbb C$ (or the analytification of such a morphism).  For $\mathcal M$ a $D_{\mathcal X}$-module, the image of $f_+\mathcal M$ under  the natural functor $D(\operatorname{Mod}(D_{\mathcal Y})^{\operatorname{right}})\to  D_{fake}(D_{\mathcal Y})$ (see \eqref{E:genCompareDfake}) is $f_+^{ah}(\mathcal M)$ (see \Cref{L:CmodGenDescADmod}).  In other words, there is a commutative diagram 
$$
\xymatrix{
\operatorname{Mod}(D_{\mathcal X})^{\operatorname{right}} \ar[r] \ar[d]^{f_+}&   D_{fake}(D_{\mathcal X}) \ar[d]^{f^{ah}_+}\\
D(\operatorname{Mod}(D_{\mathcal Y})^{\operatorname{right}}) \ar[r]&   D_{fake}(D_{\mathcal Y}).
}
$$
\end{cor}

\begin{proof}
If $\mathcal M$ is a $D_{\mathcal X}$-module, then it comes equipped with descent data, and so it makes sense to apply $f^{ah}_+$.  Then, considering \eqref{E:Canon-Iso+Sp-3}, one sees  that the image of $f_+\mathcal M$ in $D_{fake}(D_{\mathcal Y})$ agrees with the definition of $f^{ah}_+\mathcal M$. 
\end{proof}

Taking cohomology in \eqref{E:PushD-BD} defines  the $i$-th push forward  functor on $D_{\mathcal X}$-modules:
$$
\mathcal H^if_+:\operatorname{Mod}(D_{\mathcal X})^{\operatorname{right}}\longrightarrow \operatorname{Mod}(D_{\mathcal Y})^{\operatorname{right}}
$$
\begin{equation}\label{E:def-Hif+M}
\mathcal M\mapsto \mathcal H^i(f_+\mathcal M)
\end{equation}

\begin{cor}\label{C:Canon-Iso-Hi}
In the notation of \Cref{L:Canon-Iso+Sp}, there are canonical isomorphisms 
$$
p^*\mathcal H^i(f_+\mathcal M)\stackrel{\sim}{\longrightarrow}\mathcal H^i(f'_+(p_{\mathcal X}^*\mathcal M)).
$$

\end{cor}

\begin{proof}
This follows by taking cohomology in  \eqref{E:Canon-Iso+Sp-1}.
\end{proof}

\begin{cor}\label{C:Hif+=Hifah+}
Let $f:\mathcal X\to \mathcal Y$ be  schematic proper morphism  of smooth separated integral DM stacks of finite type over $\mathbb C$ (or the analytification of such a morphism). 
Then the functors 
$$
\mathcal H^if_+, \ \mathcal H^if^{ah}_+:\operatorname{Mod}(D_{\mathcal X})^{\operatorname{right}}\longrightarrow \operatorname{Mod}(D_{\mathcal Y})^{\operatorname{right}}
$$
of \eqref{E:C:PervDescADmod} and \eqref{E:def-Hif+M} agree. 

\end{cor}

\begin{proof}

Taking cohomology in \Cref{C:Im f+=fah+}  gives the identification $\mathcal H^i(f_+\mathcal M)= \mathcal H^i(f^{ah}_+\mathcal M)$ (see \Cref{R:ah-push-compatibleHH}). 
\end{proof}

\subsection{Compatibility  of algebraic and analytic push forwards}

In this subsection we establish the compatibility of the algebraic and analytic push forwards, which essentially comes down to the fact that the push forward can be computed using \v Cech complexes.  
 We will use the superscript $b$ (resp.~subscript $coh$) to indicate full subcategories of derived categories whose objects have bounded (resp.~coherent) cohomology sheaves; e.g., we will denote by $D^b_{BD,coh}(\mathcal X, \Omega_{\X})\subseteq D^b_{BD}(\mathcal X, \Omega_{\X})$ the full subcategory of objects whose cohomology sheaves  are coherent, and are non-zero in only a finite number of places.

\begin{lem}\label{L:fBDanCom}
Let $f:\mathcal X\to \mathcal Y$ be  schematic proper morphism  of smooth separated integral DM stacks of finite type over $\mathbb C$. 
The analytification functor and Beilinson--Drinfeld push forward 
\eqref{E:DBpushforward}
commute; i.e., 
there is a commutative diagram
$$
\begin{tikzcd}
D^b_{BD,coh}(\mathcal X, \Omega_{\X}) \arrow[r, "an"] \arrow[d, "f_{BD *}"] & D^b_{BD,coh}(\mathcal X^{an},\Omega_{\X^{an}}) \arrow[d, "f_{BD *}"] \\
D^b_{BD,coh}(\mathcal Y, \Omega_{\Y}) \arrow[r, "{an}"]                  & D^b_{BD,coh}(\mathcal Y^{an}, \Omega_{\Y^{an}})                 
\end{tikzcd}
$$
\end{lem}

\begin{proof}
 We need to check that the diagram commutes via a canonical natural transformation. 
 To this end, choosing an \'etale presentation $p:U_{\mathcal X}\to \mathcal X$, and using the fact that  $f_{BD*}$ can be computed using a \v Cech resolution (\Cref{R:CechComplex}), 
it suffices to show there is a canonical isomorphism 
$$
f_{\Omega *}  \check{C}(\mathcal F^{an},U_{\mathcal X}^{an})\simeq (f_{\Omega *}  \check{C}(\mathcal F,U_{\mathcal X}))^{an}.
$$
In other words, we need to check that there is a commutative diagram
$$
\begin{tikzcd}
C_{\operatorname{BD}}(\mathcal X,\Omega_{\mathcal X})  \arrow[r, "\operatorname{an}"] \arrow[d, "{f_{\Omega *}\check{C}( (-),U_{\mathcal X})}"] & C_{\operatorname{BD}}(\mathcal X^{an},\Omega_{\mathcal X^{an}})  \arrow[d, "{f_{\Omega *}\check{C}( (-),U_{\mathcal X}^{an})}"] \\
C_{\operatorname{BD}}(\mathcal Y,\Omega_{\mathcal Y}) \arrow[r, "\operatorname{an}"]                                                              & C_{\operatorname{BD}}(\mathcal Y^{an},\Omega_{\mathcal Y^{an}})                                                              
\end{tikzcd}
$$
which follows from the definitions. 
\end{proof}

With this we can show that the push forward $f_+$ and analytification functor commute:

\begin{pro}\label{P:f+alg=f+an}
Let $f:\mathcal X\to \mathcal Y$ be  schematic proper morphism  of smooth separated integral DM stacks of finite type over $\mathbb C$, and  let
$\mathcal M$ be in $D_{coh}(\operatorname{Mod}(D_{\mathcal X})^{\operatorname{right}})$ (e.g., $\mathcal M$ is a coherent $D$-module on $\mathcal X$). Then we have
$$
(f_+\mathcal M)^{an}= f_+(\mathcal M^{an}).
$$

\end{pro}

\begin{proof}
The first observation is that $\mathbf D$ and $\mathbf \Omega$ commute with the analytification functor.  
This can be deduced by first considering the functors at the level of complexes; e.g., considering $\mathbf D$, one observes first that, by definition, one has a commutative diagram 
$$
\begin{tikzcd}
C_{\operatorname{BD}}(\mathcal X,\Omega_{\mathcal{X}})\arrow[r, "\operatorname{an}"] \arrow[d, "\mathbf D"] & C_{\operatorname{BD}}(\mathcal X^{an},\Omega_{\mathcal{X}^{an}}) \arrow[d, "\mathbf D"] \\
C_{\operatorname{BD}}(\mathcal X,D_{\mathcal{X}}) \arrow[r, "\operatorname{an}"]                     & C_{\operatorname{BD}}(\mathcal X^{an},\Omega_{\mathcal{X}^{an}}).       
\end{tikzcd}
$$
This then induces a commutative diagram 
$$
\begin{tikzcd}
D_{\operatorname{BD}}(\mathcal X,\Omega_{\mathcal{X}})\arrow[r, "\operatorname{an}"] \arrow[d, "\mathbf D"] & D_{\operatorname{BD}}(\mathcal X^{an},\Omega_{\mathcal{X}^{an}}) \arrow[d, "\mathbf D"] \\
D_{\operatorname{BD}}(\mathcal X,D_{\mathcal{X}}) \arrow[r, "\operatorname{an}"]                     & D_{\operatorname{BD}}(\mathcal X^{an}, D_{\mathcal{X}^{an}})       
\end{tikzcd}
$$
Here we are using that the analytification is an exact functor to deduce that there is a well-defined induced analytification functor after localization (e.g., \cite[\href{https://stacks.math.columbia.edu/tag/015F}{Lem.~015F}]{stacks-project}).  
Via the equivalences  \eqref{E:Omega-DBD=D} and  \eqref{E:D-DBD=D}, this then induces commutative diagrams
$$
\begin{tikzcd}
D_{\operatorname{BD}}(\mathcal X,\Omega_{\mathcal{X}})\arrow[r, "\operatorname{an}"] \arrow[d, "\mathbf D"] & D_{\operatorname{BD}}(\mathcal X^{an},\Omega_{\mathcal{X}^{an}}) \arrow[d, "\mathbf D"] \\
D(\operatorname{Mod}(D_{\mathcal X})^{\operatorname{right}}) \arrow[r, "\operatorname{an}"]                     & D(\operatorname{Mod}(D_{\mathcal X^{an}})^{\operatorname{right}})     
\end{tikzcd}
$$
From this, using \Cref{L:fBDanCom}, one has 
$$(f_+\mathcal M)^{an} := (\mathbf Df_{BD*}\mathbf \Omega \mathcal M)^{an} = \mathbf Df_{BD*}\mathbf \Omega (\mathcal M^{an})=: f_+(\mathcal M^{an}),$$
completing the proof.
\end{proof}

\begin{rem} Combining \Cref{P:f+alg=f+an} and \Cref{C:Hif+=Hifah+} provides a second proof of 
\Cref{P:alg-and-an-ah-push-agree}
and \Cref{C:alg-and-an-ah-push-agree}.
\end{rem}

\subsection{Coherent, regular, and holonomic $D$-modules}
\label{S:Regular/holo-are-etale}
We now address push forwards for coherent, regular, and holonomic $D$-modules; we reviewed the terminology in 
\Cref{S:Regular/holo-are-etale-ah}.

\begin{pro}[Coherent, regular, and holonomic $D$-modules]\label{P:CRHpush}
Suppose  $f:\mathcal X\to \mathcal Y$ is a  schematic proper morphism  of smooth separated integral DM stacks of finite type over $\mathbb C$ (or the analytification of such a morphism).
Then for $\sharp \in \{c,h,rh\}$, the functor  $f_+$ \eqref{E:PushD-BD} induces a functor 
$$
f_+:D^b_\sharp (\operatorname{Mod}(D_{\mathcal X})^{\operatorname{right}})\longrightarrow D^b_\sharp(\operatorname{Mod}(D_{\mathcal Y})^{\operatorname{right}})
$$
In other words, suppose   $\mathcal M$ is in $D^b_\sharp (\operatorname{Mod}(D_{\mathcal X})^{\operatorname{right}})$ for $\sharp \in \{c,h,rh\}$, i.e., for all $i$ we have $\mathcal H^i(\mathcal M)$ is a coherent (resp.~regular, resp.~holonomic) $D_{\mathcal X}$-module. Then the push forward $f_+\mathcal M$ (\Cref{D:PushD-BD}) is in $D^b_\sharp (\operatorname{Mod}(D_{\mathcal Y})^{\operatorname{right}})$, i.e., for all $i$ one has  $\mathcal H^i(f_+\mathcal M)$ is  a coherent (resp.~regular, resp.~holonomic) $D_{\mathcal Y}$-module.
\end{pro}

\begin{proof}
As the definitions are local, we are free to check that $\mathcal H^i(f_+\mathcal M)$ is coherent (resp.~regular, resp.~holonomic)  after an \'etale base change.  Then \Cref{C:Canon-Iso-Hi} reduces the question to morphisms of varieties, where the result is standard (e.g., \cite[Thm.~2.5.1, Thm.~3.2.3, Thm.~6.1.5]{HTT08}).
\end{proof}

\section{Push forward of filtered $D$-modules}\label[appendix]{S:FiltDModPush}

We  defined the notion of a filtered $D$-module $(\mathcal M,F_\bullet)$ in the obvious way in  \cite[Def.~3.7]{CMZpositivity}, as well as the notion of a good filtration 
\cite[\S 3.1.5]{CMZpositivity}. 
The definition of coherent $D_{\mathcal X}$-modules we made via a local condition (see \Cref{S:Regular/holo-are-etale-ah})  then leads to a natural question, namely,  whether a $D_{\mathcal X}$-module being coherent is equivalent to the existence of a good filtration (not just the existence of good filtrations locally). The following proposition establishes this:

\begin{pro}\label{Existence-of-good-filtr-on-stack}
A quasi-coherent $D_{\mathcal X}$-module $\mathcal M$ is coherent if and only if there exists a good filtration $F_\bullet \mathcal M$ on $\mathcal M$.
  
\end{pro}

\begin{proof} 
Let $\mathcal{M}$ be a coherent $D_{\mathcal{X}}$-module on  $\mathcal{X}$.
    By \cite[Prop. 15.4]{LMB}, as $\mathcal M$ is by definition a  quasi-coherent $\mathcal{O}_{\mathcal{X}}$-module on $\mathcal{X}$, it is the colimit of its coherent $\mathcal O_{\mathcal X}$-subsheaves $\mathcal{N}_{\lambda}$. Consider $\mathcal{M}_{\lambda}:=\operatorname{Im}(D_{\mathcal{X}}\otimes_{\mathcal O_{\mathcal X}} \mathcal{N}_{\lambda} \to \mathcal{M})$. 
    Although it is not necessary, we note that each $\mathcal M_\lambda$ is evidently quasi-coherent as an $\mathcal O_{\mathcal X}$-module, and finitely generated as a $D_{\mathcal X}$-module, and so is a coherent $D_{\mathcal X}$-module (e.g., \cite[Prop.~1.4.9(ii)]{HTT08}). 
    The collection of  $D_{\mathcal X}$-submodules  $\mathcal M_\lambda \subseteq \mathcal{M}$ satisfies the ascending chain condition (ACC); indeed,  we can pull back to an \'etale cover $p:U\to \mathcal{X}$, with $U$ a finite union of affine varieties,  and there ACC holds since we are assuming that $\mathcal M|_U$ is a coherent $D_U$-module (e.g., \cite[Prop. 1.4.6]{HTT08}).
       So, there exists $\mathcal{N}_{\lambda}$, a  coherent $\mathcal{O}_{\mathcal{X}}$-module, for which $\mathcal{M}_{\lambda}=\mathcal{M}$. Then we obtain a filtration $F_{\bullet}$ on $\mathcal{M}$ by $F_k\mathcal{M}:=\operatorname{Im}(F_k D_{\mathcal{X}}\otimes_{\mathcal O_{\mathcal X}} \mathcal{N}_{\lambda}\to \mathcal{M}).$ It is an exhaustive filtration by coherent $\mathcal{O}_{\mathcal{X}}$-modules, it satisfies $F_k\mathcal{M}=0$ for $k<0$, and $F_iD_{\mathcal{X}}\cdot F_k\mathcal{M}=F_{k+i}\mathcal{M}$, for all $i\geq 0, k\geq 0$, and so we obtain a good filtration on $\mathcal M$.
\end{proof}

It follows from \Cref{Existence-of-good-filtr-on-stack} and \Cref{P:CRHpush} that the $i$-th push forward of a coherent $D$-module admits a good filtration; but at this stage, we do not have much control over this filtration. To say more, we use the Rees algebra approach; the terminology was reviewed in \S \ref{S:Tilde-notation-from-MHM-project}.

To begin, for  a filtered $D_{\mathcal X}$-module $(\mathcal M,F_\bullet)$,   we can define a filtration $F_\bullet$ on $\mathbf \Omega (\mathcal M)$ using the filtration in \S \ref{S:dR+SCFilt}, namely, by setting 
$$
F_k((\mathbf \Omega(\mathcal M))^i):=F_{k+i}\mathcal M \otimes _{\mathcal O_{\mathcal X}}\bigwedge ^{-i}  \mathcal T_{\mathcal X}.
$$
Then viewing $\mathbf D\mathbf \Omega \mathcal M$ as a tensor power of filtered modules via \eqref{E:Alt-D}, we obtain a filtration $F_\bullet$ on $\mathbf D\mathbf \Omega \mathcal M$ given by
$$
F_k((\mathbf D\mathbf \Omega \mathcal M)^i):=\sum_{j+\ell=k} F_{j+i}\mathcal M \otimes _{\mathcal O_{\mathcal X}}\bigwedge ^{-i}  \mathcal T_{\mathcal X}\otimes_{\mathcal O_{\mathcal X}} F_\ell D_{\mathcal X}.
$$
Finally,  from the definition of the relative Spencer complex in \eqref{E:DDOmegaSp} and the alternate formulation in 
\eqref{E:AltSpencf-1}, one obtains a filtration $F_\bullet$ on $\operatorname{Sp}_{\mathcal X\to \mathcal Y}^\bullet (\mathcal M)$ as  in \S \ref{S:dR+SCFilt} 
given by 
$$
F_k((\operatorname{Sp}_{\mathcal X\to \mathcal Y}^\bullet (\mathcal M))^i):=\sum_{j+\ell=k} F_{j+i}\mathcal M \otimes _{\mathcal O_{\mathcal X}}\bigwedge ^{-i}  \mathcal T_{\mathcal X}\otimes_{f^{-1}\mathcal O_{\mathcal Y}} f^{-1}F_\ell D_{\mathcal Y}.
$$
All of the above filtrations are filtrations by $\mathcal O_{\mathcal X}$-submodules. 

Given the filtrations, we can take the associated Rees modules for the filtered complexes above (see  \S \ref{S:Tilde-notation-from-MHM-project} for a review of the notation) and from the descriptions of the filtrations above, we obtain an identification 
\begin{equation}
\widetilde{\operatorname{Sp}^\bullet_{\mathcal X\to\mathcal Y}(M)}=\widetilde{\mathbf{D}\mathbf{\Omega}(M)}\otimes_{\widetilde{D_{\X}}}\widetilde{f^* D_{\Y}} 
\end{equation}
as chain complexes of $\widetilde{f^{-1}D_{\mathcal{Y}}}$-modules.  
Then, applying \Cref{lem:graded-functorial-inj} to $\mathcal{A}=f^{-1}\widetilde{D_{\mathcal Y}}$,  $\mathcal{B}=\operatorname{Mod}(\widetilde{D_{\mathcal Y}})^{\operatorname{right}}$, and $F=f_*$,  yields a derived push forward 
$$Rf_*:D^b(f^{-1}\widetilde{D_{\Y}})\rightarrow D^b(\widetilde{D_{\Y}}).$$
Viewing $\widetilde {\operatorname{Sp}_{X\to Y}^\bullet \mathcal M }$ as an object of $D^b(\widetilde {f^{-1}D_{\mathcal Y}})$, we obtain a push forward  
$$
Rf_* \widetilde{\operatorname{Sp}^\bullet_{\mathcal X\to\mathcal Y}(\mathcal M)}
$$
in $D^b(\widetilde{D_{\mathcal{Y}}})$.
  By taking cohomology, we obtain a $\widetilde {D_{\mathcal Y}}$-module 
$$
R^if_* \widetilde{\operatorname{Sp}^\bullet_{\mathcal X\to\mathcal Y}(\mathcal M)}.
$$

\begin{pro}\label{FilteredDmodPush}
Let $f:\mathcal X\to \mathcal Y$ be  schematic proper morphism  of smooth separated integral DM stacks of finite type over $\mathbb C$ (or the analytification of such a morphism), and  let  $(\mathcal M,F_\bullet)$ be a filtered $D_{\mathcal X}$-module. 
For all $i$ and $p$  we have that 
$$(R^if_* \widetilde{\operatorname{Sp}^\bullet _{\mathcal X\to\mathcal Y}(\mathcal M)})_p=R^if_* F_p \operatorname{Sp}^\bullet_{\mathcal X\to\mathcal Y}(\mathcal M),$$
where on the left $Rf_*$ is the derived functor $Rf_*:D^b(f^{-1}\widetilde{D_{\mathcal{Y}}})\rightarrow D^b(\widetilde{D_{\mathcal{Y}}})$ and on the right it is the derived functor   $Rf_*: D^b(f^{-1}\mathcal{O}_{\mathcal{Y}}) \rightarrow D^b(\mathcal{O}_{\mathcal{Y}})$. Moreover, the $D_{\mathcal Y}$-module $R^if_*\operatorname{Sp}_{\mathcal X\to \mathcal Y}^\bullet (\mathcal M)$ is canonically isomorphic to the $D_{\mathcal Y}$-module obtained by applying the derived functor 
$R^if_*:D^b(f^{-1}\widetilde{D_{\mathcal{Y}}})\rightarrow D^b(\widetilde{D_{\mathcal{Y}}})$ to $\operatorname{Sp}^\bullet_{\mathcal X\to \mathcal Y}(\mathcal M)$; i.e., we are free to apply either derived functor.   
\end{pro}

\begin{proof}
For any abelian category $\mathcal{A}$, we denote by $\operatorname{Ch}(\mathcal A)$ the chain complexes of objects in $\mathcal A$. We work with $\operatorname{gr-Mod}(f^{-1}\widetilde{\mathcal D}_{\mathcal Y})$-modules on $\X$ (equivalently, one can work in $\operatorname{gr-Mod}(f^{-1}\widetilde{\mathcal O}_{\mathcal Y})$-modules; the argument is the same).
For an object $N=\bigoplus_{q\in\mathbb Z}N_q\in \operatorname{gr-Mod}(f^{-1}\widetilde{\mathcal D}_{\mathcal Y})$ we write $N_p$ for its degree-$p$ component. For a complex $K^\bullet$ in $\operatorname{Ch}(\operatorname{gr-Mod}(f^{-1}\widetilde{\mathcal D}_{\mathcal Y}))$, we write $(K^\bullet)_p$ for the complex obtained by taking degree-$p$ components termwise.

For any graded sheaf $N$ on $\X$ and any open $ U\rightarrow \Y$ we have
\[
(f_*N)_p(U)
=\bigl(f_*N(U)\bigr)_p
= N(f^{-1}U)_p
= N_p(f^{-1}U)
=\bigl(f_*N_p\bigr)(U).
\]
Thus there is a canonical identification of functors on $\operatorname{gr-Mod}(f^{-1}\widetilde{\mathcal D}_{\mathcal Y})$:
\begin{equation}\label{eq:commute-abelian}
(f_*N)_p \cong f_*(N_p).
\end{equation}

The functor
\[
(-)_p:\operatorname{gr-Mod}(f^{-1}\widetilde{\mathcal D}_{\mathcal Y})\to \operatorname{Mod}(f^{-1}\mathcal{O}_{\Y}),\qquad N=\bigoplus_q N_q \mapsto N_p
\]
is exact. Also, $(-)_p$ is right adjoint to the exact functor (see \cite[Prop.~1.1.3]{HTT08})
\[
i_p:\operatorname{Mod}(f^{-1}\mathcal{O}_{\mathcal Y})\to \operatorname{gr-Mod}(f^{-1}\widetilde{ D}_{\mathcal Y}),\qquad A\mapsto \bigr(f^{-1}\widetilde{D}_{\Y}\otimes_{f^{-1}\widetilde{\mathcal{O}}_{\mathcal{ Y}}}A \bigr)(p),
\]
where $N^{\bullet}(p)$ is the module whose grading is shifted by $p$. To see the adjunction, observe,
\[
\operatorname{Hom}_{\operatorname{gr-Mod}(f^{-1}\widetilde{\mathcal D}_{\mathcal Y})}(i_pA,N)\cong \operatorname{Hom}_{\operatorname{Mod}(f^{-1}\mathcal{O}_{\mathcal Y})}(A,N_p),
\]
because the action of positively graded pieces of the filtration of $f^{-1}\widetilde{\mathcal D}_{\mathcal Y}$ send $i_pA$ to $0$ in $N$ and thus the data of $i_p A$ is precisely captured by the $f^{-1}\mathcal{O}_{\mathcal Y}$ action on $A$. Similarly for the maps.

Since $i_p$ is exact, its right adjoint, $(-)_p$, preserves injectives; hence if $I$ is injective as an object in $\operatorname{gr-Mod}(f^{-1}\widetilde{\mathcal D}_{\mathcal Y})$,
then $I_p$ is injective as a (non-graded) sheaf.

Let $K^\bullet$ be a (bounded below) complex in $\operatorname{gr-Mod}(f^{-1}\widetilde{\mathcal D}_{\mathcal Y})$.
Choose an injective resolution $K^\bullet\to I^\bullet$ in $\operatorname{Ch}(\operatorname{gr-Mod}(f^{-1}\widetilde{\mathcal D}_{\mathcal Y}))$.
Then
\[
Rf_*K^\bullet \simeq f_*I^\bullet.
\]
Applying degree-$p$ and using \eqref{eq:commute-abelian} termwise gives
\[
(Rf_*K^\bullet)_p \;\simeq\; (f_*I^\bullet)_p \;\cong\; f_*(I^\bullet_p).
\]
By the above discussion, $I^\bullet_p$ is a complex of injective sheaves, and since $(-)_p$ is exact,
the map $(K^\bullet)_p\to (I^\bullet)_p$ is a quasi-isomorphism. Hence $I^\bullet_p$ computes $Rf_*(K^\bullet_p)$, i.e.  $
Rf_*(K^\bullet_p)\simeq f_*(I^\bullet_p).$
Thus,
\begin{equation}\label{eq:degree-commutes-Rf}
(Rf_*K^\bullet)_p \;\simeq\; Rf_*(K^\bullet_p).
\end{equation}
Applying \eqref{eq:degree-commutes-Rf} with $K^\bullet=\operatorname{Sp}_{\X\to \Y}^\bullet (\widetilde {\mathcal{M}})$ 
gives the claim.
\end{proof}

The inclusions $F_p\operatorname{Sp}_{\mathcal X\to \mathcal Y}^\bullet (\mathcal M)\hookrightarrow \operatorname{Sp}_{\mathcal X\to \mathcal Y}^\bullet (\mathcal M)$ induce $\mathcal O_{\mathcal X}$-module morphisms  
\begin{equation}\label{E:RiFptoRi}
R^if_*F_p\operatorname{Sp}_{\mathcal X\to \mathcal Y}^\bullet (\mathcal M)\rightarrow R^if_*\operatorname{Sp}_{\mathcal X\to \mathcal Y}^\bullet (\mathcal M),
\end{equation}
which in turn induces a canonical morphism
$
\varinjlim_{p}R^if_*F_p\operatorname{Sp}_{\mathcal X\to \mathcal Y}^\bullet (\mathcal M)\rightarrow R^if_*\operatorname{Sp}_{\mathcal X\to \mathcal Y}^\bullet (\mathcal M)
$,
where $R^if_*$ is the derived functor 
$R^if_*: D^b(f^{-1}\mathcal{O}_{\mathcal{Y}}) \rightarrow D^b(\mathcal{O}_{\mathcal{Y}})$.

\begin{lem}\label{L:limRif=Rif} In the notation above, 
for a filtered $D_{\mathcal X}$-module $(\mathcal M,F_\bullet)$, the canonical morphism $$\varinjlim_{p}R^if_*F_p\operatorname{Sp}_{\mathcal X\to \mathcal Y}^\bullet (\mathcal M)\rightarrow R^if_*\operatorname{Sp}_{\mathcal X\to \mathcal Y}^\bullet (\mathcal M)$$ is an isomorphism, where $R^if_*$ is the derived functor 
$R^if_*: D^b(f^{-1}\mathcal{O}_{\mathcal{Y}}) \rightarrow D^b(\mathcal{O}_{\mathcal{Y}})$.  

\end{lem}

\begin{proof}
This essentially comes down to \cite[\href{https://stacks.math.columbia.edu/tag/07TA}{\S 07TA}]{stacks-project}, that colimits commute with higher direct images under mild hypotheses.  
We will be showing that
\[
\varinjlim_p R^i f_*(K_p)\xrightarrow{\sim} R^i f_*(K),\]
and applying it to
\[
K_p:=F_p\operatorname{Sp}^\bullet_{\mathcal X \to \mathcal Y}(M), \qquad
K:=\operatorname{Sp}^\bullet_{\mathcal X \to \mathcal Y}(M)=\varinjlim_p K_p .
\]

Choose integers $a \le b$ such that each $K_p$ is concentrated in degrees
$[a,b]$. We prove the claim by induction on $b-a$.

If $b=a$, then $K_p=\mathcal F_p[-a]$ with $\mathcal F_p:=K_p^a$ quasi-coherent, and
$K=(\varinjlim_p \mathcal F_p)[-a]$. Hence
\[
R^i f_*(K_p)=R^{i-a}f_*(\mathcal F_p), \qquad
R^i f_*(K)=R^{i-a}f_*\!\left(\varinjlim_p \mathcal F_p\right).
\]
Crucially, by \cite[\href{https://stacks.math.columbia.edu/tag/07TA}{\S 07TA}]{stacks-project},
\[
\varinjlim_p R^i f_*(K_p)
=
\varinjlim_p R^{i-a}f_*(\mathcal F_p)
\cong
R^{i-a}f_*\!\left(\varinjlim_p \mathcal F_p\right)
=
R^i f_*(K).
\]

Assume now $b>a$. For each $k$, let
\[
L_p:=\sigma_{\ge a+1}K_p, \qquad A_p:=K_p^a[-a],
\]
where $\sigma_{\ge a+1}$ denotes the stupid truncation. Then there is a short exact
sequence of complexes
\[
0 \longrightarrow L_p \longrightarrow K_p \longrightarrow A_p \longrightarrow 0 .
\]
Since filtered colimits are exact in $\operatorname{Mod}(\mathcal O_{\mathcal{X}})$
\cite[\href{https://stacks.math.columbia.edu/tag/01AH}{Lem.~01AH}]{stacks-project}, taking filtered colimits we get
\[
0 \longrightarrow L \longrightarrow K \longrightarrow A \longrightarrow 0,
\]
where
\[
L=\sigma_{\ge a+1}K, \qquad A=K^a[-a]
      =\left(\varinjlim_p K_p^a\right)[-a].
\]

Hence, we have long exact sequences of cohomology sheaves:
\[
\cdots \to R^i f_*(L_p)\to R^i f_*(K_p)\to R^i f_*(A_p)\to R^{i+1}f_*(L_p)\to \cdots
\]
and
\[
\cdots \to R^i f_*(L)\to R^i f_*(K)\to R^i f_*(A)\to R^{i+1}f_*(L)\to \cdots.
\]

As filtered colimits are exact in $\operatorname{Mod}(\mathcal O_Y)$, taking $\varinjlim_p$ of the
first long exact sequence yields a long exact sequence
\[
\cdots \to \varinjlim_p R^i f_*(L_p)
\to \varinjlim_p R^i f_*(K_p)
\to \varinjlim_p R^i f_*(A_p)
\to \varinjlim_p R^{i+1}f_*(L_p)\to \cdots.
\]

Now $L_p$ is concentrated in degrees $[a+1,b]$, so by the induction hypothesis, the canonical maps are isomorphisms
\[
\varinjlim_p R^i f_*(L_p)\xrightarrow{\sim} R^i f_*(L),
\qquad
\varinjlim_p R^{i+1} f_*(L_p)\xrightarrow{\sim} R^{i+1} f_*(L).
\]
Also $A_p$ is concentrated in a single degree, so by the base case
\[
\varinjlim_p R^i f_*(A_p)\xrightarrow{\sim} R^i f_*(A).
\]
Comparing the two long exact sequences, the five lemma gives
\[
\varinjlim_p R^i f_*(K_p)\xrightarrow{\sim} R^i f_*(K).
\]
Hence, we are done.
\end{proof}

We now define a filtration on the push forward. To begin,  define a filtration $F_\bullet$ on $\mathcal H^i(f_+\mathcal M)=R^if_*\operatorname{Sp}_{\mathcal X\to \mathcal Y}(\mathcal M)$ via the images of the morphisms  \eqref{E:RiFptoRi}, i.e., 
\begin{equation}\label{E:FiltOnHif+M}
F_k(R^if_*\operatorname{Sp}_{\mathcal X\to \mathcal Y}(\mathcal M)):=\operatorname{Im}\left(R^if_*F_k\operatorname{Sp}_{\mathcal X\to \mathcal Y}^\bullet (\mathcal M)\rightarrow R^if_*\operatorname{Sp}_{\mathcal X\to \mathcal Y}^\bullet (\mathcal M)\right),
\end{equation}
where $R^if_*$ is the derived functor 
$R^if_*: D^b(f^{-1}\mathcal{O}_{\mathcal{Y}}) \rightarrow D^b(\mathcal{O}_{\mathcal{Y}})$. 
The content of \Cref{L:limRif=Rif} is that 
 the filtration in \eqref{E:FiltOnHif+M} is the filtration induced on $\varinjlim_{p}R^if_*F_p\operatorname{Sp}_{\mathcal X\to \mathcal Y}^\bullet (\mathcal M)=R^if_*\operatorname{Sp}_{\mathcal X\to \mathcal Y}^\bullet (\mathcal M)$ via the Rees construction applied to $R^if_* \widetilde{\operatorname{Sp}^\bullet_{\mathcal X\to\mathcal Y}(\mathcal M)}$.  Finally, we recall \Cref{R:Hodge-Strict-Saito}, discussing when the morphisms \eqref{E:RiFptoRi} are known to be injective; i.e., when the $F_k\mathcal H^i(\mathcal M,F_\bullet)$ are equal to $R^if_*F_k\operatorname{Sp}_{\mathcal X\to \mathcal Y}^\bullet (\mathcal M)$. 
 
In summary, we have defined an $i$-th push forward of filtered $D$-modules.  More precisely, for a smooth separated integral DM stack $\mathcal X$ of finite type over $\mathbb C$ (or the analytification of such a stack) let $MF(D_{\mathcal X})^{\operatorname{right}}$ denote the category of pairs $(\mathcal M,F_\bullet)$ such that $\mathcal M$ is a $D_{\mathcal X}$-module and $F_\bullet$ is a filtration on $\mathcal M$.
Then if $f:\mathcal X\to \mathcal Y$ is  schematic proper morphism  of smooth separated integral DM stacks of finite type over $\mathbb C$ (or the analytification of such a morphism), then for each $i$ we have a functor
\begin{equation*}
\mathcal H^if_+: MF(D_{\mathcal X})^{\operatorname{right}}\longrightarrow MF(D_{\mathcal Y})^{\operatorname{right}}
\end{equation*}
\begin{equation}\label{E:PushFiltD}
\mathcal H^if_+(\mathcal M,F_\bullet):=(\mathcal H^if_+\mathcal M,F_\bullet)
\end{equation}   
where $\mathcal H^if_+\mathcal M$ is defined in \eqref{E:def-Hif+M} and the filtration on $\mathcal H^if_+\mathcal M$ is defined in \eqref{E:FiltOnHif+M}.

\begin{pro}\label{L:Canon-Iso+SpHi}
Let $f:\mathcal X\to \mathcal Y$ be a  schematic proper morphism  of smooth separated integral DM stacks of finite type over $\mathbb C$ (or the analytification of such a morphism), and  let $(\mathcal M,F_\bullet)$  be a filtered $D_{\mathcal X}$-module. Let $p:U\to \mathcal Y$ be an \'etale morphism from a smooth variety (or complex analytic space).
In the notation of diagram \eqref{E:TubPresDpf'}, 
there are canonical isomorphisms 
\begin{equation}\label{E:Canon-Iso+Sp-1Hi}
\alpha_\phi^i:p^*\mathcal H^if_+(\mathcal M,F_\bullet) \stackrel{\sim}{\longrightarrow}\mathcal H^if'_+(p_{\mathcal X}^*(\mathcal M,F_\bullet))
\end{equation}

\begin{equation}\label{E:Canon-Iso+Sp-3Hi}
p^*F_k\mathcal H^if_+(\mathcal M,F_\bullet)\stackrel{\sim}{\longrightarrow} \operatorname{Im}\left(R^if'_*F_k \operatorname{Sp}_{U_{\mathcal X}\to U}p_{\mathcal X}^*\mathcal M\to  R^if'_*\operatorname{Sp}_{U_{\mathcal X}\to U}p_{\mathcal X}^*\mathcal M \right) .
\end{equation}

Moreover, if $F_\bullet \mathcal M$ is a good filtration, then so is the filtration $F_\bullet \mathcal H^if_+\mathcal M$.  
\end{pro}

\begin{proof}
The first isomorphism comes from the fact that in the definition of the filtration \eqref{E:FiltOnHif+M}, $R^if_*$ is the derived functor 
$R^if_*: D^b(f^{-1}\mathcal{O}_{\mathcal{Y}}) \rightarrow D^b(\mathcal{O}_{\mathcal{Y}})$, and this functor satisfies \'etale base change.  The second isomorphism is established using the same property, along with the fact that the relative Spencer complex is local \eqref{E:Canon-Iso+Sp-2}.  The statement about good filtrations can be checked locally, and therefore comes from \eqref{E:Canon-Iso+Sp-1Hi}, and the case for varieties, which is well known (see \cite[2.3.5]{Saito}). 
\end{proof}

It follows from \Cref{L:Canon-Iso+SpHi} and \Cref{P:CRHpush} 
that for $\sharp \in \{c,h,rh\}$, the functor  $\mathcal H^if_+$ \eqref{E:PushFiltD} induces a functor 
\begin{equation}\label{E:Hif+rh}
\mathcal H^if_+:MF_\sharp(D_{\mathcal X})^{\operatorname{right}}\longrightarrow 
MF_\sharp(D_{\mathcal Y})^{\operatorname{right}}. 
\end{equation}

\begin{rem}[Filtered \emph{Ad hoc} push forward] \label{R:FiltahPF}
We make the following observation following \Cref{L:Canon-Iso+SpHi}.
Given  $(\mathcal M,F_\bullet)$  in $MF(D_{\mathcal X})^{\operatorname{right}}$ viewed as an object $\{ (\mathcal M_U,F_U),\theta_\phi\}$ in $MF(D_{(-)})^{\operatorname{right}}$ with descent data on $\mathcal X$ (\S \ref{S:ConcreteDesc}), then, in the notation of diagram  \eqref{E:TubPresDpf'},  \eqref{E:Canon-Iso+Sp-1Hi}, and \eqref{E:Canon-Iso+Sp-3Hi}, the data
$$
\{\mathcal H^if'_+(\mathcal M_{U_{\mathcal X}},F_{U_{\mathcal X}}),\tau_\phi:=\alpha^i_\phi \circ \mathcal H^i(f''_+(\theta_{\phi_{\mathcal X}}))\}
$$
defines an object in $MF(D_{(-)})^{\operatorname{right}}$ with descent data on $\mathcal Y$, i.e., an object of $MF(D_{\mathcal Y})^{\operatorname{right}}$. From 
 \Cref{L:Canon-Iso+SpHi} one obtains  that this is the object with descent data associated with  $\mathcal H^i(f_+(\mathcal M,F_\bullet))$.  
\end{rem}

\section{\emph{Ad hoc} push forward of Hodge modules}\label{S:PushHodge}
For a smooth separated integral DM stack $\mathcal X$ of finite type over $\mathbb C$,  let  $$MF_{rh}(D_{\mathcal X}, \mathbb Q)$$ be the category of regular holonomic $D_{\mathcal X}$-modules with $\mathbb Q$-structure.  More precisely, this is the category of triples $(\mathcal M,F_\bullet,K)$ such that $\mathcal M$ is  a regular holonomic $D_{\mathcal X}$-module, $F_\bullet$ is a good filtration on $\mathcal M$, and $K$ in $\operatorname{Perv}(\mathbb Q_{\mathcal X})$ is such that $^pDR(\mathcal M^{an})= K\otimes_{\mathbb Q_{\mathcal X}}\mathbb C_{\mathcal X}$.
Using the identification $\mathcal H^if_+=\mathcal H^if^{ah}_+$ of \Cref{C:Hif+=Hifah+}, then from \Cref{R:RHF-push}, we conclude that the functor \eqref{E:Hif+rh} induces a functor
$$
\mathcal H^if^{ah}_+:MF_{rh}(D_{\mathcal X},\mathbb Q)\longrightarrow 
MF_{rh}(D_{\mathcal Y},\mathbb Q)
$$
\begin{equation}\label{E:Hif+rhQQ}
(\mathcal M,F_\bullet,K)\mapsto (\mathcal H^if_+(\mathcal M,F_\bullet),R^if^{ah}_*K).
\end{equation}

\begin{pro}[\emph{Ad hoc} push forward for Hodge modules] \label{P:Hif+Hodge}
Let $f:\mathcal X\to \mathcal Y$ be a  schematic proper morphism  of smooth separated integral DM stacks of finite type over $\mathbb C$.  Then 
the \emph{ad hoc} push forward functor $\mathcal H^if^{ah}_+$ of \eqref{E:Hif+rhQQ} on filtered regular holonomic $D_{\mathcal X}$-modules with $\mathbb Q$-structure induces an \emph{ad hoc} push forward functor on pure Hodge modules of weight $w$:
$$
\mathcal H^if^{ah}_+:\mathsf {HM}(\mathcal X,w) \longrightarrow 
\mathsf {HM}(\mathcal Y,w+i)
$$
\begin{equation}\label{E:Hif+Hodge}
(\mathcal M,F_\bullet,K)\mapsto (\mathcal H^if_+(\mathcal M,F_\bullet),R^if^{ah}_*K).
\end{equation}
\end{pro}

\begin{proof}
If $(\mathcal M,F_\bullet,K)$ is a pure Hodge module on $\mathcal X$ of weight $w$, the only thing to check is that the regular holonomic $D_{\mathcal Y}$-module  with $\mathbb Q$-structure $(\mathcal H^if_+(\mathcal M,F_\bullet),R^if^{ah}_*K)$ is a pure Hodge module of weight $w+i$.  By definition, this is a local question, and so using \eqref{E:Canon-Iso+Sp-1Hi} and the definition of the \emph{ad hoc} push forward of perverse sheaves (\Cref{C:PervDescA}), this reduces to the case of varieties, which is well known.
\end{proof}

\begin{rem}[\emph{Ad hoc} push forward construction for Hodge modules] \label{R:FiltahPFHodge}
We make the following observation following \Cref{P:Hif+Hodge}.
Given $\mathsf M$  in $\mathsf {HM}(\mathcal X,w)$,
then, in the notation of diagram  \eqref{E:TubPresDpf'}, tacitly contained in \Cref{P:Hif+Hodge} is the statement that 
 there are canonical isomorphisms 
\begin{equation}\label{E:Canon-Iso+Sp-1HiHodge}
\alpha_\phi^i:p^*\mathcal H^if^{ah}_+\mathsf M \stackrel{\sim}{\longrightarrow}\mathcal H^if'_+(p_{\mathcal X}^*\mathsf M),
\end{equation}
obtained via the canonical isomorphisms of \eqref{E:Canon-Iso+Sp-1Hi} and the definition of the \emph{ad hoc} push forward of perverse sheaves (\Cref{C:PervDescA}).
Moreover, 
given  $\mathsf M$  in $\mathsf {HM}(\mathcal X,w)$ viewed as an object $\{ \mathsf M_U,\theta_\phi\}$ in $\mathsf {HM}((-),w)$ with descent data on $\mathcal X$ (\S \ref{S:ConcreteDesc}), then, in the notation of diagram  \eqref{E:TubPresDpf'} and \eqref{E:Canon-Iso+Sp-1HiHodge}, the data
$$
\{\mathcal H^if'_+\mathsf M_{U_{\mathcal X}},\tau_\phi:=\alpha^i_\phi \circ \mathcal H^i(f''_+(\theta_{\phi_{\mathcal X}}))\}
$$
defines an object in $\mathsf {HM}((-),w+i)$ with descent data on $\mathcal Y$, i.e., an object of $\mathsf {HM}(\mathcal Y,w+i)$. From 
 \Cref{P:Hif+Hodge} one obtains  that this is the object with descent data associated with  $\mathcal H^if^{ah}_+\mathsf M$.  
\end{rem}

\fi 

 \bibliographystyle{amsalpha}
 \bibliography{mhm_bib}

@article{Rybakov2015DGModulesDeRham,
  author       = {Sergey Rybakov},
  title        = {DG-modules over de Rham DG-algebra},
  journal      = {European Journal of Mathematics},
  year         = {2015},
  volume       = {1},
  pages        = {25--53},
  doi          = {10.1007/s40879-014-0014-4}
}

@misc{CMZpositivity,
      title={Positivity in the context of {H}odge modules and {H}iggs bundles on {D}eligne--{M}umford stacks}, 
      author={Sebastian Casalaina-Martin and Shend Zhjeqi},
      year={2026},
      eprint={2605.22989},
      archivePrefix={arXiv},
      primaryClass={math.AG},
      url={https://arxiv.org/abs/2605.22989}, 
      note={preprint, \url{https://arxiv.org/abs/2605.22989}}, 
}

@misc{CMZslope_stability,
      title={Movable curve classes and slope stability on {D}eligne--{M}umford stacks}, 
      author={Sebastian Casalaina-Martin and Shend Zhjeqi},
      year={2026},
      eprint={2605.26101},
      archivePrefix={arXiv},
      primaryClass={math.AG},
      url={https://arxiv.org/abs/2605.26101},
      note={preprint, \url{https://arxiv.org/abs/2605.26101}} 
}

@misc{CMZfoliations,
      title={Foliations, slope stability, and positivity of log canonical bundles on {D}eligne--{M}umford stacks}, 
      author={Sebastian Casalaina-Martin and Shend Zhjeqi},
      year={2026},
      eprint={2605.26443},
      archivePrefix={arXiv},
      primaryClass={math.AG},
      note={preprint, \url{https://arxiv.org/abs/2605.26443}}, 
}

@article {PS17,
    AUTHOR = {Popa, Mihnea and Schnell, Christian},
     TITLE = {Viehweg's hyperbolicity conjecture for families with maximal
              variation},
   JOURNAL = {Invent. Math.},
  FJOURNAL = {Inventiones Mathematicae},
    VOLUME = {208},
      YEAR = {2017},
    NUMBER = {3},
     PAGES = {677--713},
      ISSN = {0020-9910,1432-1297},
   MRCLASS = {14D06 (14D07 14E30 14F10)},
  MRNUMBER = {3648973},
MRREVIEWER = {Andreas\ H\"oring},
       DOI = {10.1007/s00222-016-0698-9},
       URL = {https://doi.org/10.1007/s00222-016-0698-9},
}

@book {FGAE,
    AUTHOR = {Fantechi, Barbara and G\"ottsche, Lothar and Illusie, Luc and
              Kleiman, Steven L. and Nitsure, Nitin and Vistoli, Angelo},
     TITLE = {Fundamental algebraic geometry},
    SERIES = {Mathematical Surveys and Monographs},
    VOLUME = {123},
      NOTE = {Grothendieck's FGA explained},
 PUBLISHER = {American Mathematical Society, Providence, RI},
      YEAR = {2005},
     PAGES = {x+339},
      ISBN = {0-8218-3541-6},
   MRCLASS = {14-06 (14A15 14D15 14F20)},
  MRNUMBER = {2222646},
MRREVIEWER = {Liam\ O'Carroll},
       DOI = {10.1090/surv/123},
       URL = {https://doi.org/10.1090/surv/123},
}

@article {WW23,
    AUTHOR = {Wei, Chuanhao and Wu, Lei},
     TITLE = {Hyperbolicity for log smooth families with maximal variation},
   JOURNAL = {Int. Math. Res. Not. IMRN},
  FJOURNAL = {International Mathematics Research Notices. IMRN},
      YEAR = {2023},
    NUMBER = {1},
     PAGES = {708--743},
      ISSN = {1073-7928,1687-0247},
   MRCLASS = {14D06 (14A30 14E30 32Q45)},
  MRNUMBER = {4530119},
MRREVIEWER = {Guolei\ Zhong},
       DOI = {10.1093/imrn/rnab280},
       URL = {https://doi.org/10.1093/imrn/rnab280},
}

@book {HTT08,
    AUTHOR = {Hotta, Ryoshi and Takeuchi, Kiyoshi and Tanisaki, Toshiyuki},
     TITLE = {{$D$}-modules, perverse sheaves, and representation theory},
    SERIES = {Progress in Mathematics},
    VOLUME = {236},
   EDITION = {Japanese},
 PUBLISHER = {Birkh\"auser Boston, Inc., Boston, MA},
      YEAR = {2008},
     PAGES = {xii+407},
      ISBN = {978-0-8176-4363-8},
   MRCLASS = {32C38 (14F05 14F10 17B10)},
  MRNUMBER = {2357361},
MRREVIEWER = {Corrado\ Marastoni},
       DOI = {10.1007/978-0-8176-4523-6},
       URL = {https://doi.org/10.1007/978-0-8176-4523-6},
}

@unpublished{BDstacks,
AUTHOR = {A. Beilinson and V. Drinfeld},
TITLE = {Quantization of {H}itchin's integrable system and {H}ecke
eigensheaves},
NOTE = {\href{https://math.uchicago.edu/~drinfeld/langlands/QuantizationHitchin.pdf}{preprint}},
URL = {https://math.uchicago.edu/~drinfeld/langlands/QuantizationHitchin.pdf}
}

@article {KV04,
    AUTHOR = {Kresch, Andrew and Vistoli, Angelo},
     TITLE = {On coverings of {D}eligne-{M}umford stacks and surjectivity of
              the {B}rauer map},
   JOURNAL = {Bull. London Math. Soc.},
  FJOURNAL = {The Bulletin of the London Mathematical Society},
    VOLUME = {36},
      YEAR = {2004},
    NUMBER = {2},
     PAGES = {188--192},
      ISSN = {0024-6093,1469-2120},
   MRCLASS = {14A20 (14D20 14F22)},
  MRNUMBER = {2026412},
MRREVIEWER = {Martin\ M\"oller},
       DOI = {10.1112/S0024609303002728},
       URL = {https://doi.org/10.1112/S0024609303002728},
}

@incollection {kresch09,
    AUTHOR = {Kresch, Andrew},
     TITLE = {On the geometry of {D}eligne-{M}umford stacks},
 BOOKTITLE = {Algebraic geometry---{S}eattle 2005. {P}art 1},
    SERIES = {Proc. Sympos. Pure Math.},
    VOLUME = {80, Part 1},
     PAGES = {259--271},
 PUBLISHER = {Amer. Math. Soc., Providence, RI},
      YEAR = {2009},
      ISBN = {978-0-8218-4702-2},
   MRCLASS = {14A20 (14D23 14L24)},
  MRNUMBER = {2483938},
MRREVIEWER = {Fabio\ Perroni},
       DOI = {10.1090/pspum/080.1/2483938},
       URL = {https://doi.org/10.1090/pspum/080.1/2483938},
}

@misc{stacks-project,
  author       = {The {Stacks project authors}},
  title        = {The stacks project},
  howpublished = {\url{https://stacks.math.columbia.edu}},
  year         = {2026},
}

@misc{kerodon,
  author       = {Jacob Lurie},
  title        = {Kerodon},
  howpublished = {\url{https://kerodon.net}},
  year         = {2026},
}

@misc{LurieHA,
  author       = {Lurie, Jacob},
  title        = {Higher Algebra},
  year         = {2017},
  month        = sep,
  date         = {2017-09-18},
    howpublished           = {\url{https://www.math.ias.edu/~lurie/papers/HA.pdf}},
  urldate      = {2026-07-03}
}

@book {LMB,
    AUTHOR = {Laumon, G\'erard and Moret-Bailly, Laurent},
     TITLE = {Champs alg\'ebriques},
    SERIES = {Ergebnisse der Mathematik und ihrer Grenzgebiete. 3. Folge. A
              Series of Modern Surveys in Mathematics [Results in
              Mathematics and Related Areas. 3rd Series. A Series of Modern
              Surveys in Mathematics]},
    VOLUME = {39},
 PUBLISHER = {Springer-Verlag, Berlin},
      YEAR = {2000},
     PAGES = {xii+208},
      ISBN = {3-540-65761-4},
   MRCLASS = {14A20 (14D20)},
  MRNUMBER = {1771927},
MRREVIEWER = {Dan\ Edidin},
}

@book {kollar_singularities_MMP,
    AUTHOR = {Koll\'ar, J\'anos},
     TITLE = {Singularities of the minimal model program},
    SERIES = {Cambridge Tracts in Mathematics},
    VOLUME = {200},
      NOTE = {With a collaboration of S\'andor Kov\'acs},
 PUBLISHER = {Cambridge University Press, Cambridge},
      YEAR = {2013},
     PAGES = {x+370},
      ISBN = {978-1-107-03534-8},
   MRCLASS = {14E30 (14B05)},
  MRNUMBER = {3057950},
MRREVIEWER = {Tommaso\ De Fernex},
       DOI = {10.1017/CBO9781139547895},
       URL = {https://doi.org/10.1017/CBO9781139547895},
}

@book {kollar_families,
    AUTHOR = {Koll\'ar, J\'anos},
     TITLE = {Families of varieties of general type},
    SERIES = {Cambridge Tracts in Mathematics},
    VOLUME = {231},
      NOTE = {With the collaboration of Klaus Altmann and S\'andor J.
              Kov\'acs},
 PUBLISHER = {Cambridge University Press, Cambridge},
      YEAR = {2023},
     PAGES = {xviii+471},
      ISBN = {978-1-009-34610-8},
   MRCLASS = {14J10 (14D20 14E30 14J29)},
  MRNUMBER = {4566297},
MRREVIEWER = {Chenyang\ Xu},
}

@article {Pat12,
    AUTHOR = {Patakfalvi, Zsolt},
     TITLE = {Viehweg's hyperbolicity conjecture is true over compact bases},
   JOURNAL = {Adv. Math.},
  FJOURNAL = {Advances in Mathematics},
    VOLUME = {229},
      YEAR = {2012},
    NUMBER = {3},
     PAGES = {1640--1642},
      ISSN = {0001-8708,1090-2082},
   MRCLASS = {14J10 (14E05)},
  MRNUMBER = {2871152},
MRREVIEWER = {Atsushi\ Moriwaki},
       DOI = {10.1016/j.aim.2011.12.013},
       URL = {https://doi.org/10.1016/j.aim.2011.12.013},
}

@article {KK08,
    AUTHOR = {Kebekus, Stefan and Kov\'acs, S\'andor J.},
     TITLE = {Families of varieties of general type over compact bases},
   JOURNAL = {Adv. Math.},
  FJOURNAL = {Advances in Mathematics},
    VOLUME = {218},
      YEAR = {2008},
    NUMBER = {3},
     PAGES = {649--652},
      ISSN = {0001-8708,1090-2082},
   MRCLASS = {14J10 (14E30)},
  MRNUMBER = {2414316},
MRREVIEWER = {Jaros\l aw\ A.\ Wi\'sniewski},
       DOI = {10.1016/j.aim.2008.01.005},
       URL = {https://doi.org/10.1016/j.aim.2008.01.005},
}

@article {BHPS13,
    AUTHOR = {Bhatt, Bhargav and Ho, Wei and Patakfalvi, Zsolt and Schnell,
              Christian},
     TITLE = {Moduli of products of stable varieties},
   JOURNAL = {Compos. Math.},
  FJOURNAL = {Compositio Mathematica},
    VOLUME = {149},
      YEAR = {2013},
    NUMBER = {12},
     PAGES = {2036--2070},
      ISSN = {0010-437X,1570-5846},
   MRCLASS = {14D22 (14D23 14J15 14J17)},
  MRNUMBER = {3143705},
MRREVIEWER = {Wenfei\ Liu},
       DOI = {10.1112/S0010437X13007288},
       URL = {https://doi.org/10.1112/S0010437X13007288},
}

@article {KP17proj,
    AUTHOR = {Kov\'acs, S\'andor J. and Patakfalvi, Zsolt},
     TITLE = {Projectivity of the moduli space of stable log-varieties and
              subadditivity of log-{K}odaira dimension},
   JOURNAL = {J. Amer. Math. Soc.},
  FJOURNAL = {Journal of the American Mathematical Society},
    VOLUME = {30},
      YEAR = {2017},
    NUMBER = {4},
     PAGES = {959--1021},
      ISSN = {0894-0347,1088-6834},
   MRCLASS = {14J10},
  MRNUMBER = {3671934},
MRREVIEWER = {Atsushi\ Moriwaki},
       DOI = {10.1090/jams/871},
       URL = {https://doi.org/10.1090/jams/871},
}

@article {EV90effective,
    AUTHOR = {Esnault, H\'el\`ene and Viehweg, Eckart},
     TITLE = {Effective bounds for semipositive sheaves and for the height
              of points on curves over complex function fields},
      NOTE = {Algebraic geometry (Berlin, 1988)},
   JOURNAL = {Compositio Math.},
  FJOURNAL = {Compositio Mathematica},
    VOLUME = {76},
      YEAR = {1990},
    NUMBER = {1-2},
     PAGES = {69--85},
      ISSN = {0010-437X,1570-5846},
   MRCLASS = {14H10 (11G35 14F10 14J25)},
  MRNUMBER = {1078858},
MRREVIEWER = {Bruce\ Hunt},
       URL = {http://www.numdam.org/item?id=CM_1990__76_1-2_69_0},
}

@book {EVvan92,
    AUTHOR = {Esnault, H\'el\`ene and Viehweg, Eckart},
     TITLE = {Lectures on vanishing theorems},
    SERIES = {DMV Seminar},
    VOLUME = {20},
 PUBLISHER = {Birkh\"auser Verlag, Basel},
      YEAR = {1992},
     PAGES = {vi+164},
      ISBN = {3-7643-2822-3},
   MRCLASS = {14F17 (14F40 32L10 32L20)},
  MRNUMBER = {1193913},
MRREVIEWER = {Marko\ Roczen},
       DOI = {10.1007/978-3-0348-8600-0},
       URL = {https://doi.org/10.1007/978-3-0348-8600-0},
}

@article {PW16,
    AUTHOR = {Popa, Mihnea and Wu, Lei},
     TITLE = {Weak positivity for {H}odge modules},
   JOURNAL = {Math. Res. Lett.},
  FJOURNAL = {Mathematical Research Letters},
    VOLUME = {23},
      YEAR = {2016},
    NUMBER = {4},
     PAGES = {1139--1155},
      ISSN = {1073-2780,1945-001X},
   MRCLASS = {14C30 (14D07 14F10 32C38 32G20)},
  MRNUMBER = {3554504},
MRREVIEWER = {Fumio\ Hazama},
       DOI = {10.4310/MRL.2016.v23.n4.a8},
       URL = {https://doi.org/10.4310/MRL.2016.v23.n4.a8},
}

@article {CPFol19,
    AUTHOR = {Campana, Fr\'ed\'eric and P\u{a}un, Mihai},
     TITLE = {Foliations with positive slopes and birational stability of
              orbifold cotangent bundles},
   JOURNAL = {Publ. Math. Inst. Hautes \'Etudes Sci.},
  FJOURNAL = {Publications Math\'ematiques. Institut de Hautes \'Etudes
              Scientifiques},
    VOLUME = {129},
      YEAR = {2019},
     PAGES = {1--49},
      ISSN = {0073-8301,1618-1913},
   MRCLASS = {37F75 (14E30 14M22 32S65)},
  MRNUMBER = {3949026},
MRREVIEWER = {Carla\ Novelli},
       DOI = {10.1007/s10240-019-00105-w},
       URL = {https://doi.org/10.1007/s10240-019-00105-w},
}

@incollection {BBD82,
    AUTHOR = {Beilinson, A. A. and Bernstein, J. and Deligne, P.},
     TITLE = {Faisceaux pervers},
 BOOKTITLE = {Analysis and topology on singular spaces, {I} ({L}uminy,
              1981)},
    SERIES = {Ast\'erisque},
    VOLUME = {100},
     PAGES = {5--171},
 PUBLISHER = {Soc. Math. France, Paris},
      YEAR = {1982},
   MRCLASS = {32C38},
  MRNUMBER = {751966},
MRREVIEWER = {Zoghman\ Mebkhout},
}

@book {GW20,
    AUTHOR = {G\"ortz, Ulrich and Wedhorn, Torsten},
     TITLE = {Algebraic geometry {I}. {S}chemes---with examples and
              exercises},
    SERIES = {Springer Studium Mathematik---Master},
   EDITION = {Second},
 PUBLISHER = {Springer Spektrum, Wiesbaden},
      YEAR = {[2020] \copyright 2020},
     PAGES = {vii+625},
      ISBN = {978-3-658-30732-5; 978-3-658-30733-2},
   MRCLASS = {14-01},
  MRNUMBER = {4225278},
       DOI = {10.1007/978-3-658-30733-2},
       URL = {https://doi.org/10.1007/978-3-658-30733-2},
}

@book {matsumura,
    AUTHOR = {Matsumura, Hideyuki},
     TITLE = {Commutative ring theory},
    SERIES = {Cambridge Studies in Advanced Mathematics},
    VOLUME = {8},
   EDITION = {Second},
      NOTE = {Translated from the Japanese by M. Reid},
 PUBLISHER = {Cambridge University Press, Cambridge},
      YEAR = {1989},
     PAGES = {xiv+320},
      ISBN = {0-521-36764-6},
   MRCLASS = {13-01},
  MRNUMBER = {1011461},
}

@article {Hart80,
    AUTHOR = {Hartshorne, Robin},
     TITLE = {Stable reflexive sheaves},
   JOURNAL = {Math. Ann.},
  FJOURNAL = {Mathematische Annalen},
    VOLUME = {254},
      YEAR = {1980},
    NUMBER = {2},
     PAGES = {121--176},
      ISSN = {0025-5831,1432-1807},
   MRCLASS = {14F05 (14D22)},
  MRNUMBER = {597077},
MRREVIEWER = {Klaus\ Hulek},
       DOI = {10.1007/BF01467074},
       URL = {https://doi.org/10.1007/BF01467074},
}

@article{tubach_2024, title={Mixed {H}odge modules on stacks}, volume={13}, DOI={10.1017/fms.2025.10122}, journal={Forum of Mathematics, Sigma}, author={Tubach, Swann}, year={2025}, pages={e175}}

@article {RS20Intersection,
    AUTHOR = {Richarz, Timo and Scholbach, Jakob},
     TITLE = {The intersection motive of the moduli stack of shtukas},
   JOURNAL = {Forum Math. Sigma},
  FJOURNAL = {Forum of Mathematics. Sigma},
    VOLUME = {8},
      YEAR = {2020},
     PAGES = {Paper No. e8, 99},
      ISSN = {2050-5094},
   MRCLASS = {20G05 (14D23 14F42 19E15)},
  MRNUMBER = {4061978},
MRREVIEWER = {Ilya\ Karzhemanov},
       DOI = {10.1017/fms.2019.32},
       URL = {https://doi.org/10.1017/fms.2019.32},
}

@article {aoki23,
    AUTHOR = {Aoki, Ko},
     TITLE = {Tensor triangular geometry of filtered objects and sheaves},
   JOURNAL = {Math. Z.},
  FJOURNAL = {Mathematische Zeitschrift},
    VOLUME = {303},
      YEAR = {2023},
    NUMBER = {3},
     PAGES = {Paper No. 62, 27},
      ISSN = {0025-5874,1432-1823},
   MRCLASS = {18G80 (18F20 18F70 18N60)},
  MRNUMBER = {4549105},
MRREVIEWER = {Maosong\ Xiang},
       DOI = {10.1007/s00209-023-03210-z},
       URL = {https://doi.org/10.1007/s00209-023-03210-z},
}

@misc{paulin13,
      title={The {R}iemann-{H}ilbert Correspondence for Algebraic Stacks}, 
      author={Alexander Paulin},
      year={2013},
      eprint={1308.5890},
      archivePrefix={arXiv},
      primaryClass={math.AG},
      url={https://arxiv.org/abs/1308.5890}, 
}

@incollection {CMW18,
    AUTHOR = {Casalaina-Martin, Sebastian and Wise, Jonathan},
     TITLE = {An introduction to moduli stacks, with a view towards {H}iggs
              bundles on algebraic curves},
 BOOKTITLE = {The geometry, topology and physics of moduli spaces of {H}iggs
              bundles},
    SERIES = {Lect. Notes Ser. Inst. Math. Sci. Natl. Univ. Singap.},
    VOLUME = {36},
     PAGES = {199--399},
 PUBLISHER = {World Sci. Publ., Hackensack, NJ},
      YEAR = {2018},
      ISBN = {978-981-3229-08-2},
   MRCLASS = {14D23 (14H60 18F99)},
  MRNUMBER = {3837871},
MRREVIEWER = {Arvid\ Siqveland},
}

@book{Deligne1970,
  author    = {Pierre Deligne},
  title     = {Équations différentielles à points singuliers réguliers},
  series    = {Lecture Notes in Mathematics},
  volume    = {163},
  publisher = {Springer-Verlag},
  year      = {1970},
}

@article{Saito1990,
  author    = {Morihiko Saito},
  title     = {Mixed Hodge Modules},
  journal   = {Publications of the Research Institute for Mathematical Sciences},
  volume    = {26},
  number    = {2},
  year      = {1990},
  pages     = {221--333},
}

@article {PTW18,
    AUTHOR = {Popa, Mihnea and Taji, Behrouz and Wu, Lei},
     TITLE = {Brody hyperbolicity of base spaces of certain families of
              varieties},
   JOURNAL = {Algebra Number Theory},
  FJOURNAL = {Algebra \& Number Theory},
    VOLUME = {13},
      YEAR = {2019},
    NUMBER = {9},
     PAGES = {2205--2242},
      ISSN = {1937-0652,1944-7833},
   MRCLASS = {14D07 (14D23 14E30 14J10 14J15 14J29)},
  MRNUMBER = {4039502},
MRREVIEWER = {Kimio\ Miyajima},
       DOI = {10.2140/ant.2019.13.2205},
       URL = {https://doi.org/10.2140/ant.2019.13.2205},
}

@misc{MHM-project,
  author = {Christian Schnell and Claude Sabbah},
  title = {The MHM Project},
  note = {Lecture notes, unpublished manuscript},
  year = {2017},
  url = {https://www.math.stonybrook.edu/~cschnell/mhm-project.pdf}
}

@article{PopaSchnell2013,
  author    = {Mihnea Popa and Christian Schnell},
  title     = {Generic vanishing theory via mixed Hodge modules},
  journal   = {Forum of Mathematics, Sigma},
  volume    = {1},
  pages     = {e1, 60pp},
  year      = {2013},
  doi       = {10.1017/fms.2013.1},
  mrnumber  = {3090229}
}

@article {fujita_fiber_spaces,
    AUTHOR = {Fujita, Takao},
     TITLE = {On {K}\"ahler fiber spaces over curves},
   JOURNAL = {J. Math. Soc. Japan},
  FJOURNAL = {Journal of the Mathematical Society of Japan},
    VOLUME = {30},
      YEAR = {1978},
    NUMBER = {4},
     PAGES = {779--794},
      ISSN = {0025-5645,1881-1167},
   MRCLASS = {32G13 (14E99 32J99)},
  MRNUMBER = {513085},
       DOI = {10.2969/jmsj/03040779},
       URL = {https://doi.org/10.2969/jmsj/03040779},
}

@article {FFS14Remarks,
    AUTHOR = {Fujino, Osamu and Fujisawa, Taro and Saito, Morihiko},
     TITLE = {Some remarks on the semipositivity theorems},
   JOURNAL = {Publ. Res. Inst. Math. Sci.},
  FJOURNAL = {Publications of the Research Institute for Mathematical
              Sciences},
    VOLUME = {50},
      YEAR = {2014},
    NUMBER = {1},
     PAGES = {85--112},
      ISSN = {0034-5318,1663-4926},
   MRCLASS = {14D07 (32G20)},
  MRNUMBER = {3167580},
MRREVIEWER = {Christian\ Schnell},
       DOI = {10.4171/PRIMS/125},
       URL = {https://doi.org/10.4171/PRIMS/125},
}

@article {FF14variationsHMS,
    AUTHOR = {Fujino, Osamu and Fujisawa, Taro},
     TITLE = {Variations of mixed {H}odge structure and semipositivity
              theorems},
   JOURNAL = {Publ. Res. Inst. Math. Sci.},
  FJOURNAL = {Publications of the Research Institute for Mathematical
              Sciences},
    VOLUME = {50},
      YEAR = {2014},
    NUMBER = {4},
     PAGES = {589--661},
      ISSN = {0034-5318,1663-4926},
   MRCLASS = {14D07 (14C30 32G20)},
  MRNUMBER = {3273305},
MRREVIEWER = {Jan\ Nagel},
       DOI = {10.4171/PRIMS/145},
       URL = {https://doi.org/10.4171/PRIMS/145},
}

@misc{AlperStacksBook,
      title={Stacks and Moduli: working draft}, 
      author={Jarod Alper},
      year={2025},
      url={https://sites.math.washington.edu/~jarod/moduli.pdf},
       howpublished = {\url{https://sites.math.washington.edu/~jarod/moduli.pdf}},
      note = {Draft of book},  
}

@article {Kawamata,
    AUTHOR = {Kawamata, Yujiro},
     TITLE = {Characterization of abelian varieties},
   JOURNAL = {Compositio Math.},
  FJOURNAL = {Compositio Mathematica},
    VOLUME = {43},
      YEAR = {1981},
    NUMBER = {2},
     PAGES = {253--276},
      ISSN = {0010-437X,1570-5846},
   MRCLASS = {14J10 (32J15)},
  MRNUMBER = {622451},
MRREVIEWER = {Daniel\ Comenetz},
       URL = {http://www.numdam.org/item?id=CM_1981__43_2_253_0},
}

@incollection {kollar_sub,
    AUTHOR = {Koll\'ar, J\'anos},
     TITLE = {Subadditivity of the {K}odaira dimension: fibers of general
              type},
 BOOKTITLE = {Algebraic geometry, {S}endai, 1985},
    SERIES = {Adv. Stud. Pure Math.},
    VOLUME = {10},
     PAGES = {361--398},
 PUBLISHER = {North-Holland, Amsterdam},
      YEAR = {1987},
      ISBN = {0-444-70313-6},
   MRCLASS = {14J10 (14C30 14D20)},
  MRNUMBER = {946244},
MRREVIEWER = {Yujiro\ Kawamata},
       DOI = {10.2969/aspm/01010361},
       URL = {https://doi.org/10.2969/aspm/01010361},
}

@article{Saito,
  author  = {Saito, Morihiko},
  title   = {Modules de Hodge polarisables},
  journal = {Publ. Res. Inst. Math. Sci.},
  volume  = {24},
  year    = {1988},
  number  = {6},
  pages   = {849--995},
}

@article {Zuo00negativity,
    AUTHOR = {Zuo, Kang},
     TITLE = {On the negativity of kernels of {K}odaira-{S}pencer maps on
              {H}odge bundles and applications},
      NOTE = {Kodaira's issue},
   JOURNAL = {Asian J. Math.},
  FJOURNAL = {Asian Journal of Mathematics},
    VOLUME = {4},
      YEAR = {2000},
    NUMBER = {1},
     PAGES = {279--301},
      ISSN = {1093-6106,1945-0036},
   MRCLASS = {32G20 (14D07 32J25)},
  MRNUMBER = {1803724},
MRREVIEWER = {Jan\ Nagel},
       DOI = {10.4310/AJM.2000.v4.n1.a17},
       URL = {https://doi.org/10.4310/AJM.2000.v4.n1.a17},
}

@article {VZ01isotriv,
    AUTHOR = {Viehweg, Eckart and Zuo, Kang},
     TITLE = {On the isotriviality of families of projective manifolds over
              curves},
   JOURNAL = {J. Algebraic Geom.},
  FJOURNAL = {Journal of Algebraic Geometry},
    VOLUME = {10},
      YEAR = {2001},
    NUMBER = {4},
     PAGES = {781--799},
      ISSN = {1056-3911,1534-7486},
   MRCLASS = {14D06 (14J10)},
  MRNUMBER = {1838979},
MRREVIEWER = {S\'andor\ J.\ Kov\'acs},
}

@incollection{VZ02,
  author = {Viehweg, Eckart and Zuo, Kang},
  title = {Base spaces of non-isotrivial families of smooth minimal models},
  booktitle = {Complex Geometry (Göttingen, 2000)},
  publisher = {Springer},
  address = {Berlin},
  year = {2002},
  pages = {279--328},
  mrnumber = {1922109}
}

@article{VZ03,
  author = {Viehweg, Eckart and Zuo, Kang},
  title = {On the Brody hyperbolicity of moduli spaces for canonically polarized manifolds},
  journal = {Duke Mathematical Journal},
  volume = {118},
  number = {1},
  pages = {103--150},
  year = {2003},
  mrnumber = {1978884}
}

@article {hassett_weighted,
    AUTHOR = {Hassett, Brendan},
     TITLE = {Moduli spaces of weighted pointed stable curves},
   JOURNAL = {Adv. Math.},
  FJOURNAL = {Advances in Mathematics},
    VOLUME = {173},
      YEAR = {2003},
    NUMBER = {2},
     PAGES = {316--352},
      ISSN = {0001-8708,1090-2082},
   MRCLASS = {14H10 (14D22 14E30)},
  MRNUMBER = {1957831},
MRREVIEWER = {Ivan\ S.\ Kausz},
       DOI = {10.1016/S0001-8708(02)00058-0},
       URL = {https://doi.org/10.1016/S0001-8708(02)00058-0},
}
\end{document}